\documentclass[11pt,reqno]{amsart}
\usepackage[english]{babel}
\usepackage[a4paper,twoside]{geometry}
\usepackage{microtype}
\usepackage[shortlabels]{enumitem}

\usepackage{csquotes}
\usepackage{longtable}
\usepackage{bm}

\usepackage{xspace}
\patchcmd{\subsection}{-.5em}{.5em}{}{}
\usepackage[hyperfootnotes=false]{hyperref} 
\usepackage{color}
\hypersetup{
	colorlinks = true,
	linkcolor = {blue},
	urlcolor = {blue},
	citecolor = {blue}
}

\makeatletter
\newcommand{\mysubsubsection}[1]{\subsubsection*{\bfseries #1}}
\renewcommand{\tocsection}[3]{
  \indentlabel{\@ifnotempty{#2}{\ignorespaces#1 #2\quad}}\bfseries#3}
\renewcommand{\tocsubsection}[3]{
  \indentlabel{\@ifnotempty{#2}{\ignorespaces#1 #2\quad}}#3}

\newcommand\@dotsep{4.5}
\def\@tocline#1#2#3#4#5#6#7{\relax
  \ifnum #1>\c@tocdepth
  \else
    \par \addpenalty\@secpenalty\addvspace{#2}
    \begingroup \hyphenpenalty\@M
    \@ifempty{#4}{
      \@tempdima\csname r@tocindent\number#1\endcsname\relax
    }{
      \@tempdima#4\relax
    }
    \parindent\z@ \leftskip#3\relax \advance\leftskip\@tempdima\relax
    \rightskip\@pnumwidth plus1em \parfillskip-\@pnumwidth
    #5\leavevmode\hskip-\@tempdima{#6}\nobreak
    \leaders\hbox{$\m@th\mkern \@dotsep mu\hbox{.}\mkern \@dotsep mu$}\hfill
    \nobreak
    \hbox to\@pnumwidth{\@tocpagenum{\ifnum#1=1\bfseries\fi#7}}\par
    \nobreak
    \endgroup
  \fi}
\AtBeginDocument{
\expandafter\renewcommand\csname r@tocindent0\endcsname{0pt}
}
\def\l@subsection{\@tocline{2}{0pt}{2.5pc}{5pc}{}}
\makeatother

\makeatletter
\renewcommand\part{%
   \if@noskipsec \leavevmode \fi
   \par
   \addvspace{4ex}%
   \@afterindentfalse
   \secdef\@part\@spart}

\def\@part[#1]#2{%
    \ifnum \c@secnumdepth >\m@ne
      \refstepcounter{part}%
      \addcontentsline{toc}{part}{Part \thepart.\hspace{0.5em}#1 \vspace{0.3em}}%
    \else
      \addcontentsline{toc}{part}{#1}%
    \fi
    {\parindent \z@ \raggedright
     \interlinepenalty \@M
     \normalfont
     \ifnum \c@secnumdepth >\m@ne
       \Large\bfseries \partname\nobreakspace\thepart.
       \nobreak
     \fi
     \Large \bfseries #2%
     \par}%
    \nobreak
    \vskip 3ex
    \@afterheading}
\def\@spart#1{%
    {\parindent \z@ \raggedright
     \interlinepenalty \@M
     \normalfont
     \huge \bfseries #1\par}%
     \nobreak
     \vskip 3ex
     \@afterheading}
\makeatother
\renewcommand{\thepart}{\Roman{part}}

\usepackage{amsmath}
\usepackage{amsthm}
\usepackage{amssymb}
\usepackage{mathrsfs} 
\usepackage{eucal}
\usepackage{mathtools}

\usepackage{graphicx}
\usepackage{tikz}
\usetikzlibrary{cd}

\newcounter{results}[section]

\theoremstyle{plain}
\newtheorem{theorem}[results]{Theorem}
\newtheorem{lemma}[results]{Lemma}
\newtheorem{proposition}[results]{Proposition}
\newtheorem{corollary}[results]{Corollary}
\newtheorem{definition}[results]{Definition}

\theoremstyle{remark}
\newtheorem{remark}[results]{Remark}
\newtheorem{example}[results]{Example}

\theoremstyle{definition}

\numberwithin{equation}{section}

\newcommand{\R}{\ensuremath{\mathbb R}} 
\newcommand{\N}{\ensuremath{\mathbb N}} 

\newcommand{\eps}{\ensuremath{\varepsilon}} 
\newcommand{\restr}[1]{\lower3pt\hbox{$|_{#1}$}} 

\newcommand{\dom}{\ensuremath{\mathrm{D}}}
\newcommand{\domf}[1]{\ensuremath{\mathrm{D}_f(#1)}}
\newcommand{\pdom}[2]{\ensuremath{\mathrm{D}_{#1}(#2)}}

\renewcommand{\P}{\mathbb P} 
\newcommand{\E}{\mathbb E} 

\newcommand{\upds}{{\frac{\d}{\d s}}^{\kern-3pt +}} 
\newcommand{\updt}{{\frac{\d}{\d t}}^{\kern-3pt +}} 
\newcommand{\lodt}{{\frac{\d}{\d t}}_{\kern-1pt +}} 

\newcommand{\la}{\langle} 
\newcommand{\ra}{\rangle} 

\newcommand{\de}{\ensuremath{\,\mathrm d}} 
\renewcommand{\d}{\mathrm d}  

\newcommand{\intt}[1]{\operatorname{int}\left(#1\right)} 

\newcommand{\weakto}{\ensuremath{\rightharpoonup}} 

\newcommand{\bb}{\bm b} 
\newcommand{\ff}{\bm f}

\newcommand{\Bb}{\bm B} 
\newcommand{\Gg}{\bm G} 
\newcommand{\Mm}{\bm M} 
\newcommand{\ggamma}{\bm\gamma} 
\newcommand{\rrho}{\bm\rho} 
\newcommand{\ii}{\bm i} 
\newcommand{\jJ}{\bm J} 
\newcommand{\jj}{\bm j} 
\renewcommand{\ss}{\bm s} 
\newcommand{\mmu}{\bm\mu} 
\newcommand{\rr}{\bm r} 
\newcommand{\ssigma}{\bm\sigma} 
\newcommand{\ttheta}{\bm\vartheta} 
\newcommand{\Ttheta}{\bm\Theta} 
\newcommand{\iiota}[1]{\iotaT_{#1}}
\newcommand{\uu}{\bm u} 
\newcommand{\vv}{\bm v} 
\newcommand{\ww}{\bm w} 
\newcommand{\xx}{{\bm x}} 
\newcommand{\yy}{{\bm y}} 
\newcommand{\zz}{{\bm z}} 
\newcommand{\pfrF}[1]{\bm B_{#1}} 

\renewcommand{\H}{\mathcal H} 
\newcommand{\sfv}{\mathsf v} 
\newcommand{\sfx}{\mathsf x} 
\newcommand{\sfy}{\mathsf y} 
\newcommand{\X}{\mathsf X} 
\newcommand{\Y}{\mathsf Y} 

\newcommand{\interval}{\mathcal I} 
\newcommand{\cN}{{\mathfrak N}} 
\newcommand{\cB}{{\mathcal B}} 
\newcommand{\IIN}{{\mathscr I_N}} 

\newcommand{\EVI}{{\rm EVI}\xspace} 
\newcommand{\wEVI}{{\rm EVI}\xspace} 
\newcommand{\MPVF}{{\rm MPVF}\xspace} 
\newcommand{\PVF}{{\rm PVF}\xspace} 

\newcommand{\leb}{\mathscr{L}} 
\newcommand{\relcP}[3]{\prob_{#1}(#3|#2)} 
\newcommand{\prob}{\ensuremath{\mathcal{P}}} 
\DeclareMathOperator{\supp}{supp} 

\newcommand{\CondGamma}[3]{\Gamma({#2},{#3}|#1)} 
\DeclareMathOperator{\Tan}{Tan} 

\newcommand{\sqm}[1]{\mathsf m_2^2(#1)} 
\newcommand{\rsqm}[1]{\mathsf m_2(#1)} 
\newcommand{\iotaT}{\iota^2}

\newcommand{\TRd}{\mathsf{T}\R^d} 
\newcommand{\TX}{\mathsf {T\kern-1.5pt X}} 
\newcommand{\TY}{\mathsf {T\kern-1.5pt Y}} 

\newcommand{\rmC}{\mathrm C} 
\DeclareMathOperator{\Cyl}{Cyl} 
\newcommand{\Lip}{\mathrm {Lip}} 
\newcommand{\testsw}[1]{\rmC^{sw}_{2}(#1)} 
\newcommand{\sfO}{{\mathsf O}} 
\newcommand{\cO}[1]{\mathcal O_{#1}} 
\newcommand{\Sym}[1]{{\mathrm{Sym}(#1)}} 

\newcommand{\scalprod}[2]{\ensuremath{\langle #1, #2\rangle}} 
\newcommand{\bram}[2]{\ensuremath{\left [ #1, #2\right ]_{r}}} 
\newcommand{\brap}[2]{\ensuremath{\left [ #1, #2\right ]_{l}}} 

\newcommand{\directionalm}[3]{[#1,#2]_{r,#3}} 
\newcommand{\directionalp}[3]{[#1,#2]_{l,#3}} 

\newcommand{\frF}{{\bm{\mathrm F}}}
\newcommand{\frG}{{\bm{\mathrm G}}}
\newcommand{\fF}{{\bm B}}
\newcommand{\core}{\rmC} 
\newcommand{\maxim}[2]{{\hat{#1}_{#2}}} 
\newcommand{\clconv}[1]{\overline{\operatorname{co}}\left(#1\right)} 
\newcommand{\clo}[1]{\operatorname{cl}(#1)} 
\newcommand{\bri}[1]{\operatorname{bar}\left (#1\right )}
\newcommand{\maps}[1]{\operatorname{map}\left (#1\right )}

\newcommand{\bry}[1]{\boldsymbol b_{#1}} 

\newcommand{\conv}[1]{\operatorname{co}(#1)} 

\newcommand{\rmS}{\mathrm S} 

\newcommand{\dir}[1]{\operatorname{dir}(#1)}
\newcommand{\dN}{N}
\newcommand{\udN}{\dN^{-1}}
\newcommand{\symg}[1]{{\mathrm{Sym}(#1)}} 
\newcommand{\DDom}[1]{\mathcal D_{#1}}
\newcommand{\newDDom}[1]{\mathcal D_{#1}}
\newcommand{\newODDom}[1]{{\mathcal C_{#1}}}

\newcommand{\resolvent}[1]{{\jJ_{#1}}}

\newcommand{\mmo}{{\boldsymbol B}}

\newcommand{\mm}{\mathfrak m}
\newcommand{\Sgp}{\bm S}

\newcommand{\cY}{\mathcal Y}

\newcommand{\cH}{\mathcal X}
\newcommand{\Sp}[1]{\mathcal{S}\left(#1\right)}

\newcommand{\menolambda}{-\lambda}
\newcommand{\piulambda}{\lambda}
\newcommand{\piulambdapar}{\lambda}
\newcommand{\menolambdapar}{(-\lambda)}

\title[A Lagrangian approach to totally dissipative evolutions in Wasserstein spaces]{A Lagrangian approach to totally dissipative evolutions in Wasserstein spaces}

\author{Giulia Cavagnari}
\address{Giulia Cavagnari: Politecnico di Milano, Dipartimento di Matematica, Piazza Leonardo Da Vinci 32, 20133 Milano (Italy)}
\email{giulia.cavagnari@polimi.it}

\author{Giuseppe Savar\'e}
\address{Giuseppe Savar\'e: Bocconi University,
  Department of Decision Sciences and BIDSA, Via Roentgen 1, 20136 Milano (Italy)}
\email{giuseppe.savare@unibocconi.it}

\author{Giacomo Enrico Sodini}
\address{Giacomo Enrico Sodini: Institut für Mathematik - Fakultät für Mathematik - Universität Wien, Oskar-Morgenstern-Platz 1, 1090 Wien (Austria)}
\email{giacomo.sodini@univie.ac.at}

\subjclass{Primary: 34A06, 47B44, 49Q22; Secondary: 34A12, 34A60, 28D05}
 \keywords{ Measure differential equations/inclusions in Wasserstein spaces, probability vector fields, dissipative operators, measure-preserving isomorphisms, geodesically convex functionals, JKO scheme.}

\begin{document}

\begin{abstract}
We 
introduce 
and study the class of 
\emph{totally dissipative}
multivalued probability vector fields (\MPVF) $\frF$ on the Wasserstein space
$(\prob_2(\X),W_2)$ of Euclidean or Hilbertian probability measures.
We show that such class of {\MPVF}s is 
in one to one correspondence with law-invariant 
dissipative operators in a Hilbert space $L^2(\Omega,\cB,\P;\X)$ 
of random variables, preserving
a natural maximality property.
This allows us to import in the Wasserstein framework many of the powerful tools from the theory of 
maximal dissipative operators in Hilbert spaces, 
deriving existence, uniqueness, stability, and approximation results
for the flow generated by a maximal totally dissipative \MPVF
and the equivalence of its Eulerian and Lagrangian characterizations.

We will show that 
demicontinuous single-valued probability vector fields
satisfying a metric dissipativity condition as in \cite{CSS} are in fact totally dissipative. 
Starting from a sufficiently rich set of discrete measures, we 
will also show how to recover a unique maximal totally dissipative version of a \MPVF, proving that its flow
provides a general mean field characterization of the asymptotic limits of the corresponding family of discrete particle systems.
Such an approach also reveals new interesting structural properties 
for gradient flows of displacement convex functionals 
with a core of discrete measures dense in energy.
\end{abstract}
\thispagestyle{empty}
\dedicatory{Dedicated to Ricardo H.~Nochetto on the occasion of his $70^{th}$ birthday.}
\maketitle
\vspace{-1cm}
\tableofcontents

\section{Introduction}
The theory of evolutions of probability measures is experiencing an ever growing interest from the scientific community. On one side, this is justified by its numerous applications in modeling real-life dynamics: social dynamics, crowd dynamics for multi-agent systems, opinion formation, evolution of financial markets just to name a few. We refer the reader to the recent survey \cite{Piccoli-soa23} for a more complete overview of the many applications of control theory for multi-agent systems, i.e.~large systems of interacting particles/individuals.
On the other side, dealing with mean-field evolutions, expecially in the framework of optimal control theory in Wasserstein spaces \cite{FLOS,CLOS,CapuaniM-lift}, provides interesting insights into mathematical research. We mention for instance the recent contributions \cite{Averb2022,bonnet2020mean,BF2023} for the study of a well-posedness theory for differential inclusions in Wasserstein spaces,
\cite{AK-pmp,BF-pmp,pogodaev2016} for
necessary conditions for optimality in the form of a Pontryagin maximum principle,
the references \cite{AOZ,BadF-hjb,CMPrainbow,Jimenez,JMQ} for the study of Hamilton-Jacobi-Bellman equations in this framework. Finally, other contributions devoted to the development of a
viability theory for control problems in the space of probability measures are e.g.~\cite{AMQ2021,BF-viability,BadF-viability,CMQ}.

In addition to these studies, we have all the applications of the theory of gradient flows  in Wasserstein spaces \cite{ags}
which are impossible to summarize here even briefly.
In particular, in the case of geodesically convex (resp.~$\lambda$-convex) functionals \cite{McCann97}, the geometric viewpoint and the variational approach introduced by \cite{Otto01,JKO98} have been 
extremely powerful to construct a semigroup of contractions (resp.~Lipschitz maps) \cite{ags}, which provides a robust background for various applications.
\medskip

In the present paper, we continue the project, started in \cite{CSS}, to extend the theory beyond gradient flows. 
Our aim is to investigate the  evolution semigroups generated by a $\lambda$-dissipative multivalued probability vector field (in short, \MPVF)
$\frF$ in the Wasserstein space 
$(\mathcal P_2(\X),W_2)$. The space $\mathcal P_2(\X)$ denotes the set of Borel probability measures with finite quadratic moment on a separable Hilbert space $\X$. 
The geometric notion of dissipativity is
intimately related to the $L^2$-Kantorovich-Rubinstein-Wasserstein distance $W_2$ between two measures $\mu_0,\mu_1\in \mathcal P_2(\X)$, which can be expressed 
by the solution of the Optimal Transport problem
\begin{equation}
    \label{eq:Wass-intro}
    W_2^2(\mu_0,\mu_1):=
    \min\Big\{\int_{\X^2} |x_0-x_1|^2\,\d\mmu(x_0,x_1):
    \mmu\in \Gamma(\mu_0,\mu_1)\Big\},
\end{equation}
where $\Gamma(\mu_0,\mu_1)$
denotes the set of couplings $\mmu\in \mathcal P_2(\X\times \X)$ with
marginals $\mu_0$ and $\mu_1$.
It is well known that the set $\Gamma_o(\mu_0,\mu_1)$ where the 
minimum in \eqref{eq:Wass-intro} is attained
is a nonempty compact and convex subset of $\Gamma(\mu_0,\mu_1).$

We refer to \cite{CSS} for a detailed discussion of the various approaches to such kind of problems; let us only mention here the Cauchy-Lipschitz approach via vector fields \cite{bonnet2020mean,BF2023}, the barycentric approach in \cite{Piccoli_2019, Piccoli_MDI,Camilli_MDE} and the variational approach to characterize limit solutions of an Explicit Euler Scheme for evolution equations driven by dissipative {\MPVF}s in \cite{CSS}.

Let us just recall here the main features of this approach. 
A \MPVF $\frF$ 
can be identified with a subset of 
the set of probability measures
$\mathcal P_2(\TX)$ on the space-velocity tangent bundle $\TX=\{(x,v)\in \X\times \X\}$, with proper domain
$\dom(\frF):=\{\sfx_\sharp \Phi:\Phi\in \frF\}$ and sections $\frF[\mu]:=\{\Phi\in \frF:\sfx_\sharp \Phi=\mu\}$,
where $\sfx(x,v):=x$ is the projection on the first coordinate in $\TX$.
Since every element $\Phi\in \frF$
has finite quadratic moment in the tangent bundle, the $L^2$-norm of the velocity marginal
\begin{equation*}
    |\Phi|_2^2:=
    \int_\TX |v|^2\,\d\Phi(x,v)\quad\text{is finite.}
\end{equation*}
 The disintegration $\{\Phi_x\}_{x\in \X}$ of $\Phi\in \frF[\mu]$ with respect to $\mu$ 
provides a Borel field of probability measures on the space of velocity vectors, which can be interpreted as
a probabilistic description of the velocity prescribed by $\frF$ at every position/particle $x$, given the distribution $\mu$.
An important case, which is simpler to grasp, occurs when $\frF$ is concentrated on maps 
and therefore $\Phi_x=\delta_{\ff(x)}$ is a Dirac mass concentrated on the deterministic velocity $\ff$
(in this case we say that $\frF$ is deterministic): 
 for every measure $\mu\in \dom(\frF)$ 
\begin{equation}   
\label{eq:maps}
\text{the elements $\Phi\in \frF[\mu]$ 
have the form
$(\ii_\X, \ff)_\sharp\mu$
for a vector field $\ff\in L^2(\X,\mu;\X),$}
\end{equation} where $\ii_\X$ denotes the identity map on $\X$. In this case, $\frF$ is \emph{dissipative} if for every $\Phi_i=(\ii_\X, \ff_i)_\sharp \mu_i\in \dom(\frF)$, $i=0,1$,
\begin{equation}
    \label{eq:diss-intro}
    \exists\,  \mmu\in \Gamma_o(\mu_0,\mu_1)
    \quad\text{optimal, such that}
    \quad
    \int_{\X^2} 
    \langle \ff_0(x_0)-
    \ff_1(x_1),x_0-x_1\rangle
    \,\d\mmu(x_0,x_1)\le 0.
\end{equation}
Notice however that, even in the deterministic case,
the realization of $\frF[\mu]$ as an element/subset
of $\prob_2(\TX)$ is crucial to deal with varying base measures $\mu$,
since for different $\mu_0,\mu_1\in \dom(\frF)$ 
the representation \eqref{eq:maps} 
yields corresponding 
maps $\ff_0,\ff_1$ which belong to different $L^2$ spaces
and therefore are not easy to compare.

When $\frF$ is not concentrated on maps, the dissipativity condition 
between two elements 
$\Phi_0\in \frF[\mu_0],\Phi_1\in \frF[\mu_1]$
guarantees the existence
of a coupling $\bm\vartheta\in \Gamma(\Phi_0,\Phi_1)\subset 
\prob_2(\TX\times \TX)$ 
such that 
the ``space'' marginal projection
$(\sfx_0,\sfx_1)_\sharp\bm\vartheta$ is optimal,
thus belongs to $\Gamma_o(\mu_0,\mu_1)$, and
moreover
\begin{equation}
    \label{eq:generaldiss-intro}
     \int_{\TX^2} \langle v_1-v_0,x_1-x_0
        \rangle \,\d\bm\vartheta(x_0,v_0;x_1,v_1)\le 0.
\end{equation}
Such a property appears as a natural generalization of 
the corresponding 
condition introduced in \cite{ags} 
for the Wasserstein subdifferentials of geodesically convex functionals.

The geometric interpretation of this condition becomes apparent by considering its equivalent characterization
in terms of the first order expansion
of the squared Wasserstein distance: 
in the case \eqref{eq:maps} 
it can be written as 
\begin{displaymath}
    W_2^2((\ii_\X+h\ff_0)_\sharp \mu_0,
    (\ii_\X+h\ff_1)_\sharp \mu_1)\le 
    W_2^2(\mu_0,\mu_1)+o(h)\quad\text{as }h\downarrow0.    
\end{displaymath}
In principle, one may interpret the flow generated by $\frF$ in terms of 
absolutely continuous (w.r.t.~the Wasserstein metric) 
curves $\mu:[0,+\infty)\to\prob_2(\X)$ in $\dom(\frF)$ solving the continuity equation
\begin{equation*}
\partial_t\mu_t+\nabla\cdot(\mu_t\,\ff_t)=0\quad\text{in }(0,+\infty)\times \X,\quad
(\ii_\X, \ff_t)_\sharp\mu_t\in \frF,
\end{equation*}
and obeying a Cauchy condition $\mu\restr{t=0}=\mu_0.$
However, the derivation of such a precise formulation is not a simple task and, in general, it requires more
restrictive assumptions on $\frF$
as
\begin{equation}
\begin{gathered}
    \label{eq:closure-intro}
    \dom(\frF)=\mathcal P_2(\X),\quad
    \frF[\mu]=(\ii_\X, \ff[\mu])_\sharp\mu
    \quad \text{(thus $\frF$ is single-valued),}\\
    \mu_n\to \mu \quad\Longrightarrow\quad
    (\ii_\X, \ff[\mu_n])_\sharp \mu_n
    \to
    (\ii_\X, \ff[\mu])_\sharp \mu.
    \end{gathered}
    \end{equation}
We introduced in \cite{CSS}
the more flexible condition of \EVI solutions, borrowed from the theory of gradient flows \cite{ags} and from the B\'enilan notion of integral solutions to dissipative evolutions in Hilbert/Banach spaces \cite{Benilan}: 
a continuous curve $\mu:[0,+\infty)\to\prob_2(\X)$ 
with values in $\overline{\dom(\frF)}$ is an \EVI solution  (we say it solves $\dot\mu_t\in\frF[\mu_t]$)  if it solves the system of Evolution Variational Inequalities
\begin{equation}
    \label{eq:EVI-intro}
    \frac\d{\d t}\frac 12 W^2_2(\mu_t,\nu)\le 
    -[\Phi,\mu_t]_r\quad\text{in }\mathscr D'((0,+\infty)),\quad
    \text{for every }\Phi\in \frF[\nu],\,\nu\in \dom(\frF),
\end{equation}
where for every 
$\Phi=(\ii_\X, \ff)_\sharp\nu\in \frF$ the duality pairing $[\Phi,\mu]_r$ is defined by 
\begin{equation*}
    [\Phi,\mu]_r:=
    \min\Big\{\int_{\X^2} \langle \ff(x_0),x_0-x_1\rangle\,\d\mmu(x_0,x_1):
    \mmu\in \Gamma_o(\nu,\mu)\Big\}.
\end{equation*}
In \cite{CSS}, we studied the properties of the flow in $\mathcal P_2(\X)$ generated by $\frF$ by means of  the \emph{explicit Euler method} 
and we proved that, under suitable conditions, every family of discrete approximations obtained by the explicit Euler method 
converges to an \EVI solution
when the step size vanishes, also providing an optimal error estimate.

The use of the explicit Euler method is simple to implement and quite powerful  when the domain of $\frF$ coincides with the whole $\mathcal P_2(\X)$
and $\frF$ is locally bounded 
\cite[Cor.~5.23]{CSS}, i.e.
$|\Phi|_2$ remains uniformly bounded 
in a suitable neighborhood of every measure
$\mu\in \mathcal P_2(\X)$ (but much more general conditions are thoroughly discussed in \cite{CSS}).
Dealing with constrained evolutions or with operators which are not locally bounded requires a better understanding of the implicit Euler method.

\subsubsection*{\bfseries Maximal totally dissipative {\MPVF}s}
One of the starting points of the present investigation (see Sections \ref{subsec:maps} and 
\ref{sec:constructionFlagr}) is the nontrivial fact that
a large class of $\lambda$-dissipative {\MPVF}s 
including the demicontinuous fields \eqref{eq:closure-intro} satisfies a much stronger dissipativity condition, which we call \emph{total $\lambda$-dissipativity}:
in the simplest case $\lambda=0$ when \eqref{eq:maps} holds and $\frF$ is single-valued, such a property reads as 
\begin{equation}
    \label{eq:strongdiss-intro}
    \int_{\X^2}
    \langle
    \ff[\mu_0](x_0)-
    \ff[\mu_1](x_1),x_0-x_1\rangle
    \,\d\mmu(x_0,x_1)\le 0
    \quad\text{for \emph{every} $\mmu\in \Gamma(\mu_0,\mu_1)$}
\end{equation} 
and can be compared with
the notion of L-monotonicity of \cite[Definition 3.31]{CD18}. Total dissipativity thus holds along
arbitrary couplings between pairs of measures $\mu_0,\mu_1$ in the domain of $\frF$,
whereas the metric dissipativity condition 
\eqref{eq:diss-intro}
involves only optimal couplings.
The relaxed version of \eqref{eq:strongdiss-intro} allowing for $\lambda$-dissipativity 
includes the class of Lipschitz probability vector fields $\ff$ satisfying 
\begin{equation*}
    \Big|\ff[\mu_0](x_0)-\ff[\mu_1](x_1)\Big|\le L\Big(|x_0-x_1|+W_2(\mu_0,\mu_1)\Big)\quad\text{for every }x_i\in \X,\ \mu_i\in \mathcal P_2(\X)
\end{equation*}
for $\lambda=2L$ (see Example \ref{ex:Lip}).

Motivated by this remarkable property, it is natural to extend 
the notion of total dissipativity
to a
general \MPVF $\frF$. 
Here there are two possible 
approaches: the weakest one, modeled
on the general definition of metric dissipativity \eqref{eq:generaldiss-intro}, 
would require that 
for every $\Phi_0\in \frF[\mu_0],\Phi_1\in \frF[\mu_1]$ and \emph{every}
coupling $\mmu\in \Gamma(\mu_0,\mu_1)$
($\mmu$ is not optimal) 
there exists $\bm\vartheta
\in \Gamma(\Phi_0,\Phi_1)$
such that 
$(\sfx_0,\sfx_1)_\sharp\bm\vartheta=\mmu$ and
\eqref{eq:generaldiss-intro}
holds.

The strongest condition, which we will systematically investigate in this paper, 
requires that
    \begin{equation}
    \label{eq:total-diss-intro}
    \text{for every $\Phi_0,\Phi_1\in \frF$
    and every $\bm \vartheta\in \Gamma(\Phi_0,\Phi_1)$}\quad
    \int_{\TX^2} \langle v_1-v_0,x_1-x_0
        \rangle \,\d\bm\vartheta(x_0,v_0;x_1,v_1)\le 0.
    \end{equation}
It is 
clear that 
total dissipativity for arbitrary \MPVF{s} 
is much stronger than
the metric dissipativity condition \eqref{eq:generaldiss-intro}.
We address two main questions:
\begin{enumerate}[label=\rm $\langle$Q.\arabic*$\rangle$]
    \item\label{q1} What are the structural properties of totally dissipative {\MPVF}s satisfying the stronger condition \eqref{eq:total-diss-intro}
    and their relation with Lagrangian representations by dissipative 
        operators in the Hilbert space 
    \[\cH:=L^2(\Omega,\cB,\P;\X),\]
    where $\P$ is a nonatomic probability measure 
    on a standard Borel 
    space $(\Omega,\cB)$, 
    which provides the domain of the parametrization. A similar lifting approach has been used also in e.g.~\cite{Lions,carda,gangbotudo,CD18,Jimenez,JMQ}, in particular 
    for functions defined in $\prob_2(\X)$ and 
    their Fr\'echet differential.
    This is the content of \textbf{Part \ref{partI}} and in particular of Section \ref{sec:invmpvf} and \ref{sec:totdissMPVF-flow}, with applications to the case of gradient flows in Section
    \ref{sec:jkobis}.
    \item\label{q2} Under which conditions 
    a dissipative \MPVF 
    is totally dissipative and, more generally,
    is it possible to recover a unique
     maximal totally dissipative ``version'' of the initial \MPVF
    starting from
    a sufficiently rich set of discrete measures.
    This is investigated first in Section 
    \ref{subsec:maps} and then more extensively in \textbf{Part \ref{partII}}, in particular in Section \ref{sec:constructionFlagr},
    starting from the results of Sections \ref{sec:coupl} and \ref{sec:strong-dissipative}
    on the geometry of discrete measures.
\end{enumerate}
\subsubsection*{\bfseries Lagrangian representations}
Concerning the first question \ref{q1}, in Section \ref{sec:3.2} 
we will show that 
\begin{quote}
\em there is a one-to-one correspondence between totally dissipative {\MPVF}s 
and law invariant 
dissipative operators in the Hilbert space
$\cH:=L^2(\Omega,\cB,\P;\X)$;
such a correspondence preserves maximality.
\end{quote}
This representation is very useful to import in the metric space 
$(\mathcal P_2(\X), W_2)$
all the powerful tools and results
concerning semigroups of contractions generated by maximal dissipative operators in Hilbert spaces, see e.g.~\cite{BrezisFR}.
This approach overcomes most of the technical limits of the explicit Euler method adopted in \cite{CSS} and allows for a more general theory of existence, well posedness, and stability of solutions.
In particular, even if the results are new and relevant
also in the finite dimensional Euclidean case, the theory does not rely on any compactness argument and thus can be fully developed in a
infinite dimensional separable Hilbert space $\X$.
We can in fact lift 
a totally dissipative \MPVF $\frF$ to a dissipative operator $\mmo
\subset \cH\times \cH$, that we call \emph{Lagrangian representation of $\frF$},  defined by
\begin{equation*}
    (X,V)\in \mmo\quad\Longleftrightarrow\quad
    (X,V)_\sharp \P\in \frF.
\end{equation*}
It turns out that $\mmo$ is law invariant
(i.e.~if $(X,V)\in \mmo$ and $(X',V')$ has the same law as $(X,V)$, then $(X',V')\in \mmo$ as well) and admits a maximal dissipative extension $\hat \mmo$ which is law invariant
and corresponds to a maximal extension 
of $\frF$ still preserving total dissipativity. 
In particular, $\frF$ is maximal in the class of totally dissipative {\MPVF}s if and only if $\mmo$ is a law invariant operator which is maximal dissipative. 

Such a crucial result depends on two important properties: 
first of all, if the graph of $\mmo$ is strongly-weakly closed in $\cH\times \cH$ (in particular if $\mmo$ is maximal)
then law invariance can also be characterized by invariance w.r.t.~measure-preserving isomorphisms of $\Omega$, i.e.~essentially injective maps $g:\Omega\to\Omega$ such that
$g_\sharp\P=\P=g^{-1}_\sharp \P$ (Theorem \ref{thm:invTOlawinv}). The second property
(Theorem \ref{thm:maximal-dissipativity})
guarantees that every dissipative operator in $\cH$ which 
is invariant by measure-preserving isomorphisms has a maximal dissipative extension enjoying the same invariance (and thus also law invariance). Such a result has been obtained in 
\cite{CSS2piccolo} and exploits remarkable results of \cite{BauWang2009,BauWang2010} providing an explicit construction of a maximal extension of a monotone operator. 

The equivalence between law-invariance and invariance by measure-preserving tranformations  also plays a crucial role to prove that the resolvents of $\mmo$, its Yosida approximations, and the generated semigroup of contractions $(\bm S_t)_{t\ge0}$ in $\cH$ 
are still law invariant.  
The family 
$(\bm S_t)_t$ thus induces
a projected semigroup of contractions in  $\mathcal P_2(\X)$ defined by
\begin{equation}
    \label{eq:lagrangian-intro}
    S_t(\mu_0):=(\bm S_t X_0)_\sharp\P\quad\text{whenever}\quad
    (X_0)_\sharp\P=\mu_0\in \dom(\frF),
\end{equation}
which is independent of the choice of $X_0$
parametrizing the initial law $\mu_0$,
which satisfies the \EVI formulation \eqref{eq:EVI-intro} and the stability property (here for arbitrary $\lambda\in \R$) 
\begin{equation}
    \label{eq:stability-intro}
    |\bm S_t X_0-\bm S_t Y_0|_\cH\le 
    \mathrm e^{\lambda t} |X_0-Y_0|_{\cH},\quad
    W_2(S_t(\mu_0),S_t(\nu_0))\le \mathrm e^{\lambda t}W_2(\mu_0,\nu_0).
\end{equation}
Another crucial property of totally dissipative {\MPVF}s concerns the \emph{barycentric projection}, which can be obtained by taking the expected value of the disintegration 
$\{\Phi_x\}_{x\in \X}$ of an element $\Phi\in \frF$ with respect to its first marginal $\mu=\sfx_\sharp \Phi$:
\begin{equation*}
    \bry\Phi(x):=
    \int_\X v\,\d\Phi_x(v)
    \quad\text{for $\mu$-a.e.~$x\in \X$;}\quad
    \bry\Phi\in L^2(\X,\mu;\X).
\end{equation*}
The barycenter $\bry\Phi$ also represents the conditional expectation 
$\E[V|X]$
of $V$ given (the $\sigma$-algebra generated by) $X$, for every 
$(X,V)\in \frF$ with $(X,V)_\sharp\P=\Phi$:
\begin{equation*}
    \E[V|X]=\bry\Phi\circ X\quad\text{in $L^2(\Omega,\sigma(X),\P;\X)$.}
\end{equation*}
It turns out that, if 
$\frF$ is maximal totally dissipative 
(or, equivalently, its Lagrangian representation  $\mmo$ is maximal dissipative), then $\frF$ is invariant 
with respect to the barycentric projection:
\begin{equation}
    \label{eq:bary-invariance-intro}
    (X,V)_\sharp\P=\Phi\in \frF
    \quad\Longrightarrow
    \quad
    (\ii_\X, \bry\Phi)_\sharp \mu\in \frF,\quad
    (X,\E[V|X])\in \mmo.
\end{equation}
Thanks to \eqref{eq:bary-invariance-intro}, for every $\mu_0\in \dom(\frF)$, the solution $\mu_t$ expressed by the Lagrangian formula \eqref{eq:lagrangian-intro} can be 
characterized as a Lipschitz curve in $\mathcal P_2(\X)$
satisfying the continuity equation 
\begin{equation}
    \label{eq:cont-intro1}
    \frac\d{\d t}\int_\X \zeta\,\d\mu_t=
    \int_\X \langle \vv_t(x),\nabla\zeta(x)\rangle
    \,\d\mu_t(x)
    \quad\text{for every $\zeta\in \Cyl(\X)$ and a.e.~$t>0$}
\end{equation}
for a Borel vector field $\vv$
satisfying
\begin{equation}
\label{eq:diff-inclusion-intro}
    t\mapsto \int_\X |\vv_t(x)|^2\,\d\mu_t(x)
    \quad\text{is locally integrable in $[0,+\infty)$},\quad
    (\ii_\X, \vv_t)_\sharp \mu_t\in \frF\ \text{for a.e.~$t>0$.}
\end{equation}
 We can also characterize the solution $\mu_t$ to \eqref{eq:cont-intro1}, \eqref{eq:diff-inclusion-intro} by 
requiring  that 
there exists a Borel family $\Phi_t$, $t>0$, such that 
\begin{equation}\label{eq:cucu}
\begin{split}
    \Phi_t\in \frF[\mu_t]\quad\text{for a.e.~$t>0$,}\quad
    t\mapsto \int_{\TX}|v|^2\,\d\Phi_t\quad\text{is locally integrable in $[0,+\infty)$,}\\
    \frac\d{\d t}\int_\X \zeta\,\d\mu_t=
    \int_\TX \langle v,\nabla\zeta(x)\rangle
    \,\d\Phi_t(x,v)
    \quad\text{for every $\zeta\in \Cyl(\X)$ and a.e.~$t>0$.}
    \end{split}
\end{equation}
 Indeed the validity of \eqref{eq:cont-intro1}, \eqref{eq:diff-inclusion-intro} gives that \eqref{eq:cucu} holds with $\Phi_t=(\ii_\X, \vv_t)_\sharp \mu_t$; on the other hand, assuming \eqref{eq:cucu}, we get \eqref{eq:cont-intro1}, \eqref{eq:diff-inclusion-intro} with $\vv_t=\bry{\Phi_t}$ which belongs to $\frF[\mu_t]$ by \eqref{eq:bary-invariance-intro}.

When $\frF$ is maximal totally dissipative, 
a more precise formulation of \eqref{eq:cont-intro1} and 
\eqref{eq:diff-inclusion-intro}
can be obtained by introducing the minimal selection $\mmo^\circ$ (i.e.~the element of minimal norm) of $\mmo$: 
we will prove that 
for every $X\in \dom(\mmo)$ with
$X_\sharp\P=\mu$,
$\mmo^\circ$ is associated with 
a vector field $\ff^\circ\in L^2(\X,\mu;\X)$ through the formula
\begin{equation}
    V=\mmo^\circ X,\quad X_\sharp\P=\mu\quad
    \Longleftrightarrow\quad
    V=\ff^\circ[\mu](X).
\end{equation}
The measure $(\ii_\X, \ff^\circ[\mu])_\sharp \mu$ can be characterized as
the unique element $\Phi\in\frF[\mu]$ minimizing $|\Phi|_2$
and the solution $\mu:[0,+\infty)\to\prob_2(\X)$ provided by \eqref{eq:lagrangian-intro}
is also the unique Lipschitz curve satisfying the continuity equation
\begin{equation}
    \label{eq:cont-intro}
    \partial_t \mu_t+\nabla\cdot(
    \mu_t \,\ff^\circ[\mu_t])=0
    \quad\text{in $(0,+\infty)\times \X$}
\end{equation}
with initial datum $\mu_0\in \dom(\frF)$.
It is remarkable that 
a maximal totally dissipative \MPVF 
always admits a minimal selection which is concentrated on a map. 

It turns out that the evolution driven by $\frF$ preserves the class of discrete measures with finite
support; if moreover $\mu_0=\frac 1N\sum_{n=1}^N\delta_{x_n}\in \dom (\frF)$
(or also in $\overline{\dom(\frF)}$ if $\X$ has finite dimension) then
the unique solution of \eqref{eq:cont-intro} can be expressed 
in the form $\mu_t=\frac 1N\sum_{n=1}^N \delta_{\mathsf x_n(t)}$
where $t\mapsto \mathsf x_n(t)$ are locally Lipschitz curves satisfying the system of ODEs
\begin{equation}
    \label{eq:ODE-system-intro}
    \dot {\mathsf x}_n(t)=\ff^\circ[\mu_t](\mathsf x_n(t))\quad\text{a.e.~in $(0,+\infty)$},\quad
    \mathsf x_n(0)=x_n,\quad n=1,\cdots,N.
\end{equation}
Thanks to \eqref{eq:stability-intro}, if a sequence of 
discrete initial measures $\mu_0^N=\frac1N\sum_{n=1}^N \delta_{x_n^N}$
converges to a limit $\mu_0$ in $\prob_2(\X)$ as $N\to+\infty$, then 
the corresponding evolving measures $\mu^N_t$ obtained by solving
\eqref{eq:ODE-system-intro} starting from $\mu_0^N$ will converge to $\mu_t=S_t(\mu_0)$. As a general fact \cite{Sznitman91}, this correspond to the propagation of chaos 
for the sequence of symmetric particle systems \eqref{eq:ODE-system-intro}.

Maximality also shows that \EVI curves are unique; when they are differentiable
(in particular when $\mu_0\in \dom(\frF)$) 
we can recover the representation 
\eqref{eq:cont-intro} and the Lagrangian
representation \eqref{eq:lagrangian-intro},
in an even more refined version involving characteristic curves.
This representation immediately yields regularity, stability, perturbation, and approximation results thanks to the corresponding statements in the Hilbertian framework.

Among the possible applications,  we just recall that we can also use {\em the Implicit Euler Method}
(corresponding to the JKO scheme for gradient flows)
to construct the flow (Corollary \ref{cor:IES}).
Starting from $M_\tau^0:=\mu_0\in \dom(\frF)$, for every step size $\tau>0$ we 
can find a (unique) sequence $(M^n_\tau)_{n\in \N}$ 
in $\dom(\frF)$ which at each step $n\in \N$ solves
\begin{equation}
    (\sfx-\tau\sfv)_\sharp \Phi^{n+1}_\tau=M^n_\tau
    \quad\text{for some }\Phi^{n+1}_\tau\in \frF[M^{n+1}_\tau].
\end{equation}
Selecting $\tau:=t/N$, the sequence $\left(M^N_{t/N}\right)_{N\in\N}$ converges to $S_t(\mu_0)$ as $N\to+\infty$ 
with the a-priori error estimate
\begin{equation}
    W_2(S_t(\mu_0),M^{N}_{t/N})\le \frac {2t}{\sqrt{N}}\|\ff^\circ [\mu_0]\|_{L^2(\X,\mu_0;\X)}.
\end{equation}
When $\dom(\frF)=\mathcal P_2(\X)$ and
$\frF$ is single-valued as in  \eqref{eq:closure-intro},
it follows that maximality is equivalent to
the following demicontinuity condition:
for every sequence $(\mu_n)_{n\in \N}$ converging to $\mu$ in
$\prob_2(\X)$ one has
\begin{equation}
\label{eq:hemicontinuity}
\sup_{n\to+\infty}\int_\X |\ff[\mu_n]|^2\,\d\mu_n<+\infty,\quad
(\ii_\X, \ff[\mu_n])_\sharp \mu_n\to
(\ii_\X, \ff[\mu])_\sharp \mu\quad\text{in }\mathcal P(\X\times \X^w),
  \end{equation}
  where $\X^w$ denotes the Hilbert space endowed with its weak topology.
Clearly, in this case
the map $\ff$ representing $\frF$ coincides with $\ff^\circ$. 
Notice that \eqref{eq:hemicontinuity} surely holds
if $\frF$ is represented by a map $\ff:\mathcal{P}_2(\X)\to
\Lip(\X;\X)$ (see also the map $F'$ in \cite[Section 2.3]{Chizat}) satisfying
the integrated Lipschitz-like condition along arbitrary couplings
\begin{equation}
\label{eq:Lipschitzintro}
        \int_{\X^2} \Big|\ff[\mu_0](x_0)-
    \ff[\mu_1](x_1)\Big|^2\,\d\mmu(x_0,x_1)
    \le L^2\int_{\X^2} |x_0-x_1|^2\,\d\mmu(x_0,x_1)
    \quad\text{for every }
    \mmu\in \Gamma(\mu_0,\mu_1).
  \end{equation}
On the other hand,
this class of regular 
dissipative {\PVF}s
is sufficiently rich to approximate the minimal selection of any 
maximal totally dissipative \MPVF $\frF$: in fact, by using the Yosida approximation,
it is possible to find a sequence of regular {\PVF}s $\frF_n$ associated to 
Lipschitz fields $\ff_n$ according to \eqref{eq:Lipschitzintro} 
(w.r.t.~a possibly diverging sequence of Lipschitz constant $L_n$)
satisfying the dissipativity condition \eqref{eq:strongdiss-intro} and
\begin{equation*}
    \lim_{n\to+\infty}\int_\X\big|\ff_n[\mu](x)-\ff^\circ[\mu](x)\big|^2\,\d\mu(x)=0
    \quad\text{for every }\mu\in \dom(\frF).
\end{equation*}
So, the class of totally dissipative {\MPVF}s arises as a natural closure
of more regular {\PVF}s concentrated on dissipative Lipschitz maps. 
This statement (Corollary \ref{cor:regular-appr-MPVF})
justifies a posteriori 
the choice of the 
strongest notion of 
total dissipativity 
given in \eqref{eq:total-diss-intro}.

\subsubsection*{\bfseries Construction of
  a maximal totally dissipative \MPVF from a discrete core.}

We investigate the second 
issue \ref{q2} in Section \ref{sec:constructionFlagr},
i.e.~how to recover a (unique) maximal totally dissipative ``version'' of
a  (totally or metrically) 
$\lambda$-dissipative \MPVF $\frF$ defined on  
a sufficiently 
rich \emph{core} $\core$ of discrete measures. 
This corresponds to the derivation of a mean-field description
from a compatible family of discrete particle systems.

Just to give an idea of a simple case of core,
we consider a totally convex subset $\mathrm D$ of 
the set $\prob_f(\X)$ of discrete measures with finite support:
total convexity here means that, 
whenever the marginals 
$\sfx^i_\sharp\mmu$, $i=0,1$,
of $\mmu\in \prob_f(\X\times \X)$ 
belong to $\mathrm D$,
then also $((1-t)\sfx^0+t\sfx^1)_\sharp \mmu$
belong to $\mathrm D$ for every $t\in (0,1).$ 

For every 
$N\in \N$ 
we consider 
the 
collection $\core_N$ 
of uniform discrete measures 
$\mu_\xx=\frac 1N\sum_{n=1}^N \delta_{x_n}$ 
belonging to $\mathrm D$,
where $\xx=(x_1,\cdots,x_N)$ is a vector 
in $\mathsf X^N$ with distinct coordinates.
The set $\core_N$ corresponds to 
a subset $\mathsf C_N$ of  $\mathsf X^N$
which is invariant under the action of 
the group of permutations $\Sym N$ of the components,
\begin{displaymath}
    \sigma \xx:=(x_{\sigma(1)},\cdots,x_{\sigma (N)}),
    \quad
    \text{for every}\quad
    \sigma\in \Sym N,\ \xx=(x_1,\cdots,x_N)\in \X^N.
\end{displaymath}

We will suppose that $\mathsf C_N$ is relatively 
open in $\X^N$
for every $N\in \N.$
Examples of $\mathrm D$ are provided by the collection of all
the discrete measures $\mu$ such that $\supp(\mu)$ 
is contained in a given convex open subset $\mathsf U$
of $\X$. Another interesting case, assuming $0\in \mathsf U$,
is given by all
the discrete measures such that 
$\supp(\mu)-\supp(\mu)\subset \mathsf U.$ The case of the whole
set $\prob_f(\X)$ is
still interesting.

Suppose that we
have a deterministic single-valued \PVF $\frF$ 
defined in $\core=\bigcup_N {\core_N}$ (when $\frF$ is not deterministic, the construction is more subtle).
We can then represent $\frF$ on each $\core_N$
by a vector field $\ff^N: \mathsf C_N
\to \X^N$ 
satisfying 
the invariance property 
$\ff^N(\sigma\xx)=\sigma \ff^N(\xx)$, so that 
\[
\frF[\mu_\xx]=\frac 1N\sum_{n=1}^N \delta_{(x_n,\ff^N_n(\xx))}
\quad
\text{for every }\xx\in 
\mathsf C_N,\]
and, at least for a short time  when no collisions occur,
the evolution of discrete measures in $\core_N$
can be described by 
$\mu_t=\frac 1N\sum_{n=1}^N \delta_{x_n(t)}=
\mu_{\xx(t)}$
where the vector $\xx(t)=(x_1(t),\cdots x_N(t))\in \mathsf C_N$
solves the system
\begin{equation}
\label{eq:particle-system}
    \dot \xx(t)=\ff^N(\xx(t)).
\end{equation}
We assume the following 
$\lambda$-dissipativity conditions on the maps $\ff^N$:
for every pair of integers 
$M,N\in \N$ with $M\mid N$,
if $
\xx\in \mathsf C_M$, $\yy\in \mathsf C_N$
and $\theta$ is an optimal correspondence from 
$\{1,\cdots, N\}$ to $\{1,\cdots,M\}$, i.e.

\[
\frac 1N\sum_{n=1}^N |y_n-x_{\theta(n)}|^2
=W_2^2(\mu_\xx,\mu_\yy),
\]
 
then 
\begin{displaymath}
    \sum_{n=1}^N 
    \langle \ff^N_n(\yy)-\ff^M_{\theta(n)}(\xx),
    y_n-x_{\theta(n)}\rangle \le 
    \lambda \sum_{n=1}^N |y_n-x_{\theta(n)}|^2.
\end{displaymath}
We will show that 
$\frF$ is in fact totally $\lambda$-dissipative
and admits a unique maximal extension $\hat \frF$, 
whose flow
can be interpreted as the unique mean-field limit 
of the particle systems driven by \eqref{eq:particle-system}.
This fact guarantees  two interesting properties: 
the local in time evolution corresponding to 
\eqref{eq:particle-system} admits a unique
global extension which induces a semigroup $(S^N_t)_{t\ge0}$
on $\overline {\core_N}$
which corresponds to the restriction to $\overline{\core_N}$ 
of the semigroup $S_t$ generated by $\hat\frF$
(and characterized e.g.~by the continuity equation
\eqref{eq:cont-intro} and by \eqref{eq:ODE-system-intro}).
Moreover, thanks to \eqref{eq:stability-intro}
for every $\mu_0\in \overline{\core}$ and every sequence $(\mu^N_0)_{N\in \N}$
with $\mu^N_0\in\core_N$ and converging to $\mu_0$ as $N\to+\infty$
we have $S^N_t(\mu^N_0)\to S_t(\mu_0)$ in $\prob_2(\X)$ 
locally uniformly w.r.t.~$t\in [0,+\infty)$.

Thanks to the stability properties of 
the Lagrangian flow, Theorem \ref{thm:stability} also shows that 
the trajectories of the discrete particle system 
uniformly converge
in a measure-theoretic sense to the characteristics of 
the mean-field system.

As a byproduct, we obtain that when
the domain of a totally dissipative
\MPVF $\frF$ contains a dense core then
its maximal extension is unique and can be characterized by a suitable
explicit construction starting from the core itself
and its flow has a natural mean-field interpretation.
\smallskip

Our result also provides interesting applications to geodesically
convex functionals and their approximations (see Sections
\ref{sec:jkobis},\ref{sec:jko}).

First of all, if the proper domain of a lower semicontinuous
and geodesically convex functional $\phi:\mathcal P_2(\X)\to(-\infty,+\infty]$
contains a discrete core $\core$ which is dense in energy,
then $\phi$ is totally convex, i.e.~it is convex along all the
linear interpolations induced by arbitrary couplings.
An important class is provided by
continuous and everywhere defined geodesically convex functionals,
which thus turn out to be totally convex.

The same property holds for any functional $\phi:\mathcal
P_2(\X)\to (-\infty,+\infty]$ which arises as Mosco-like limit of a
sequence of continuous and geodesically convex functionals which are everywhere finite. In particular, such approximation is impossible for all the functionals
which are not totally convex, as the relative entropy functionals
w.r.t.~log-concave measures.

\subsubsection*{\bfseries Contributions and applications.} One reason this study is relevant is that it enables the application of the well-developed Hilbertian theory into the framework of dissipative evolutions in the $2$-Wasserstein space. In particular, we are allowed to apply the implicit Euler scheme to maximal totally dissipative \MPVF{s} — an approach not available in general, or at least not yet clearly implementable, for \MPVF{s} that are only metrically dissipative. As in Hilbertian theory, the implicit scheme does not require local boundedness of the operator, which is instead necessary for the explicit scheme (cf. \cite{CSS}). Furthermore, the correspondence between maximal dissipative operators in Hilbert spaces and maximal totally dissipative \MPVF{s} allows for a refined description of the evolutions; see in particular Section \ref{sec:totdissMPVF-flow}.

Following the same principle — that is the application of Hilbertian techniques to the Wasserstein context — we aim to study the following further aspects in a future review paper:
\begin{itemize}
    \item Regularizing effects under suitable assumptions on $\frF$;
    \item Asymptotic behaviour and periodic solutions;
    \item Error estimates for the Yosida regularization and for time discretizations (see also \cite{CSS}), Chernoff and Trotter formulas;
    \item Stability and convergence of sequences of $\lambda$-contractive semigroups;
    \item Discrete-to-continuous limit and chaos propagation;
    \item The case of time-dependent \MPVF{s}.
\end{itemize}

In \cite{CSSnew}, we initiated this program and compared the explicit approach of \cite{CSS} and the implicit approach of the present work. There, we studied the convergence of stochastic time-discretization schemes for evolution equations driven by random velocity fields, including examples such as stochastic gradient descent and interacting particle systems. Under suitable dissipativity and boundedness conditions, we proved that the laws of the interpolated trajectories converge to those of a limiting evolution governed by a maximal dissipative extension of the associated barycentric field. This provides a general measure-theoretic study of the convergence of stochastic schemes in continuous time.

\subsubsection*{\bfseries Plan of the paper.}  The plan of the paper is as follows.\\
\quad \\
\textbf{Part \ref{partI}} develops the theory of totally dissipative {\MPVF}s and it is devoted to answer \ref{q1}. After a quick review in \textbf{Section \ref{sec:preliminaries}} of the main tools on Wasserstein spaces used in the sequel, we summarize in Subsection \ref{sec:prelimCSS} the notation and the results concerning Multivalued Probability Vector Fields and \EVI solutions.

In \textbf{Section \ref{sec:invmpvf}}, we introduce the notion of \emph{totally dissipative} \MPVF and we study its consequences in terms of existence and description of Lagrangian solutions: in Subsection \ref{subsec:maxim} we study the properties of the Yosida approximations, the resolvent operator and the minimal selection of law-invariant operators in the Hilbert space $\cH$ of parametrizations, Subsection \ref{sec:3.2} deals with the relation between dissipativity for such law-invariant subsets of  $\cH$ and the corresponding total dissipativity for their law.
These results are used in Subsection \ref{subsec:maps} to study the particular case of deterministic totally dissipative {\PVF}s.

\textbf{Section \ref{sec:totdissMPVF-flow}} contains the main 
existence, uniqueness, stability, and approximation results for 
the Lagrangian flow generated by a totally dissipative \MPVF, together with its various equivalent characterizations.

In \textbf{Section \ref{sec:jkobis}}, we study the behaviour of functionals $\phi: \prob_2(\X)\to (-\infty, +\infty]$ which are convex along any coupling, proving the existence of gradient flows (equivalently, \EVI solutions for the \MPVF given by their subdifferential) still exploiting their representation in terms of a convex functional $\psi$ defined in the parametrization space $\cH$. \\
\smallskip \noindent

\textbf{Part \ref{partII}} studies the 
characterization of maximal extensions of totally dissipative MPVF, their relation with metric dissipativity, and it is devoted to answer \ref{q2}.
\textbf{Section \ref{sec:coupl}} is devoted to study the properties of couplings between discrete measures, in particular showing that such couplings are ``piece-wise" optimal.
This property is then exploited in \textbf{Section \ref{sec:strong-dissipative}} where we show that a dissipative \MPVF is totally dissipative along discrete couplings.

In \textbf{Section \ref{sec:constructionFlagr}} we show that starting from a dissipative \MPVF $\frF$
defined on a sufficiently rich \emph{core} $\core$ of discrete measures, it is possible to construct a maximal totally dissipative \MPVF
$\hat\frF$, in a unique canonical way. 

\textbf{Section \ref{sec:jko}} is in the same spirit but in the case of a geodesically convex functional $\phi$: under analogous approximation properties, it is possible to show that $\phi$ is actually totally convex and then satisfies the assumptions of Section \ref{sec:jkobis}.

\smallskip\noindent
Finally, \textbf{Appendix \ref{sec:brezis}} contains many useful results related to $\lambda$-dissipative operators in Hilbert spaces that are more commonly known for $\lambda=0$ (the main reference is \cite{BrezisFR}), while \textbf{Appendix \ref{sec:appborel}} lists some of the results of \cite{CSS2piccolo} related to Borel partitions and approximations of couplings that are used in the present work.

\mysubsubsection{Acknowledgments.} G.S.~and G.E.S.~gratefully acknowledge the support of the Institute for Advanced Study of the Technical University of Munich, funded by the German Excellence Initiative. 
G.C.~acknowledges the partial support of MIUR-PRIN projects, of the group GNAMPA of the Istituto Nazionale di Alta Matematica (INdAM), and of the funds FSR Politecnico di Milano Prog.TDG3ATEN02. 
G.S. has been partially supported
by the INDAM project E53C23001740001
and by funding from
the European Research Council (ERC) under the European Union’s Horizon Europe research and innovation programme (grant
agreement No. 101200514, project acronym OPTiMiSE). Views and opinions expressed are however those of the author(s) only
and do not necessarily reflect those of the European Union or the European Research Council Executive Agency. Neither the
European Union nor the granting authority can be held responsible for them.

\part{Total dissipativity}\label{partI}

\section{Preliminaries}
\label{sec:preliminaries}
In the following table, we provide a list of the adopted symbology for the reader's convenience. We then recall the main notions and results of Optimal Transport theory and finally, in Subsection \ref{sec:prelimCSS}, we collect the fundamental objects and basic results taken from \cite{CSS} needed to develop our analysis.
As a general rule, we will use bold letters to denote
maps (even multivalued) with values in the Hilbert space $\X$ or 
measures/sets of measures in product spaces as couplings in 
$\X\times \X$ or probability vector fields in $\TX$.
{\small\begin{longtable}{ll}
$\bry{\Phi}$&the barycenter of $\Phi\in\prob(\TX)$ as in Definition \ref{def:wassmom};\\
$\mmo^\lambda$&the $\lambda$-transformation of a set $\mmo$ as in Remark \ref{rem:transff};\\
$\mmo_\tau$&Yosida approximation of a maximal dissipative $\mmo$, see Appendix \ref{sec:brezis};\\
$\mmo^\circ$&minimal selection of a maximal dissipative $\mmo$, see Appendix \ref{sec:brezis};\\
$\clo{\frF}$& the sequential closure of $\frF$, see Proposition \ref{prop:closure};\\
$\conv{E}, \clconv{E}$&convex and closed and convex envelope of a set $E$ in a Hilbert space;\\
$\dom(\frF)$&the proper domain of a set-valued function as in Definition \ref{def:MPVF};\\
$\dom(\phi)$&the proper domain of a functional $\phi$;\\
$\ff^\circ$&the map defined in Theorem \ref{thm:minimal};\\
$\frF, \frF[\mu]$&a multivalued probability vector field and its section at $\mu \in \prob_2(\X)$, see Definition \ref{def:MPVF};\\
$\frF^\lambda$&the $\lambda$-transformation of $\frF$ as in \eqref{eq:Flambda};\\
$\Gamma(\mu,\nu)$&the set of admissible couplings between $\mu,\nu$, see \eqref{def:admplans};\\
$\Gamma_o(\mu,\nu)$&the set of optimal couplings between $\mu,\nu$,
see Definition \ref{def:wassmom};\\
$\iota, \iotaT, \iota_X, \iotaT_{X,Y}$&the maps as in the beginning of Section \ref{sec:invmpvf};\\
$\ii_X$&the identity map on a set $X$;\\
$\jJ_\tau$&the resolvent operator of a maximal dissipative $\mmo$, see Appendix \ref{sec:brezis};\\
$\rsqm{\nu}$&the $2$-nd moment of $\nu\in\prob_2(\mathscr{X})$ as in Definition \ref{def:wassmom};\\
$|\Phi|_2$&the partial  $2$-nd moment of $\Phi\in\prob_2(\TX)$ as in \eqref{eq:defsqmPhi};\\
$\cN$&a directed subset of $\N$ w.r.t.~the order induced by $\preccurlyeq$, see Appendix \ref{sec:appborel};\\
$(\Omega, \cB, \P)$&a standard Borel space  endowed with a nonatomic probability measure, Def. \ref{def:sbs};\\
$(\Omega, \cB, \P, (\mathfrak P_N)_{N \in \cN})$&a $\cN$-refined standard Borel probability space, see Definition \ref{def:segm};\\
$\prob(\mathscr{X})$&the set of Borel probability measures on the topological
space $\mathscr{X}$;\\
$\prob_f(\mathscr{X})$, $\prob_{f,N}(\mathscr{X})$&the sets defined in \eqref{eq:70},\eqref{eq:30};\\
$\prob_{f,\cN}, \prob_{\#\cN}(\X)$&the sets in \eqref{eq:137};\\
$\prob_c(\mathscr{X}), \prob_2(\mathscr{X})$&measures in $\prob(\mathscr{X})$ with compact support or finite quadratic moment, see \eqref{condTanTX};\\
$\prob_2^{sw}(\TX)$&the space $\prob_2(\TX)$ endowed with the strong-weak topology as in Definition \ref{def:p2sw};\\
$\pi^i, \pi^{i,j}, \pi^{i,j,k}, \pi^{i,j,k,l}$&projections from a product space to one or more factors as in \eqref{def:admplans};\\
$\bram{\cdot}{\cdot}$, $\brap{\cdot}{\cdot}$&the pseudo scalar products as in Definition \ref{def:pairings};\\
$\directionalm{\Phi}{\ttheta}t$, $\directionalp{\Phi}{\ttheta}t$&the duality pairings as in Definition \ref{def:pairings};\\
$\directionalm \frF\mmu t$, $\directionalp \frF\mmu t$&the duality pairings as in Definition \ref{def:plangeodomV};\\
$\rmS(\Omega)=\rmS(\Omega, \cB, \P)$&measure-preserving isomorphisms on $(\Omega, \cB, \P)$,
see Appendix \ref{sec:appborel};\\
$\rmS_N(\Omega)$&subset of $\rmS(\Omega, \cB, \P)$ of maps that are $\cB_N-\cB_N$ measurable;\\
$S_t, \ss_t$&Eulerian and Lagrangian semigroups, Def.~\ref{def:semig};\\
$\Sgp_t$&semigroup generated by a maximal dissipative $\mmo$, see Appendix \ref{sec:brezis};\\
$\Sp{\X, D},\,\Sp\X$&the subsets of $\X\times \prob_2(\X)$ as in \eqref{eq:2};\\
$W_2(\mu,\nu)$&the $L^2$-Wasserstein distance between $\mu$ and $\nu$, see Definition \ref{def:wassmom};\\
$\X$&a separable Hilbert space;\\
$\cH,\cH_N, \cH_\infty$&the Hilbert spaces $L^2(\Omega, \cB, \P; \X)$, $L^2(\Omega, \cB_N, \P; \X)$ and the union of the $\cH_N$ respectively;\\
$\X^s, \X^w$& a separable Hilbert space endowed with its strong and weak topologies;\\
$\TX$&the tangent bundle to $\X$, usually endowed with the strong-weak
topology;\\
$\sfx, \sfx^i, \sfv,\sfv^i$&the projection maps defined in \eqref{eq:projandexp} and in Section \ref{sec:prelimCSS};\\
$\sfx^t$&the evaluation map defined in \eqref{eq:mapxt}.\\
\end{longtable}}

In this first section of general measure theory preliminaries, we consider $\mathscr{X}, \mathscr{Y}$ to be Lusin completely regular topological spaces. We recall that a topological space $\mathscr{X}$ is \emph{completely regular} if it is Hausdorff and for every closed set $C$ and point $x\in
\mathscr{X}\setminus C$ there exists $f: \mathscr{X} \to [0,1]$ continuous function
s.t.~$f(C)=\{1\}$ and $f(x)=0$.  A Hausdorff topological space is \emph{Lusin} if its topology is coarser than a Polish topology. 
This general setting is convenient for our analysis which deals with Borel probability measures defined in (subsets of) a separable Hilbert space $\mathscr{X}$, which could be endowed with the strong or the weak topology.

We denote by $\prob(\mathscr{X})$ the set of Borel probability measures on $\mathscr{X}$ endowed with the weak/narrow topology induced by the duality with the space of real valued continuous and bounded functions $\rmC_b(\mathscr{X})$. Thus, given a directed set $\mathbb A$, we say that a net $(\mu_\alpha)_{\alpha\in \mathbb A} \subset \prob(\mathscr{X})$ converges narrowly to $\mu \in \prob(\mathscr{X})$, and we write $\mu_\alpha \to \mu$ in $\prob(\mathscr{X})$, if
\begin{equation*} \lim_{\alpha} \int_\mathscr{X} \varphi \de \mu_\alpha = \int_\mathscr{X} \varphi \de \mu\quad\text{for every }\varphi\in\rmC_b(\mathscr{X}). \end{equation*}

Given $\mu\in \prob(\mathscr{X})$ and a Borel function $f: \mathscr{X} \to \mathscr{Y}$, we define the \emph{push-forward} $f_\sharp\mu\in\prob(\mathscr{Y})$ of $\mu$ through $f$ by
\[ \int_{\mathscr{Y}} \varphi \de (f_{\sharp}\mu) = \int_\mathscr{X} \varphi \circ f \de \mu \]
for every $\varphi:\mathscr{Y}\to\R$ bounded (or nonnegative) Borel function.

We recall the so-called \emph{disintegration theorem} (see e.g. \cite[Theorem 5.3.1]{ags}).
\begin{theorem}\label{thm:disintegr}
Let $\mathscr{W}, \mathscr{X}$ be Lusin completely regular topological spaces, $\mmu\in\prob(\mathscr{W})$ and $r:\mathscr{W}\to \mathscr{X}$ a Borel map. Denote with $\mu=r_{\sharp}\mmu\in\prob(\mathscr{X})$. Then there exists a $\mu$-a.e.~uniquely determined Borel family of probability measures $\{\mmu_x\}_{x\in \mathscr{X}}\subset\prob(\mathscr{W})$ such that $\mmu_x(\mathscr{W}\setminus r^{-1}(x))=0$ for $\mu$-a.e. $x\in \mathscr{X}$, and
\[\int_{\mathscr{W}}\varphi(w)\de\mmu(w)=\int_\mathscr{X}\left(\int_{r^{-1}(x)}\varphi(w)\de\mmu_x(w)\right)\de\mu(x)\]
for every bounded Borel map $\varphi:\mathscr{W}\to\R$.
\end{theorem}

\begin{remark}\label{rmk:disintegr}
When $\mathscr{W}=\mathscr{X}_1\times \mathscr{X}_2$ and $r$ is the projection $\pi^1$ on the first component, we can canonically
identify the disintegration $\{\mmu_x\}_{x\in \mathscr{X}_1} \subset
\prob(\mathscr{W})$ of $\mmu\in\prob(\mathscr{X}_1\times \mathscr{X}_2)$ w.r.t.~$\mu=\pi^1_\sharp\mmu$
with a family of probability measures $ \{\mu_{x_1}\}_{x_1\in \mathscr{X}_1} \subset
\prob(\mathscr{X}_2)$. We write
$\mmu=\int_{\mathscr{X}_1}\mu_{x_1}\,\d \mu(x_1)$.
\end{remark}

Given $\mu\in\prob(\mathscr{X})$, $\nu\in\prob(\mathscr{Y})$, we define the set of admissible transport plans
\begin{equation}\label{def:admplans}
 \Gamma(\mu, \nu) := \left  \{ \ggamma \in \prob(\mathscr{X} \times \mathscr{Y}) \mid \pi^{1}_{\sharp} \ggamma = \mu \, , \, \pi^{2}_{\sharp} \ggamma = \nu \right \},
\end{equation}
where we denoted by $\pi^i$, $i=1,2$, the projection on the $i$-th component and we call $\pi^{i}_{\sharp}\ggamma$ the $i$-th marginal of $\ggamma$.

\subsection{Wasserstein distance in Hilbert spaces and strong-weak topology}
From now on, we denote by $\X$ a separable (possibly infinite dimensional) Hilbert space with
norm $|\cdot|$ and scalar product $\la\cdot,\cdot\ra$. When it is necessary to specify it, we denote by $\X^s$ (resp.~$\X^w$) the Hilbert space $\X$ endowed with its strong (resp.~weak) topology. We remark that $\X^w$ is
a Lusin completely regular space and that $\X^s$ and $\X^w$
share the same class of Borel sets
and thus of Borel probability measures. Therefore, we are allowed to adopt the simpler notation $\prob(\X)$ and to use the heavier $\prob(\X^s)$ and $\prob(\X^w)$
only when we will refer to the corresponding topology.\\
We adopt the notation 
$\TX$ for the tangent bundle to $\X$, which is identified with the cartesian product $\X\times\X$
with the induced norm $|(x,v)|:=\big(|x|^2+|v|^2\big)^{1/2}$ and the 
strong-weak topology of $\X^s \times \X^w$(i.e.~the product of the strong topology on the
first component and the weak topology on the second one). The set $\prob(\TX)$ is defined thanks to the identification of
$\TX$ with $\X\times \X$ and it is endowed with
the narrow topology induced by the strong-weak topology in $\TX$.\\
We will denote by $\sfx,\sfv:\TX\to\X$ the projection maps defined by
\begin{equation}\label{eq:projandexp}
  \sfx(x,v):=x,\quad \sfv(x,v)=v.
\end{equation}
When dealing with the product space $\X^2$
we use the notation
\begin{alignat}{3}\label{eq:switch}
\mathsf s&: \X^2 \to \X^2, \quad &&\mathsf s (x_0,x_1):=(x_1,x_0),&&\\
\label{eq:mapxt}
  \sfx^t&:\X^2\to \X,\quad &&\sfx^t(x_0,x_1):=(1-t)x_0+tx_1,\quad &&t\in[0,1].
\end{alignat}

\begin{definition}\label{def:wassmom} 
  Given $\mu \in \prob(\X)$ and $\Phi \in \prob(\TX)$
  we define
\begin{equation}\label{eq:defsqmPhi}
 \sqm \mu := \int_\X |x|^2 \de \mu(x), \quad |\Phi|_2^2:= \int_{\TX} |v|^2 \de \Phi(x,v)
\end{equation}
and the spaces
\begin{equation}\label{condTanTX}
\prob_2(\X) := \{ \mu \in \prob(\X) \mid \rsqm\mu < + \infty \}, \quad \relcP 2{\mu}\TX := \Big\{\Phi\in \prob(\TX): \sfx_\sharp \Phi = \mu, \, |\Phi|_2<+\infty\Big\}.
\end{equation}
Given $\Phi\in\relcP 2\mu\TX$, the \emph{barycenter of $\Phi$} is the
function
$\bry{\Phi}\in L^2(\X,\mu;\X)$ defined by
\begin{equation}
\label{eq:barycenter}
    \bry{\Phi}(x):=\int_\X v\de\Phi_x(v) \quad\text{for }\mu\text{-a.e. }x\in \X,
\end{equation}
where $\{\Phi_x\}_{x\in \X}\subset\prob_2(\X)$ is the disintegration
of $\Phi$ w.r.t.~$\mu$.
We set $\bri\Phi:=(\ii_\X,\bry\Phi)_\sharp\mu$.
We say that $\Phi$ is \emph{concentrated on a map} (or that it is
\emph{deterministic}) if $\Phi=\bri\Phi$.
\end{definition}

For the following recalls on Wasserstein spaces we refer e.g. to \cite[\S 7]{ags}. 
On $\prob_2(\X)$ we define the $L^2$-Wasserstein distance $W_2$ by
\begin{align} \label{eq:w21} W_2^2(\mu, \nu) &:= \inf \left \{ \int_{\X^2} |x-y|^2 \de \ggamma(x,y) \mid \ggamma \in \Gamma(\mu, \nu) \right \}.
\end{align}
For the sequel, the set $\Gamma_o(\mu, \nu)$ denotes the subset of admissible plans in $\Gamma(\mu, \nu)$ realizing the infimum in
\eqref{eq:w21}.
We say that a measure $\ggamma\in \prob_2(\X\times\X)$ is optimal if
$\ggamma\in
\Gamma_o(\pi^1_\sharp\ggamma,\pi^2_\sharp\ggamma)$. We recall that $\ggamma \in \prob_2(\X \times \X)$ is optimal if and only if its support is cyclically monotone i.e.
\begin{equation}
  \label{eq:3}
  \begin{gathered}
    \text{for every $N\in \N$ and
      $\{(x_n,y_n)\}_{n=1}^N\subset \supp\ggamma$ with $x_0:=x_N$ we
      have}\\
    \sum_{n=1}^N \la y_n,x_n-x_{n-1}\ra\ge0.
  \end{gathered}  
\end{equation}
We recall that the metric space $(\prob_2(\X), W_2)$ is a complete and separable metric space and the $W_2$-convergence (sometimes denoted with $\overset{W_2}{\longrightarrow}$) is stronger than the narrow convergence. More precisely, if $(\mu_n)_{n\in\N}\subset\prob_2(\X)$ and $\mu\in\prob_2(\X)$, the following holds (see \cite[Remark 7.1.11]{ags})
\begin{equation*}
  \mu_n\overset{W_2}{\longrightarrow}\mu,\text{ as }n\to+\infty \quad\Longleftrightarrow\quad\begin{cases}\mu_n\to\mu \text{ in }\prob(\X^s),\\
    \rsqm{\mu_n}\to\rsqm\mu,
  \end{cases}
  \text{ as }n\to+\infty.\\
\end{equation*}
In the following Definition \ref{def:p2sw} and Proposition \ref{prop:finalmente}, we recall the topology of $\prob_2^{sw}(\TX)$ (see \cite{NaldiSavare,CSS}).

\begin{definition}[Strong-weak topology in $\prob_2(\TX)$]
\label{def:p2sw} We denote by $\prob_2^{sw}(\TX)$ the space $\prob_2(\TX)$ endowed with the coarsest topology which makes the following functions continuous
    \begin{equation*}
      \Phi\mapsto \int_{\TX}\zeta(x,v)\,\d\Phi(x,v),\quad \zeta \in \testsw \TX,
    \end{equation*}
    where $\testsw \TX$ is the Banach space of test functions $  \zeta:\TX\to \R$ such that
\begin{gather*}
  \label{eq:36}
  \zeta\text{ is sequentially continuous in $\X^s\times \X^w$,}\\
  \forall\,\eps>0\ \exists\,A_\eps\ge 0: |\zeta(x,v)|\le
  A_\eps(1+|x|^2)+\eps |v|^2\quad\text{for every }(x,v)\in \TX.
\end{gather*}
  \end{definition}
The following proposition (whose proof can be found in \cite{NaldiSavare}) summarizes some of the properties of the topology of $\prob_2^{sw}(\TX)$.

\begin{proposition}
    \label{prop:finalmente}
    \
    \begin{enumerate}
    \item If $(\Phi_n)_{n \in\mathbb N}\subset \prob_2(\TX)$ is
      a sequence and $\Phi \in \prob_2(\TX)$, then $\Phi_n\to\Phi$ in $\prob_2^{sw}(\TX)$ as $n\to+\infty$ if and only if 
      \begin{enumerate}
      \item $\Phi_n\to\Phi$ in $\prob(\TX)= \prob(\X^s\times \X^w)$,
      \item $\displaystyle \lim_{n \to + \infty}\int_{\TX} |x|^2\,\d\Phi_n(x,v)=\int_{\TX} |x|^2\,\d\Phi(x,v)$,
      \item $\displaystyle \sup_{n} \int_{\TX}|v|^2\,\d\Phi_n(x,v)<+\infty$.
      \end{enumerate}
    \item For every compact set $\mathcal K\subset \prob_2(\X^s)$ and
      every constant $c<+\infty$ the sets
      \begin{equation*}
        \mathcal K_c:=\Big\{\Phi\in \prob_2(\TX):
        \sfx_\sharp\Phi
        \in
      \mathcal K,\quad
      \int_{\TX} |v|^2\,\d\Phi(x,v)\le c\Big\}
    \end{equation*}
    are sequentially compact in
    $\prob_2^{sw}(\TX)$.
    \end{enumerate}
  \end{proposition}

\medskip

For the sequel, we recall the concept and main properties of geodesics in $\prob_2(\X)$. Given an interval $\interval\subset\R$, we denote equivalently by $\mu(t)$ or $\mu_t$ the evaluation at time
$t\in \interval$ of a curve $\mu:\interval\to \prob_2(\X)$.
\begin{definition}[Geodesics]\label{def:W2geodesic}
A curve $\mu:[0,1]\to\prob_2(\X)$ is said to be a \emph{(constant speed) geodesic} if for all $0\leq s\leq t\leq 1$
we have
\[W_2(\mu_s,\mu_t)=(t-s)W_2(\mu_0,\mu_1).\]
We also say that $\mu$ is a \emph{geodesic from $\mu_0$ to $\mu_1$}. 
\end{definition}

\begin{definition}[Geodesic and total convexity]
We say that $A\subset\prob_2(\X)$ is a \emph{geodesically convex} set if for any pair $\mu_0,\mu_1\in A$ there exists a geodesic $\mu:[0,1]\to\prob_2(\X)$ from $\mu_0$ to $\mu_1$ such that $\mu_t\in A$ for all $t\in[0,1]$.\\
We say that $A\subset\prob_2(\X)$ is \emph{totally convex}  if for any pair $\mu_0,\mu_1\in A$ and any coupling $\ggamma\in\Gamma(\mu_0,\mu_1)$, we have that $(\sfx^t)_\sharp\ggamma\in A$ for any $t\in [0,1]$.
\end{definition}
\begin{remark}
\label{rem:total-conv}
 Since total convexity will play a crucial role in the present paper, 
let us recall a few examples of totally convex sets in $\prob_2(\X)$, which
are induced by a lower semicontinuous and convex function
$P:\X\to (-\infty,+\infty]$ 
and a real number $c$: the sets of measures $\mu\in \prob_2(\X)$ satisfying one of the following conditions:
\begin{equation*}
        P\Big(\int_\X x\,\d\mu(x)\Big)\le c,\quad
        \int_\X P(x)\,\d\mu(x)\le c,\quad
        \int_{\X^2} P(x-y)\,\d\left(\mu\otimes \mu\right)(x,y)\le c.
\end{equation*}
    Clearly, one can replace large with strict inequalities in the previous
    formulae. Choosing $P$ as the indicator function of 
    a convex set $\mathsf U\subset \X$ (i.e.~$P(x)=0$ if $x\in \mathsf U$,
    $P(x)=+\infty$ otherwise), one obtains conditions confining
    the barycenter, $\supp\mu$, or 
    $\supp\mu-\supp\mu$ to a given set $\mathsf U$.
\end{remark}
The following useful result (see \cite[Theorem 7.2.1, Theorem 7.2.2]{ags}  for the first part and \cite[Lemma 5.29]{santambrogio} or the proof of \cite[Lemma 3.20]{CSS} for the last assertion) on geodesics also points out that total convexity is stronger than geodesic convexity.
\begin{theorem}[Properties of geodesics] \label{theo:chargeo}
Let $\mu_0,\mu_1\in \prob_2(\X)$ and $\mmu\in \Gamma_o(\mu_0,\mu_1)$.
Then $\mu:[0,1]\to\prob_2(\X)$ defined by
\begin{equation}\label{eq:charact_geodesic}
\mu_t := (\sfx^t)_{\sharp} \mmu,\quad t\in[0,1],
\end{equation}
is a (constant speed) geodesic from $\mu_0$ to $\mu_1$. Conversely, any (constant speed) geodesic  $\mu$ from $\mu_0$ to $\mu_1$ admits the representation~\eqref{eq:charact_geodesic} for a suitable plan $\mmu\in \Gamma_o(\mu_0,\mu_1)$.\\
If $\mu$ is a geodesic connecting $\mu_0$ to $\mu_1$, then for every $t \in (0,1)$ there exists a unique optimal plan $\mmu_{t0}$ between $\mu_t$ and $\mu_0$ (resp.~$\mmu_{t1}$ between $\mu_t$ and $\mu_1$) and it is concentrated on a map w.r.t.~$\mu_t$, meaning that there exist Borel maps $\rr_t, \rr'_t: \X \to \X$ such that 
\[ \mmu_{t0}= (\ii_\X, \rr_t)_\sharp \mu_t, \quad \mmu_{t1}= (\ii_\X, \rr_t')_\sharp \mu_t.\]
 Finally, the map $\sfx^t$ is $\mmu$-essentially injective. 
\end{theorem}

\medskip

The following defines the counterpart of $\rmC^{\infty}_c(\R^d)$ when $\R^d$ is replaced by $\X$.
\begin{definition}[The space $\Cyl(\X)$ of cylindrical functions]\label{def:Cyl}
 Given $d \in \N$, we  denote by $L_d(\X)$  the space of  all  linear maps $\pi:\X\to
\R^d$ of the form $\pi(x)=(\la x,e_1\ra,\cdots,\la x,e_d\ra )$
where $\{e_1,\dots,e_d\}$ is any orthonormal family of vectors in  $\X$.
A function $\varphi: \X \to \R$ belongs to the space of cylindrical functions on $\X$, $\Cyl(\X)$, if it is of the form
\[ \varphi = \psi \circ \pi\]
where $\pi\in  L_d(\X)$ and $\psi \in \mathrm C^\infty_c(\R^d)$  for some $d \in \N$. 
\end{definition}

Given $\nu\in\prob_2(\X)$, we define the tangent space to $\prob_2(\X)$ at $\nu$ by
\[\Tan_{\nu}\prob_2(\X) :={} \overline{\{ \nabla \varphi \mid
    \varphi \in \Cyl(\X) \}}^{L^2(\X,\nu;\X)}.\]
If $\interval\subset\R$ is an open interval and $\mu:\interval \to\prob_2(\X)$ is a locally absolutely
continuous curve, we define the \emph{metric velocity} of $\mu$ at $t\in\interval$ as
\[|\dot \mu_t|^2:=\lim_{h\to
    0}\frac{W_2^2(\mu_{t+h},\mu_t)}{h^2},\]
which exists for a.e. $t\in\interval$.

The following result (see \cite[Theorem 8.3.1, Proposition
8.4.5 and Proposition 8.4.6]{ags}) characterizes
locally absolutely continuous curves in $\prob_2(\X)$.

\begin{theorem}[Wasserstein velocity field]
  \label{thm:tangentv}
Let $\mu:\interval \to\prob_2(\X)$ be a locally absolutely
continuous curve defined in an open interval $\interval\subset \R$.
There exists a Borel vector field
$\vv:\interval\times \X\to \X$
and
a set $A(\mu) \subset \interval$ with
$\leb(\interval \setminus A(\mu))=0$ such that
for every $t\in A(\mu)$ the following hold
\begin{enumerate}
    \item $\vv_t\in\Tan_{\mu_t}\prob_2(\X)$;
    \item $\int_{\X} |\vv_t|^2\,\d\mu_t=|\dot \mu_t|^2$;
    \item the continuity equation $\partial_t\mu_t+\nabla\cdot(\vv_t\mu_t)=0$
holds in the sense of distributions in $\interval\times \X$, i.e.
\begin{equation*}
    \frac\d{\d t}\int_\X \zeta\,\d\mu_t=
    \int_\X \langle \vv_t(x),\nabla\zeta(x)\rangle
    \,\d\mu_t(x)
    \quad\text{for every $\zeta\in \Cyl(\X)$ and a.e.~$t\in\interval$.}
\end{equation*}
\end{enumerate}
Moreover, $\vv_t$ is uniquely determined
in $L^2(\X,\mu_t;\X)$ for $t\in A(\mu)$ and 
\begin{equation*}
 \lim_{h \to 0} \frac{W_2((\ii_\X+h\vv_t)_{\sharp}\mu_t,
   \mu_{t+h})}{|h|} =0 \quad \text{for every }t \in A(\mu).
\end{equation*}
\end{theorem}

\subsection{Duality pairings}\label{sec:prelimCSS}
In this subsection we collect the main objects involving
duality pairings between measures in $\prob_2(\TX).$
We report here a summary of the results needed in the sequel and we refer to \cite{CSS} for a wider discussion on this matter.

As usual, we denote by $\sfx^0,\sfv^0,\sfx^1:\TX\times\X\to\X$ the projection maps of a point $(x_0,v_0,x_1)$ into $x_0$, $v_0$ or $x_1$, respectively (and similarly with $\sfx^0,\sfv^0,\sfx^1,\sfv^1$ when they are defined in $\TX\times\TX$).
\begin{definition}[Metric-duality pairings] \label{def:pairings} For every $\Phi_0, \Phi_1 \in \prob_2(\TX)$, $\mu_1 \in \prob_2(\X)$, $\ttheta\in \prob_2(\X \times \X)$, $t \in [0,1]$ and $\Psi \in \relcP2{\sfx^t_\sharp \ttheta}{\TX}$, we set
\begin{align*}
    \Lambda(\Phi_0, \mu_1)&:= \left \{ \ssigma \in \Gamma(\Phi_0, \mu_1)
    \mid
    (\sfx^0,\sfx^1)_{\sharp}\ssigma \in  \Gamma_o(\sfx_\sharp \Phi_0, \mu_1) \right
      \},\\
      \Lambda(\Phi_0, \Phi_1)&:= \left \{ \Ttheta \in \Gamma(\Phi_0, \Phi_1)
                         \mid (\sfx^0,\sfx^1)_{\sharp} \Ttheta\in
                          \Gamma_o(\sfx_\sharp \Phi_0, \sfx_\sharp \Phi_1)
    \right \},\\
    \Gamma_t(\Psi,\ttheta)&:=
    \left \{ \ssigma \in \prob_2(\TX\times \X)
    \mid
    (\sfx^0,\sfx^1)_{\sharp}\ssigma =\ttheta,\quad
    (\sfx^t \circ(\sfx^0,\sfx^1),\sfv^0)_\sharp \ssigma=\Psi\right \}.
  \end{align*}
  We set
  \begin{align*}
    \bram{\Phi_0}{\mu_1} &:= \min \left \{ \int_{\TX \times
                                   \X} \scalprod{x_0-x_1}{v_0} \de \ssigma
                                   \mid \ssigma \in \Lambda(\Phi_0,
                                   \mu_1) \right \}, \\
    \brap{\Phi_0}{\mu_1} &:= \max \left \{ \int_{\TX \times
                                   \X} \scalprod{x_0-x_1}{v_0} \de \ssigma
                                   \mid \ssigma \in \Lambda(\Phi_0,
                                   \mu_1) \right \}, \\
\bram{\Phi_0}{\Phi_1} &:= \min \left \{ \int_{\TX \times \TX} \scalprod{x_0-x_1}{v_0-v_1} \de \Ttheta \mid \Ttheta \in \Lambda(\Phi_0, \Phi_1) \right \},\\
\brap{\Phi_0}{\Phi_1} &:= \max \left \{ \int_{\TX \times \TX} \scalprod{x_0-x_1}{v_0-v_1} \de \Ttheta \mid \Ttheta \in \Lambda(\Phi_0, \Phi_1) \right \},\\
    \directionalm{\Psi}{\ttheta}t
    & := \min \left \{ \int_{\TX\times\X}
      \scalprod{x_0-x_1}{v_0} \de \ssigma(x_0,v_0,x_1)
      \mid \ssigma \in \Gamma_t(\Psi,\ttheta)
      \right \}, \\
    \directionalp{\Psi}{\ttheta}t
    & := \max \left \{ \int_{\TX\times\X}
      \scalprod{x_0-x_1}{v_0} \de \ssigma(x_0,v_0,x_1)
      \mid \ssigma \in \Gamma_t(\Psi,\ttheta)
      \right \}.
  \end{align*}
\end{definition}
The following theorem summarizes some of the properties of duality pairings analyzed in \cite{CSS}.
\begin{theorem}\label{thm:all} The following properties hold.
\begin{enumerate}
\item \emph{(Inversion)} For every $\ttheta\in \prob_2(\X^2)$, $t\in[0,1]$,
  $\Psi\in \relcP2{\sfx^t_\sharp \ttheta}\TX$ it holds 
\[\directionalm{\Psi}{\ttheta}t=- \directionalp{\Psi}{\mathsf s_\sharp\ttheta}{1-t},\]
where $\mathsf s$ is as in \eqref{eq:switch}.
\item \emph{(Comparison)} For every $\mu_0, \mu_1 \in \prob_2(\X)$ and every $\Phi_0\in\relcP2{\mu_0}\TX$, $\Phi_1\in\relcP2{\mu_1}\TX$, it holds
\begin{alignat*}{2}
&\bram{\Phi_0}{\mu_1} = \min_{\ttheta \in \Gamma_o(\mu_0, \mu_1)} \directionalm{\Phi_0}{\ttheta}0, \quad  &&\brap{\Phi_0}{\mu_1} = \max_{\ttheta \in \Gamma_o(\mu_0, \mu_1)} \directionalp{\Phi_0}{\ttheta}0,\\
&\bram{\Phi_0}{\mu_1} + \bram{\Phi_1}{\mu_0} \le \bram{\Phi_0}{\Phi_1}, \quad &&\brap{\Phi_0}{\mu_1} + \brap{\Phi_1}{\mu_0} \ge \brap{\Phi_0}{\Phi_1},
\end{alignat*}
and
\begin{equation*}
 \bram{\Phi_0}{\Phi_1}\le \bram{\Phi_0}{\mu_1} + \brap{\Phi_1}{\mu_0} \le \brap{\Phi_0}{\Phi_1}.
\end{equation*}
\item \emph{(Restriction)} For every $\ttheta \in \prob_2(\X^2)$, every $0 \le s < t \le 1$ and every $\Phi \in \relcP{2}{\sfx^s_\sharp \ttheta}{\TX}$, $\Psi \in \relcP{2}{\sfx^t_\sharp \ttheta}{\TX}$ we have 
\begin{equation} \label{eq:restrsp} 
\directionalm {\Phi} \ttheta s = \frac{1}{t-s} \directionalm{\Phi}{(\sfx^s, \sfx^t)_{\sharp}\ttheta}{0}, \quad \directionalp {\Psi} \ttheta t = \frac{1}{t-s} \directionalp{\Psi}{(\sfx^s, \sfx^t)_{\sharp}\ttheta}{1}.
\end{equation}
 
\item \emph{(Trivialization)} If $\ttheta\in \prob_2(\X^2)$, $t\in[0,1]$,
  $\Psi\in \relcP2{\sfx^t_\sharp \ttheta}\TX$ and $\sfx^t:\X^2\to\X$ is
  $\ttheta$-essentially injective
  or $\Psi$ is concentrated on a map,
  then $\Gamma_t(\Psi,\ttheta)$ 
  contains a unique element and 
  \begin{equation}
    \label{eq:29}
    \directionalm{\Psi}{\ttheta}t=
    \directionalp{\Psi}{\ttheta}t=\int_{\X^2}\scalprod{\bry{\Psi}\left(\sfx^t(x_0,x_1)\right)}{x_0-x_1}\de\ttheta(x_0,x_1),
  \end{equation}
  with $\bry{\Psi}$ the barycenter of $\Psi$ as in Definition \ref{def:wassmom}.
  
  \item \emph{(Semicontinuity)} Let $(\Phi_n^i)_{n\in \N} \subset \prob_2(\TX)$ be
  converging to $\Phi^i$ in $\prob_2^{sw}(\TX)$, $i=0,1$, let $(\ttheta_n)_{n \in \N} \subset \prob_2(\X^2)$ be converging to $\ttheta$ in $\prob_2(\X^2)$, let $(\nu_n )_{n\in\N} \subset \prob_2(\X)$ be converging to $\nu$ in $
  \prob_2(\X) $ and let $t \in [0,1]$. Then
\begin{alignat*}{2}
    &\liminf_{n\to+\infty} \bram{\Phi^0_n}{\nu_n} \ge \bram{\Phi^0}{\nu}, \quad &&\limsup_{n\to+\infty} \brap{\Phi^0_n}{\nu_n} \le \brap{\Phi^0}{\nu},\\
 &\liminf_{n\to+\infty}\bram{\Phi^0_n}{\Phi^1_n}\ge
    \bram{\Phi^0}{\Phi^1},\quad
    &&\limsup_{n\to+\infty}\brap{\Phi^0_n}{\Phi^1_n}\le
    \brap{\Phi^0}{\Phi^1},\\
     &\liminf_{n\to+\infty}\directionalm{\Phi_n^0}{\ttheta_n}t\ge
    \directionalm{\Phi^0}{\ttheta}t,\quad
    &&\limsup_{n\to+\infty}\directionalp{\Phi_n^0}{\ttheta_n}t\le
    \directionalp{\Phi^0}{\ttheta}t.
\end{alignat*}

\item  Let $\interval \subset \R$ be an open interval, let $\mu^1,\mu^2:\interval
  \to \prob_2(\X)$ be locally absolutely continuous curves and let $\vv^1,\vv^2: \interval \times \X\to \X$ be Borel vector fields such that $\|\vv^i_t\|_{L^2(\X,\mu^i_t; \X)} \in L^1_{loc}(\interval)$, $i=1,2$, and such that
  \[\partial_t\mu^i_t+\nabla\cdot(\vv^i_t\mu^i_t)=0\]
holds in the sense of distributions in $\interval\times \X$, $i=1,2$. Let $A({\mu^1}), A({\mu^2})\subset\interval$ be as in Theorem \ref{thm:tangentv}. Then
\begin{enumerate}
    \item for every $\nu \in \prob_2(\X)$ and every $t \in A(\mu^i)$, $i=1,2$, it holds
    \begin{align*}
\lim_{h \downarrow 0} \frac{W_2^2(\mu^i_{t+h}, \nu)-W_2^2(\mu^i_t,
  \nu)}{2h}
      &= \bram{(\ii_\X , \vv^i_t)_{\sharp}\mu^i_t}{\nu},\\
      \notag\lim_{h \uparrow 0} \frac{W_2^2(\mu^i_{t+h}, \nu)-W_2^2(\mu^i_t, \nu)}{2h}                                                  &=\brap{(\ii_\X , \vv^i_t)_{\sharp}\mu^i_t}{\nu};
\end{align*}
    \item there exists a subset $A\subset A({\mu^1})\cap A({\mu^2})$ of full Lebesgue
measure
such that $s \mapsto W_2^2(\mu^1_s, \mu^2_s)$ is
differentiable in $A$ and for every $t\in A$ it holds
\begin{equation*}
  \begin{aligned}
    \frac12\frac{\d}{\d t}W_2^2(\mu^1_t,\mu^2_t)
    &= \bram{(\ii_\X ,
      \vv^1_t)_{\sharp}\mu^1_t}{(\ii_\X , \vv^2_t)_{\sharp}\mu^2_t}=
    \brap{(\ii_\X , \vv^1_t)_{\sharp}\mu^1_t}{(\ii_\X ,
      \vv^2_t)_{\sharp}\mu^2_t}.
  \end{aligned}
\end{equation*}
\end{enumerate}
\end{enumerate}
\end{theorem}
\begin{proof} We give a few references for the proofs. Property (1) is \cite[(3.27)]{CSS}. Property (2) comes from the definition and \cite[Corollary 3.7]{CSS}. We sketch the proof only for the last property in (2): take $\ssigma \in \Lambda(\Phi_0,\mu_1)$ such that
\[\bram{\Phi_0}{\mu_1}=\int_{\TX \times\X} \scalprod{x_0-x_1}{v_0} \de \ssigma,\]
and consider $\Ttheta\in\prob_2(\TX\times\TX)$ such that $(\sfx^0,\sfv^0,\sfx^1)_\sharp\Ttheta=\ssigma$ and $(\sfx^1,\sfv^1)_\sharp\Ttheta=\Phi_1$. { Notice that such a measure $\Ttheta$ exists by disintegration and gluing arguments.} Then $\Ttheta\in\Lambda(\Phi_0,\Phi_1)$, so that
\begin{align*}
\bram{\Phi_0}{\Phi_1}&\le\int_{\TX \times\TX} \scalprod{x_0-x_1}{v_0-v_1} \de \Ttheta\\
&=\int_{\TX \times\X} \scalprod{x_0-x_1}{v_0} \de \ssigma+\int_{\TX \times\TX} \scalprod{x_1-x_0}{v_1} \de \Ttheta\\
&\le \bram{\Phi_0}{\mu_1}+\brap{\Phi_1}{\mu_0}.    
\end{align*}
The strategy for proving the remaining inequality in (2) is identical.

Assertion (3) follows from the fact that, if we define $T^{s,t}: \TX \times \X \to \TX \times \X$ and $\mathcal{L}: \prob_2(\TX \times \X) \to \R$ as
  \[ T^{s,t}(x_0, v_0, x_1) := (\sfx^s(x_0, x_1), v_0, \sfx^t(x_0, x_1)), \quad \quad \mathcal{L}(\ssigma) := \int_{\TX \times \X} \scalprod{v_0}{x_0-x_1}\de \ssigma(x_0, v_0, x_1), \]
it is clear that 
 \[ \directionalm {\Phi} \mmu s = \inf \left \{ \mathcal{L}(\ssigma) \mid \ssigma \in \Gamma_s( \Phi, \mmu) \right \}, \quad \directionalm{\Phi}{(\sfx^s, \sfx^t)_{\sharp}\mmu}{0} = \inf \left \{ \mathcal{L}(\ssigma) \mid \ssigma \in \Gamma_0( \Phi, (\sfx^s, \sfx^t)_{\sharp}\mmu) \right \}.\]
Then, the first equality in the statement follows noting that $T^{s,t}_{\sharp}(\Gamma_s( \Phi, \mmu))=\Gamma_0( \Phi, (\sfx^s, \sfx^t)_{\sharp}\mmu)$ and that $\mathcal{L}(T^{s,t}_{\sharp} \ssigma) = (t-s) \mathcal{L}(\ssigma)$ for every $\ssigma \in \prob_2(\TX \times \X)$. The second equality follows from the first one and (1). Item (4) is \cite[Remark 3.19]{CSS}. Item (5) easily follows by \cite[Lemma 3.15]{CSS}. Finally, item (6) is provided by \cite[Theorem 3.11, Theorem 3.14, Remark 3.12]{CSS}.
\end{proof}
\subsection{Multivalued probability vector fields, metric dissipativity and EVI solutions}
\label{subsec:prelimCSS2}
We recall now the 
main definition of Multivalued Probability Vector Field and of metric dissipativity.
\begin{definition}[Multivalued Probability Vector Field - \MPVF] \label{def:MPVF}
  A \emph{multivalued probability vector field} $\frF$ is a nonempty subset of
  $\prob_2(\TX)$ with $\dom(\frF) := \sfx_\sharp(\frF)=
  \{ \sfx_\sharp\Phi:\Phi\in \frF \}$.
  Given any $\mu \in \prob_2(\X)$, we define the \emph{section} $\frF[\mu]$
  of $\frF$ as 
  \[ \frF[\mu] := \left \{ \Phi \in \frF
      \mid \sfx_{\sharp}\Phi = \mu \right \}. \]
 We say that $\frF$ is a 
 {\em Probability Vector Field
 (\PVF)} if $\sfx_\sharp $ is injective in $\frF$, i.e.~$\frF[\mu]$ contains a
  unique element for every $\mu\in \dom(\frF)$.\\
  A \emph{selection} $\frF'$ of a \MPVF $\frF$ is a \PVF
   such that $\frF'\subset \frF$ and $\dom(\frF')=\dom(\frF)$.\\
  A \MPVF\ $\frF\subset\prob_2(\TX)$ is \emph{deterministic} 
  or \emph{concentrated on maps} if
  every $\Phi\in \frF$ is deterministic (see Definition \ref{def:wassmom}).
\end{definition}
 Starting from a \MPVF $\frF$, 
 the barycentric projection \eqref{eq:barycenter}
 induces a deterministic \MPVF 
 which we call $\bri\frF$, defined by

\begin{equation}\label{eq:brimpvf}
\bri\frF[\mu]:= \left \{ \bri\Phi=(\ii_\X, \bry\Phi)_\sharp\mu : \Phi \in \frF[\mu] \right \}, \quad \mu \in \dom(\frF).
\end{equation} 
We will also use the notation
\begin{equation}\label{eq:mapss}
\maps \frF [\mu]:= \left \{ \ff \in L^2(\X, \mu; \X) : (\ii_\X, \ff)_\sharp \mu \in \frF[\mu] \right \}, \quad \mu \in \dom(\frF),
\end{equation}
to extract the deterministic part of a \MPVF $\frF$: notice 
that a \MPVF $\frF$ is deterministic if and only if
$\frF=\bri\frF=\left \{ (\ii_\X,\ff)_\sharp \mu : \ff \in \maps{\frF}[\mu], \, \mu \in \dom(\frF) \right \}$.
\newcommand{\jointmap}[1]{{\bm {#1}}}
Conversely, 
for a given set $D \subset \prob_2(\X)$, define
\begin{equation}\label{eq:2}
    \Sp{\X, D}:= \left \{ (x,\mu) \in \X \times D \mid x \in \supp(\mu) \right \},\quad
    \Sp\X:=\Sp{\X, \prob_2(\X)},
\end{equation}
and let us consider a continuous map $\jointmap f:\Sp{\X, D}\to \X$.
If, for every $\mu\in D$,
the integral
$\int_\X |\ff(x,\mu)|^2\,\d\mu(x)$
is finite, 
then $\jointmap f$ induces a \PVF 
$\frF$ defined by
$$\frF=\big\{(\ii_\X,\jointmap f(\cdot,\mu))_\sharp\mu:\mu\in D\big\},\quad
\dom(\frF)=D.$$ 
We often adopt the convention to write $\ff[\mu]$ for the function
\begin{equation*}
    \ff[\mu](x):=\ff(x,\mu),\quad x\in \supp(\mu),
\end{equation*}
in particular when $\ff[\mu]$ is just an element
of $L^2(\X,\mu;\X)$.

\begin{definition}[Metrically $\lambda$-dissipative \MPVF] \label{def:dissipative}
  A \MPVF $\frF \subset \prob_2(\TX)$ is (metrically) \emph{$\lambda$-dissipative},
  $\lambda\in \R$, if
    \begin{equation}
      \bram{\Phi_0}{\Phi_1} \le \lambda W_2^2(\mu_0,\mu_1)
      \quad \forall\,\Phi_0,\Phi_1\in \frF,\ \mu_0=\sfx_\sharp \Phi_0,\ 
    \mu_1=\sfx_\sharp \Phi_1.
  \label{eq:33}
\end{equation}
When $\lambda=0$, we simply say that $\frF$ is dissipative.
\end{definition}
\begin{remark}\label{rmk:equivdiss}
  Thanks to  Theorem \ref{thm:all}(2), \eqref{eq:33} implies the weaker condition
  \begin{equation}\label{Hdiss}
    \bram{\Phi_0}{\mu_1}+\bram{\Phi_1}{\mu_0} \le \lambda W_2^2(\mu_0,
    \mu_1),\quad \forall\, \Phi_0,\Phi_1 \in \frF,\ \mu_0=\sfx_\sharp \Phi_0,\ 
    \mu_1=\sfx_\sharp \Phi_1.
  \end{equation}
\end{remark}

Given a \MPVF $\frF\subset\prob_2(\TX)$, we define its $\lambda$-transformation, $\frF^\lambda$, and its opposite, $-\frF$, as
\begin{align}\label{eq:Flambda}
\frF^\lambda&:=L^\lambda_\sharp\frF=\left\{L^\lambda_\sharp\Phi\,:\,\Phi\in\frF\right\},\\ \label{eq:opposite}
-\frF&:= \left \{ (\sfx,-\sfv)_\sharp \Phi : \Phi \in \frF \right \},
\end{align}
where $L^\lambda:\TX\to\TX$ is the bijective map defined by
  \begin{equation*}
  L^\lambda:=(\sfx,\sfv-\lambda \sfx).
\end{equation*}
Similar to Remark \ref{rem:transff} for the case of operators in Hilbert spaces, we recall the following result (cf. \cite[Lemma 4.6]{CSS})

\begin{lemma}
    $\frF\subset\prob_2(\TX)$ is a $\lambda$-dissipative \MPVF (resp.~satisfies \eqref{Hdiss}) if and only if $\frF^\lambda$ is dissipative, i.e. $0$-dissipative (resp.~satisfies \eqref{Hdiss} with $\lambda=0$).
\end{lemma}

\begin{definition}\label{def:plangeodomV}
  Let $\frF \subset \prob_2(\TX)$, $\mu_0,\mu_1\in \dom(\frF)$. We define the set 
  \[ \CondGamma\frF{\mu_0}{\mu_1} := \left \{ \mmu \in \Gamma(\mu_0, \mu_1) \mid \sfx^t_\sharp \mmu \in \dom(\frF) \text{ for every }t \in [0,1] \right \}.\]
  If $\mmu \in \CondGamma\frF{\mu_0}{\mu_1}$ and $t \in [0,1]$, we define
\begin{align*}
  \directionalm \frF\mmu t := \sup \left \{
                          \directionalm{\Phi}{\mmu}t \mid \Phi \in \frF[\mu_t]
                          \right \},
                                    \qquad
  \directionalp \frF\mmu t := \inf \left \{
                          \directionalp{\Phi}{\mmu}t 
                          \mid \Phi \in \frF[\mu_t] \right \}.
\end{align*}
\end{definition}

In the following theorem we discuss the behaviour of duality pairings with $\frF$ along geodesics.
\begin{theorem} \label{theo:propflfr}
   Let $\frF$ be a \MPVF, let $\mu_0,\mu_1\in \dom(\frF)$ and let
  $\mmu\in\CondGamma\frF{\mu_0}{\mu_1} \cap \Gamma_o(\mu_0, \mu_1)$. 
If $\frF$ satisfies \eqref{Hdiss}, then the following properties hold.
\begin{enumerate}
\item $\directionalp \frF\mmu t \le \directionalm \frF\mmu t$ for every
  $t \in (0,1)$;
\item $\directionalm \frF\mmu s \le \directionalp \frF\mmu t+\lambda (t-s)\, W_2^2(\mu_0,\mu_1)$
  for every $0 \le s< t \le 1$;
\item $t\mapsto \directionalm \frF\mmu t+\lambda t\, W_2^2(\mu_0,\mu_1)$ and
  $t\mapsto \directionalp \frF\mmu t+\lambda t\,W_2^2(\mu_0,\mu_1) $ are increasing respectively in $[0,1)$ and in $(0,1]$;
\item
  $\directionalp \frF\mmu t = \directionalm \frF\mmu t$ at every
  point $t \in (0,1)$ where one of them is continuous and thus coincide outside a countable set.
\end{enumerate}
\end{theorem}
\begin{proof}
Item (1) immediately follows from the definition. Item (2) is proven in \cite[Theorem 4.9]{CSS}, while (3) and (4) follow from (2).
\end{proof}

\begin{proposition}
  \label{prop:closure}
  If $\frF$ is a $\lambda$-dissipative \MPVF then
  its sequential closure
  \begin{equation}
    \clo{\frF}:=\Big\{\Phi\in \prob_2(\TX):
    \exists\,\Phi_n\in \frF:\Phi_n\to\Phi\ \text{in
    }\prob_2^{sw}(\TX)\Big\}.
    \label{eq:sw-closure}
\end{equation} is $\lambda$-dissipative as well.
\end{proposition}
\begin{proof}
  It follows from Theorem \ref{thm:all}(5). See also \cite[Proposition 4.15]{CSS}.
\end{proof}

We recall the definition of $\lambda$-\wEVI solution for a \MPVF.
\begin{definition}[$\lambda$-Evolution Variational Inequality] \label{def:evi}
  Let $\frF$ be a \MPVF and let $\lambda \in \R$.
  We say that a continuous curve $\mu: \interval  \to \overline{\dom(\frF)}$
  is a \emph{$\lambda$-\wEVI solution} for the \MPVF $\frF$ if 
\begin{equation*}
    \begin{aligned}
      \frac12\frac\d{\d t} W_2^2(\mu_t,\sfx_\sharp \Phi)
      &\le \lambda W_2^2(\mu_t,\sfx_\sharp \Phi)-
      \bram{\Phi}{\mu_t} \text{ in } \mathscr{D}'(\intt{\interval}) \text{ for every } \Phi \in \frF,
    \end{aligned}
  \end{equation*}
 where the writing $\mathscr{D}'(\intt{\interval})$ means that the expression has to be understood in the distributional sense in $\intt{\interval}$.
\end{definition}

\begin{remark}\label{rem:added}
    In the classical theory, if $\mmo\subset \mathcal{H}\times\mathcal{H}$ is a $\lambda$-dissipative operator in a separable Hilbert space $\mathcal{H}$, then any differentiable solution to $\dot x(t)\in \mmo[x(t)]$ satisfies the associated $\lambda$-\wEVI, i.e.
    \[\frac12\frac\d{\d t} |x(t)-y|^2\le \lambda |x(t)-y|^2-\langle w,y-x(t)\rangle,\quad\text{for every }(y,w)\in\mmo.\]
    Maximality of $\mmo$ gives also the reverse implication.
    In our case of the space $\prob_2(\X)$, a full characterization of the $\lambda$-\wEVI notion of solution in Definition \ref{def:evi} with the solution of a continuity equation formulation of the measure differential equation $\dot\mu_t\in\frF[\mu_t]$ is done later in Section \ref{sec:totdissMPVF-flow}, in particular in Theorem \ref{thm:final-total}, following a Lagrangian approach. This requires appropriate assumptions on the \MPVF $\frF$. We refer the reader to \cite{CSS} for an alternative, metric-based, approach to this subject.
\end{remark}

\begin{remark}\label{rmk:EVIright}
In light of Theorem \ref{thm:all}(6a) and recalling \cite[Remark 5.2]{CSS}, an absolutely continuous curve $\mu: \interval  \to \overline{\dom(\frF)}$ is a \emph{$\lambda$-\wEVI solution} for the \MPVF $\frF$ if and only if
\begin{equation*}
    \begin{aligned}
      \lim_{h \downarrow 0} \frac{W_2^2(\mu_{t+h}, \nu)-W_2^2(\mu_t,
  \nu)}{2h}
      &\le \lambda W_2^2(\mu_t,\sfx_\sharp \Phi)-
      \bram{\Phi}{\mu_t}
    \end{aligned}
    \quad
    \text{for every}\
    t\in A(\mu)
    \text{ and every}\ \Phi\in \frF,
  \end{equation*}
  where $A(\mu) \subset\interval$ is as in Theorem \ref{thm:tangentv}.
\end{remark}

\section{Invariant dissipative operators in Hilbert spaces and totally dissipative {\MPVF}s} \label{sec:invmpvf}
From now on, $\X$ will denote a separable Hilbert space;
we will also consider 
a standard Borel 
space $(\Omega, \cB)$ 
endowed with a nonatomic probability measure $\P$ 
(see Appendix \ref{sec:appborel} and in particular Definition
\ref{def:sbs})
and the Hilbert space  $\cH$ defined by
\[\cH:=L^2(\Omega, \cB, \P; \X).\] 
We will use capital letters $X,Y,V,\ldots$ to denote elements of
$\cH$ (i.e.~$\X$-valued random variables).

We denote by $\iota:\cH \to \prob_2(\X)$
and
$\iotaT:\cH\times\cH \to \prob_2(\X\times\X)\equiv\prob_2(\TX)$
the push-forward operators
\begin{equation}
\iota(X):=X_\sharp \P,\qquad
\iotaT(X,V):=(X,V)_\sharp\P.\label{eq:1}
\end{equation}
We frequently use the notations $\iota_X= \iota(X)$ and $\iotaT_{X,V}=\iotaT(X,V)$.

\begin{definition}[Measure-preserving isomorphisms]
  \label{def:MPI}
  We denote by $\rmS(\Omega)$ the class of $\cB$-$\cB$-measurable maps
  $g:\Omega\to\Omega$ which are essentially injective and measure
  preserving, meaning that there exists a full $\P$-measure set
  $\Omega_0 \in \cB$ such that $g$ is injective on $\Omega_0$ and
  $g_\sharp \P=\P$.  Every $g\in \rmS(\Omega)$ has an inverse
  $g^{-1}\in \rmS(\Omega)$ (defined up to a $\P$-negligible set) such
  that $g^{-1}\circ g=g\circ g^{-1}=\ii_\Omega$ $\P$-a.e.~in $\Omega$.
\end{definition}
In Section \ref{subsec:maxim} we report some properties (see
\cite{CSS2piccolo} for details and proofs) of the resolvent operator,
the Yosida approximation and the minimal selection of a maximal
$\lambda$-dissipative operator $\mmo\subset\cH\times\cH$
which is invariant by measure-preserving isomorphisms.
In Section \ref{sec:3.2} we study the relation between $\lambda$-dissipativity for an invariant subset $\mmo$ of $\cH\times\cH$, and corresponding total $\lambda$-dissipativity of the image/law $\frF$ of $\mmo$ in $\prob_2(\TX)$. The particular case of deterministic \MPVF{s} is considered in Section \ref{subsec:maps}.
These results are then used, in Section \ref{sec:totdissMPVF-flow}, to analyze well-posedness of the Eulerian flow for $\frF$ generated by the corresponding Lagrangian one for $\mmo$ and the generation of $\lambda$-\EVI solutions in $\prob_2(\X)$. 
\medskip

\subsection{Law invariant dissipative operators}
\label{subsec:maxim}

Given a set 
$\mmo\subset \cH\times\cH$
(as usual, we will identify subsets of $\cH\times \cH$ with
multivalued operators),
we define $\mmo(X):=\{V\in \cH:(X,V)\in
\mmo\}$ and the domain
$\dom(\mmo):=\{X\in \cH:\mmo(X)\neq \emptyset\}$.

When $\mmo$ is maximal $\lambda$-dissipative,  the sections $\mmo(X)$ are closed and convex subsets of $\cH$, for $X\in\dom(\mmo)$, hence they contain a unique element of minimal norm, denoted by $\mmo^\circ(X) $. 
For every $0<\tau<1/\lambda^+$, the resolvent operator $\resolvent\tau:=(\ii_{\cH}-\tau\mmo)^{-1}$ of $\mmo$ is a $(1-\lambda \tau)^{-1}$-Lipschitz map defined on the whole $\cH$, where we set $\lambda^+:=\lambda \vee 0$ and $1/\lambda^+=+\infty$ if $\lambda^+=0$. 
In particular, given $X\in\cH$, $\resolvent\tau(X)$ is the unique solution of the inclusion $Y-X\in\tau\mmo(Y)$, so that
\[\left(\resolvent\tau(X),\frac{\resolvent\tau(X)-X}{\tau}\right)\in\mmo,\]
or, equivalently, we can write $\resolvent\tau(X)=X+\tau V$, for some $V\in\mmo(\resolvent\tau(X))$.

The minimal selection $\mmo^\circ:\dom(\mmo)\to\cH$ of $\mmo$ is also characterized by
\begin{equation*}\mmo^\circ(X)=\displaystyle\lim_{\tau \downarrow 0}\frac{\resolvent\tau(X)-X}{\tau}.
\end{equation*}
The Yosida approximation of $\mmo$ is defined by $\mmo_\tau:=\frac{\resolvent\tau-\ii_{\cH}}{\tau}$.
For every $0<\tau<1/\lambda^+$, $\mmo_\tau$ is maximal $\lambda/(1-\lambda \tau)$-dissipative and $\frac{2-\lambda \tau}{\tau(1-\lambda \tau)}$-Lipschitz continuous.  We refer to Appendix \ref{sec:brezis} for a recall of the main properties of the operators $\mmo^\circ,\resolvent\tau,\mmo_\tau$ associated to $\mmo$.

\medskip

If $\mmo$ is a maximal $\lambda$-dissipative operator,
then there exists (cf. Theorems \ref{thm:brezis3},\ref{thm:brezis4} in Appendix \ref{sec:brezis}) a semigroup of $e^{\lambda t}$-Lipschitz transformations $(\Sgp_t)_{t \ge
  0}$ with $\Sgp_t: \overline{\dom(\mmo)} \to
\overline{\dom(\mmo)}$ s.t.~for every $X_0\in \dom(\mmo)$
the curve $t \mapsto \Sgp_t X_0$ is included in $\dom(\mmo)$ and it
is the unique locally Lipschitz continuous solution of the differential inclusion
\begin{equation*}
\begin{cases}
  \dot{X}_t \in \mmo (X_t) \quad \text{ a.e. } t>0, \\
  X\restr{t=0} = X_0.
\end{cases}
\end{equation*}

By Theorem \ref{thm:brezis3}(3), we also have
\begin{equation*}
  \lim_{h\downarrow0}\frac{\Sgp_{t+h} (X_0)-\Sgp_t (X_0)}h =
  \mmo^\circ(\Sgp_t(X_0)),\quad \text{for every $X_0\in \dom(\mmo)$ and every $t\ge 0$.}
\end{equation*}
Let us now consider the particular classes of
operators which are invariant by measure-preserving isomorphisms
or law-invariant. 
\begin{definition}[Invariant operators]
  \label{def:inv} We say that a set
  (or a multivalued operator) $\mmo \subset \cH \times \cH$ is \emph{invariant by measure-preserving isomorphisms} if for every $g \in \rmS(\Omega)$ it holds
\[(X,V) \in \mmo \,\Rightarrow\, (X \circ g, V \circ g) \in \mmo. \]
A set $\mmo \subset \cH \times \cH$ is \emph{law invariant} if it holds
\[(X,V) \in \mmo,\,\, X',V' \in \cH, \,\, \iotaT_{X,V} =
  \iotaT_{X',V'} \,\Rightarrow\, (X',V') \in \mmo.\]
An operator $\bm A:\cH\supset \dom(\bm A)\to \cH$, 
is invariant by measure-preserving isomorphisms (resp.~law invariant) if its graph is invariant by measure-preserving isomorphisms (resp.~law invariant).
\end{definition}
\newcommand{\dommmo}{\iota\big(D(\mmo)\big)}
 We recall that  $\dommmo=
\Big\{\iota_X\,:\,X\in \dom(\mmo)\Big\}$
is 
the image
in $\mathcal P_2(\X)$
of the domain of
$\mmo$.
The results in the following Lemma \ref{le:noncera} and Theorem \ref{thm:invTOlawinv} are presented in \cite[Section 4]{CSS2piccolo} to which we refer for the proofs.

\begin{lemma}[Closed invariant sets]
\label{le:noncera} Let $\mmo \subset \cH \times \cH$ be a closed set. Then $\mmo$ is invariant by measure-preserving isomorphisms if and only if it is law invariant.
\end{lemma}

 For the following, recall that $\Sp{\X,D}$ is defined in \eqref{eq:2}. 
\begin{theorem}
[Representation of resolvents, Yosida approximations, and semigroups]\label{thm:invTOlawinv}
  Let $\mmo \subset \cH \times \cH$ be a maximal $\lambda$-dissipative operator which is invariant by measure-preserving isomorphisms. Then for every $0<\tau<1/\lambda^+,\ t\ge0$
  the operators $\mmo, \mmo_\tau, \resolvent\tau,\Sgp_t,\mmo^\circ$ are
  law invariant. Moreover there exist (uniquely defined) continuous maps $\jj_\tau:\Sp\X\to \X$, $\bb_\tau: \Sp \X \to \X$, and $\ss_t:\Sp{\X,\overline{\dommmo}}\to \X$ such that:
   \begin{align} \label{eq:7-1} \text{ for every $X \in \cH$, } \resolvent\tau (X)(\omega)=\jj_\tau(X(\omega),\iota_X) \text{ for $\P$-a.e.~$\omega \in \Omega$,}\\
   \label{eq:7-3} \text{ for every $X \in \cH$, } \mmo_\tau (X)(\omega)=\bb_\tau(X(\omega),\iota_X) \text{ for $\P$-a.e.~$\omega \in \Omega$,}\\ 
  \label{eq:7-2} \text{ for every $X \in \overline{\dom(\mmo)}$, } \Sgp_t (X)(\omega)=\ss_t(X(\omega),\iota_X) \text{ for $\P$-a.e.~$\omega \in \Omega$;}
  \end{align}
  Furthermore,
    \begin{itemize}
  \item the following invariance and semigroup properties are satisfied
  \begin{equation}
    \label{eq:8}
    \begin{split}
      \mu\in \overline{\dommmo}&\quad\Rightarrow\quad
      \ss_t(\cdot,\mu)_\sharp\mu\in  \overline{\dommmo};\\
      \mu\in {\dommmo}&\quad\Rightarrow\quad
      \ss_t(\cdot,\mu)_\sharp\mu\in  {\dommmo};\\
      \ss_{t+h}(x,\mu)=\ss_h(\ss_t(x,\mu),\ss_t(\cdot,\mu)_\sharp\mu)&\quad
      \text{for every }(x,\mu) \in \Sp{\X,\overline{\dommmo}},\ t,h\ge0;
    \end{split}
  \end{equation}
  \item for every $\mu\in \dommmo$, there exists a map
  $\bb^\circ[\mu] \in L^2(\X, \mu; \X)$ such that
  for every $X\in \cH$
  \begin{equation}
    \label{eq:9}
    \text{ if $\iota_X=\mu$ then $X \in \dom(\mmo)$, } \mmo^\circ (X)(\omega)=\bb^\circ[\mu](X(\omega)) 
    \text{ for $\P$-a.e.~$\omega \in \Omega$.}
  \end{equation}
  The map $\bb^\circ[\mu]$ is $\lambda$-dissipative in a set $\X_0\subset \X$
  of full $\mu$-measure and satisfies
  \begin{equation}
    \label{eq:10}
    \lim_{h\downarrow0}\int_\X\bigg|\frac 1h(\ss_{t+h}(x,\mu)-\ss_t(x,\mu))-
    \bb^\circ[\ss_t(\cdot,\mu)_\sharp \mu](\ss_t(x,\mu))\bigg|^2\,\d\mu(x)=0,\quad
    t\ge0;
  \end{equation}
  \item the following regularity properties hold
  \begin{enumerate}
  \item for every $\mu \in \prob_2(\X)$, the map  $\jj_\tau(\cdot,\mu): \supp(\mu) \to \X$ is $(1-\lambda \tau)^{-1}$-Lipschitz continuous, for $0<\tau< 1/\lambda^+$;
  \item  for every $\mu \in \prob_2(\X)$, the map  $\bb_\tau(\cdot,\mu): \supp(\mu) \to \X$ is $\frac{2-\lambda \tau}{\tau(1-\lambda \tau)}$-Lipschitz continuous, for $0<\tau< 1/\lambda^+$;
  \item for every $\mu \in \overline{\dommmo}$, the map $\ss_t(\cdot,\mu): \supp(\mu) \to \X$ is $e^{\lambda t}$-Lipschitz continuous.
  \end{enumerate}
\end{itemize}
\end{theorem}
Notice that when $\mu\in \iota(\dom(\mmo))$, 
\eqref{eq:8} and \eqref{eq:10} yield
\begin{equation}
    \label{eq:10bis}
    \lim_{h\downarrow0}\int_\X\bigg|\frac 1h(\ss_{h}(x,\mu)-x)-
    \bb^\circ[\mu](x)\bigg|^2\,\d\mu(x)=0. 
  \end{equation}
\begin{remark}
  \label{rem:aggiungi}
 By Theorem \ref{thm:brezis2}(1) and Lemma \ref{le:noncera},
a maximal $\lambda$-dissipative operator $\mmo \subset \cH \times \cH$, $\lambda\in\R$, is law invariant if and only if it is invariant by measure-preserving isomorphisms.
Hence, in this case, we will simply use the word \emph{invariant}. Notice moreover that if $\mmo$ is law invariant, then also
$\dom(\mmo)$ is \emph{law invariant}
in the sense that if $X\in \dom(\mmo)$ and $\iota_Y=\iota_X$
then also $Y$ belongs to $\dom(\mmo)$. It is an immediate consequence of \eqref{eq:9}.
\end{remark}

\subsection{Totally dissipative {\MPVF}s }\label{sec:3.2}
The aim of this section is to study the properties of \MPVF{s} enjoying a strong dissipativity property that we call total dissipativity.
\begin{definition}[Total dissipativity]
    \label{def:total-dissipativity}
    We say that a \MPVF $\frF \subset \prob_2(\TX)$ is \emph{totally $\lambda$-dissipative}, $\lambda\in\R$, if 
    for every $\Phi_0,\Phi_1\in \frF$
    and every $\bm \vartheta\in \Gamma(\Phi_0,\Phi_1)$ we have
    \begin{equation}
        \label{eq:total-dissipativity}
        \int_{\TX^2} \langle v_1-v_0,x_1-x_0
        \rangle \,\d\bm\vartheta(x_0,v_0,x_1,v_1)\le \lambda \int_{\TX^2} |x_1-x_0|^2\d\bm\vartheta(x_0,v_0,x_1,v_1).
    \end{equation}
    We say that $\frF$ is \emph{maximal totally $\lambda$-dissipative} if it is 
    maximal in the class of totally $\lambda$-dissipative \MPVF{s}: if $\frF'\supset \frF$
    and $\frF'$ is totally $\lambda$-dissipative, then $\frF'=\frF.$
\end{definition}
Of course, total $\lambda$-dissipativity implies
$\lambda$-dissipativity (see Definition \ref{def:dissipative}).
\begin{remark}
\label{rem:deterministic}    
Notice that for a deterministic \MPVF (recall Definition \ref{def:MPVF})
total $\lambda$-dissipativity is equivalent
to the following condition
(when $\lambda=0$ see the analogous 
notion of L-monotonicity
of \cite[Def.~3.31]{CD18}):
for every 
$\mu_i\in \dom(\frF)$ and
$\ff_i\in \maps{\frF[\mu_i]}$, $i=0,1$, and every $\mmu\in \Gamma(\mu_0,\mu_1)$ it holds
\begin{equation}
    \label{eq:deterministic-total-diss}
    \int_{\X^2} 
    \langle \ff_1(x_1,\mu_1)-
    \ff_0(x_0,\mu_0),x_1-x_0
    \rangle\,\d\mmu(x_0,x_1)\le 
\lambda\int_{\X^2} |x_1-x_0|^2\,\d\mmu(x_0,x_1).    
\end{equation}
\end{remark}

We introduce now the natural notion of Lagrangian representation of
a \MPVF, 
based on the maps $\iota$, $\iotaT$ introduced in \eqref{eq:1}.
\begin{definition}
  [Lagrangian representations and Eulerian images]\label{def:representations}
  Given $\mmo \subset \cH
  \times \cH$ and $\frF \subset \prob_2(\TX)$, we say that $\mmo$ is
  the \emph{Lagrangian representation} of $\frF$ if
  \[\mmo=(\iotaT)^{-1}(\frF)=
  \Big\{(X,V)\in \cH\times \cH\,:\,
  \iotaT_{X,V} \in \frF\Big\}.
  \]  
  Conversely, if
  $\mmo\subset \cH\times \cH$
  we say that $\frF$ is the
\emph{Eulerian
image} of $\mmo$ if 
\[\frF=\iota^2(\mmo)=
\Big\{\iotaT_{X,V}\,:\,(X,V)\in \mmo\Big\}.\]
\end{definition}
Clearly, the Lagrangian representation $\mmo$ of $\frF$ is law
invariant, moreover $\mmo$ is the Lagrangian representation of $\frF$
if and only if $\frF$ is the Eulerian image of $\mmo$ and $\mmo$ is law invariant.

Similarly to Remark \ref{rem:transff} concerning operators in Hilbert
spaces, we highlight the following result which allows a
reduction of
many arguments to the dissipative case $\lambda=0$.
\begin{lemma}\label{lem:lambda0}
The following hold:
\begin{enumerate}
    \item $\frF\subset\prob_2(\TX)$ is totally $\lambda$-dissipative if and only if $\frF^\lambda$ (cf. \eqref{eq:Flambda}) is totally $0$-dissipative;
    \item $\frF\subset\prob_2(\TX)$ is maximal totally $\lambda$-dissipative if and only if $\frF^\lambda$  is maximal  totally $0$-dissipative;
    \item $\mmo\subset\cH\times\cH$ is invariant by measure-preserving isomorphisms (resp.~law invariant) if and only if $\mmo^\lambda:=\mmo-\lambda\ii_\cH$ is invariant by measure-preserving isomorphisms (resp.~law invariant);
    \item $\mmo\subset\cH\times\cH$ is the Lagrangian representation of $\frF\subset\prob_2(\TX)$ if and only if $\mmo^\lambda$ is the Lagrangian representation of $\frF^\lambda$.
\end{enumerate}
\end{lemma}
\begin{proof}
The proof of item (1) is similar to \cite[Lemma 4.6]{CSS} and is based on the bijectivity of the map $L^\lambda:=(\sfx,\sfv-\lambda\sfx):\TX\to\TX$. Hence, if $\Phi_i\in\frF$ and $\Phi_i^\lambda:=L^\lambda_\sharp\Phi_i\in\frF^\lambda$, $i=1,2$, then $\bm\vartheta\in\Gamma(\Phi_0,\Phi_1)$ if and only if $\bm\vartheta^\lambda\in\Gamma(\Phi_0^\lambda,\Phi_1^\lambda)$, with $\bm\vartheta^\lambda=(\sfx^0,\sfv^0-\lambda\sfx^0,\sfx^1,\sfv^1-\lambda\sfx^1)_\sharp\bm\vartheta$. We can thus prove only the left-to-right implication, the other will follow from the same procedure. We have
\begin{align*}
\int_{\TX^2} \langle v_1-v_0,x_1-x_0
        \rangle \,\d\bm\vartheta^\lambda(x_0,v_0,x_1,v_1)&=\int_{\TX^2} \langle v_1-v_0-\lambda(x_1-x_0),x_1-x_0
        \rangle \,\d\bm\vartheta(x_0,v_0,x_1,v_1)\\
        &=\int_{\TX^2} \langle v_1-v_0,x_1-x_0
        \rangle \,\d\bm\vartheta-\lambda\int_{\TX^2} |x_1-x_0|^2\d\bm\vartheta\\
        &\le0,
\end{align*}
by total $\lambda$-dissipativity of $\frF$.\\
Items (2), (3) and (4) are straightforward.
\end{proof}

A first basic fact is stated by the following proposition.
\begin{proposition}
    \label{prop:basic-relation}
Let $\mmo\subset \cH\times \cH$ be the Lagrangian representation of $\frF
\subset \mathcal P_2(\TX)$ according to Definition \ref{def:representations}. 
Then $\frF$ is totally $\lambda$-dissipative 
if and only if $\mmo$ is $\lambda$-dissipative.
\end{proposition}
\begin{proof}
By Lemma \ref{lem:lambda0} and Remark \ref{rem:transff}, it is sufficient to prove the result in the case $\lambda=0$.
Let us first assume that $\frF$ is totally dissipative.
Let $(X_0,V_0), (X_1,V_1)\in \mmo$.
Since $\Phi_0=\iotaT_{X_0,V_0}\in \frF$,
$\Phi_1=\iotaT_{X_1,V_1}\in \frF$
and $\bm\vartheta:=
(X_0,V_0,X_1,V_1)_\sharp\P\in \Gamma(\Phi_0,\Phi_1)$, 
\eqref{eq:total-dissipativity}
yields
\begin{displaymath}
    \int_\Omega \langle V_1-V_0,X_1-X_0\rangle\,\d\P
    =
    \int_{\TX^2}\langle v_1-v_0,x_1-x_0\rangle
    \,\d\bm\vartheta\le 0.
\end{displaymath}
In order to prove the converse implication, 
let us assume that $\mmo$ is dissipative and 
take $\Phi_0, \Phi_1\in\frF$, $\ttheta\in\Gamma(\Phi_0,\Phi_1)$ and $(X_0,V_0,X_1,V_1)\in\cH^4$ such that $(X_0,V_0,X_1,V_1)_\sharp\P=\ttheta$. 
Since $\Phi_0, \Phi_1\in\frF$, 
there exist $(X_0',V_0')\in\mmo$ and $(X_1',V_1')\in\mmo$ such that\[\iotaT_{X_0',V_0'}=\Phi_0=\iotaT_{X_0,V_0},\qquad \iotaT_{X_1',V_1'}=\Phi_1=\iotaT_{X_1,V_1}.\]
By the law invariance of $\mmo$, we have that $(X_0,V_0),(X_1,V_1)\in\mmo$, so that
\[ \int_{\TX^2} \langle v_1-v_0,x_1-x_0\rangle
    \,\d\bm\vartheta=
    \la V_1-V_0, X_1-X_0 \ra_{\cH} \le 0\]
by the dissipativity of $\mmo$.
\end{proof}
\begin{example}
\label{ex:Lip}
        Let us consider a map $\ff:\Sp\X \to \X$ 
    (recall \eqref{eq:2}) 
    such that there exists $L>0$ for which we have
\begin{equation*}
|\ff(x_1,\mu_1)-\ff(x_0, \mu_0)| \le L \left (W_2(\mu_0, \mu_1) + |x_0-x_1| \right ) \quad \text{ for every } (x_0,\mu_0),\ (x_1,\mu_1) \in \Sp{\X}.
\end{equation*}
We can also identify $\ff$ with the map sending $\mu \mapsto f(\cdot, \mu) \in \Lip(\X;\X)$ (compare with the framework analyzed by Bonnet and Frankowska in \cite{bonnet2020mean,BF2023} and with the hypoteses in \cite{CLOS, AFMS21}).
Let us define the map $\Bb: \cH \to \cH$ and the (single-valued, deterministic) \PVF $\frF \subset \prob_2(\TX)$ as
\begin{align*}
    \Bb(X)(\omega)&:= \ff(X(\omega), \iota_X), \quad X \in \X, \, \omega \in \Omega,\\
    \frF[\mu]&:= (\ii_{\X}, \ff(\cdot, \mu))_\sharp \mu, \quad\mu \in \prob_2(\X).
\end{align*}
It is not difficult to check that $\Bb$ is $2L$-Lipschitz and that
$\frF$ is maximal $2L$-totally dissipative. Indeed, for every $X,Y \in \cH$, we have
\begin{align*}
|\Bb (X)-\Bb (Y)|_{\cH} &= \left ( \int_{\Omega} |\Bb (X)(\omega)-\Bb (Y)(\omega)|^2 \de \P(\omega) \right )^{1/2} \\
&= \left ( \int_{\Omega} |\ff(X(\omega), \iota_X))-\ff(Y(\omega), \iota_Y)|^2 \de \P(\omega) \right)^{1/2}\\
& \le L \left ( \int_{\Omega} \left(W_2(\iota_X, \iota_Y) + |X(\omega)-Y(\omega)|\right)^2 \de \P(\omega)\right )^{1/2} \\
& \le L \left ( \left (\int_{\Omega} W_2^2(\iota_X, \iota_Y) \de \P(\omega) \right )^{1/2} +\left (\int_{\Omega} |X(\omega) - Y(\omega)|^2 \de \P(\omega) \right )^{1/2} \right ) \\
& \le 2L|X-Y|_{\cH}
\end{align*} 
so that $\Bb$ is $2L$-dissipative and therefore
$\frF$ is $2L$-totally dissipative as well by
Proposition \ref{prop:basic-relation}. Maximality follows by the maximality of $\Bb$ and the next theorem.
\end{example}
\begin{theorem}[Maximal dissipativity]
\label{thm:maximal-dissipativity}\
    \begin{enumerate}
    \item Every $\lambda$-dissipative operator $\mmo \subset \cH \times \cH$ 
    which is invariant by measure-preserving isomorphisms has 
    a maximal $\lambda$-dissipative extension with domain included in $\clconv{\dom(\mmo)}$ which is 
    invariant by measure-preserving isomorphisms (and therefore also law invariant).
    \item
    Let us suppose that 
    $\mmo \subset \cH \times \cH$ is the $\lambda$-dissipative Lagrangian representation
    of the totally $\lambda$-dissipative \MPVF $\frF \subset \prob_2(\TX)$.
    Then $\mmo$ is maximal $\lambda$-dissipative if and only if 
    $\frF$ is maximal totally $\lambda$-dissipative. 
    \item If $\frF \subset \prob_2(\TX)$ is a totally $\lambda$-dissipative \MPVF with domain included in a closed and totally convex set $\core$, then there exists a maximal totally $\lambda$-dissipative extension of $\frF$ with domain included in $\core$. 
    \end{enumerate}
\end{theorem}
\begin{proof}
By Lemma \ref{lem:lambda0} and Remark \ref{rem:transff}, it is sufficient to prove the result in case $\lambda=0$.
    Item (1) is \cite[Theorem 4.5]{CSS2piccolo}. Notice that, since it is maximal $\lambda$-dissipative and invariant by measure-preserving isomorphisms, a maximal $\lambda$-dissipative extension of $\mmo$ is also law invariant by Lemma \ref{le:noncera}.

    Item (2) follows by 
    the equivalence result of 
    Proposition \ref{prop:basic-relation}
    and by item (1). In fact, if $\mmo$ is
    maximal dissipative it is clear that $\frF$ is maximal. Conversely,
    suppose that $\frF$ is maximal and $\mmo$
    is its Lagrangian representation.
    By contradiction, if $\mmo$ is not maximal, Item (1) shows that 
    there exists a maximal and proper extension $\hat\mmo$
    of $\mmo$ which is law invariant.
    Therefore, $\hat \mmo$
    induces a strict extension of $\frF$
    which is totally dissipative.

    Item (3) is a consequence of items (1) and  (2).
  \end{proof}
  \begin{remark}
    \label{rem:closure-of-domain}
    Notice that if $\mmo$ is the Lagrangian representation
    of a maximal totally $\lambda$-dissipative \MPVF\ $\frF$,
    then
    $\iota^{-1}\big(\overline{\dom(\frF)}\big)=\overline{\dom(\mmo)}$.
    In fact, it is sufficient to prove that if $\iota_X=\mu\in
    \overline{\dom(\frF)} $ then $X\in \overline{\dom(\mmo)}$, since the converse
    inclusion is trivial. Given such a $\mu=\iota_X\in\overline{\dom(\frF)}$, we can find a sequence $(\mu_n)_{n\in\N}\subset
    \dom(\frF)$ converging to $\mu$ in $\prob_2(\X)$. Applying the last
    statement of Theorem \ref{thm:gpfinal}
    we can then find a sequence $(X_n)_{n\in\N}\subset \cH$ 
    such that $\iota_{X_n}=\mu_n$ and
    $\lim_{n\to+\infty}|X_n-X|_\cH=0$.
    We deduce
    that $X_n\in \dom(\mmo)$ by Remark \ref{rem:aggiungi}
    and therefore $X\in \overline{\dom(\mmo)}$.    
  \end{remark} 

The uniqueness of a maximal totally dissipative extension of a given totally dissipative \MPVF is investigated in Part \ref{partII} and, in particular, in Theorem \ref{thm:total-case} of which we report a simplified version here.

\begin{theorem} Let $\mathsf{U} \subset \X$ be open, convex, non-empty and let $\frF \subset \prob_2(\TX)$ be a totally $\lambda$-dissipative \MPVF whose domain satisfies
\[ \prob_f(\mathsf U) \subset \dom(\frF) \subset \prob_2(\overline{\mathsf U}),\]
where
$\prob_f(\mathsf U)
      :=\{\mu\in \prob(\mathsf U)\,:\,\supp(\mu) \text{ is finite}\}$.
Then there exists a unique maximal totally $\lambda$-dissipative extension $\hat{\frF}$ of $\frF$ with domain included in $\prob_2(\overline{\mathsf U})$.
\end{theorem}  
  
We now apply Theorem \ref{thm:maximal-dissipativity} to get
useful insights on the structure of totally dissipative {\MPVF}s.
The first result concerns the existence of a solution to the resolvent equation,
which provides an equivalent characterization of maximality and will
be the crucial tool to implement the Implicit Euler method, see
Corollary \ref{cor:IES}.
\begin{theorem}
[Solution to the resolvent equation]
    \label{cor:resolvent-max-tot}
    A totally $\lambda$-dissipative \MPVF $\frF \subset \prob_2(\TX)$ is maximal $\lambda$-dissipative if and only if 
    for every $\mu\in \mathcal P_2(\X)$
    and every $0<\tau<1/\lambda^+$ 
    there exists $\Phi\in\frF$
    such that $(\sfx-\tau \sfv)_\sharp \Phi=\mu.$  Moreover, if $\frF$ is a maximal totally $\lambda$-dissipative \MPVF, then for every $\mu\in\prob_2(\X)$ and $0<\tau<1/\lambda^+$, such a $\Phi$ is unique. 
\end{theorem}
\begin{proof}
    Let
    $\mmo$ be the Lagrangian representation of $\frF$ that is $\lambda$-dissipative by Proposition \ref{prop:basic-relation}.
    If $\frF$ is maximal $\lambda$-dissipative, then $\mmo$
    is maximal $\lambda$-dissipative as well by Theorem \ref{thm:maximal-dissipativity}(3), so that 
    for every $Y\in \cH$ with
    $\iota_Y=\mu$ and $0<\tau<1/\lambda^+$
    there exists  a unique 
    $(X,V)\in \mmo$ such that 
    $X-\tau V=Y$ (cf.~Theorem \ref{thm:brezis1}(1))
    so that $\Phi:=\iotaT_{X,V}\in \frF$
    satisfies $(\sfx-\tau\sfv)_\sharp \Phi=\mu$.  Moreover, we can prove that such $\Phi$ is unique. Indeed, assume there exists $\Phi'\in\frF$ such that $(\sfx-\tau\sfv)_\sharp \Phi'=\mu$. Let $(X',V')\in\mmo$ such that $\Phi'=\iotaT_{X',V'}$ and $Y':=X'-\tau V'$. By definition, we have
    \[\jJ_\tau(Y)=Y+\tau V=X,\quad \jJ_\tau(Y')=Y'+\tau V'=X'.\]
    By Theorem \ref{thm:invTOlawinv}, there exists a map $\jj_{\tau}$ representing $\jJ_{\tau}$. In particular, defining the map $\boldsymbol{a}_\tau^\mu:\supp(\mu)\to\X\times\X$,
    \[\boldsymbol{a}^\mu_\tau(x):=\left(\jj_\tau(x,\mu),\frac{\jj_\tau(x,\mu)-x}{\tau}\right), \quad x \in \supp(\mu),\]
    we have that $\boldsymbol{a}^\mu_\tau(Y)=(X,V)$ and $\boldsymbol{a}^\mu_\tau(Y')=(X',V')$. Since $\iota_Y=\iota_{Y'}=\mu$, we get $\iota^2_{X',V'}=\iota^2_{X,V}$. In particular, $\Phi'=\Phi$. 
\smallskip

    Conversely,  we prove the reverse implication of the statement.  Let us now suppose
    that $\frF$ is not maximal $\lambda$-dissipative, 
    so that $\mmo$ is not maximal $\lambda$-dissipative 
    and it admits a proper maximal $\lambda$-dissipative law invariant extension $\hat\mmo$ by Theorem \ref{thm:maximal-dissipativity}. Consider the following objects:
    \[(\tilde X,\tilde V)\in \hat\mmo\setminus \mmo,\quad 0<\tau<1/\lambda^+,\quad 
    \tilde Y:=\tilde X-\tau \tilde V,\quad\text{and }\mu:=\iota_{\tilde Y}.\]
    We claim that the equation
    $\Phi\in \frF$, $(\sfx-\tau\sfv)_\sharp\Phi=\mu$
    has no solution. We argue by contradiction,
    and we suppose that 
    $\Phi\in \frF$ is a solution: we could find
    $(X,V)\in \mmo$
    such that setting 
    $\iotaT_{X,V}=\Phi$
    and setting
    $Y:=X-\tau V$
    we have $ \iota_Y=\mu.$

    We use the maximal $\lambda$-dissipativity of $\hat\mmo$ and we denote by $\hat{\jJ}_\tau$ the resolvent associated to $\hat\mmo$,
    by $\hat \jj_\tau$ the map 
    induced by Theorem \ref{thm:invTOlawinv} as in \eqref{eq:7-1},
    and we set 
    \[\hat \bb_\tau(x):=
    \frac 1\tau(\hat \jj_\tau(x,\mu)-x),\quad x \in \supp(\mu).\]
    
    We have
    \begin{align*}
        \tilde X&=\hat{\jJ}_\tau\tilde (Y)=\hat \jj_\tau(\tilde Y,\mu);\\
        X&=\hat{\jJ}_\tau (Y)=\hat \jj_\tau (Y,\mu);\\
        \tilde V&=\frac 1\tau(\tilde X-\tilde Y)=\hat \bb_\tau(\tilde Y);\\
        V&=\frac 1\tau(X-Y)=\hat \bb_\tau(Y).
    \end{align*}
    It follows that 
    $\iotaT_{\tilde X,\tilde V}=
    (\hat \jj_\tau(\cdot,\mu),
    \hat \bb_\tau)_\sharp 
    \mu=\iotaT_{X,V}=\Phi\in \frF$
    so that $(\tilde X,\tilde V)$
    has the same law of $(X,V)$
    and therefore belongs to $\mmo$, 
    a contradiction.
\end{proof}
\medskip\noindent
We now show that a maximal totally $\lambda$-dissipative \MPVF
is sequentially closed in the strong-weak topology of
$\prob_2^{sw}(\TX)$, recall   \eqref{eq:sw-closure}.
\begin{proposition}
[Strong-weak closure]
  \label{prop:sw-closedness}
  The sequential strong-weak closure $\clo\frF$
  of a totally $\lambda$-dissipative \MPVF $\frF$
  is totally $\lambda$-dissipative as well.
  In particular, if $\frF$ is maximal, then $\clo\frF=\frF$.  
\end{proposition}
\begin{proof}
  As usual, it is sufficient to check the property for $\lambda=0$.
  Let $\Phi',\Phi''\in \clo\frF$ and $\bm\vartheta\in
  \Gamma(\Phi',\Phi'')$.
  Denoting by $\{e_i\}_{i \in \N}$ an orthonormal system for $\X$, we introduce on $\X$ and on $\TX$ respectively the distances
  \begin{displaymath}
    \mathsf d^w(v_1,v_2):=
    \sum_{i=1}^{+\infty} 2^{-i} \big(|\langle v_1-v_2,e_i\rangle|\land 1\big),
    \quad
    \mathsf d^{sw}((x_1,v_1),(x_2,v_2)):=\Big(|x_1-x_2|_\X^2+
    \mathsf d^w(v_1,v_2)^2\Big)\Big)^{1/2}
  \end{displaymath}
whose induced topologies are weaker than the weak (resp.~the strong-weak)
  topology of $\X$ (resp.~$\TX$), see also the proof of \cite[Proposition 3.4]{NaldiSavare}. 
  Denoting by $W_2^{sw}$ the $2$-Wasserstein distance on
  $\prob_2(\TX)$ induced by $\mathsf d^{sw}$, we have 
  \begin{equation*}
    \Phi_n\to\Phi\quad\text{in }\prob_2^{sw}(\TX)\quad\Rightarrow\quad
    W_2^{sw}(\Phi_n,\Phi)\to0.
  \end{equation*}
  By definition of $\clo\frF$ we can find two
  sequences
  $(\Phi_n')_{n\in \N}, (\Phi_n'')_{n\in \N}$ in $\frF$
  respectively  converging to $\Phi'$ and $\Phi''$ in
  $\prob_2^{sw}(\TX)$. We denote by $\bm \gamma_n'\in
  \Gamma_o^{sw}(\Phi_n',\Phi')$ and 
  $\bm \gamma_n''\in
  \Gamma_o^{sw}(\Phi'',\Phi_n'')$ 
  the corresponding optimal plans for $W_2^{sw}$.
  
  Denoting the elements of $\TX^4$ by
  $(x_1',v_1',x_1,v_1,x_2,v_2,x_2'',v_2'')$ and using the gluing lemma we can find
  a plan $\bm \sigma_n\in
  \prob_2(\TX^4)$ such that 
  $(\sfx_1',\sfv_1',\sfx_1,\sfv_1)_\sharp \ssigma_n=\ggamma_n' $,
  $(\sfx_1,\sfv_1,\sfx_2,\sfv_2)_\sharp \ssigma_n=\bm\vartheta$,
  $(\sfx_2,\sfv_2,\sfx_2'',\sfv_2'')_\sharp \ssigma_n=\ggamma_n'' $.
  We also have
  \begin{gather*}
    \lim_{n\to+\infty}\int_{\TX^4} \left(|x_1'-x_1|^2+|x_2-x_2''|^2+d^w(v_1',v_1)^2+
    d^w(v_2'',v_2)^2\right)\,\d\ssigma_n=0,\\
    \sup_{n\in \N} \int_{\TX^4} \Big(|v_1'|^2+|v_1|^2+|v_2|^2+|v_2''|^2\Big)\,\d\ssigma_n<+\infty,
  \end{gather*}
  so that
  setting $\tilde\ssigma_n:=
      (\sfx_1',\sfx_2'',\sfv_1',\sfv_2'')_\sharp\ssigma_n$ we have
  \begin{displaymath}
    \tilde \ssigma_n
    \to (\sfx_1,\sfx_2,\sfv_1,\sfv_2)_\sharp \bm\vartheta
    \quad\text{in }\prob_2^{sw}(\X^2\times\X^2).
  \end{displaymath}
  Since $(\sfx_1',\sfv_1',\sfx_2'',\sfv_2'')_\sharp \ssigma_n\in
  \Gamma(\Phi_n',\Phi_n'')$, the total dissipativity of $\frF$ yields
  \begin{equation}
    \label{eq:6}
    \int_{\X^2\times\X^2} \la \sfv_1-\sfv_2,\sfx_1-\sfx_2\ra\,\d\tilde\ssigma_n=
    \int_{\TX^4} \la \sfv_1'-\sfv_2'',\sfx_1'-\sfx_2''\ra\,\d\ssigma_n\le
    0\quad
    \text{for every }n\in \N.
  \end{equation}
  Since the function $\zeta(x_1,x_2;v_1,v_2):=\la
  \sfv_1-\sfv_2,\sfx_1-\sfx_2\ra $
  belongs to $\rmC^{sw}_2(\X^2\times \X^2)$ (cf. Definition \ref{def:p2sw}),
  the convergence in $\prob_2^{sw}(\X^2\times\X^2)$ is sufficient
  to pass to the limit in \eqref{eq:6} and thus get
  \begin{displaymath}
    \int_{\TX^2} \la
    \sfv_1-\sfv_2,\sfx_1-\sfx_2\ra\,\d\bm\vartheta\le 0.
    \qedhere
  \end{displaymath}
 \end{proof}
We can also prove that
the sections $\frF[\mu]$ of a maximal totally
dissipative \MPVF are (conditionally) totally convex. 
In the following statement
we consider the space $\X\times \X^N$
whose variables are denoted by $(x,v_1,\cdots,v_N)$ and
the corresponding projections are 
$\sfx(x,v_1,\cdots,v_N):=x,$
$\sfv_i(x,v_1,\cdots,v_N):=v_i.$
\begin{proposition}
[Total convexity of sections of maximal totally dissipative \MPVF]
    \label{cor:easy}
    If $\frF \subset \prob_2(\TX)$ is a maximal totally
    $\lambda$-dissipative \MPVF, 
    then for every $\mu\in \dom(\frF)$ 
    the section $\frF[\mu]$
    satisfies the following total convexity property:
    \begin{equation}
              \label{eq:total-convexity}
              \begin{gathered}
                \text{if $\Lambda \in \mathcal P_2(\X\times \X^N)$
    satisfies
    $(\sfx,\sfv_i)_\sharp \Lambda\in \frF[\mu]$
    and $\alpha_i\ge0$, $i=1,\cdots,N$ with $\sum_i\alpha_i=1$, then}\\
           (\sfx,\sum_i \alpha_i\sfv_i)_\sharp \Lambda\in \frF[\mu].
         \end{gathered}
       \end{equation}   
\end{proposition}
\begin{proof}
  Since $\frF$ is maximal totally $\lambda$-dissipative, by Theorem \ref{thm:maximal-dissipativity}, its Lagrangian representation $\mmo \subset \cH \times \cH$ is maximal $\lambda$-dissipative.

We can find $(X,V_1,V_2,\cdots V_N)\in \cH\times \cH^N$
    such that $(X,V_1,V_2,\cdots V_N)_\sharp\P=\Lambda.$
    We deduce that $(X,V_i)\in \mmo$ since $\iotaT_{X,V_i}\in \frF$.
    Since the sections of $\mmo$ are convex, we deduce that 
    $\left(X,\sum_i \alpha_i V_i\right)\in \mmo$ as well, so that
    \begin{equation*}
        \Big(\sfx,\sum_i \alpha_i\sfv_i\Big)_\sharp \Lambda=
        \Big(X,\sum_i \alpha_i V_i\Big)_\sharp\P
        \in \frF.\qedhere
    \end{equation*}   
  \end{proof}
  We can now derive a remarkable information on the structure of
  a totally dissipative \MPVF,
  which involves the barycentric projection
  introduced in \eqref{eq:brimpvf}.

\begin{theorem}[Barycentric projection]
\label{thm:bary-proj}
Let $\frF$ be a \MPVF and
 $\mu\in \dom(\frF)$ such that
$\frF[\mu]$ is closed in $\prob_2(\TX)$ and
satisfies the total convexity property \eqref{eq:total-convexity}.
Then $\bri\frF[\mu]\subset \frF[\mu]$.
In particular, if $\frF$ is a maximal totally $\lambda$-dissipative \MPVF,
then $\bri \frF \subset \frF$. 
\end{theorem}
\begin{proof}
    We use an argument which is clearly inspired by the law of large numbers.
    
    Let $\{\Phi_x\}_{x\in \X}$
    be the disintegration of $\Phi \in \frF$ w.r.t.~its first marginal $\mu\in \dom(\frF)$.    
    For a given integer $N$ 
    and every $x\in \X$ we define the product measure
    $\Phi_x^N:=(\Phi_x)^{\otimes N}
    \in \mathcal P_2(\X^N)$
    and the corresponding plan
    \begin{equation*}
        \Lambda^N:=
        \int_\X \delta_x\otimes \Phi_x^N\,\d\mu(x)\in 
        \mathcal P_2(\X\times \X^N).
    \end{equation*}
    It is clear that 
    $\Lambda^N$ satisfies the condition
    of Proposition \ref{cor:easy}:
    choosing $\alpha_i:=1/N$ 
    we deduce that 
    $\Psi^N:=(\sfx,\frac 1N\sum_i\sfv_i)_\sharp \Lambda^N\in \frF[\mu].$

    Let now $\Psi:=(\ii_\X,\bry{\Phi})_\sharp \mu.$
    We can easily estimate
    the squared Wasserstein distance between $\Psi$ and $\Psi^N$ by
    \begin{align*}
        W_2^2(\Psi^N,\Psi)&\le 
        \int_{\X\times \X^N} \Big|\frac 1N\sum_i v_i-
        \bry\Phi(x)\Big|^2\,\d\Lambda^N
        =\frac 1N 
        \int_{\TX} \left|v-
        \bry\Phi(x)\right|^2\,\d\Phi
    \end{align*}
    where we used the following
    orthogonality for $i\neq j$
    \begin{align*}
        &\int_{\X\times \X^N} \langle v_i-
        \bry\Phi(x),
        v_j-\bry\Phi(x)\rangle\,\d\Lambda^N\\
        &=
        \int_\X\Big( \int_{\X\times\X} 
        \langle v_i-
        \bry\Phi(x),
        v_j-\bry\Phi(x)\rangle\,\d\Phi_x(v_i)\otimes \Phi_x(v_j)\Big)\,\d\mu(x)\\
        &=0
    \end{align*}
    and
    the fact that  
    \begin{align*}
        \int_{\X\times \X^N} \left|v_i-
        \bry\Phi(x)\right|^2\,\d\Lambda^N
        &=
        \int_\X\Big( \int_\X 
        \left| v_i-\bry\Phi(x)\right|^2\,\d\Phi_x(v_i)\Big)\,\d\mu(x)=
        \int_\TX \left|v-
        \bry\Phi(x)\right|^2\,\d\Phi.
    \end{align*}
    We deduce that $\Psi^N\to\Psi$
    in $\mathcal P_2(\TX)$ as $N\to+\infty$,
    so that $\Psi\in \frF[\mu]$
    as well.
  \end{proof}

   Even in case $\frF$ is not maximal, or it does not contain its barycentric projection, we can still derive a compatibility relation between $\frF$ and $\bri \frF$ as follows. 

  \begin{corollary}
    \label{cor:bary-works-well}
    Let $\frF \subset \prob_2(\TX)$ be a totally $\lambda$-dissipative \MPVF. Then
    the extended \MPVF $\tilde \frF$ defined 
    by
    \begin{equation*}
         \tilde\frF
          :=\frF
          \cup \bri \frF,
    \end{equation*}
    with $\bri \frF$ as in \eqref{eq:brimpvf}, is totally $\lambda$-dissipative. In particular, for every
    $\Phi_i\in \frF[\mu_i]$, $i=1,2$, and every $\mmu\in
    \Gamma(\mu_1,\mu_2)$,
    \begin{equation*}
      \int_{\X^2} \langle
      \bry{\Phi_1}(x_1)-\bry{\Phi_2}(x_2),x_1-x_2\rangle\,\d\mmu(x_1,x_2)
      \le \lambda\int_{\X^2} |x_1-x_2|^2\,\d\mmu(x_1,x_2).
    \end{equation*}
  \end{corollary}
  \begin{proof}
    It is sufficient to consider an arbitrary maximal totally
    $\lambda$-dissipative extension $\hat\frF$ of $\frF$: by the
    previous Theorem \ref{thm:bary-proj}
    clearly $\hat \frF\supset \tilde\frF$.
  \end{proof}

 In analogy with the Hilbertian theory, in the following theorem we state the existence of a unique selection of minimal norm for a maximal totally $\lambda$-dissipative \MPVF. It turns out that such a minimal selection is concentrated on a map which coincides with that coming from the Lagrangian representation of the \MPVF. 

\begin{theorem}[The minimal selection]
  \label{thm:minimal}
  Let $\frF\subset\prob_2(\TX)$ be a maximal totally $\lambda$-dissipative
  \MPVF.
  \begin{enumerate}
  \item For every $\mu\in \dom(\frF)$ there exists a unique vector field
    $\ff^\circ[\mu]\in L^2(\X,\mu;\X)$ such that
    \begin{equation}\label{eq:defq}
      (\ii_\X, \ff^\circ[\mu])_\sharp\mu\in \frF[\mu],\quad
      \int_\X |\ff^\circ[\mu]|^2\,\d\mu\le 
      \int_\TX |v|^2\,\d\Phi\quad\text{for every }\Phi\in\frF[\mu].
    \end{equation}
    We denote the \emph{minimal selection of $\frF$} at $\mu$ by
    \begin{equation}\label{eq:minsel}
    \frF^\circ[\mu]:=(\ii_\X, \ff^\circ[\mu])_\sharp \mu.
    \end{equation}
    \item If $\mmo$ is the Lagrangian representation of $\frF$, then for
    every $\mu\in \dom(\frF)$, we have 
    \[\ff^\circ[\mu]=\bb^\circ[\mu]\quad\mu\text{-a.e.},\]
    where $\bb^\circ$ has been defined in \eqref{eq:9}
    and, if $0<\tau<1/\lambda^+$, the following hold
    \begin{align}
      \label{eq:23}
      \int_\X \big|\bb_\tau(x,\mu)-\ff^\circ[\mu](x)\big|^2\,\d\mu
      &\le
      \int_\X \big|\ff^\circ[\mu]\big|^2\,\d\mu-
      (1-2\lambda\tau)
        \int_\X \big|\bb_\tau(x,\mu)\big|^2\,\d\mu,\\
      \label{eq:23bis}
      (1-\lambda\tau)^2
      \int_\X \big|\bb_\tau(x,\mu)\big|^2\,\d\mu
      &\uparrow 
      \int_\X \big|\ff^\circ[\mu]\big|^2\,\d\mu\quad
      \text{as }\tau\downarrow0
    \end{align}
    with $\bb_\tau$ as in \eqref{eq:7-3}.
    \item The map $|\frF|_2:\prob_2(\X)\to [0,+\infty]$ defined by
    \begin{equation}
      \label{eq:21}
      |\frF|_2(\mu):=
      \begin{cases}
        \displaystyle
        \int_\X |\ff^\circ[\mu]|^2\,\d\mu&\text{if }\mu\in \dom(\frF),\\
        +\infty&\text{if }\mu \not\in \dom(\frF)
      \end{cases}
    \end{equation}
    is lower semicontinuous.
      
  \item
    Finally, if $\Y$ is a Polish space, $\mmu\in \prob(\X\times \Y)$
    with marginal $\nu=\pi^2_\sharp \mmu$ and the disintegration
    $\{\mu_y\}_{y\in \Y}$ of $\mmu$ w.r.t.~$\nu$ satisfies
    \begin{equation}
      \label{eq:22}
      \int_{\X\times \Y} |x|^2\,\d\mmu(x,y)+
      \int_\Y
      |\frF|_2(\mu_y)\,\d\nu(y)<+\infty,
    \end{equation}
    then the map $\ff(x,y):=\ff^\circ[\mu_y](x)$ belongs to
    $L^2(\X\times \Y,\mmu;\X)$ (in particular it is uniquely defined up
    to a $\mmu$-negligible set and it is $\mmu$-measurable).
  \end{enumerate}
\end{theorem}
\begin{proof}
  Item (1) is an immediate consequence of the closure of $\frF$ in
  $\prob_2^{sw}(\TX)$
  (so that the map $\Phi\mapsto |\Phi|_2$
  has compact sublevels in the set $\relcP 2{\mu}\TX$)
  and of the previous Theorem \ref{thm:bary-proj}.
  
  To prove the second item,
  it is enough to notice that, trivially, $\bb^\circ(\cdot, \mu)$
  satisfies \eqref{eq:defq}. Estimates \eqref{eq:23} and \eqref{eq:23bis}
  follow
  by Theorem \ref{thm:brezis2}(5).

  Item (3) still follows immediately by the
  closure of $\frF$ in
  $\prob_2^{sw}(\TX)$
  and the fact that the map $\Phi\mapsto |\Phi|_2^2$
  defined by \eqref{eq:defsqmPhi} is lower semicontinuous
  w.r.t.~the topology of $\prob_2^{sw}(\TX)$.

  Let us now prove item (4).
  We first notice that \eqref{eq:22} yields $\mu_y\in \dom(\frF)$ for
  $\nu$-a.e.~$y\in \Y$.
  Let us now prove that
  the map $\bb_\tau(x,y):=\bb_\tau(x,\mu_y)$ is
  $\mmu$-measurable.

  Recall that the set \[\mathcal{S}_0:=\big\{(x,\mu)\in \X\times \prob(\X):
  x\in\supp(\mu)\big\}\] is a $G_\delta$ (thus Borel, cf.~\cite[Formula (4.3)]{FSS22}) subset
  of $\X\times \prob(\X)$.
  Since the inclusion map of $\X\times\prob_2(\X)$ in
  $\X\times \prob(\X)$ is continuous,
  we deduce that
  \[ \mathcal{S}:= 
  \mathcal{S}_0 \cap \left(\X\times\prob_2(\X)\right)\]
  is a $G_\delta$ set in $\X\times\prob_2(\X)$.  

  Since the map $j(x,y):= (x,\mu_y)$ is Borel
  from $\X\times \Y$ to $\X\times \prob(\Y)$,
  we deduce that the set
  $\mathcal S':=j^{-1}(\mathcal S) = (\big\{(x,y)\in \X\times \Y:
  x\in \supp(\mu_y)\big\}$ is Borel in $\X\times \Y$ and it is immediate to check that
  $\mmu$ is concentrated on $\mathcal S'$. Since the map $(x,\mu) \mapsto \bb_\tau(x,\mu)$ is continuous in $\mathcal S$ (cf.~Theorem \ref{thm:invTOlawinv}) then its composition with $j$ (which is the map $(x,y) \mapsto \bb_\tau(x,\mu_y))$ is $\mmu$-measurable. 
Passing to the limit as $\tau\downarrow0$ and using \eqref{eq:23} and \eqref{eq:23bis}
  we conclude that $\bb_\tau \to \ff$ in $L^2(\X\times \Y,\mmu; \X)$ so that also $\ff$ is $\mmu$-measurable.
\end{proof}

We now show that discrete measures are sufficient to reconstruct a
maximal totally $\lambda$-dissipative \MPVF. For a general Polish space $\mathscr{X}$, we consider the following set of discrete probability measures
\begin{equation}
      \label{eq:70}
      \prob_f(\mathscr{X})
      :=\Big\{\mu\in \prob(\mathscr{X})\,:\,\supp(\mu) \text{ is finite} \Big\}.
    \end{equation}
    Given $N\in\N$, we denote by $\prob_{f,N}(\mathscr{X})$ the set of empirical measures with weights in $\frac{1}{N}\N$,
\begin{equation}
  \label{eq:30}
  \begin{split}
  \prob_{f,N}(\mathscr{X})&:=\Big\{\mu\in \prob_f(\mathscr{X}):N\mu(A)\in \N\quad
  \text{for every }A\subset X \Big\},\\
  \prob_{f,\infty}(\mathscr{X})&:=\bigcup_{N\in \N}\prob_{f,N}(\mathscr{X}).
  \end{split}
\end{equation}
\begin{corollary}
  \label{cor:discrete-are-enough}
  Let $\frF\subset\prob_2(\TX)$ be a maximal totally $\lambda$-dissipative \MPVF
  and let
  \begin{equation*}
    \dom_{f,\infty}(\frF):=\prob_{f,\infty}(\X)\cap \dom(\frF).
  \end{equation*}
  Then for every $\mu \in \dom(\frF)$ there exists a sequence
  $(\mu_n)_{n\in\N}\subset     \dom_{f,\infty}(\frF) $ such that
  $\frF^\circ[\mu_n]\to \frF^\circ[\mu]$ in $\prob_2(\TX)$ as $n\to+\infty$, where $\frF^\circ$ has been defined in \eqref{eq:minsel}.
  Moreover, a measure $\Phi\in \prob_2(\TX)$ with $\sfx_\sharp\Phi\in
  \overline{\dom(\frF)}$ belongs to $\frF$ if and only if
  for every $\mu\in     \dom_{f,\infty}(\frF) $ and every $\ggamma\in \Gamma(\Phi,\mu)$ we have
  \begin{equation*}
    \int_{\TX\times\X} \langle v-\ff^\circ(y,\mu),x-y\rangle \d\ggamma(x,v,y)\le
    \lambda \int_{\TX\times\X} |x-y|^2\,\d\ggamma(x,v,y).
  \end{equation*}
\end{corollary}
\begin{proof}
  We denote by $\Bb\subset \cH\times \cH$
  the Lagrangian representation of $\frF$ and
  we set $D:=\iota^{-1}(\prob_{f,\infty}(\X))$.
  Since $\prob_{f,\infty}(\X)$ is dense in $\prob_2(\X)$, by e.g.~the last part of Theorem \ref{thm:gpfinal} we have that $D$ is dense in $\cH$ and by Theorem
  \ref{thm:invTOlawinv} (see in particular \eqref{eq:7-1}) it satisfies
  $\jJ_\tau (D)\subset D$.
  We can thus apply Corollary \ref{cor:speriamo-che-sia-ultimo}.
\end{proof}

\subsection{Totally dissipative \PVF{s}\ concentrated on maps}
\label{subsec:maps}
We devote this section to the study of the important case
of deterministic {\MPVF}s and, in particular, of single-valued and everywhere defined {\PVF}s.  Recall that a \MPVF $\frF \subset \prob_2(\TX)$ is deterministic if every $\Phi\in\frF$ is concentrated on a map (cf. Definition \ref{def:wassmom}). 
Recall also that 
for a deterministic \PVF, total $\lambda$-dissipativity can be equivalently stated as in Remark \ref{rem:deterministic}.

\begin{definition}
[Demicontinuity]
  \label{def:demicontinuous}
  A single-valued \PVF $\frF$ is {\em demicontinuous} if
  the map $\mu\mapsto \frF[\mu]$ satisfies
  \begin{equation*}
    \mu_n\to\mu\quad\text{in }\prob_2(\X)\quad\Rightarrow\quad
    \frF[\mu_n]\to \frF[\mu]\quad
    \text{in }\prob_2^{sw}(\TX).
  \end{equation*}
   A single-valued \PVF\ $\frF$ is {\em hemicontinuous} if its domain is totally convex and, for every $\ggamma \in \prob_2(\X \times \X)$ with marginals in $\dom(\frF)$, the restriction of $\frF$ to the
  set $\{\sfx^t_\sharp\ggamma: t\in [0,1]\}$ is demicontinuous.
\end{definition}
 In infinite dimension, hemicontinuity plays a crucial role, as it reduces the problem of verifying continuity to a one-dimensional setting, which is usually easier to handle (see \cite{BrezisFR}).

\begin{theorem}[Characterization of deterministic totally dissipative {\PVF}s]
  \label{thm:trivial-but-useful-to-fix}
  Let $\frF$ be a \emph{single-valued} totally $\lambda$-dissipative
  \PVF.
  \begin{enumerate}
  \item If $\frF$ is maximal, then it is deterministic and
    $\frF[\mu]=(\ii_\X, \ff^\circ[\mu])_\sharp\mu$ for every $\mu\in \dom(\frF)$, where $\ff^\circ$ is the minimal selection of $\frF$ as in Theorem \ref{thm:minimal}.
  \item If $\dom(\frF)=\prob_2(\X)$, then $\frF$ is maximal if and only if
    it is deterministic and demicontinuous
    (or, equivalently, deterministic and hemicontinuous)
  \item If $\dom(\frF)=\prob_2(\X)$ and
    $\frF[\mu]=(\ii_\X, \ff[\mu])_\sharp \mu$
    for every $\mu\in \prob_2(\X)$, then $\frF$ is maximal
    if and only if for every $\zeta\in \rmC_2^{sw}(\TX)$
    and for every sequence $\mu_n\to\mu$ in $\prob_2(\TX)$
    \begin{equation*}
      \lim_{n\to+\infty}\int_\X \zeta(x,\ff[\mu_n](x))\,\d\mu_n(x)
      =\int_\X\zeta(x,\ff[\mu](x))\,\d\mu(x).
    \end{equation*}
  \end{enumerate}
\end{theorem}
\begin{proof}
  Item (1) is an obvious consequence of Theorem \ref{thm:bary-proj} and of the fact that $\frF$ is single-valued.

   We prove item (2):  let us first assume that
  $\frF$ is maximal and let $\mmo$ be its Lagrangian representation.
  Since $\dom(\mmo)=\cH$, $\mmo$ is locally bounded (see Theorem \ref{thm:brezis2}(3)) so that
  if a sequence $(\mu_n)_{n\in\N}$ is converging to $\mu$ in
  $\prob_2(\X)$ and $\Phi_n=\frF[\mu_n]$, we can assume that 
  there exists a constant $C>0$ such that
  \begin{equation*}
    \int_\TX |v|^2\,\d\Phi_n(x,v)\le C\quad\text{for every }n\in \N.
  \end{equation*}
  The compactness
  criterion of Proposition \ref{prop:finalmente}
  shows that $(\Phi_n)_{n\in\N}$ is relatively compact in
  $\prob_2^{sw}(\TX)$.
  On the other hand,
  since $\frF=\clo\frF$
  by Proposition \ref{prop:sw-closedness}, we know
  that any accumulation point of $\Phi_n$ belongs to $\frF$
  and therefore it should coincide with $\frF[\mu]$.

  In order to prove the converse implication,
  it is sufficient to consider the case
  $\lambda=0$ and $\frF$ deterministic and hemicontinuous;
  we reproduce the argument of \cite{BrezisFR} in the measure
  theoretic framework.
  
  We first observe that
  the Lagrangian representation $\mmo$ of $\frF$
  is everywhere defined and single-valued, since
  $\iota_X=\mu$, and $\iotaT_{X,V}=\frF[\mu]=(\ii_\X, \ff)_\sharp \mu$ yield
  $V=\ff\circ X$.

  Let $(Y,W)\in \cH\times \cH$
  satisfying 
  \begin{displaymath}
    \la \mmo (X)-W,X-Y\ra_{\cH}\le 0\quad\text{for every }X\in \cH.
  \end{displaymath}
  Replacing $X$ with $Y_t:=(1-t)Y+tX$, $t\in (0,1)$ and setting
  $\mu_t:=\iota_{Y_t}$, $\ff_t:=\ff[\mu_t]$, $V_t:=\ff_t\circ Y_t
  = \mmo(Y_t)$, we get 
  \begin{equation*}
    \la V_t-W,X-Y_t \ra_{\cH} = \frac{1-t}{t} \la V_t-W,Y_t -Y \ra_{\cH} \le 0 \quad\text{for every }X\in \cH,
  \end{equation*} 
  so that 
  \begin{equation}
    \label{eq:18}
    \la V_t-W,X-Y_t \ra_{\cH} \le 0 \quad\text{for every }X\in \cH.
  \end{equation}
  Let us now set $\bm\vartheta_t:=(X ,Y_t  ,V_t)_\sharp \P\in
  \prob_2(\X^2 \times \X)$. Denoting by $\sfx,\sfy,\sfv$ the projections
  of the points of $\X^3$ to their components, since $(\sfy,\sfv)_\sharp
\bm\vartheta_t=\frF[\mu_t]$,
  by hemicontinuity assumption we know that
  \[(\sfy,\sfv)_\sharp\bm\vartheta_t\to
  (Y,\ff_0\circ Y)_\sharp \P=\frF[\mu_0],\quad\text{in $\prob_2^{sw}(\X\times \X)$ as $t\downarrow0$}.\] 
  On the other hand, $(\sfx,\sfy)_\sharp\bm\vartheta_t=
  \iota^2_{X,Y_t}$ converges to $\iota^2_{X,Y}$ in $\prob_2(\X^2)$ so that
  by compactness, we can also find a sequence  $\left(t(n)\right)_{n\in\N}$, with $t(n)\downarrow0$, 
  such that $\bm\vartheta_{t(n)}\to \bm\vartheta$
  in $\prob_2^{sw}(\X^2\times \X)$. Clearly
  $(\sfy,\sfv)_\sharp \bm\vartheta=
  (\ii_\X, \ff_0)_\sharp \mu_0$ is concentrated on a graph,
  so that $\bm\vartheta=(X,Y,\ff_0\circ Y)_\sharp\P$.

  Since
  \begin{displaymath}
    \la V_t,X-Y_t\ra_{\cH}=
    \int_{\X^2\times\X} \langle v,x-y\rangle\,\d\bm\vartheta_t
  \end{displaymath}
  and the function $\zeta(x,y,v):=\langle v,x-y\rangle$
  belongs to $\rmC_2^{sw}(\X^2\times \X)$ we deduce that
  \begin{displaymath}
    \lim_{n\to+\infty}\la V_{t(n)},X-Y_{t(n)}\ra_{\cH}=
    \int_{\X^2\times\X} \langle v,x-y\rangle\,\d\bm\vartheta=
    \la \ff_0( Y),X-Y\ra_{\cH}.
  \end{displaymath}
  Thus, we can pass to the limit in \eqref{eq:18} obtaining 
  \begin{displaymath}
    \la \ff_0(Y)-W,X-Y\ra_{\cH}\le 0\quad\text{for every }X\in \cH,
  \end{displaymath}
  in particular it holds for $X=\ff_0(Y)-W+Y$.
We deduce that $W=\ff_0\circ Y=\mmo (Y)$ so that $\mmo$ is maximal
  and $\frF$ is maximal as well.
  
   Finally, item (3)  is just the equivalent way to express the demicontinuity
  of $\frF$, recalling Definition \ref{def:p2sw}.
\end{proof}

An important example of single-valued, everywhere defined
demicontinuous \PVF is provided by the Yosida approximation:
starting from a maximal totally $\lambda$-dissipative \MPVF\ $\frF$
and its Lagrangian representation $\mmo$, 
for every $\tau\in (0,1/\lambda^+)$ we consider its Yosida approximation $\mmo_\tau:=\frac{(\ii_{\cH}-\tau\mmo)^{-1}-\ii_{\cH}}{\tau}$ and define the corresponding (single-valued) \PVF 
\begin{equation}\label{eq:ftaudef}
    \frF_\tau:=\iotaT(\mmo_\tau).
\end{equation}
Notice that $\frF_\tau$ is maximal totally $\lambda/(1-\lambda\tau)$-dissipative (see Theorem \ref{thm:brezis2}). Moreover, by Theorem \ref{thm:trivial-but-useful-to-fix}(1), \eqref{eq:7-3} and \eqref{eq:ftaudef} we get that
\[\frF_\tau[\mu]=(\ii_\X, \ff_\tau[\mu])_\sharp\mu, \quad\text{for all }\mu\in\prob_2(\X),\]
where
$\ff_\tau:
\Sp\X 
\to \X$ are given by 
$ \ff_\tau[\mu](\cdot):=\bb_\tau(\cdot,\mu)$
with $\bb_\tau$ as in \eqref{eq:7-3}; notice that $\ff_\tau$ admits a continuous version
defined in $\Sp\X$ and $\ff_\tau(\cdot,\mu)$ 
belongs to $\Lip(\supp(\mu);\X)$ 
for every $\mu\in \prob_2(\X)$ 
and clearly admits a Lipschitz extension 
to $\X$ 
(see Theorem \ref{thm:invTOlawinv}).
Setting $L_\tau:=\frac 1\tau (2-\lambda\tau)/(1-\lambda\tau)$, by $L_\tau$-Lipschitz continuity of $\mmo_\tau$ and the representation \eqref{eq:7-3}, we get the following Lipschitz condition
\begin{equation}
\label{eq:Lipschitz}
        \int_{\X^2} \Big|\ff_\tau(x_0,\mu_0)-
    \ff_\tau(x_1,\mu_1)\Big|^2\,\d\mmu(x_0,x_1)
    \le L_\tau^2\int_{\X^2} |x_0-x_1|^2\,\d\mmu(x_0,x_1)
    \quad\text{for every }
    \mmu\in \Gamma(\mu_0,\mu_1),
\end{equation}
which clearly implies demicontinuity of $\frF_\tau$.
We have thus proved the following result, recalling also Theorem \ref{thm:minimal}(2).
\begin{corollary}\label{cor:regular-appr-MPVF}
  Let $\frF\subset\prob_2(\TX)$ be a maximal totally $\lambda$-dissipative \MPVF.
  There exist sequences $\lambda_n, L_n\in\R$ and a sequence of maps
  $\ff_n:\prob_2(\X)\to \Lip(\X,\X)$
  satisfying the Lipschitz condition \eqref{eq:Lipschitz} with $L_n$ in place of $L_\tau$
  inducing a sequence of single-valued
  maximal totally $\lambda_n$-dissipative {\PVF}s $\frF_n$, and
  satisfying 
  \[ \lim_{n\to+\infty}\int_\X\big|\ff_n[\mu](x)-\ff^\circ[\mu](x)\big|^2\,\d\mu(x)=0
    \quad\text{for every }\mu\in \dom(\frF),\]
where $\ff^\circ$ is as in Theorem \ref{thm:minimal}.
\end{corollary}

 To conclude this section, devoted to deterministic {\MPVF}s, we anticipate  a result  which gives a sufficient condition to pass from dissipativity to total dissipativity  in the deterministic case.  Its proof, in a more general framework, is deferred to Section \ref{sec:constructionFlagr} (see in particular Theorem \ref{thm:demi-case}).  We will see how the required condition on the dimension of $\X$ will allow us to play with measures with finite support so to slightly perturb non-optimal couplings into optimal ones, at least for a small interval. This perturbation argument is presented in Section \ref{sec:coupl} and then applied later in Section \ref{sec:strong-dissipative} to get first interesting relations between metric and total dissipativity. 

\begin{theorem}\label{thm:antic} Assume that $\dim(\X) \ge 2$. Let $\mathsf{U} \subset \X$ be an open, convex, non-empty subset of $\X$ and let $\frF \subset \prob_2(\TX)$ be a  deterministic $\lambda$-dissipative \MPVF with domain $\dom(\frF)=\prob_f(\mathsf U)$. Then $\frF$ is totally $\lambda$-dissipative.
\end{theorem}

\section{Lagrangian and Eulerian flow generated by a totally dissipative \MPVF}\label{sec:totdissMPVF-flow}
In this section, making use of the results obtained in the previous Section \ref{sec:invmpvf}, 
we study the well-posedness for $\lambda$-\EVI solutions driven by a maximal totally $\lambda$-dissipative \MPVF $\frF$. These curves are characterized (time by time) as the law of the unique semigroup of Lipschitz transformations $\Sgp_t$ of the Lagrangian representation $\mmo$ of $\frF$. 
 As in the previous section, we will consider a standard Borel
space $(\Omega, \cB)$ 
endowed with a nonatomic probability measure $\P$ 
and the Hilbert space
$\cH:=L^2(\Omega, \cB, \P; \X)$.

\begin{definition}[Lagrangian flow]\label{def:semig}
Let $\frF \subset \prob_2(\TX)$ be a 
maximal totally $\lambda$-dissipative \MPVF.
We call \emph{Lagrangian flow}
the family of maps $\ss_t:\mathcal S(\X,\overline{\dom(\frF)})\to \X$
defined by Theorem
\ref{thm:invTOlawinv}
starting from the Lagrangian representation
$\mmo$ of $\frF$. \\
The Lagrangian flow induces 
a semigroup of $(\prob_2(\X),W_2)$-Lipschitz transformations
$S_t:\overline{\dom(\frF)}
\to \overline{\dom(\frF)}$
defined by 
$S_t(\mu_0):= \ss_t(\cdot,\mu_0)_\sharp \mu_0$.

We say that the continuous 
curve $\mu:[0,+\infty)\to\mathcal P_2(\X)$
is a  \emph{Lagrangian solution}
of the flow generated by $\frF$
if $\mu_t=
S_t (\mu_0)=
\ss_t(\cdot,\mu_0)_\sharp\mu_0$
for every $t\ge0$.
\end{definition}
Notice that, if $\mu$ is a Lagrangian solution, the semigroup property \eqref{eq:8} 
of the Lagrangian flow $\ss_t$ yields in particular
\begin{equation*}
    \mu_t=\ss_{t-s}(\cdot,\mu_s)_\sharp\mu_s\quad\text{for every }
    0\le s\le t.
\end{equation*}
 In particular, to construct a Lagrangian solution starting from $\mu_0\in \dom(\frF)$ 
it is sufficient to choose an arbitrary map $X_0\in \cH$ satisfying $\iota_{X_0}=\mu_0$
and set $\mu_t:=\iota_{X_t}$ for the (unique) locally Lipschitz solution 
$X\in \Lip_{\rm loc}([0,+\infty);\cH)$ of 
\begin{displaymath}
    \frac \d{\d t }X_t=\mmo^\circ (X_t)\quad\text{a.e.~in $(0,+\infty)$,}\quad
    X\restr{t=0}=X_0.
\end{displaymath}
An immediate consequence of Theorem
\ref{thm:invTOlawinv}
is the following result.
\begin{theorem}[Existence of Lagrangian solutions]
    \label{thm:existence-Lagrangian}
    If $\frF \subset \prob_2(\TX)$ is a maximal totally $\lambda$-dissipative \MPVF then for
    every $\mu_0\in \overline{\dom(\frF)}$
    there exists a unique Lagrangian solution $\mu:[0,+\infty)\to\mathcal P_2(\X)$
    starting from $\mu_0$.

    If $\mu_0\in \dom(\frF)$,
    then
    $\mu_t\in \dom(\frF)$
    for every $t\ge0$, 
    the curve $\mu:[0,+\infty)\to\mathcal P_2(\X)$
    is locally Lipschitz continuous, and 
    \begin{equation}
    \label{eq:apriori-estimate}
    \int_\X |\ff^\circ(x,\mu_t)|^2\,\d\mu_t(x)\le 
    e^{\lambda t}\int_\X |\ff^\circ(x,\mu_0)|^2\,\d\mu_0(x)
    \quad\text{for every }t\ge0,
\end{equation}
where $\ff^\circ$
    is defined in Theorem \ref{thm:minimal}
    and induces a map $(x,t)\mapsto \ff^\circ(x,\mu_t)$
    which is $\mmu$-measurable
    with respect to $\mmu=\int \mu_t \,\d t$ in every
    set $\X\times (0,T)$, $T>0$.\\
    Moreover, 
    $\mu$ is the unique \emph{Eulerian solution}
    of the flow generated by $\frF$ in the following sense: 
    $\mu:[0,+\infty)\to\mathcal P_2(\X)$ is the unique
    distributional solution
    of
    \begin{equation}
    \label{eq:cont-3}
    \partial_t \mu_t+\nabla\cdot(
    \mu_t \,\ff^\circ(\cdot, \mu_t))=0
    \quad\text{in $(0,+\infty)\times \X$}
\end{equation}
among the class of locally absolutely continuous curves satisfying
$\mu_{t=0}=\mu_0\in \dom(\frF)$
and
    \begin{equation}
    \label{eq:L2estimate}
        \int_0^T
        \int_\X |\ff^\circ(x,\mu_t)|^2\,\d 
        \mu_t\,\d t<+\infty
        \quad\text{for every }T>0.
    \end{equation}
    Finally, for every $\mu_0\in \overline{\dom(\frF)}$ and $t>0$ we have
    \begin{enumerate}
    \item if $\supp(\mu_0)$
      is finite, then $\supp(\mu_t)$ is finite    and its
      cardinality is nonincreasing w.r.t.~$t$.
      In particular, if $\mu_0\in \prob_{f,N}(\X)$ for some $N\in \N$
      (recall \eqref{eq:30})
      then $\mu_t\in \prob_{f,N}(\X)$ for every $t\ge0$;
    \item
      if $\supp(\mu_0)$ is compact, then
      $\supp(\mu_t)$ is compact;
    \item 
      if $\supp(\mu_0)$ is bounded, then
      $\supp(\mu_t)$ is bounded and
      $\operatorname{diam}(\supp(\mu_t))\le
      \mathrm e^{\lambda t}\operatorname{diam}(\supp(\mu_0))$;
    \item
      if
      $\int_\X |x|^p \de \mu_0(x)< + \infty$ for some $p \ge 1$,
      then
      $\int_\X |x|^p \de \mu_t(x)< + \infty$ and
      \begin{equation*}
        \int_{\X\times\X} \big|x-y\big|^p\,\d\mu_t\otimes \mu_t\le
        \mathrm e^{p\lambda t}
        \int_{\X\times\X} \big|x-y\big|^p\,\d\mu_0\otimes\mu_0.
      \end{equation*}
    \end{enumerate}
\end{theorem}
\begin{proof}
    The existence and the regularity properties of Lagrangian solutions
    follow by Theorem 
    \ref{thm:invTOlawinv}, while \eqref{eq:apriori-estimate} follows by Theorem \ref{thm:brezis3}(4).

    Property \eqref{eq:10} clearly implies
    \eqref{eq:cont-3}.  Indeed, by definition of Lagrangian solution, we have $\mu_t=\ss_t(\cdot,\mu_0)_\sharp\mu_0$. Thus, by \eqref{eq:10} we have
    \begin{equation}
    \label{eq:10bis+}
    \lim_{h\downarrow0}\frac 1h(\ss_{t+h}(\cdot,\mu_0)-\ss_{t}(\cdot,\mu_0))=\bb^\circ[\mu_t](\ss_{t}(\cdot,\mu_0))\equiv\ff^\circ[\mu_t](\ss_{t}(\cdot,\mu_0)) \quad\text{in }L^2(\X,\mu_0;\X),
  \end{equation}
  where the last equivalence is provided in Theorem \ref{thm:minimal}(2).
  Thus $\mu_t$ satisfies
\begin{align*}
\frac\d{\d t}\int_\X \zeta(x)\,\d\mu_t(x)&=\frac\d{\d t}\int_\X \zeta(\ss_t(x,\mu_0))\,\d\mu_0(x)\\
&=\int_\X \langle\nabla\zeta(\ss_t(x,\mu_0)),\ff^\circ[\mu_t](\ss_{t}(x,\mu_0))\rangle
    \,\d\mu_0(x)\\
&=\int_\X \langle\nabla\zeta(x),\ff^\circ[\mu_t](x)\rangle
    \,\d\mu_t(x)
\end{align*}
for every $\zeta\in \Cyl(\X)$ and a.e.~$t>0$. Hence \eqref{eq:cont-3}.

    Concerning uniqueness of solutions to 
    \eqref{eq:cont-3}
    satisfying 
    \eqref{eq:L2estimate}, we have
    \begin{align*}
        \frac\d{\d t}W_2^2(\mu^1_t,\mu^2_t)
        &\le 2
        \int_{\X^2} \langle \ff^\circ(x_1,\mu^1_t)
        -\ff^\circ(x_2,\mu^2_t),x_1-x_2
        \rangle\,\d\mmu_t\\
        &\le 2\lambda \int_{\X^2} |x_1-x_2|^2\d\mmu_t\\
        &= 2\lambda W_2^2(\mu^1_t,\mu^2_t)
    \end{align*}
    for a.e.~$t\ge0$ and every
    $\mmu_t\in \Gamma_o(\mu^1_t,\mu^2_t)$, by Theorem \ref{thm:all}(6b) thanks to \eqref{eq:L2estimate}, the total $\lambda$-dissipativity of $\frF$ and \eqref{eq:defq}. Hence, by Gr\"onwall inequality, we get
    \[W_2(\mu^1_t,\mu^2_t)\le e^{\lambda t}W_2(\mu^1_0,\mu^2_0).\]

   The $\mmu$-measurability of the map $(x,t)\mapsto \ff^\circ(x,\mu_t)$ follows by continuity of $t\mapsto\mu_t$ together with Theorem \ref{thm:minimal}(4) with $\Y=[0,T]$. Indeed, \eqref{eq:22} holds thanks to \eqref{eq:apriori-estimate}.
    
    The last assertions (1-4) come from the fact that $\mu_t=\ss_t(\cdot, \mu_0)_\sharp \mu_0$ and this map is $e^{\lambda t}$-Lipschitz continuous (cf.~Theorem \ref{thm:invTOlawinv}(3)).
\end{proof}

\begin{remark}[A sticky-particle interpretation]
\label{rem:sticky}
We may interpret property (1) of the previous Theorem \ref{thm:existence-Lagrangian}
by saying that the flows of totally dissipative {\MPVF}s model sticky particle evolutions,
(see also \cite{Natile-Savare09}). This fact reflects at a dynamic level the 
barycentric projection property stated in Theorem \ref{thm:bary-proj}. 
In contrast, we immediately see that the example of $\frac{1}{2}$-dissipative \PVF, with $\X=\R$, analysed in \cite[Section 7.1]{Piccoli_2019}, \cite[Section 6]{Camilli_MDE} and later discussed in \cite[Section 7.5, Example 7.11]{CSS}, cannot be maximally total $\frac{1}{2}$-dissipative since it produces a $\frac{1}{2}$-\EVI solution which splits the mass for positive times if e.g. $\mu_0=\delta_0$.
Notice indeed that, as highlighted in the following Theorem \ref{thm:flow-generation}, if $\frF$ is maximal totally dissipative then Lagrangian and \EVI solutions coincide.
\end{remark}
It is remarkable that 
the Lagrangian flow $\ss_t$ 
provides an explicit representation
of the flow of Lipschitz transformations
generated by the unique $\lambda$-\EVI solution,
see \cite[Definition 5.21]{CSS} and Definition \ref{def:evi}.
\begin{theorem}
[EVI solutions and contraction]
    \label{thm:flow-generation}
    If $\frF \subset \prob_2(\TX)$ is a maximal totally $\lambda$-dissipative
    \MPVF, then 
    for every $\mu_0\in \overline{\dom(\frF)}$, the curve $\mu:[0,+\infty)\to\prob_2(\X)$,
    $\mu_t:=S_t(\mu_0)$, is the unique $\lambda$-EVI solution starting from 
    $\mu_0$ and 
    $S_t$ is a semigroup of $e^{\lambda t}$-Lipschitz transformations
    satisfying
    \begin{equation*}
        W_2(S_t(\mu_0),S_t(\mu_1))
        \le e^{\lambda t}
        W_2(\mu_0,\mu_1)
        \quad\text{for every }
        \mu_0,\mu_1\in \overline{\dom(\frF)},\,t\ge0.
    \end{equation*}
\end{theorem}
\begin{proof}
    The proof is an immediate consequence
    of 
    \cite[Theorem 5.22(e)]{CSS}
    and Theorem \ref{thm:existence-Lagrangian}. Indeed notice that \cite[Theorem 5.22(e)]{CSS} can be applied even if the absolutely continuous curve $\mu$ satisfies the differential inclusion
    \begin{equation}\label{eq:diffinclv}
    (\ii_\X,\vv_t)_\sharp\mu_t\in\frF[\mu_t]
    \end{equation}
    w.r.t.~to any Borel vector field $\vv_t$ s.t.~$(\mu, \vv)$ solves the continuity equation and $t \mapsto |\vv_t|_{L^2(\X, \mu_t; \X)} \in L^1_{loc}(0,+\infty)$. For instance it holds for the vector field $\ff^\circ$. Indeed, the proof of \cite[Theorem 5.22(e)]{CSS} relies on \cite[Theorem 5.17(2)]{CSS} which holds even if the differential inclusion  \eqref{eq:diffinclv}, with $\vv$ the Wasserstein vector field, is replaced by a general velocity field $\vv$ as above. See also \cite[Remark 3.12]{CSS}.
\end{proof}
As a further consequence,
in the case of maximal $\lambda$-totally dissipative \MPVF
all the various definitions of solutions
coincide.
\begin{theorem}
    \label{thm:final-total}
    Let $\frF \subset \prob_2(\TX)$ be a maximal totally $\lambda$-dissipative \MPVF,
    let $\mu_0\in \overline{\dom(\frF)}$
    and let $\mu:[0,+\infty)\to\mathcal P_2(\X)$
    be a continuous curve
    starting from $\mu_0$.
    The following properties are equivalent:
    \begin{enumerate}
        \item $\mu$ is a Lagrangian solution.
        \item $\mu$ is a $\lambda$-\EVI solution.
     \end{enumerate}
    If moreover 
    $\mu_0\in \dom(\frF)$
    or there exists 
    a sequence $t_n\downarrow0$ for which $\mu(t_n)\in \dom(\frF)$,
    the above conditions are also equivalent to the following ones:
    \begin{enumerate}
    \setcounter{enumi}{2}
        \item 
       there exists 
       a Borel vector field 
        $\ww_t$
satisfying
\begin{equation}
\label{eq:diff-inclusion}
\begin{split}
    &t\mapsto \int_\X |\ww_t(x)|^2\,\d\mu_t(x)
    \quad\text{is locally integrable in 
     $(0,+\infty)$},\\
    &(\ii_\X, \ww_t)_\sharp \mu_t\in \frF\quad\text{for a.e.~$t>0$}
    \end{split}
\end{equation}
      and the pair $(\mu,\ww)$ satisfies the continuity equation 
\begin{equation}
    \label{eq:cont1}
    \partial_t \mu_t+\nabla\cdot(
    \mu_t \,\ww_t)=0
    \quad\text{in $(0,+\infty)\times \X$};
\end{equation}
\item 
there exists a Borel family $\Phi_t$, $t>0$, such that 
\begin{gather}
\label{eq:diff-inclusion-weak1}
    \Phi_t\in \frF[\mu_t]\quad\text{for a.e.~$t>0$,}\quad
    t\mapsto \int_{\TX}|v|^2\,\d\Phi_t\quad\text{is locally integrable in 
     $(0,+\infty)$,} \\
    \int_0^{+\infty}
    \Big(\int_\X \partial_t \zeta(t,x)\,\d\mu_t+
    \int_\TX \langle v,\nabla\zeta(t,x)\rangle
    \,\d\Phi_t(x,v)\Big)\,\d t=0
    \quad\text{for every $\zeta\in \Cyl((0,+\infty)\times \X)$;}
\end{gather}
\item
$\mu_t\in \dom(\frF)$ 
for every $t>0$,
 $\mu$ is locally Lipschitz in $(0,+\infty)$,
it satisfies \eqref{eq:cont-3} and
 \begin{equation*}
      t\mapsto  \int_\X |\ff^\circ(x,\mu_t)|^2\,\d 
        \mu_t
        \quad\text{is locally bounded in }(0,+\infty).
    \end{equation*}
    \end{enumerate}     
\end{theorem}
\begin{proof}
    The equivalence between 
    (1) and (2) is a consequence of 
    Theorem \ref{thm:flow-generation}.
    
     We can now consider the case when $\mu_0\in \dom(\frF)$
    (the argument for the case $\mu(t_n)\in \dom(\frF)$ along an infinitesimal sequence
    $t_n$ is completely analogous).
    Theorem \ref{thm:existence-Lagrangian} 
    clearly yields (1) $\Rightarrow$
    (5). The implication (5) $\Rightarrow$
    (3) is obvious.
    Theorem 
    \ref{thm:bary-proj}
    shows that (3) and (4) are equivalent. Indeed (3) implies (4) by choosing $\Phi_t :=(\ii_\X, \ww_t)_\sharp \mu_t$ and (4) implies (3) by choosing $\ww_t:=\bry{\Phi_t}$.
    The implication (3) $\Rightarrow$ (2)
    follows by Theorem 5.4(1) of
    \cite{CSS}.
\end{proof}
In the case when $\mu_0\in \prob_f(\X)$ has finite support (recall \eqref{eq:70},
\eqref{eq:30}), 
we can obtain a more refined characterization, which 
also yields a regularization effect when $\X$ has finite dimension
and recovers the characterization \eqref{eq:ODE-system-intro} 
anticipated in the Introduction.
Recall that 
by Theorem \ref{thm:existence-Lagrangian}(1) any
Lagrangian solution starting from $\mu_0\in \prob_{f,N}(\X)$
must stay in $\prob_{f,N}(\X)$ for every time $t\ge0$.
\begin{corollary}[Regularization effect and Wasserstein velocity field for discrete measures]
    \label{cor:non-disprezzabile}
    Let $\frF \subset \prob_2(\TX)$ be a maximal totally $\lambda$-dissipative \MPVF, 
    let $\mu_0\in \overline{\dom(\frF)}\cap \prob_{f,N}(\X)$ for some $N\in \N$
    and let $\mu:[0,+\infty)\to \prob_{f,N}(\X)$
    be a continuous curve
    starting from $\mu_0$. 
    Assume moreover that at least one of the following properties holds:
    \begin{enumerate}[(a)]
        \item $\mu_0\in \dom(\frF)$,
        \item $\dom(\frF) \cap \prob_{f,N}(\X)$ has non empty relative interior in 
        $\prob_{f,N}(\X)$,
        \item $\X$ has finite dimension.
    \end{enumerate}
    Then conditions (1),\ldots,(5) 
    of Theorem 
    \ref{thm:final-total} are equivalent
    and, in this case, the minimal selection $\ff^\circ$ of $\frF$ (cf. Theorem \ref{thm:minimal})
    coincides with the Wasserstein velocity field $\vv$ of $\mu$ (cf. Theorem \ref{thm:tangentv})
    and $\mu$ also satisfies the right-differentiability property
    \begin{equation}
        \label{eq:right-diff}
        \vv_t=\lim_{h\downarrow0}\frac1h \Big(\bm{t}_t^{t+h}-\ii_\X\Big)=
        \ff^\circ[\mu_t]\quad\text{in }L^2(\X,\mu_t;\X)
        \quad\text{for every }t>0,
    \end{equation}
    where $\bm t_{t}^{t+h}$ 
    is the optimal transport map pushing $\mu_t$ into $\mu_{t+h}$.
    
    Finally, $\mu$ is a Lagrangian solution for $\frF$ starting from $\mu_0$
    if and only if
    there are curves 
    $\mathsf x_n\in \rmC([0,+\infty);\X)$, $n=1,\cdots,N$ 
    which are locally Lipschitz in $(0,+\infty)$  such that 
$\mu_t=\frac1N \sum_{n=1}^N\delta_{\mathsf x_n(t)}$ 
for every $t\ge0$
and the curves  $\{\mathsf x_n(t)\}_{n=1}^N$  solve 
the system of ODEs
\begin{equation}
    \label{eq:ODE}
    \dot{\mathsf x}_n(t)=\ff^\circ(\mathsf x_n(t),\mu_t)\quad
    \text{a.e.~in }(0,+\infty).
\end{equation}
\end{corollary}
\begin{proof}
    Case (a) is part of Theorem \ref{thm:final-total}. In order to prove the first equivalence statement in cases (b) and (c), we briefly 
    anticipate an argument that we will develop more extensively in Section 
    \ref{sec:constructionFlagr}:
    we 
    introduce the standard Borel space $\Omega:=[0,1)$ endowed with the Lebesgue measure 
    (still denoted by $\P$), the Lagrangian representation $\mmo$ of $\frF$,
    and we consider the closed subspace $\cH_N\subset\cH$
    of maps $X:\Omega\to \X$ which are constant on each interval 
    $[(k-1)/N,k/N)$, $k=1,\cdots, N.$
    
    Thanks to Theorem \ref{thm:invTOlawinv}, $\cH_N$ is invariant with respect to the action
    of the resolvent map $\jJ_\tau$, i.e.~if $X\in\cH_N$ then $\jJ_\tau(X)\in\cH_N$. Indeed, by Theorem \ref{thm:invTOlawinv}, if $k\in\{1,\dots,N\}$ and $\omega,\omega'\in[(k-1)/N,k/N)$, then
    \[\resolvent\tau (X)(\omega)=\jj_\tau(X(\omega),X_\sharp\P)=\jj_\tau(X(\omega'),X_\sharp\P)=\resolvent\tau (X)(\omega').\]
    
    We can thus apply Proposition \ref{prop:brezz} obtaining that the operator $\mmo_N:=\mmo\cap (\cH_N\times \cH_N)$ is maximal $\lambda$-dissipative in $\cH_N$ and, if we select a Lagrangian parametrization $X_0\in \overline{\dom(\mmo_N)}$ of $\mu_0$, still by Proposition \ref{prop:brezz}(ii), we get that $\Sgp_t X_0$, the semigroup generated by $\mmo$, coincides with $\Sgp_t^N (X_0)$, the semigroup generated by $\mmo_N$ and, under any of the conditions (b) and (c), $\bm S^N$ has a regularizing effect (see Theorem \ref{thm:brezreg}, Corollary \ref{cor:vipregobasta} and notice that, in case (c), $\cH_N$ has finite dimension)    
    so that $\bm S^N_t (X_0)\in \dom(\mmo_N)\subset \dom(\mmo)$ 
    for every $t>0$.
    We immediately obtain that the conditions (1), \dots, (5) of 
    Theorem \ref{thm:final-total} are equivalent.

    In order to check \eqref{eq:right-diff}, we can use 
    \eqref{eq:10bis} observing that, 
    for sufficiently small $h$, $(\ii_\X,\ss_h)_\sharp\mu_t$ is an optimal coupling
    between $\mu_t$ and $\mu_{t+h}$, since $\mu_t\in \prob_{f,N}(\X)$, 
    see the next Lemma \ref{le:quantitative}. 

    Finally, in order to check the last representation formula, it is sufficient to write
    $\mu_0$ as $\frac1N\sum_{n=1}^N x_n$ for suitable points $x_n\in \X$
    and to set $\mathsf x_n(t):=\ss_t(x_n,\mu_0)$. 
\end{proof}

 A further application concerns the convergence of the Implicite Euler Scheme
(also called JKO method in the framework of gradient flows, see 
Proposition \ref{prop:JtauvsJKObis}). We just recall here
the main Crandall-Liggett estimate, referring to 
\cite{NSV00,Nochetto-Savare06} for more refined a-priori and a-posteriori error estimates.
\begin{corollary}[Implicit Euler Scheme]
\label{cor:IES}
    Let $\frF \subset \prob_2(\TX)$ be a maximal totally $\lambda$-dissipative \MPVF, $\mu \in \prob_2(\X)$, and $0<\tau<1/\lambda^+$. Then, denoting by $\Phi \in \prob_2(\TX)$ the unique element of $\frF$ such that 
    \begin{equation}
        \label{eq:JKO-diss}
        (\sfx-\tau \sfv)_\sharp\Phi=\mu
    \end{equation}
    coming from Theorem \ref{cor:resolvent-max-tot}, we have  $M_\tau:= \sfx_\sharp \Phi=\jj_\tau(\cdot,\mu)_\sharp\mu$, where $\jj_\tau$ is as in Theorem \ref{thm:invTOlawinv} applied to the Lagrangian representation of $\frF$.
    If $\mu_0 \in \overline{\dom(\frF)}$, then setting $M^0_\tau:=\mu_0,$
$M^{n+1}_\tau:=\jj_\tau(\cdot,M^n_\tau)_\sharp M^n_\tau$, $n\in \N$, we have
    \begin{equation}\label{eq:brconv}
        \lim_{N\to+\infty}
        M^N_{t/N}=\mu_t\quad\text{for every $t\ge0$},
    \end{equation}
    where $\mu_t=S_t(\mu_0)$ with $S_t$ as in Definition \ref{def:semig}.
    Moreover, for every $T \ge 0$ there exist $N(\lambda, T) \in \N$
    and $C(\lambda, T)>0$ (with $C(0,T)=2T$)
    such that
    \begin{equation}\label{eq:estimImpl}
        W_2(M^N_{t/N},\mu_t)\le 
        \frac{C(\lambda, T)}{\sqrt{N}}\,
        \big|\ff^\circ[\mu_0]\big|_{L^2(\X,\mu_0;\X)},
    \end{equation}
    for every $0 \le t \le T$, $n \ge N(\lambda, T)$ and every $\mu_0 \in \dom(\frF)$,
    where $\ff^\circ$ is as in Theorem \ref{thm:existence-Lagrangian}.
\end{corollary}
\begin{proof} 
The approximation in \eqref{eq:brconv} follows by the Lagrangian one 
  \begin{equation*}
    \Sgp_t (X)=\lim_{N\to+\infty} (\resolvent{t/N})^N(X)
  \end{equation*}
 holding for any $X\in\overline{\dom(\mmo)}$ (see Theorem \ref{thm:brezis4}), $\mmo$ the Lagrangian representation of $\frF$.

 Finally, \eqref{eq:estimImpl} follows by Theorem \ref{thm:brezis4}.
\end{proof}
We conclude this section with two results concerning the uniqueness
and the stability
of the characteristic system representing the solution of
\eqref{eq:diff-inclusion} and \eqref{eq:cont1}.

Using the notation of Theorem \ref{thm:invTOlawinv}, we preliminarily observe that choosing $\mu_0\in \dom(\frF)$
the maps $\ss_t(x):=\ss_t(x,\mu_0)$ belong
to $\Lip(\supp(\mu_0);\X)$ and the curve $t\mapsto \ss_t$ is Lipschitz in
$L^2(\X,\mu_0;\X)$ with derivative $\bb^\circ_t(\ss_t)$ where
$\bb^\circ_t(\cdot):=\bb^\circ(\cdot,(\ss_t)_\sharp \mu_0)$.
It follows that for every $T>0$ and for 
$\mu_0$-a.e.~$x$ the curve
$t\mapsto \ss_t(x)$ belongs to $H^1(0,T;\X)$
and satisfies $\dot \ss_t(x)=\bb^\circ_t(\ss_t(x))$.
We can thus associate to $(\ss_t)_{t \ge 0}$ a $\mu_0$-measurable map 
\begin{equation}
    \label{eq:pedante}
    \mathrm s:\X\to H^1(0,T;\X),\quad
    \mathrm s[x](t):=\ss_t(x,\mu_0).
\end{equation}
In a similar way, if $X_0\in \cH$ with $\iota_{X_0}=\mu_0$,
we can define 
\begin{equation}
    \label{eq:better-representative}
    \mathrm X(\omega, t):=\ss_t(X_0(\omega),\mu_0),\quad
    \mathrm X[\omega]:=\mathrm s\circ X_0,
\end{equation}
obtaining a distiguished Caratheodory representative of 
$\bm S_t (X_0)$ which satisfies
\begin{equation}
    \label{eq:better1}
    \mathrm X(\omega,t)=(\bm S_t (X_0))(\omega)
    \quad\text{for $\P$-a.e.~$\omega\in \Omega$,
    for every $t>0$}
\end{equation}
and
\begin{equation}
    \label{eq:better2}
    \mathrm X(\omega,\cdot)\in H^1(0,T;\X)\quad\text{for $\P$-a.e.~$\omega$},\quad
    \int_\Omega \Big(\int_0^T |\partial_t {\mathrm X}(\omega,t)|^2\,\d t\Big)
    \,\d\P(\omega)\le Te^{2\lambda_+T}|\Bb^\circ (X_0)|_\cH^2,
\end{equation}
since
\begin{align*}
   \int_\Omega \Big(\int_0^T |\partial_t {\mathrm X}(\omega,t)|^2\,\d t\Big)
    \,\d\P(\omega)&=
    \int_\Omega \Big(\int_0^T |\bb^\circ_t({\mathrm X}(\omega,t))|^2\,\d t\Big)\,\d\P(\omega)\\
    &=\int_0^T \left|\Bb^\circ ({\mathrm X}(\cdot,t))\right|_\cH^2\,\d t \\
    &\le Te^{2\lambda_+T}|\Bb^\circ (X_0)|_\cH^2,
\end{align*} where we have used Theorem \ref{thm:brezis3}(4).
It follows that $\mathrm X$ can be identified with a 
$\P$-measurable map $\omega\mapsto \mathrm X[\omega]$ 
which belongs to $L^2(\Omega, \cB,\P;H^1(0,T;\X))$.

\newcommand{\eeta}{\bm \eta}
\begin{theorem}[Uniqueness of the characteristic fields]
  \label{thm:uniqueness} Let $\frF \subset \prob_2(\TX)$ be a maximal totally $\lambda$-dissipative \MPVF, let us fix $T>0$ and let us suppose that $(\mu,\vv)$ is
  a solution to \eqref{eq:diff-inclusion} and \eqref{eq:cont1}
  in the interval $[0,T]$ starting from $\mu_0\in \dom(\frF)$.
  Let $\bm \eta\in \prob(\rmC([0,T];\X))$ be a probability measure
  concentrated on absolutely continuous curves and satisfying
  the following properties:
  \begin{enumerate}
  \item $(\mathsf e_t)_\sharp \eeta=\mu_t$ for every $t\in [0,T]$,
    where $\mathsf e_t (\gamma):=\gamma(t)$ for every $\gamma\in
    \rmC([0,T];\X)$;
  \item $\eeta$-a.e.~$\gamma$ is an integral solution of the
    differential equation
    $\dot\gamma(t)=\vv_t(\gamma(t))$ a.e.~in $[0,T]$.
  \end{enumerate}
  Then $\eeta=
  \mathrm s_\sharp \mu_0$, where $\mathrm s$ 
  is defined as in \eqref{eq:pedante}. 
  In particular $\eeta$ is unique
  and $\vv_t(x)=\bb^\circ_t(x)$ $\mu_t$-a.e.~in $\X$.
\end{theorem}
\begin{proof}
  We can find a Borel map 
  $\mathrm Z:\Omega\to  \rmC([0,T];\X)$
  such that $\mathrm Z_\sharp \P=\eeta$. Let $\Bb$ be the Lagrangian representation of $\frF$.
  We can then define $X_t:=\mathsf e_t \circ \mathrm Z$.
  Since $\iota_{X_t}=\mu_t\in\dom(\frF)$ by Theorem \ref{thm:final-total}(5), recalling Remark \ref{rem:aggiungi} we see that $X_t\in \dom(\mmo)\subset
  \cH$.
  It is also clear that for $\P$-a.e.~$\omega$ we have
  \[X_{t+h}(\omega)-X_t(\omega)=
  \int_t^{t+h} \vv_s(X_s(\omega))\,\d s\]
  and therefore
  $|X_{t+h}-X_t|_\cH\le
  \int_t^{t+h}\|\vv_s\|_{L^2(\X,\mu_s;\X)}\,\d s$
  so that $t\mapsto X_t$ belongs to $H^1(0,T;\cH)$.
  At every differentiability point we have
  $\dot X_t=\vv_t(X_t)$ so that
  $\iotaT_{X_t,\dot X_t}=(\ii_\X,\vv_t)_\sharp \mu_t\in \frF[\mu_t]$
  and eventually $\dot X_t\in \mmo (X_t)$.
  We conclude that $X_t(\omega)=\ss_t(X_0(\omega))$
  and therefore $\eeta=\mathrm s_\sharp \mu_0$.
\end{proof}

\begin{theorem}[Stability of the Lagrangian flows]
        \label{thm:stability}
    Under the same conditions of the previous Theorem \ref{thm:uniqueness},
    let $(\mu^n_0)_{n\in \N}$ 
    be a sequence in $\dom(\frF)$ satisfying the following properties:
    \begin{enumerate}
        \item $(\mu^n_0)_{n\in \N}$ converges to $\mu_0$ in $\prob_2(\X)$, as $n\to+\infty$;
        \item $\sup_n |\frF|_2(\mu^n_0)<+\infty$, where $|\frF|_2(\cdot)$ is defined in \eqref{eq:21}.
    \end{enumerate}
    If $\mathrm s^n,\mathrm s:\X\to \rmC([0,T];\X)$
    are the Lagrangian maps defined as in 
    \eqref{eq:pedante} starting from $\mu^n_0$
    and $\mu_0$ respectively,
    then 
    $(\ii_\X,\mathrm s^n)_\sharp \mu^n_0\to 
    (\ii_\X,\mathrm s)_\sharp \mu_0$ in 
    $\prob_2(\X\times \rmC([0,T];\X))$ as $n\to+\infty$.
\end{theorem}
\begin{proof}
    By the last part of Theorem \ref{thm:gpfinal}, we can select a sequence $(X_0^n)_{n\in \N}$ in $\cH$
    strongly converging to $X_0$ such that 
    $\iota_{X_0^n}=\mu_0^n$ and $\iota_{X_0}=\mu_0.$
    We now consider the family of 
    $\P$-measurable maps 
    $\mathrm X^n:\Omega\to H^1(0,T;\X)
    \subset \mathrm C([0,T];\X)$ 
    defined as in \eqref{eq:better-representative}
    starting from $X^n_0$ and the corresponding $\mathrm{X}$ defined starting from $X_0$.
Our thesis follows if we prove that 
    $\mathrm X^n\to\mathrm X$ 
    in $L^2(\Omega,\cB,\P;\mathrm C([0,T];\X)).$
    
The equivalence    \eqref{eq:better1} and 
    the contraction estimates on $\bm S_t$ (cf.~\eqref{smgprop})
    show that 
    \begin{align*}
        \|\mathrm X^n-\mathrm X\|_{L^2(\Omega;L^2(0,T;\X))}^2
        &=
        \int_\Omega \Big(\int_0^T |\mathrm X^n(\omega, t)
        -\mathrm X(\omega,t)|^2\,\d t\Big)\,\d\P(\omega)
        \\&
        =
        \int_0^T 
        \Big(\int_\Omega |\mathrm X^n(\omega, t)
        -\mathrm X(\omega,t)|^2\,\d\P(\omega)\Big)\,\d t\\
        &=\int_0^T |\bm S_t (X^n_0)-
        \bm S_t (X_0)|_{\cH}^2\,\d t
        \\&\le T\mathrm e^{2\lambda_+ T} |X^n_0-X_0|^2_\cH
        \to 0\quad\text{as }n\to+\infty.
    \end{align*}
    Moreover, recalling \eqref{eq:better2}, we have
    \begin{displaymath}
       \sup_n \| \dot{\mathrm X}^n\|_{L^2(\Omega;L^2(0,T;\X))}^2
        \le Te^{2\lambda_+T} 
        \sup_n |\Bb^\circ (X^n_0)|^2_\cH<+\infty
        \quad\text{for every }n\in \N,
    \end{displaymath}
    so that $(\mathrm X^n)_{n\in\N}$ is uniformly bounded in 
    $L^2(\Omega,\cB, \P;H^1(0,T;\X))$
    by some finite constant $S>0$. The interpolation inequality (cf.~\cite[p.233 (iii)]{Breziss})
    \begin{displaymath}
        \|Y\|^2_{\mathrm C([0,T];\X)}
        \le C\|Y\|_{L^2(0,T;\X)}\|Y\|_{H^1(0,T;\X)}
        \quad\text{for every }Y\in H^1(0,T;\X),
    \end{displaymath}
    gives that 
    the sequence $(\mathrm X^n)_{n\in\N}$ 
    strongly converges to $\mathrm X$
    in $L^2(\Omega,\cB,\P;\mathrm C([0,T];\X))$, since
    \begin{align*}
       \|\mathrm X^n-
       \mathrm X&\|^2_{L^2(\Omega, \cB, \P;\mathrm C([0,T];\X))}
       =
       \int_\Omega \|\mathrm X^n[\omega]-
       \mathrm X[\omega]\|_{\rmC([0,T];\X)}^2\,\d\P
       \\\quad
       &\le C
       \int_\Omega 
       \|\mathrm X^n[\omega]-
       \mathrm X[\omega]\|_{L^2(0,T;\X)}
       \|\mathrm X^n[\omega]-
       \mathrm X[\omega]\|_{H^1(0,T;\X)}\,\d\P
       \\\quad&\le 
       C
       \Big(\int_\Omega 
       \|\mathrm X^n[\omega]-
       \mathrm X[\omega]\|_{L^2(0,T;\X)}^2\,\de\P\Big)^{1/2}
       \Big(\int_\Omega 
       \|\mathrm X^n[\omega]-
       \mathrm X[\omega]\|_{H^1(0,T;\X)}^2\,\de\P\Big)^{1/2}
       \\\quad&\le C\left(S+\|\mathrm X\|_{L^2(\Omega, \cB, \P; H^1(0,T;\X))}\right)
         \|\mathrm X^n-\mathrm X\|_{L^2(\Omega;L^2(0,T;\X))}.
    \end{align*}
\end{proof}

\section{Totally convex functionals in \texorpdfstring{$\prob_2(\X)$}{P}}\label{sec:jkobis}

In this section we analyze the case of a proper and lower
semicontinuous functional $\phi$ which satisfies a strong convexity property.

\begin{definition}[Total ($\menolambda$)-convexity] Let $\phi:\prob_2(\X) \to (-\infty, + \infty]$ and let $\lambda \in \R$. We say that $\phi$ is \emph{totally $\menolambdapar$-convex} if it is $\menolambdapar$-convex along any coupling, i.e.
\begin{equation*}
    \phi(\sfx^t_\sharp \mmu) \le (1-t) \phi(\sfx^0_\sharp \mmu) + t
    \phi(\sfx^1_\sharp \mmu) +
    \frac{\lambda}{2}t(1-t) \int_{\X \times \X} |x-y|^2 \de \mmu(x,y)
\end{equation*}
for every $\mmu \in \prob_2(\X \times \X)$, $t\in[0,1]$.
\end{definition}

Notice that, in particular, if $\phi$ is totally $\menolambdapar$-convex then it is $\menolambdapar$-convex along
generalized geodesics \cite[Definition 9.2.4]{ags} and thus also
geodesically
$\menolambdapar$-convex.
It is also easy to check that
$\phi$ is totally $\menolambdapar$-convex if and only if
\begin{equation*}
\phi^\lambda(\mu):=\phi(\mu)+\frac{\piulambda}2\int |x|^2\,\d\mu
  \quad\text{is {totally convex}}.
\end{equation*}
\\
 Referring to \cite[Definition 10.3.1]{ags},  we recall that the Wasserstein subdifferential $\bm{\partial}
\phi(\mu)\subset \prob_2(\TX)$ of $\phi$  at $\mu$ is defined as
the set of $\Psi\in \prob_2(\TX)$ such that  $\sfx_\sharp \Psi = \mu \in \dom(\phi)$ and
\begin{equation}\label{eq:Wassubdphi}
  \phi(\nu) -
  \phi(\mu) \ge \inf_{\ssigma\in\Lambda(\Psi,\nu)}\int_{\TX\times\X}\scalprod{y-x}{v}\d\ssigma(x,v,y) +o\left (W_2(\mu,\nu)\right ) \quad
   \text{as }\nu\to \mu \text{ in }\prob_2(\X). 
\end{equation}
Equivalently, using the notation of duality pairings introduced in Definition \ref{def:pairings}, we can write \eqref{eq:Wassubdphi} as follows

\begin{equation*}
  \phi(\nu) -
  \phi(\mu) \ge -\brap{\Psi}{\nu} +o\left (W_2(\mu,\nu)\right ) \quad
   \text{as }\nu\to \mu \text{ in }\prob_2(\X). 
\end{equation*}
When
$\phi$ is geodesically $\menolambdapar$-convex,
then it is possible to show that
$\Psi$ belongs to
$\bm{\partial}\phi$ if and only if
$\Psi$ and $\mu=\sfx_\sharp \Psi\in \dom(\phi)$ satisfy
\begin{equation}\label{eq:defsubdnonfr}
\phi(\nu) - \phi(\mu) \ge -\brap{\Psi}{\nu} - \frac{\lambda}{2} W_2^2(\mu,\nu) \quad \text{for every } \nu \in \prob_2(\X). 
\end{equation}
It is easy to check that $-\bm{\partial}\phi$ (cf.\eqref{eq:opposite})
is a $\lambda$-dissipative \MPVF (see also \cite[Section
7.1]{CSS}), but in general not totally $\lambda$-dissipative,  as shown in the following remark.

\begin{remark}[A non totally dissipative subdifferential]\label{rem:nontotd}
We show that the (opposite of the) Wasserstein subdifferential of the Shannon's entropy functional $\mathcal{E}$ in $\R^d$, $d \ge 2$, is not totally dissipative. We recall that $\mathcal{E}:\prob_2(\R^d) \to (-\infty, +\infty]$ is defined as
\begin{equation}
\mathcal{E}(\mu) = \begin{cases}
    \displaystyle\int_{\R^d} \rho \log(\rho) \de \mathcal{L}^d \quad &\text{ if } \mu = \rho \mathcal{L}^d \ll \mathcal{L}^d, \\
    + \infty \quad &\text{ else,}
\end{cases} \quad \mu \in \prob_2(\R^d).
\end{equation}
The \MPVF $\frF:= - \boldsymbol{\partial} \mathcal{E}$ is $0$-dissipative, see \cite[Theorem 7.1]{CSS}. We show that we can find $\Phi_0, \Phi_1 \in \frF$ and $\ggamma \in \Gamma(\sfx_\sharp \Phi_0, \sfx_\sharp \Phi_1)$ such that
\[
\int_{\TRd\times\TRd} \langle v_1-v_0,x_1-x_0
        \rangle \,\d\ttheta(x_0,v_0,x_1,v_1) >0, \quad\text{for any }\ttheta \in \Gamma(\Phi_0, \Phi_1)\text{ s.t. }(\sfx^0, \sfx^1)_\sharp \ttheta = \ggamma.
\]
Let $T: \R^d \to \R^d$, $T(z):=-z$, i.e. the reflection w.r.t.~the origin. Define $f: [0,+\infty) \to [0,1]$ by
\[
f(r):=  \begin{cases} 1 \quad &\text{ if } r \in [0,1],\\
2-r \quad &\text{ if } r \in [1,2],\\
0 \quad &\text{ if } r \in [2,+\infty).
\end{cases}
\]
Let $a_0 \in \R^d$ be any point such that $|a_0|\ge 3$ and consider the density $\rho_0(z):= c_0f(|z-a_0|)$, $z \in \R^d$, with corresponding probability measure $\mu_0 := \rho_0 \mathcal{L}^d \in \prob_2(\R^d)$, and $c_0>0$ a normalization constant such that $\int_{\R^d} \rho_0 \de \mathcal{L}^d =1$. We set
\[
\rho_1:= \rho_0 \circ T, \quad \mu_1:= \rho_1 \mathcal{L}^d, \quad \Phi_i := (\ii_{\R^d}, -\nabla \rho_i/ \rho_i)_\sharp \mu_i, \, \,\, i=0,1, \quad \ggamma := (\ii_{\R^d}, T)_\sharp \mu_0.
\]
By \cite[Theorems 10.4.6, 10.4.13]{ags}, we have that $\Phi_0, \Phi_1 \in \frF$ with $\sfx_\sharp \Phi_i = \mu_i$ for $i=0,1$, and so $\ggamma \in \Gamma(\mu_0, \mu_1)$. Since $\Phi_i$, $i=0,1$, and $\ggamma$ are induced by maps, then the set of $\ttheta \in \Gamma(\Phi_0, \Phi_1)$ with $(\sfx^0, \sfx^1)_\sharp \ttheta = \ggamma$ is a singleton, whose unique element is given  by
\[
\ttheta := (\sfx^0, (-\nabla \rho_0 /\rho_0) \circ \sfx^0, \sfx^1, (-\nabla \rho_1 /\rho_1) \circ \sfx^1)_\sharp \ggamma.
\]
We have
\begin{align*}
    &\int_{\TRd\times\TRd} \langle v_1-v_0,x_1-x_0
        \rangle \,\d\ttheta(x_0,v_0,x_1,v_1)\\
        &= -4\int_{\R^d} \langle \nabla \rho_0(x_0) /\rho_0(x_0), x_0 \rangle \de \mu_0(x_0) \\
        &= -4 \int_{B(a_0, 2) \setminus B(a_0,1)} \langle \nabla \rho(x_0), x_0 \rangle \de \mathcal{L}^d(x_0) \\
        & = 4c_0 \int_{B(a_0, 2) \setminus B(a_0,1)} \left\langle \frac{x_0-a_0}{|x_0-a_0|}, x_0 \right\rangle \de \mathcal{L}^d(x_0) \\
        & = 4c_0 \omega_{d-1}\int_1^2 r^d \de r + 4c_0 \int_{B(0, 2) \setminus B(0,1)} \frac{1}{|x_0|} \langle a_0, x_0 \rangle \de \mathcal{L}^d(x_0) \\
        & = \frac{4c_0 \omega_{d-1}}{d+1} (2^{d+1}-1) + 0 >0,
\end{align*}
where $\omega_{d-1}$ is the surface area of the unit sphere in $\R^d$.
\end{remark}

Let us now consider a totally $\lambda$-convex, proper and lower semicontinuous
functional $\phi$.
We fix a standard Borel space $(\Omega, \cB)$ endowed with a nonatomic probability measure $\P$, with
$\cH:= L^2(\Omega, \cB, \P; \X)$ and we consider the Lagrangian parametrization of $\phi$ given by 
\begin{equation}\label{eq:psi}
  \text{$\psi: \cH \to (-\infty, + \infty]$ \quad defined as } \quad 
  \psi(X):= \phi(\iota_X) \quad \text{ for every } X \in \cH. 
\end{equation}
Clearly, $\psi$ is proper, l.s.c. and $\menolambdapar$-convex, i.e.~$X
\mapsto \psi(X)+\frac{\lambda}{2} |X|^2$ is convex.
\medskip

\newcommand{\bpartial}{\bm\partial}
\newcommand{\btpartial}{\bm\partial_{\mathrm t}}

As a preliminary result, we study the (opposite of the) subdifferential of $\psi$, showing in particular that it is an invariant maximal $\lambda$-dissipative operator. This allows to consider its resolvent operator $\jJ_\tau$ and compare, in Proposition \ref{prop:JtauvsJKObis}, the scheme generated by $\jJ_\tau$ with the Wasserstein JKO scheme (\cite{JKO}) for the functional $\phi$ in $\prob_2(\X)$. We then show relations between $-\partial\psi$ and $-\bm{\partial}\phi$, dealing in particular with the respective elements of minimal norm. Finally, in Theorem \ref{prop:GFiotabis}, we show that the Lagrangian solution to the flow generated by the maximal totally $\lambda$-dissipative \MPVF $\iotaT(-\partial\psi)$ is the unique Wasserstein gradient flow for $\phi$ and the unique $\lambda$-\EVI solution for $-\bm{\partial}\phi$. Analogously to Theorem \ref{thm:flow-generation}, this Wasserstein semigroup can be characterized as the law of the semigroup of Lipschitz transformations $\Sgp_t$ of $-\partial\psi$.

\begin{proposition}
[Total subdifferential]\label{prop:invphi} Let $\phi:\prob_2(\X) \to
  (-\infty, + \infty]$ be a proper, lower semicontinuous and totally
  $\menolambdapar$-convex functional and let $\psi$ be as in
  \eqref{eq:psi}.
  \begin{enumerate}
  \item The opposite of the subdifferential of $\psi$,
    $-\partial \psi$, is an invariant maximal
    $\piulambdapar$-dissipative operator in $\cH\times\cH$.
  \item
    The 
     \emph{total subdifferential}  
    $-\btpartial\phi:=\iotaT(-\partial \psi)$
    is maximal totally
    $\piulambdapar$-dissipative.
  \item An element $\Psi\in \prob_2(\TX)$ satisfying
    $\mu=\sfx_\sharp\Psi\in \dom(\phi)$ belongs to $-\btpartial\phi$
    if and only if for every $\nu\in \dom(\phi)$ and every
    plan $\ttheta\in \Gamma(\Psi,\nu)$ we have
    \begin{equation}
      \label{eq:25}
      \phi(\nu)-\phi(\mu)\ge
      \int_{\TX\times\X} \Big(\langle v,x-y\rangle-
      \frac{\piulambda}2
      |x-y|^2\Big)\,\d\ttheta(x,v,y).
    \end{equation}
    In particular $\btpartial\phi \subset \bm{\partial}\phi$.
  \end{enumerate}
\end{proposition}
\begin{proof}
  As usual it is sufficient to check the case $\lambda=0$.
  
   We prove item (1): 
  by maximality of the $\piulambdapar$-dissipative operator $-\partial
  \psi$ in $\cH \times \cH$ (cf.~Theorem \ref{thm:brezis2}(1) and
  Corollary \ref{cor:phi}) and thanks to Theorem
  \ref{thm:invTOlawinv}, it is enough to prove that $-\partial \psi$
  is invariant by measure-preserving isomorphisms.
  
  Let $(X,V) \in -\partial \psi$ and let $g \in \rmS(\Omega)$. We have
  \[ \psi(Y) - \psi(X) \ge \la V, X-Y \ra_{\cH}
    \quad \text{ for every } Y \in \cH.\]
For every $Z \in \cH$, choosing $Y:=Z\circ g^{-1}$ we get
\begin{align*}
\psi(Z) - \psi(X \circ g) &= \psi(Z \circ g^{-1}) - \psi(X)\\ 
  &\ge \la V, X - Z \circ g^{-1}\ra_{\cH}
  \\
                          & = \la V \circ g, X\circ g- Z \ra_{\cH}.
\end{align*}
This shows that $(X \circ g, V \circ g) \in -\partial \psi$.

Item (2) follows immediately  by Theorem
\ref{thm:maximal-dissipativity}(3).

 We prove item (3):  let us first show that an element $\Psi$ satisfying
\eqref{eq:25} belongs to $-\btpartial \phi$: it is sufficient to
take a pair $(X,V)\in \cH\times \cH$ such that
$\iota^2_{X,V}=\Psi$. For every $Y\in \dom(\psi)$, setting
$\nu:=\iota_Y\in \dom(\phi)$ and $\ttheta:=(X,V,Y)_\sharp \P$,
we get
\begin{displaymath}
  \psi(Y)-\psi(X)=
  \phi(\nu)-\phi(\mu)
  \ge
  \int_{\TX\times\X} \langle v,x-y\rangle\,\d\ttheta(x,v,y)
  =\langle V,X-Y\rangle_{\cH},
\end{displaymath}
which shows that $V\in -\partial\psi(X)$ and therefore $\Psi\in
\iota^2(-\partial\psi)=-\btpartial\phi$.

In order to prove the converse implication, we just take
$\Psi=\iota^2_{X',Y'}\in \iota^2(-\partial\psi)$ for some
$(X',Y')\in -\partial\psi$, $\nu\in \dom(\phi)$, and
$\ttheta\in \Gamma(\Psi,\nu)$.
We can find elements $X,V,Y\in \cH$ such that
$(X,V,Y)_\sharp \P=\ttheta$.
In particular $\iota_Y=\nu$ so that
$\psi(Y)=\phi(\nu)$ and
$\iotaT_{X,V}=\Psi$ so that $(X,V)\in -\partial\psi$,
since $-\partial\psi$ is law invariant and the law of $(X,V)$
coincides with the law of $(X',Y')$.
Since $\psi(X)=\phi(\iota_X)=\phi(\mu)$ and $(X,V)\in-\partial\psi$, we get \eqref{eq:25}
\begin{displaymath}
  \phi(\nu)-\phi(\mu)=\psi(Y)-\psi(X)
  \ge \langle V,X-Y\rangle_{\cH}=
  \int_{\TX\times\X} \langle v,x-y\rangle\,\d\ttheta(x,v,y). \qedhere
\end{displaymath}
\end{proof}

In view of the invariance and the maximal $\lambda$-dissipativity of $-\partial \psi$, by Theorem \ref{thm:minimal}(1,2) we have that the subdifferential of $\psi$ contains elements concentrated on maps, in the sense that for every $X \in \dom(\partial \psi)$ there exist $\ff \in L^2(\X, \iota_X; \X)$ such that $\ff \circ X \in -\partial\psi(X)$. An analogous result has been obtained in \cite[Theorem 3.19(iii)]{gangbotudo} 
for real-valued functionals when $\X$ has finite dimension (cf. also \cite[Lemma 8, Proposition 5]{JMQ}).
\smallskip

The next result gives a correspondence between the minimal selection
and the resolvent operators of $-\partial \psi$ and $-\bm\partial \phi$.
It is remarkable that
the minimal selection $\bpartial^\circ\phi$ of $\bpartial\phi$ is an element of
the smaller set $\btpartial\phi$ and therefore coincides with
$\btpartial^\circ\phi$.
This fact guarantees that the ``Eulerian-Wasserstein'' approach to the gradient flow
of $\phi$ coincides with the ``Lagrangian-Hilbertian'' construction.

In the following, $\jJ_\tau$ denotes the resolvent of the invariant
maximal $\piulambdapar$-dissipative operator $-\partial \psi$ for $0<
\tau<1/\piulambdapar^+$
 with the corresponding map $\jj_\tau$ introduced in Theorem
\ref{thm:invTOlawinv}.

\begin{proposition}[JKO scheme, Wasserstein and total subdifferential]\label{prop:JtauvsJKObis}
Let $\phi:\prob_2(\X) \to (-\infty, + \infty]$ be a proper, lower semicontinuous and totally $\menolambdapar$-convex functional and let $\psi$ be as in \eqref{eq:psi}. Then:
\begin{enumerate}
\item
  For every $\mu\in \prob_2(\X)$ and $0<\tau < 1/\piulambdapar^+$
  the measure $\mu_\tau:=\jj_\tau(\cdot,\mu)_\sharp\mu$ 
  is the unique
  solution
  of the JKO scheme for $\phi$ starting from $\mu$, i.e.~$\mu_\tau$ is the unique minimizer of 
  \begin{equation}
    \label{eq:JKO}
  \nu\mapsto \frac{1}{2\tau} W_2^2(\mu, \nu) + \phi(\nu).
\end{equation}
Equivalently, if $\mu=\iota_X$ for some $X \in \cH$, then
$\mu_\tau=\iota(\jJ_\tau (X))$.
\item For every $\mu=\iota_X\in \dom(\btpartial \phi)$,
  the element of minimal norm $\btpartial^\circ\phi[\mu]$
  (equivalently, the 
  law of the element of minimal norm of $\partial \psi(X)$)
  is the element of minimal norm of $\bm{\partial} \phi[\mu]$.
\item We have that $\iota(\dom(\partial \psi))=  \dom(\btpartial\phi)=\dom(\bm{\partial}\phi)$
  and the minimal selection $-\bm{\partial}^\circ\phi $ of
  $-\bm{\partial}\phi$ is concentrated on a map and it is
  totally $\piulambdapar$-dissipative.
    \item The \MPVF $\iotaT(-\partial \psi)$ is the unique maximal totally $\piulambdapar$-dissipative extension of $-\bm{\partial}^\circ\phi$ with domain included in $\overline{\dom(\phi)}$.
\end{enumerate}
\end{proposition}
\begin{proof}
  By Theorem \ref{prop:invphi} and Theorem \ref{thm:invTOlawinv}, we
  have that $\mu_\tau$ does not depend on the choice of $X \in \cH$
  such that $\iota_X = \mu$; if $\nu \in \prob_2(\X)$,
  $\nu\neq \mu_\tau$, we can thus find $(X',Y) \in \cH^2$ such that
  $\iotaT_{X',Y}=\in \Gamma_o(\mu, \nu)$,
  $\mu_\tau=\iota(\jJ_\tau (X'))$, and $Y\neq \jJ_\tau (X')$, since $\iota_Y=\nu\neq
  \mu_\tau=\iota(\jJ_\tau (X'))$.
  By the properties of the resolvent operator $\jJ_\tau$ (cf. Corollary \ref{cor:phi}), we have that
  \begin{align*}
      \phi(\mu_\tau) + \frac{1}{2 \tau} W_2^2(\mu_\tau, \mu) &\le
    \psi(\jJ_\tau (X')) + \frac{1}{2\tau} |\jJ_\tau (X') - X' |_{\cH}^2
   \\
   &<
    \psi(Y) + \frac{1}{2 \tau} |Y-X'|_{\cH}^2 \\
    &= \phi(\nu) +  \frac{1}{2 \tau} W_2^2(\mu, \nu),
  \end{align*} 
  which shows that $\mu_\tau$ is a strict minimizer of
  \eqref{eq:JKO}.

  To prove (2), first of all notice that, thanks to \cite[Lemma 10.3.8]{ags}, $\phi$ is a \emph{regular functional} according to \cite[Definition 10.3.9]{ags}.
Let $-\partial^\circ \psi (X)$ be the element of minimal norm in $-\partial \psi (X)$ and let us denote by $\mu:=\iota_X$ and $\Phi_\mu:=(X,-\partial^\circ \psi (X))_\sharp \P \in -\bm\partial \phi[\mu]$ by Proposition \ref{prop:invphi}. We have
\begin{equation}\label{eq:slopeidphi}
     |\Phi_\mu|_2^2 = |-\partial^\circ \psi (X)|_{\cH}^2 = \lim_{\tau \downarrow 0} \frac{\psi(X)-\psi(\jJ_\tau (X))}{\tau} = \lim_{\tau \downarrow 0} \frac{\phi(\mu)-\phi(\mu_\tau)}{\tau} = |-\bm \partial ^\circ \phi(\mu)|_2^2,
\end{equation}
where $-\bm \partial ^\circ \phi(\mu)$ denotes the unique element of minimal norm in $-\bm{\partial}\phi[\mu]$ (cf.~\cite[Theorem 10.3.11]{ags}), the last equality comes from \cite[Remark 10.3.14]{ags} and the second equality comes from Corollary \ref{cor:phi}. Since $\Phi_\mu\in-\bm\partial \phi[\mu]$ and by uniqueness of the element of minimal norm in $-\bm\partial \phi[\mu]$, we conclude that the slope identity \eqref{eq:slopeidphi} proves (2).

Also (3) follows  by Corollary \ref{cor:phi}, while the fact that $-\bm{\partial}^\circ\phi=-\btpartial^\circ\phi[\mu]$ is concentrated on a map follows by Theorem \ref{thm:minimal}(1) since $-\btpartial\phi[\mu]$ is maximal totally $\lambda$-dissipative by Proposition \ref{prop:invphi}(2).
To prove (4) it is enough to notice that, if $\frG$ is a maximal totally $\piulambdapar$-dissipative extension of $-\bm{\partial}^\circ \phi$ with domain included in $\overline{\dom(\phi)}$, then its Lagrangian representation $\mmo$ has domain included in $\overline{\dom(\psi)}$ and it is $\piulambdapar$-dissipative with every element of the minimal selection of $-\partial \psi$ (cf. Theorem \ref{thm:maximal-dissipativity}). By \eqref{eq:143} we thus get that $\mmo \subset - \partial \psi$ and thus, since they are both maximal $\piulambdapar$-dissipative, they coincide.
\end{proof}
\begin{remark}
[Comparison with similar notions of subdifferentiability]
    Part of Proposition \ref{prop:JtauvsJKObis} can be compared with
    the deep results obtained by \cite{gangbotudo}
    for the 
    Fr\'echet subdifferential of general 
    (not necessarily $\lambda$-convex) real-valued functionals when $\X$ has finite dimension.
    Using our notation, 
    \cite{gangbotudo} restricts the analysis to 
    elements of 
    the Wasserstein-Fr\'echet subdifferential 
    $\bm\partial\phi$ of $\phi$ which can be expressed by maps; 
    it is proven in \cite[Theorem 3.21, Corollary 3.22]{gangbotudo} that such a subset of 
    $\bm\partial\phi(\mu)$ 
    is nonempty if and only if the Fr\'echet subdifferential of $\psi$ at $X$ with $\mu=\iota_X$ is nonemtpty. Moreover in \cite[Theorem 3.14]{gangbotudo}
    it is proven that, given $\mu \in \dom(\phi)$, all the maps 
    $\ff$ belonging to $\Tan_\mu\prob_2(\X)$ for which $(\ii_\X, \ff)_\sharp \mu$ belongs to 
    $\bm\partial\phi(\mu)$
    correspond to elements 
    $\ff\circ X$ in $\partial\psi(X)$;
    in particular 
    \cite[Corollary 3.22]{gangbotudo} shows that the element of minimal norm of the Fr\'echet subdifferential of $\psi$ at $X$ can be written as $\ff^\circ \circ X$, where $\ff^\circ$ is the element of minimal norm of the Fr\'echet subdifferential of $\phi$ at $\iota_X$ (compare in particular with items (2),(3) in Proposition \ref{prop:JtauvsJKObis}).
    On the other hand, working with general \MPVF{s} and 
    elements in $\partial\psi(X)$
    which not necessarily have the form $\ff\circ X$
    allows to prove 
    the law invariance of $\partial\psi$ and to work with functions $\phi$ whose proper domain
    $\dom(\phi)$ is strictly contained in $\prob_2(\X)$.

We also mention that the lifting technique we are using here is of fundamental relevance for the concept of L-derivative considered in \cite[Definition 5.22]{CD18}, \cite[Definition 6.1]{carda}, and inspired by 
\cite{Lions}.
Using our notation, in \cite{CD18,carda} a function $\phi:\prob_2(\X)\to \R$ is said to be L-differentiable at $\mu=\iota_X \in \prob_2(\X)$, for $X \in \cH$, if the lifted function $\psi: \cH \to \R$ is Fr\'echet differentible at $X$. The notion of L-differentiability can also be used to define a notion of convexity (called L-convexity) for functionals $\phi: \prob_2(\X)\to \R$ which are continuously differentiable: 
we refer the interested reader to \cite[Section 5.5.1, Definition 5.70]{CD18} and we only mention that 
for such a class of regular functionals this definition is equivalent to total convexity.
\end{remark}

For clarity of explanation, we anticipate here a result linking geodesic convexity to total convexity whose proof, in a more general setting, is deferred to Section \ref{sec:jko} (see in particular Theorem \ref{prop:ab}).

\begin{theorem} Assume that $\dim(\X) \ge 2$. Let $\mathsf{U} \subset \X$ be open, convex, non-empty and let $\phi: \prob_2(\X) \to (-\infty, +\infty]$ be a proper, lower semicontinuous and geodesically $(\menolambda)$-convex functional whose domain satisfies $\prob_f(\mathsf U) \subset \dom(\phi)$ and such that $\prob_f(\mathsf U)$ is \emph{dense in energy,} meaning that for every $\mu \in \dom(\phi)$ there exists $(\mu_n)_{n\in\N} \subset \prob_f(\mathsf U)$ such that
    \[ \mu_n \to \mu \quad \text{ and } \quad \phi(\mu_n) \to \phi(\mu).\]
Then $\phi$ is totally $(-\lambda)$-convex. In particular, every continuous and geodesically $(\menolambda)$-convex functional $\phi: \prob_2(\X) \to \R$ is totally $(\menolambda)$-convex.
\end{theorem}
 \begin{theorem}[Gradient flows of totally convex functionals]\label{prop:GFiotabis} Let $\phi:\prob_2(\X) \to
   (-\infty, + \infty]$ be a proper, lower semicontinuous and totally
   $\menolambdapar$-convex functional and let $\psi$ be as in
   \eqref{eq:psi}. For every $\mu_0\in\overline{\dom(\phi)}$, let us
   denote by $(S_t)_{t\ge0}$ the family of semigroups in $\prob_2(\X)$ induced by the Lagrangian flow associated to
   the maximal total $\lambda$-dissipative \MPVF\
   $-\btpartial\phi=\iotaT(-\partial \psi)$ (cf.~Definition
   \ref{def:semig}).
   Then the locally Lipschitz curve $\mu:[0,+\infty)\to\mathcal
   P_2(\X)$, $\mu_t:=S_t(\mu_0)$, is the unique \emph{gradient flow
     for $\phi$} starting from $\mu_0$, in the sense that
   \begin{equation*}
(\ii_\X, \vv_t)_\sharp \mu_t =-\bm{\partial}^\circ \phi
[\mu_t]=-\bm{\partial}_{\mathrm t}^\circ \phi
[\mu_t]
\quad \text{ for a.e. } t>0,
\end{equation*}
where $\vv$ is the Wasserstein velocity field of $\mu$ coming from
Theorem \ref{thm:tangentv} which therefore satisfies all the properties of
\cite[Thm.~11.2.1]{ags}.\\ 
Moreover, $t\mapsto S_t(\mu_0)$ is also the unique  $\menolambdapar$-\EVI solution for the \MPVF $-\bm{\partial} \phi$ starting from $\mu_0\in\overline{\dom(\phi)}$ and $S_t$ is a semigroup of $e^{\lambda t}$-Lipschitz transformations satisfying
\begin{equation*}%
W_2(S_t(\mu_0),S_t(\mu_1))\le e^{\piulambda t} W_2(\mu_0,\mu_1)\quad\text{for any }\mu_0,\,\mu_1\in\overline{\dom(\phi)}.
\end{equation*}
\end{theorem}
\begin{proof}
  Since
  $\phi$ is lower semicontinuous and
  $\menolambdapar$-convex along generalized geodesics,
  in particular it is coercive
  thanks to \cite[Theorem 4.3]{NaldiSavare}:
  we can apply \cite[Theorem 11.2.1]{ags}
   to get that
  there exists a unique gradient flow
  $\mu:[0,+\infty)\to\prob_2(\X)$ for $\phi$ starting from
  $\mu_0$. By \cite[Theorem 5.22(e)]{CSS} this also shows that
  $\mu$ is the unique $(-\lambda)$-\EVI solution for
  $-\bm{\partial}\phi$ starting from $\mu_0$.

  Since $\bpartial^\circ\phi=\bm{\partial}^\circ_{\mathrm t}\phi$
  by Proposition \ref{prop:JtauvsJKObis}, we
  can apply Theorem \ref{thm:existence-Lagrangian} and Theorem \ref{thm:flow-generation}
  to show that $\mu$ coincides with $S_t(\mu_0)$, first for every
  $\mu_0\in \dom(\bpartial\phi)$ and then also in its closure, thanks to
  the regularization effect.
\end{proof}

We conclude the section with a pivotal example of a functional $\phi$ to which the results of this section can be applied.
\begin{example}\label{example:PWbis}
Let $P,W:\X \to (-\infty, + \infty]$ be proper, lower semicontinuous and $\menolambdapar$-convex functions, with $W$ even. We define the functional $\phi: \prob_2(\X)\to (-\infty, + \infty]$ as
\[ \phi(\mu) := \int_\X P \de \mu + \frac{1}{2} \int_{\X \times \X} W(x-y) \de (\mu \otimes\mu) (x,y), \quad \mu \in \prob_2(\X).\]
Notice that $W(0)$ is finite so that, if $x_0 \in \dom(P)$, then
$\phi(\delta_{x_0})= P(x_0) + \frac{1}{2}W(0) < + \infty$, so that
$\phi$ is proper. Moreover, by \cite[Propositions 9.3.2 and
9.3.5]{ags}, we have that $\phi$ is lower semicontinuous and totally
$(\menolambda \land 0)$-convex.
\end{example}
\newpage

\part{Metric dissipativity, total dissipativity and maximal extensions}\label{partII}

\section{Local optimality and injectivity
  of couplings}\label{sec:coupl}

In this section, we study the local optimality and injectivity of several classes of couplings.  These properties will be relevant for the analysis of the relation between metric dissipativity and total dissipativity in Section \ref{sec:strong-dissipative}. First, the fact that any coupling between discrete measures is piecewise optimal, as established in Theorem \ref{thm:easy-but-not-obvious}, implies that a metrically dissipative \MPVF $\frF$ is piecewise dissipative along such discrete couplings. To combine these piecewise dissipativity conditions and deduce that $\frF$ is dissipative along the full coupling, a key tool is the injectivity of the interpolation map $\sfx^t$, which allows one to trivialize the duality pairings as in Theorem \ref{thm:all}(4). This injectivity can either be assumed as a hypothesis (see Lemma \ref{le:injective-strong}) or derived via the perturbation argument presented in Proposition \ref{prop:perturbation} (see Theorems \ref{le:crucial} and \ref{le:crucial2}).

We first start with arbitrary couplings between discrete measures.

\subsection{Local optimality of couplings between discrete measures}
\label{sec:optcoupdiscrete}

We want to show that the linear interpolations induced by
arbitrary couplings between discrete measures
can be decomposed in a finite union of geodesics.

The main quantitative information
is contained in the following lemma.
\begin{lemma}
  \label{le:quantitative}
  Let $\mu_0,\mu_1\in \prob_2(\X)$, $\ggamma\in \Gamma(\mu_0,\mu_1)$.
  If $\mu_0$ has finite support
  $S=\{\bar x_1,\cdots,\bar x_M\}$ with $\delta:=\min\big\{|\bar
  x_i-\bar x_j|:i,j\in
  \{1,\cdots, M\},\ i\neq j\big\}>0$
  and
  \begin{equation*}
    \sup\Big\{|y-x|:(x,y)\in \supp\ggamma\Big\}\le \delta/2
  \end{equation*}
  then $\ggamma\in \Gamma_o(\mu_0,\mu_1)$ and $W_2^2(\mu_0,\mu_1)=\int|y-x|^2\,\d\ggamma$.
\end{lemma}
\begin{proof}
  It is sufficient to prove that the support of $\gamma$ satisfies the
  cyclical monotonicity condition \eqref{eq:3}.

  If $\{(x_n,y_n)\}_{n=1}^N$ are points in $\supp \ggamma$ with
  $x_0:=x_N$ and $x_n\neq x_{n-1}$ then
  \begin{align*}
    \la y_n,x_n-x_{n-1}\ra
    &=
      \la y_n-x_n,x_n-x_{n-1}\ra+
      \la x_n,x_n-x_{n-1}\ra
    \\&\ge -\frac \delta2|x_n-x_{n-1}|+\frac12 |x_n-x_{n-1}|^2+
    \frac12|x_n|^2-\frac 12|x_{n-1}|^2
    \\&\ge \frac 12|x_n|^2-\frac 12|x_{n-1}|^2
  \end{align*}
  since $|y_n-x_n|\le \delta/2$ and $|x_n-x_{n-1}|\ge \delta$.
  If $x_n=x_{n-1}$ we trivially have $\la y_n,x_n-x_{n-1}\ra=\frac
  12|x_n|^2-\frac 12|x_{n-1}|^2$, so that
  \begin{displaymath}
    \sum_{n=1}^N \la y_n,x_n-x_{n-1}\ra\ge
    \sum_{n=1}^N \frac
  12|x_n|^2-\frac 12|x_{n-1}|^2=\frac
  12|x_N|^2-\frac 12|x_0|^2=0.
  \end{displaymath}
\end{proof}

As a consequence we obtain the following result.
\begin{theorem}[Local optimality of discrete interpolations]
  \label{thm:easy-but-not-obvious}
  Let $\mu_0,\mu_1\in \prob_2(\X)$ be two measures with finite support,
  $\ggamma\in \Gamma(\mu_0,\mu_1)$
  and
  $\mu_t:=(\sfx^t)_\sharp\ggamma$, $t\in[0,1]$. Then the following properties hold.
  \begin{enumerate}
  \item For every $s\in [0,1]$ there exists $\delta>0$ such that
    for every $t\in [0,1]$ with $|t-s|\le \delta$
    $\ggamma_{s,t}:=(\sfx^s,\sfx^t)_\sharp\ggamma$
    is an optimal plan between $\mu_s$ and $\mu_t$, so that 
    \begin{equation*}
      W_2^2(\mu_s,\mu_t)=
      \int_{\X^2} |y-x|^2\,\d\ggamma_{s,t}=|t-s|^2\int_{\X^2} |y-x|^2\,\d\ggamma(x,y).
    \end{equation*}
  \item There exist a finite number of points
    $t_0=0<t_1<t_2<\cdots<t_K=1$ such that
    for every $k=1,\cdots,K$, $\mu|_{[t_{k-1},t_k]}$ is
      a minimal constant speed geodesic and
      \begin{equation*}
        W_2^2(\mu_{t'},\mu_{t''})=|t''-t'|^2\int_{\X^2}
        |y-x|^2\,\d\ggamma(x,y)\quad
        \text{for every }t',t''\in [t_{k-1},t_k].
      \end{equation*}
    \item The length of the curve $t\mapsto\mu_t$ coincides
      with
      $\Big(\int_{\X^2} |y-x|^2\,\d\ggamma\Big)^{1/2}$.
  \end{enumerate}
\end{theorem}
\begin{proof}
  The first statement follows by Lemma \ref{le:quantitative}, since
  every measure $\mu_s$ has finite support and for every $t\in [0,1]$
  \begin{align*}
    \sup\big\{|y-x|:(x,y)\in \supp\ggamma_{s,t}\big\}
    &=|t-s|
      \sup\big\{|y-x|:(x,y)\in \supp\ggamma\big\}
      \\&\le |t-s|\max\{|y-x|:
    x\in \supp\mu_0,\ y\in \supp\mu_1\big\}.
  \end{align*}

  In order to prove the second item, we define an increasing sequence
  $(t_n)_{n\in\N}\subset [0,1]$ by induction as follows:
  \begin{itemize}
  \item $t_0:=0$;
  \item if $t_n<1$ then $t_{n+1}:=\sup\Big\{t\in
    (t_n,1]:
    W_2^2(\mu_{t_n},\mu_t)=|t-t_n|^2\int_{\X^2} |y-x|^2\,\d\ggamma\Big\}$;
    \item if $t_n=1$ then $t_{n+1}=1$.
  \end{itemize}
  The sequence is well defined thanks to item (1). It is easy to see that there exists $K\in \N$ such that
  $t_K=1$. If not, $t_n$ would be strictly increasing with limit
  $t_\infty\le 1$ as $n\to+\infty$. By item (1), there would exist
  $r>0$ such that the restriction of $\mu$ to $[t_\infty-r,t_\infty]$
  is a minimal geodesic, so that whenever $t_n\ge t_\infty-r$ we should
  get $t_{n+1}=t_\infty$, a contradiction.

Item (3) follows immediately by item (2).
\end{proof}

\subsection{Injectivity of interpolation maps} 

Given two pairs of points $(a',b')$ and $(a'',b'')$ in $\X^2$
it is easy to check that
\begin{equation}
  \label{eq:94}
  (1-t)a'+tb'\neq (1-t)a''+t b''\quad\text{for every }t\in (0,1)
  \quad\Leftrightarrow\quad
  b''-b'\not\in\Big\{ -s(a''-a'):s>0\Big\}.
\end{equation}
In particular, 
given a set $A\subset \X$ we consider the
set of directions
\begin{equation}
  \label{eq:95}
  \dir A:=\Big\{s(a'-a''):s\in \R,\ a',a''\in A
  \Big\}=
  \bigcup_{s\in \R}s\big(A-A\big).
\end{equation}

\begin{definition}
Given $A,B\subset \X$ we say that \emph{the chords of $B$ are not aligned with the directions of $A$} if
\begin{equation}
  \label{eq:96}
  (B-B)\cap \dir A=\{0\}.
\end{equation}
In this case, for every $t\in (0,1)$ the map
$\sfx^t:\X^2\to\X$ is injective on $A\times B$.
\end{definition}

When $\X$ has at least dimension $2$,
it is remarkable that in the discrete setting,
it is always possible to perturb 
the elements of a finite set $B$ in order to satisfy condition \eqref{eq:96} with respect to a fixed finite set $A$.
In particular, we can always find a suitable small perturbation
of the points in $B$, so that the chords of the perturbed set are not aligned with the directions of the fixed set
$A$.
\begin{proposition}[Injectivity by small perturbations]
  \label{prop:perturbation}
  Assume that $\dim \X\ge 2$ and $A\subset \X$ be a finite
  set.
  For every finite set of distinct points
  $B=\{b_n\}_{n=1}^N\subset \X$
  there exists a finite set  $B':=\{b_n'\}_{n=1}^N$ of distinct points
  with $|b_n'-b_n|<1$
  such that,
  setting
  \begin{equation}
    \label{eq:75}
    b_n(s):=(1-s)b_n+sb_n',\quad
    B(s):=\{b_n(s)\}_{n=1}^N,
  \end{equation}
we have that $\# B(s)=N$ for all $s\in[0,1]$ and
  \begin{equation}
    \label{eq:98}
    (B(s)-B(s))\cap \dir A=\{0\}\quad\text{for every }s\in (0,1].
  \end{equation}
  In particular, for every $t\in (0,1)$ the
  restriction of the map $\sfx^t$ to
  $A\times B(s)$ is injective for every $s\in(0,1]$.
\end{proposition}
\begin{proof}
    We split the proof of the proposition in two steps.

    \noindent\textbf{Claim 1.}
\emph{there exists a finite set of distinct points $B'':=\{b_n''\}_{n=1}^N$
      with $|b_n''-b_n|<1$
      satisfying}
  \begin{equation}
    \label{eq:99}
    (B''-B'')\cap \dir A=\{0\}.
  \end{equation}
    We can argue by induction with respect to the cardinality $N$ of
  the set $B$.
  The statement is obvious in case $N=1$ (it is sufficient to
  choose $b''_1:=b_1$).

  Let us assume that the property holds for all the sets of
  cardinality $N-1\ge 1$. We can thus find a finite set of distinct points
 $B_{N-1}''=\{b_n''\}_{n=1}^{N-1}$ 
  satisfying $(B''_{N-1}-B''_{N-1})\cap\dir A=\{0\}$.
  We look for a point $b''_N\in U\setminus B_{N-1}''$, where $U:=\{x\in \X:|x-b_N|<1\}$,
  such that $ B''_N:=B''_{N-1}\cup \{b''_N\}$
  satisfies \eqref{eq:99}. The point $b''_N$ should therefore satisfy
  \begin{displaymath}
    b''_N\in U,\quad
    b''_N-b''_n\not\in \dir A\quad\text{for every }n\in\{1,\cdots,N-1\}.
  \end{displaymath}
  Such a point surely exists, since $\dir A$ is a closed set with
  empty interior (here we use the fact that the dimension of $\X$ is
  at least $2$)
  and the union $\bigcup_{n=1}^{N-1}
  \big(b''_n+\dir A\big)$ has empty interior as well, so that it
  cannot contain the open set $U$. 
  \smallskip

  \noindent\textbf{Claim 2.}
\emph{If $B''$ satisfies the properties of the
    previous claim, then there exists $\delta\in (0,1]$ such that setting}
  \begin{equation}
    \label{eq:78}
    b'_n:=(1-\delta)b_n+\delta b_n'',
  \end{equation}
  \emph{the set $B'=\{b_n'\}_{n=1}^N$ satisfies the thesis.}
  \smallskip
  
  We denote by $\mathsf a$ the cardinality $\# A$ of $A$ and   
  we first make a simple remark:
  for every $z,z''\in \X$
  \begin{equation}
    \label{eq:76}
    \# \{s\in [0,1]: z(s):=(1-s)z+sz'' \in \dir A\}>\mathsf a^2\quad
    \Rightarrow\quad
    z,z''\in \dir A.
  \end{equation}
  Indeed, the set $A-A$ contains at most
  $\mathsf a^2$ distinct
  elements,
  so that if the left hand side of \eqref{eq:76} is true, then
  there are at least two distinct values $s_1,s_2\in [0,1]$,
  $r_1,r_2\in \R$ and a
  vector $w\in A-A$ such that
  $(1-s_1)z+s_1z''=r_1w,\ (1-s_2)z+s_2z''=r_2w$.
  We then get
  \begin{displaymath}
    z(s)=
    z(s_1)+\frac {s-s_1}{s_2-s_1}(z(s_2)-z(s_1))=
    r_1w+\frac {(s-s_1)(r_2-r_1)}{s_2-s_1}w\in \dir A\quad\text{for
      every }s\in [0,1],
  \end{displaymath}
hence \eqref{eq:76}.
  As a particular consequence of \eqref{eq:76} we get that
  if $z''$ does not belong to $\dir A$,
  then the set $\{s\in (0,1]: z(s):=(1-s)z+sz'' \in \dir A\}$ is
  finite, so that
  \begin{equation}
    \label{eq:80}
    \forall\, z,z''\in \X:
    \
    z''\not\in \dir A\quad\Rightarrow\quad
    \exists\,\delta>0:\  (1-s)z+sz'' \not\in \dir A\quad\text{for
      every }s\in (0,\delta].
  \end{equation}
  Let us now apply  property \eqref{eq:80} to all the pairs $(z,z'')$ of the
  form
  $z=b_n-b_m,\ z''=b_n''-b_m''$, $n,m\in \{1,\cdots, N\}$, with
  $n\neq m$. Since $b_n''-b_m''\not\in \dir A$ we
  deduce that
  there exists $\delta_{n,m}>0$ such that
  \begin{equation}
    \label{eq:81}
    (1-s) (b_n-b_m)+s(b_n''-b_m'')\not\in \dir A\quad\text{for every
    }s\in (0,\delta_{n,m}].
  \end{equation}Setting
  \[\tilde\delta:=\min\{|b_n-b_m|\,:\,n,m\in\{1,\dots,N\},\,n\neq m\}>0\]
   and choosing $\delta:=\min_{n,m} \{\delta_{n,m},\,\tilde\delta/3\}>0$, then it is not difficult to check that $B'$ satisfies
  the thesis, with $b_n'$ as in
  \eqref{eq:78}.
  Indeed, $|b_n-b_n'|=\delta |b_n-b_n''|<1$,
  and for every $s\in [0,1]$ and $n$ we get
  \begin{displaymath}
    b_n(s):=(1-s)b_n+sb_n'=
    (1-s)b_n+s(1-\delta) b_n+s\delta b_n''=
    (1-\delta s)b_n+\delta s b_n''
  \end{displaymath}
  so that
  \begin{align*}
    b_n(s)-b_m(s)=
    (1-\delta s)(b_n-b_m)+\delta s (b_n''-b_m'')\not\in \dir A
  \end{align*}
  thanks to \eqref{eq:81} and the fact that $s\delta\le \delta_{n,m}$. 
\end{proof}

\section{Total dissipativity of {\MPVF}s along discrete measures}
\label{sec:strong-dissipative}
 In this section, we begin our analysis of the relationship between metric and total dissipativity, defined respectively in Definitions \ref{def:dissipative} and \ref{def:total-dissipativity}. Leveraging the piecewise optimality of discrete couplings established in Theorem \ref{thm:easy-but-not-obvious}, we deduce that metrically dissipative \MPVF{s} are piecewise dissipative along such couplings. To combine these piecewise dissipativity conditions, we need to trivialize duality pairings as in Theorem \ref{thm:all}(4). This is achieved either by assuming that the map $\sfx^t$ is essentially injective along the discrete coupling, or by assuming that the \MPVF is concentrated on a map along the discrete coupling. This is the content of Lemma \ref{le:injective-strong}. By an approximation procedure, we show in Theorem \ref{thm:demi-total} that suitably continuous dissipative \MPVF{s} concentrated on maps are totally dissipative. Finally, under suitable hypotheses on the geometry of the domain of the \MPVF and using the perturbation argument of Proposition \ref{prop:perturbation}, we can recover the injectivity of the map $\sfx^t$. This is the content of Theorems \ref{le:crucial} and \ref{le:crucial2}. 

We will consider the following subsets of the space $\prob_f(\mathscr{X})$ of probability measures with finite support 
in a general Polish space $\mathscr{X}$: for every $N\in \N$
  \begin{equation}
      \label{eq:16}
  \begin{aligned}
    \prob_{f,N}(\mathscr{X}):={}&\Big\{\mu\in \prob_f(\mathscr{X}): N\mu(A)\in \N\
    \forall\, A\subset \mathscr{X}\Big\},\\
    \prob_{\# N}(\mathscr{X}):={}&\Big\{\mu\in \prob_f(\mathscr{X}): N\mu(A)\in \{0,1\}\
    \forall\, A\subset \mathscr{X}\Big\}\ \ \\={}& \Big\{\mu\in \prob_{f,N}(\mathscr{X}):
    \#\supp(\mu)=N\Big\}.
  \end{aligned}
\end{equation}
Notice that every measure $\mu\in \prob_{f,N}(\mathscr{X})$ can be expressed in
the form
\begin{equation*}
\mu=\frac1N\sum_{n=1}^N\delta_{x_n}\quad\text{for some points
  }x_1,\cdots,x_N\in \mathscr{X}.
\end{equation*}
The measure $\mu$ belongs to $\prob_{\#N}(\mathscr{X})$ if the points $x_1,\cdots,x_n$ are
distinct.

If $\frF$ is a \MPVF, $\mu_0,\mu_1\in \prob(\X)$, we correspondingly set
\begin{equation}\label{eq:DfF}
  \pdom{\star} \frF:=\dom(\frF)\cap \prob_{\star}(\X),\quad
  \Gamma_{\star}(\mu_0,\mu_1):= \Gamma(\mu_0,\mu_1)\cap
  \prob_{\star}(\X\times \X),
\end{equation}
where $\star$ is replaced by one of the symbols $f,(f,N)$, $\#N$ above.

For every $\mu_0,\mu_1\in\prob_f(\X)$
we introduce the
$L^\infty$-Wasserstein distance by
\begin{equation}
  \label{eq:72}
  W_\infty(\mu_0,\mu_1):=\min\Big\{\big|
  \sfx^0-\sfx^1|_{L^\infty(\X\times\X,\mmu;\X)}:
  \mmu\in \Gamma(\mu_0,\mu_1)\Big\}.
\end{equation}

Before proceeding, we recall the main objects introduced in Section \ref{sec:prelimCSS}, which will play a central role in what follows. We refer to Section \ref{sec:prelimCSS} for their main properties. For every $\ttheta\in \prob_2(\X \times \X)$, $t \in [0,1]$ and $\Phi \in \relcP2{\sfx^t_\sharp \ttheta}{\TX}$, we set
\[    \Gamma_t(\Phi,\ttheta):=
    \left \{ \ssigma \in \prob_2(\TX\times \X)
    \mid
    (\sfx^0,\sfx^1)_{\sharp}\ssigma =\ttheta,\quad
    (\sfx^t \circ(\sfx^0,\sfx^1),\sfv^0)_\sharp \ssigma=\Psi\right \}.
\]
and
  \begin{align*}
    \directionalm{\Phi}{\ttheta}t
    & := \min \left \{ \int_{\TX\times\X}
      \scalprod{x_0-x_1}{v_0} \de \ssigma(x_0,v_0,x_1)
      \mid \ssigma \in \Gamma_t(\Psi,\ttheta)
      \right \}, \\
    \directionalp{\Phi}{\ttheta}t
    & := \max \left \{ \int_{\TX\times\X}
      \scalprod{x_0-x_1}{v_0} \de \ssigma(x_0,v_0,x_1)
      \mid \ssigma \in \Gamma_t(\Psi,\ttheta)
      \right \}.
  \end{align*}
  If $\frF\subset\prob_2(\TX)$, $\mu_0,\mu_1\in\dom(\frF)$, recall that the set $\CondGamma\frF{\mu_0}{\mu_1}$, introduced in Definition \ref{def:plangeodomV}, is defined as
\[ \CondGamma\frF{\mu_0}{\mu_1} := \left \{ \mmu \in \Gamma(\mu_0, \mu_1) \mid \sfx^t_\sharp \mmu \in \dom(\frF) \text{ for every }t \in [0,1] \right \}.\]
Given $\mmu\in\CondGamma\frF{\mu_0}{\mu_1}$, we recall the following definitions
  \begin{align*}
  \directionalm \frF\mmu t := \sup \left \{
                          \directionalm{\Phi}{\mmu}t \mid \Phi \in \frF[\mu_t]
                          \right \},
                                    \qquad
  \directionalp \frF\mmu t := \inf \left \{
                          \directionalp{\Phi}{\mmu}t 
                          \mid \Phi \in \frF[\mu_t] \right \}.
\end{align*}
In the following, we investigate the results of Theorem \ref{theo:propflfr} in the case of marginals $\mu_0,\mu_1$ with finite support, but removing the optimality requirement over the coupling $\mmu$.

\begin{lemma}
  \label{le:injective-strong}
  Let $\frF$ be a \MPVF satisfying \eqref{Hdiss} and let 
  $\mu_0,\mu_1\in\domf \frF$ with $\mmu\in
  \CondGamma\frF{\mu_0}{\mu_1}$
  satisfy at least one of the following conditions:
  \begin{enumerate}
  \item
    for every $t\in (0,1)$, $\sfx^t$ is $\mmu$-essentially injective;
  \item
    for every $t\in (0,1)$,
    there exists an element $\Phi_t\in \frF[\sfx^t_\sharp\mmu]$
    which is concentrated on a map.
  \end{enumerate}
  Then
  \begin{equation}
    \label{eq:73a}
    \directionalm \frF\mmu s-\directionalp \frF\mmu t\le \lambda (t-s)W_{\mmu}^2,\quad
    W_{\mmu}^2:=\int_{\X^2} |x_0-x_1|^2
    \,\d\mmu,
    \quad\text{for every }0\le s<t\le 1.
  \end{equation}
  In particular,
  $t\mapsto \directionalm \frF\mmu t+\lambda W_{\mmu}^2\, t$ and
  $t\mapsto \directionalp \frF\mmu t+\lambda W_{\mmu}^2\, t $ are increasing
  respectively in $[0,1)$ and in $(0,1]$,
$\directionalp \frF\mmu t = \directionalm \frF\mmu t$ at every $t \in (0,1)$ where one of them is continuous, hence they coincide outside a countable set of discontinuities.
\end{lemma}
\begin{proof}
  By Theorem \ref{theo:propflfr}, it is not restrictive to assume $\lambda=0$; moreover, thanks to \eqref{eq:restrsp}, we may also set $s=0$ and $t=1$. Indeed, if the statement of the present lemma holds for $\lambda=0$, $s=0$, and $t=1$, then for any $0 \le s < t \le 1$ we can define $\mmu^{st} := (\sfx^s, \sfx^t)_\sharp \mmu$ and observe that $\sfx^s_\sharp \mmu = \sfx^0_\sharp \mmu^{st}$ and $\sfx^t_\sharp \mmu = \sfx^1_\sharp \mmu^{st}$ belong to $\domf \frF$, with $\mmu^{st} \in \CondGamma\frF{\sfx^0_\sharp \mmu^{st}}{\sfx^1_\sharp \mmu^{st}}$. Moreover, if either condition (1) or (2) above holds for $\mmu$, the same holds for $\mmu^{st}$. Consequently we can apply \eqref{eq:73a} to the coupling $\mmu^{st}$, and have
\[
(t-s) \directionalm \frF\mmu s = \directionalm{\frF}{\mmu^{st}}{0} \le \directionalp{\frF}{\mmu^{st}}{1} = (t-s) \directionalp \frF\mmu t,
\]  
where the equalities follow by \eqref{eq:restrsp} and the definitions of $\directionalm \frF\mmu s$ and $\directionalp \frF\mmu t$. Dividing both sides by $(t-s) > 0$ yields the desired inequality in \eqref{eq:73a} for the general case $0 \le s < t \le 1$ and $\lambda=0$.

\medskip

We then devote the remainder of the proof to establishing the result in the case $\lambda=0$ with $s=0$ and $t=1$.
  
  We set $\mu_t:=\sfx^t_\sharp\mmu$ and we select
  an element $\Phi_t\in \frF[\mu_t]$
 (in case (2) we can also suppose
  that $\Phi_t$ is concentrated on a map).
  
  Applying Theorem \ref{thm:easy-but-not-obvious}, we can find
  points $t_0=0<t_1<\cdots<t_K=1$ such that
  \[\mmu^k:=(\sfx^{t_{k-1}},\sfx^{t_k})_\sharp\mmu\in
  \CondGamma\frF{\mu_{t_{k-1}}}{\mu_{t_k}}\cap\Gamma_o(\mu_{t_{k-1}},\mu_{t_k})\quad\text{for every }
k=1,\cdots,K.\]
  In particular, from \eqref{eq:restrsp} and Theorem \ref{theo:propflfr}(2), we get
  \begin{equation*}
    \directionalm{\Phi_{t_{k-1}}}\mmu{t_{k-1}}=
    \frac{1}{t_k-t_{k-1}}
    \directionalm{\Phi_{t_{k-1}}}{\mmu^k}{0}\le
    \frac{1}{t_k-t_{k-1}}
    \directionalp{\Phi_{t_{k}}}{\mmu^k}{1}=
    \directionalp{\Phi_{t_{k}}}\mmu{t_{k}}.
  \end{equation*}
  Since, for $1\le k<K$, $\sfx^{t_k}$ is $\mmu$-essentially injective
  (if assumption (1) holds) or $\Phi_{t_k}$
  is concentrated on its barycenter (if assumption (2) holds),
  Theorem \ref{thm:all}(4) yields
  $\directionalp{\Phi_{t_{k}}}\mmu{t_{k}}=
  \directionalm{\Phi_{t_{k}}}\mmu{t_{k}}$ so that

\[
\directionalm{\Phi_0}\mmu{0} \le \directionalp{\Phi_{t_{1}}}\mmu{t_{1}} = \directionalm{\Phi_{t_{1}}}\mmu{t_{1}} \le \dots \le \directionalp{\Phi_{t_{K-1}}}\mmu{t_{K-1}} = \directionalm{\Phi_{t_{K-1}}}\mmu{t_{K-1}} \le \directionalp{\Phi_{1}}\mmu{1}.
\]

  Taking the supremum w.r.t.~$\Phi_0\in \frF[\mu_0]$ and
  the infimum w.r.t.~$\Phi_1\in \frF[\mu_1]$ we obtain \eqref{eq:73a}.
  The last part of the statement follows as in the proof of Theorem \ref{theo:propflfr}.
\end{proof}

The following result shows that in case of a deterministic
demicontinuous \PVF (recall Definition \ref{def:demicontinuous})
$\lambda$-dissipativity yields total $\lambda$-dissipativity.
Similarly, we can lift the Lipschitz continuity along optimal couplings to arbitrary couplings.
\begin{theorem}[Deterministic demicontinuous dissipative \PVF{s} are totally dissipative]
  \label{thm:demi-total}
  Let $\frF\subset\prob_2(\TX)$ be a deterministic demicontinuous $\lambda$-dissipative \PVF with $\dom(\frF)=\prob_2(\X)$, of the form 
  \begin{equation}\label{eq:f-induces-F}
  \frF[\mu]:= (\ii_\X,\ff(\cdot, \mu))_\sharp \mu, \quad \mu \in \prob_2(\X),
  \end{equation}
  for a map
  $\ff: \Sp \X \to \X$, where $\Sp \X$ is as in \eqref{eq:2}.
  Then $\frF$ is maximal totally
  $\lambda$-dissipative.

  If moreover
  there exists $L>0$ for which the following condition holds: for every $\mu_0, \mu_1 \in \prob_2(\X)$ there exists $\mmu \in \Gamma_o(\mu_0, \mu_1)$ satisfying
\begin{equation}\label{eq:ass2}
\int_{\X \times \X} \left | \ff(x_1, \mu_1)-\ff(x_0, \mu_0) \right |^2 \de \mmu(x_0, x_1) \le L^2 \int_{\X \times \X} |x_1-x_0|^2 \de \mmu(x_0,x_1),
\end{equation}
then \eqref{eq:ass2} holds for every $\mmu\in \Gamma(\mu_0,\mu_1)$.
\end{theorem}
\begin{proof}
By Lemma \ref{le:injective-strong}(2) and the fact that $\frF$ is single-valued and
  concentrated on a map $\ff: \Sp \X\to \X$, recalling Theorem \ref{thm:all}(4) we know that $\frF$ is totally dissipative on finitely supported measures, i.e. it satisfies
  \eqref{eq:total-dissipativity} (or, equivalently,
  \eqref{eq:strongdiss-intro}) for every $\mu_0,\mu_1\in \prob_f(\X).$
  We use
  an approximation procedure to get the general formulation
  for every $\mu_0, \mu_1 \in \prob_2(\X)$ and every $\mmu \in
  \Gamma(\mu_0, \mu_1)$: we take sequences $(\mu_0^n)_{n\in\N}, (\mu_1^n)_{n\in\N}
  \subset \prob_f(\X)$ such that $W_2(\mu_0^n, \mu_0) \to 0$ and
  $W_2(\mu_1^n, \mu_1)\to 0$ and optimal plans $\ggamma_0^n \in
  \Gamma_o(\mu_0^n, \mu_0)$ and $\ggamma_1^n \in \Gamma_o(\mu_1,
  \mu_1^n)$. Let $\ssigma_n \in \prob(\X^4)$ be such that
  $\pi^{1,2}_\sharp \ssigma_n = \ggamma_0^n$, $\pi^{2,3}_\sharp
  \ssigma_n = \mmu$ and $\pi^{3,4}_\sharp \ssigma_n = \ggamma_1^n$.
  Notice that we also have that $\mmu_n:=\pi^{1,4}_\sharp \ssigma_n $
  belongs to $\Gamma(\mu_0^n,\mu_1^n)$ and converges to $\mmu$ in
  $\prob_2(\X^2)$
  as $n\to+\infty$.
 Thanks to the demicontinuity of $\frF$ and the fact that $\frF$ is
  concentrated on $\ff$,
  we obtain that
  $\ttheta_n:=(\ii_{\X\times\X},\ff(x_0,\mu_0^n)\times\ff(x_1,\mu_1^n))_\sharp \mmu_n$ 
  converges to $\ttheta:=(\ii_{\X\times\X},\ff(x_0,\mu_0)\times\ff(x_1,\mu_1))_\sharp \mmu$ in
  $\prob_2^{sw}(\X^2\times \X^2)$.
  We can then pass to the limit in the inequality
  \begin{displaymath}
    \int_{\X^2}\langle
    \ff(x_1,\mu_1^n)-\ff(x_0,\mu_0^n),x_1-x_0\rangle\,\d\mmu_n(x_0,x_1)=
    \int_{\X^2\times\X^2}\langle v_1-v_0,x_1-x_0\rangle\,\d\ttheta_n(x_0,x_1,v_0,v_1)\le 0
  \end{displaymath}
  obtaining
  \begin{displaymath}
    \int_{\X^2}\langle
    \ff(x_1,\mu_1)-\ff(x_0,\mu_0),x_1-x_0\rangle\,\d\mmu(x_0,x_1)=
    \int_{\X^2\times\X^2}\langle v_1-v_0,x_1-x_0\rangle\,\d\ttheta(x_0,x_1,v_0,v_1)\le 0.
  \end{displaymath}
  We can eventually apply Theorem \ref{thm:trivial-but-useful-to-fix}
  to get the maximality of $\frF$.
  
Concerning the second part of the Theorem,
  let us first show that the condition
  \eqref{eq:ass2} holds
for every $\mu_0, \mu_1 \in \prob_f(\X)$ and
  every $\mmu \in \Gamma(\mu_0,\mu_1)$:
  by Theorems \ref{thm:easy-but-not-obvious} and \ref{theo:chargeo} there exists some $K \in \N$ and points $0=t_0 <t_1 < \dots <t_{K-1} < t_K=1$ such that $(\sfx^{t_{i-1}},\sfx^{t_i})_\sharp \mmu$ is the unique element of $\Gamma_o(\sfx^{t_{i-1}}_\sharp \mmu, \sfx^{t_i}_\sharp \mmu)$ for every $i=1, \dots, K$. We thus have for every $i=1, \dots, K$ that
\[ 
\left ( \int_{\X^2} \left | \ff(\sfx^{t_i}, \sfx^{t_i}_\sharp \mmu)-\ff(\sfx^{t_{i-1}}, \sfx^{t_{i-1}}_\sharp \mmu) \right |^2 \de \mmu \right )^{1/2} \le L (t_{i}-t_{i-1}) \left ( \int_{\X^2} |x_1-x_0|^2 \de \mmu(x_0, x_1) \right )^{1/2}.
\]
Summing up these inequalities for $i=1, \dots, K$ and using the
triangular inequality in $L^2(\X \times \X, \mmu; \X)$, we get that
\eqref{eq:ass2} holds for every $\mu_0, \mu_1 \in \prob_f(\X)$ and
every $\mmu \in \Gamma(\mu_0, \mu_1)$.

By using the same approximation procedure (and the same notation)
of the first part of this
proof,
we show that \eqref{eq:ass2} holds for every $\mu_0, \mu_1 \in
\prob_2(\X)$ and every $\mmu \in \Gamma(\mu_0, \mu_1)$:
in fact we have the estimate
\begin{align*}
    &\left (\int_{\X^2} \left | \ff(x_1, \mu_1)-\ff(x_0, \mu_0) \right |^2 \de \mmu(x_0, x_1) \right )^{1/2} =  \|\ff(\pi^3,\mu_1)-\ff(\pi^2, \mu_0)\|_{L^2(\X^2, \ssigma_n; \X)} \\
    &\qquad \le \|\ff(\pi^3, \mu_1) - \ff(\pi^4, \mu_1^n) \|_{L^2(\X^2, \ssigma_n; \X)} 
    + \|\ff(\pi^4, \mu_1^n)-\ff(\pi^1, \mu_0^n) \|_{L^2(\X^2, \ssigma_n; \X)}\\&\qquad\qquad
    + \|\ff(\pi^1, \mu_0^n)- \ff(\pi^2, \mu_0) \|_{L^2(\X^2, \ssigma_n; \X)}\\
    &\qquad \le L  \Big(W_2(\mu_1^n, \mu_1) + W_2(\mu_0, \mu_0^n) \Big)
      + L\left (\int_{\X^2} |x-y|^2 \de \mmu_n(x,y) \right )^{1/2}.
\end{align*}
Passing to the limit as $n \to + \infty$, we get that \eqref{eq:ass2}
holds for every $\mu_0, \mu_1 \in \prob_2(\X)$ and every $\mmu \in
\Gamma(\mu_0, \mu_1)$.
\end{proof}

While the dissipativity obtained in Lemma \ref{le:injective-strong} is based on assumptions granting the trivialization of the duality pairings (cf. Theorem \ref{thm:all}(4)), we can improve such result using the perturbation argument developed in Proposition \ref{prop:perturbation}. This requires to assume $\dim \X\ge 2$ and to work with a finite set of $N$ distinct points.

\begin{theorem}[Self-improving dissipativity along discrete couplings]
  \label{le:crucial}
Assume that $\dim \X\ge 2$.  Let $\frF$ be a \MPVF satisfying
\eqref{Hdiss}, $N\in \N$, 
let $\mu_0,\mu_1\in\domf \frF$,
$\mmu\in
\Gamma(\mu_0,\mu_1)$ and let $\mu_t= \sfx^t_{\sharp} \mmu$, $t \in [0,1]$.
  Assume that one of the following conditions is
  satisfied:
  \begin{enumerate}
  \item
    $\mmu\in \prob_{f,N}(\X\times \X)$
    and for every $t\in (0,1)$ $\mu_t$ belongs to the 
    relative interior of
    $\dom_{f,N}(\frF)$ in $\prob_{f,N}(\X)$;
    \item
    for every $t\in (0,1)$ $\mu_t$ belongs to the interior of
    $\dom_f(\frF)$ in the metric space $(\prob_f(\X),W_\infty)$.
  \end{enumerate}
  Then
  \begin{equation}
    \label{eq:73}
    \directionalm \frF\mmu s-\directionalp \frF\mmu t\le \lambda (t-s)W_{\mmu}^2,\quad
    W_{\mmu}^2:=\int_{\X^2} |x_0-x_1|^2
    \,\d\mmu,
    \quad\text{for every }0\le s<t\le 1.
  \end{equation}
\end{theorem}
\begin{proof}
 We carry out the proof in case (1), the proof in case (2)
  is analogous.
By Theorem \ref{theo:propflfr} it is not restrictive to assume $\lambda=0$; we can also assume $s=0$ and $t=1$ thanks to \eqref{eq:restrsp} (see the beginning of the proof of Lemma \ref{le:injective-strong} for more details).  By Theorem \ref{thm:easy-but-not-obvious} we can find $0 < \delta < 1/2$ and $\tau \in (\delta, 1-\delta)$ s.t.~$\sfx^\delta, \sfx^\tau$ and $\sfx^{1-\delta}$ are $\mmu$-essentially injective and $(\sfx^0, \sfx^{\delta})_{\sharp} \mmu$, ($\sfx^{1-\delta}, \sfx^1)_\sharp \mmu$ are optimal: indeed by Theorem \ref{thm:easy-but-not-obvious}, we find $K \ge1$ and points $0=t_0< t_1 < t_2 < \cdots <t_K=1$ such that $(\sfx^{t_{k-1}}, \sfx^{t_k})_\sharp \mmu$ is optimal for every $k=1, \dots, K$; it is then enough to take any $\delta, \tau \in (0,1)$ such that
\[
0<\delta< t_1 \wedge (1-t_{K-1}) \wedge 1/2, \quad \tau \in (\delta, 1-\delta) \setminus \bigcup_{k=0}^K \{ t_k\}.
\]
In this way, $0<\delta<\tau <1-\delta<1$, $\delta \in (0,t_1)$, $1-\delta \in (t_{K-1}, 1)$, and $\tau \in (t_{k-1}, t_k)$ for some $k \in 1, \dots, K$. In particular, $(\sfx^0, \sfx^\delta)_\sharp \mmu$ (resp.~$(\sfx^{1-\delta}, \sfx^1)_\sharp \mmu$) is a restriction of the optimal plan $(\sfx^0, \sfx^{t_1})_\sharp \mmu$ (resp.~$(\sfx^{t_{K-1}}, \sfx^1)_\sharp \mmu$) hence optimal. Moreover, by Theorem \ref{theo:chargeo}, we see that $\sfx^{\delta/t_1}$ is $(\sfx^0, \sfx^{t_1})_\sharp \mmu$-essentially injective; since $(\sfx^0, \sfx^{t_1})(\supp(\mmu)) \subset \supp((\sfx^0, \sfx^{t_1})_\sharp \mmu)$ (cf. \cite[formula (5.2.6)]{ags}) and $\sfx^\delta = \sfx^{\delta/t_1} \circ (\sfx^0, \sfx^{t_1})$, we conclude that $\sfx^\delta$ is $\mmu$-essentially injective. An analogous argument shows that $\sfx^{1-\delta}$ and $\sfx^\tau$ are $\mmu$-essentially injective.

In this way, since by Theorem \ref{theo:propflfr} the relation \eqref{eq:73} is true both for the case $s=0, t= \delta$ and $s=1-\delta, t=1$, we only need to prove it for $s=\delta$ and $t=1-\delta$. \\
We set $A=\supp(\mu_{\delta})\cup \supp(\mu_{1-\delta})$ and $B=\supp(\mu_\tau)$. By compactness, we can find $\eps>0$ such that every measure in $\prob_{f,N}(\X)$ in the $W_2$-neighborhood of radius $\eps>0$ around $\mu_t$ is contained in $\dom(\frF)$ for every $\delta \le t \le 1-\delta$.\\
Applying Proposition \ref{prop:perturbation}
  we can find a map $\bb:B\to \X$ with values
  in the open ball of radius $\eps$ centered at $0$
  such that setting $\bb^s(x):=x+s\bb(x)$ for every $s\in [0,1]$
  and $x\in B$,
  the set $B^s:=\bb^s(B)$ satisfies
  $(B^s-B^s)\cap \dir A=\{0\}$ and $\#B^s= \# \supp(\mu_{\tau})$ for every $s \in (0,1]$.
  Considering the measures $\nu_{s,\tau}:=(\bb^s)_\sharp\mu_\tau$,
  we can pick $\Psi_{s,\tau}\in \frF[\nu_{s,\tau}]$ with barycenter $\vv_{s,\tau}:B^s\to \X$, i.e.
  \begin{equation*}
      \vv_{s,\tau}(y):=\int_{\TX} v\,\d\Psi_{s,\tau}(y,v).
  \end{equation*} 
  Now, for $a=\frac{\tau-\delta}{1-2\delta}$, we define maps $ \bb^{s,\tau}, \vv^{s,\tau} : \supp((\sfx^{\delta}, \sfx^{1-\delta})_{\sharp}\mmu) \to \X$ as
  \[\bb^{s,\tau}:=\bb^s \circ \sfx^a,\quad
    \vv^{s,\tau}:=\vv_{s,\tau} \circ \bb^{s,\tau}.\]
Notice that $\sfx^a(\supp((\sfx^{\delta}, \sfx^{1-\delta})_{\sharp}\mmu)) \subset B=\supp(\mu_\tau)$,  so the above definitions are well-posed. Let us consider $\Phi_{\delta} \in \frF[\mu_{\delta}]$, $\Phi_{1-\delta} \in \frF[\mu_{1-\delta}]$ and $\ssigma_\delta \in \prob(\TX \times \TX)$ s.t.~$(\sfx^0, \sfx^1)_{\sharp} \ssigma_\delta = (\sfx^{\delta}, \sfx^{1-\delta})_{\sharp}\mmu$, $(\sfx^0, \sfv^0)_{\sharp} \ssigma_\delta = \Phi_{\delta}$ and $(\sfx^1, \sfv^1)_{\sharp} \ssigma_\delta = \Phi_{1-\delta}$. On $ \supp(\ssigma_\delta)$, we have
  \begin{equation}\label{eq:QQQ}
      \begin{split}
    \la &\sfv^0-\sfv^1,\sfx^0-\sfx^1\ra=
      \la \sfv^0-\vv^{s,\tau},\sfx^0-\sfx^1\ra+
      \la \sfv^1-\vv^{s,\tau},\sfx^1-\sfx^0\ra
    \\&=
    \frac{1}{a}
    \la \sfv^0-\vv^{s,\tau},\sfx^0-\sfx^a\ra
    +
    \frac{1}{1-a}
    \la \sfv^1-\vv^{s,\tau},\sfx^1-\sfx^a\ra
    \\&=
    \frac{1}{a}
    \la \sfv^0-\vv^{s,\tau},\sfx^0-\bb^{s,\tau}\ra
    +
    \frac{1}{1-a}
    \la \sfv^1-\vv^{s,\tau},\sfx^1-\bb^{s,\tau}\ra
    \\&\qquad+
     \frac{1}{a}
    \la \sfv^0-\vv^{s,\tau},\bb^{s,\tau}-\sfx^a\ra+    
    \frac{1}{1-a}
    \la \sfv^1-\vv^{s,\tau},
    \bb^{s,\tau}-\sfx^a
    \ra
    \\&=
    \frac{1}{a}
    \la \sfv^0-\vv^{s,\tau},\sfx^0-\bb^{s,\tau}\ra
    +
    \frac{1}{1-a}
    \la \sfv^1-\vv^{s,\tau},\sfx^1-\bb^{s,\tau}\ra
    \\&\qquad+
    \frac{1}{a(1-a)}
    \la \vv^{1,\tau}-\vv^{s,\tau},\bb^{s,\tau}-\sfx^a\ra
    +
        \frac{1}{a(1-a)}
    \la (1-a) \sfv^0+a \sfv^1-\vv^{1,\tau},
    \bb^{s,\tau}-\sfx^a
    \ra
    \\&=
    \frac{1}{a}
    \la \sfv^0-\vv^{s,\tau},\sfx^0-\bb^{s,\tau}\ra
    +
    \frac{1}{1-a}
    \la \sfv^1-\vv^{s,\tau},\sfx^1-\bb^{s,\tau}\ra
    \\&\qquad+
    \frac{s}{(1-s)a(1-a)}
    \la \vv^{1,\tau}-\vv^{s,\tau},\bb^{1,\tau}-\bb^{s,\tau}\ra
    +
    \frac{s}{a(1-a)}
    \la (1-a) \sfv^0+a \sfv^1-\vv^{1,\tau},
    \bb^{1,\tau}-\sfx^a
    \ra.
    \end{split}
  \end{equation}
  We have that
  \begin{equation}\label{eqs123}
  \begin{split}
    \int_{\TX^2}
    \la \sfv^0-\vv^{s,\tau},\sfx^0-\bb^{s,\tau}\ra
    \,\d\ssigma_\delta &= \directionalm{\Phi_{\delta}}{\mmu^{s,\tau,\delta}}{0} - \directionalp{\Psi_{s}}{\mmu^{s,\tau,\delta}}{1}, \\
    \int_{\TX^2} \la
    \sfv^1-\vv^{s,\tau},\sfx^1-\bb^{s,\tau}\ra
    \,\d\ssigma_\delta &= \directionalm{\Phi_{1-\delta}}{\tilde{\mmu}^{s,\tau,\delta}}{0} - \directionalp{\Psi_{s}}{\tilde{\mmu}^{s,\tau,\delta}}{1}, \\
    \int_{\TX^2}\la \vv^{1,\tau}-\vv^{s,\tau},\bb^{1,\tau}
    -\bb^{s,\tau}\ra
    \,\d\ssigma_\delta &= \directionalm{\Psi_1}{\ttheta^{s,\tau,\delta}}{0} - \directionalp{\Psi_{s}}{\ttheta^{s,\tau,\delta}}{1},
    \end{split}
  \end{equation}
  where $\mmu^{s,\tau,\delta} = (\sfx^0, \bb^{s,\tau})_{\sharp} \ssigma_\delta$, $\tilde{\mmu}^{s,\tau,\delta} = (\sfx^1, \bb^{s,\tau})_{\sharp} \ssigma_\delta$, $\ttheta^{s,\tau,\delta} = (\bb^{1,\tau}, \bb^{s,\tau})_{\sharp} \ssigma_\delta$ and the equalities with the pseudo scalar products come from the fact that all those plans are concentrated on a map w.r.t.~their first marginal. Indeed, we can use Theorem \ref{thm:all}(4) thanks to the $\mmu$-essential injectivity of $\sfx^\delta, \sfx^\tau$, $\sfx^{1-\delta}$, and use the fact that the cardinality of $B^s$ is constant w.r.t.~$s$. By construction, these plans satisfy the hypotheses of Lemma \ref{le:injective-strong} so that all the expressions at the right-hand side of \eqref{eqs123} are nonpositive. Combining this fact with \eqref{eq:QQQ}, we end up with 
  \begin{align*}
    \int_{\TX^2} \la \sfv^0-\sfv^1,\sfx^0-\sfx^1\ra\,\d\ssigma_\delta
    &
      \le
      \frac{s}{a(1-a)}
      \int_{\TX^2}
    \la (1-a) \sfv^0+a \sfv^1-\vv^{1,\tau},
    \bb^{1,\tau}-x_a
    \ra\,\d\ssigma_\delta.
  \end{align*}
  Passing to the limit as $s\downarrow0$
  we obtain
  \begin{equation}\label{eq:combdiss1}
     \int_{\TX^2} \la \sfv^0-\sfv^1,\sfx^0-\sfx^1\ra\,\d\ssigma_\delta\le 0.
   \end{equation}
    Recalling \eqref{eq:switch} and using the same notation for the map $\mathsf s: \TX^2 \to \TX^2$, $\mathsf s (x_0,v_0,x_1,v_1):=(x_1,v_1,x_0,v_0)$, we can write the left-hand side as follows (cf. Definition \ref{def:pairings})
   \begin{align*}
       \int_{\TX^2} \la \sfv^0-\sfv^1,\sfx^0-\sfx^1\ra\,\d\ssigma_\delta&=\int_{\TX^2} \la \sfv^0,\sfx^0-\sfx^1\ra\,\d\ssigma_\delta+\int_{\TX^2} \la \sfv^1,\sfx^1-\sfx^0\ra\,\d\ssigma_\delta\\
       &=\int_{\TX^2} \la \sfv^0,\sfx^0-\sfx^1\ra\,\d\ssigma_\delta+\int_{\TX^2} \la \sfv^0,\sfx^0-\sfx^1\ra\,\d(\mathsf s_\sharp{\ssigma_\delta})\\
       &\ge \directionalm{\Phi_\delta}{(\sfx^{\delta}, \sfx^{1-\delta})_{\sharp}\mmu}{0}+\directionalm{\Phi_{1-\delta}}{\mathsf s_\sharp(\sfx^{\delta}, \sfx^{1-\delta})_{\sharp}\mmu}{0},
   \end{align*}
   indeed $(\sfx^0,\sfv^0,\sfx^1)_\sharp\ssigma_\delta\in\Gamma_0\left(\Phi_\delta,(\sfx^{\delta}, \sfx^{1-\delta})_{\sharp}\mmu\right)$ and
   $(\sfx^0,\sfv^0,\sfx^1)_\sharp(\mathsf s_\sharp\ssigma_\delta)\in\Gamma_0\left(\Phi_{1-\delta},\mathsf s_\sharp\left[(\sfx^{\delta}, \sfx^{1-\delta})_{\sharp}\mmu\right]\right)$.
   Thus, by \eqref{eq:combdiss1} and Theorem \ref{thm:all}(1)(3), we can write
   \begin{align*}
   0&\ge\directionalm{\Phi_\delta}{(\sfx^{\delta}, \sfx^{1-\delta})_{\sharp}\mmu}{0}+\directionalm{\Phi_{1-\delta}}{\mathsf s_\sharp(\sfx^{\delta}, \sfx^{1-\delta})_{\sharp}\mmu}{0}\\
   &=\directionalm{\Phi_\delta}{(\sfx^{\delta}, \sfx^{1-\delta})_{\sharp}\mmu}{0}-\directionalp{\Phi_{1-\delta}}{(\sfx^{\delta}, \sfx^{1-\delta})_{\sharp}\mmu}{1}\\
   &=\left(1-2\delta\right)\left(\directionalm{\Phi_\delta}{\mmu}{\delta}-\directionalp{\Phi_{1-\delta}}{\mmu}{1-\delta}\right).
   \end{align*} 
   Dividing by $1-2\delta>0$ and passing to the supremum w.r.t.~$\Phi_{\delta} \in \frF[\mu_{\delta}]$ and~$\Phi_{1-\delta} \in \frF[\mu_{1-\delta}]$, we get (cf. Definition \ref{def:plangeodomV})
   \[ \directionalm{\frF}{\mmu}{\delta}-\directionalp{\frF}{\mmu}{1-\delta} \le 0, \]
 which is \eqref{eq:73} with $s=\delta$ and $t=1-\delta$.
 \end{proof}

\begin{remark}
  If $\frF \subset \prob_2(\TX)$ is a $\lambda$-dissipative $\MPVF$ with $\dom(\frF)=\prob_2(\X)$, then $\frF$ is $\lambda$-dissipative along discrete couplings thanks to Theorem \ref{le:crucial} and Theorem \ref{thm:all}, i.e.
\[ \bram{\Phi}{\Psi} \le \lambda \int_{\X \times \X} |x-y|^2 \de \ggamma(x,y)\]
for every $\Phi, \Psi \in \frF$ and any $\ggamma \in \Gamma(\sfx_\sharp \Phi, \sfx_\sharp \Psi)$ such that $\sfx_\sharp \Phi, \sfx_\sharp \Psi$ belong to $\prob_f(\X)$.
\end{remark}
\medskip

 The same perturbation argument in the proof of Theorem \ref{le:crucial} can be applied in a similar situation, when we know that the \MPVF is dissipative along discrete couplings w.r.t.~which the map $\sfx^t$ is essentially injective. This leads us to the following definition.  

\begin{definition}[Convexity along collisionless couplings]
  \label{def:collisionless}
  Let $\mu_0,\mu_1\in \prob_{f}(\X)$.
  We say that
  $\mmu\in \Gamma(\mu_0,\mu_1)$ is \emph{collisionless} if $\sfx^t$ is $\mmu$-essentially injective for every 
  $t\in [0,1]$. \\
  We say that a set $C\subset \prob_f(\X)$ is \emph{convex along
  collisionless couplings} if
  for every collisionless $\mmu\in \prob_{f}(\X^2)$, with
$\sfx^0_\sharp\mmu,\sfx^1_\sharp\mmu\in C$, and every
    $t\in (0,1)$ we have
    $\sfx^t_\sharp \mmu\in C$.
\end{definition}
Notice that if $\mu_0,\mu_1\in 
 \prob_{\#N}(\X)$
a coupling $\mmu$ in $\Gamma(\mu_0,\mu_1)$ is collisionless
if and only if
\begin{equation}
  \label{eq:17}
  \mmu\in \Gamma_{\# N}(\X^2),\quad
  \sfx^t_\sharp \mmu\in \prob_{\#N}(\X)\quad\text{for every }t\in (0,1).
\end{equation}

\begin{theorem}[Self-improving dissipativity: the collisionless case]
  \label{le:crucial2}
  Assume that $\dim \X\ge 2$, $N\in \N$, let $\frF$ be a \MPVF satisfying
\eqref{Hdiss} and such that 
       $\dom_{\#N}(\frF)$ is convex along collisionless couplings,  let $\mu_0$
   belong to the
   interior of
   $\dom_{\#N}(\frF)$ in the metric space
   $(\prob_{\#N}(\X),W_2)$, $\mu_1 \in \dom_{f,N}(\frF)$, and $\mmu\in
   \Gamma_{\#N}(\mu_0,\mu_1)$. Assume that one of the following conditions is satisfied:
   \begin{enumerate}
       \item $\mu_1 \in \dom_{\#N}(\frF)$;
       \item for every $r \in (0,1)$ there exists $t \in (r, 1)$ such that $\sfx^t_\sharp \mmu \in \dom(\frF)$.
   \end{enumerate}
   Then 
  \begin{equation}
    \label{eq:73bis}
    \directionalm{\Phi}{\mmu}{0}-\directionalp{ \Psi }{\mmu}{1}\le \lambda W_{\mmu}^2,\quad
    W_{\mmu}^2:=\int_{\X^2} |x_0-x_1|^2
    \,\d\mmu
  \end{equation}
   for every $\Phi \in \frF[\mu_0]$, $\Psi \in \frF[\mu_1]$.
\end{theorem}
\begin{proof} We divide the proof into two claims, proving the result respectively in case (1) or (2). \\
\noindent\textbf{Claim 1.}
\emph{Case (1)}.
\smallskip

   The proof is very similar to the one of Theorem \ref{le:crucial}, we
  keep the same notation.
  
  Since $\mmu\in \prob_{\#N}(\X^2),$ $\sfx^0$, $\sfx^1$ are
  $\mmu$-essentially injective, so that we can select
  $\delta=0$.
   Since $\mu_0$ is in the interior of $\dom_{\#N}(\frF)$, we can find $\tau, \eps \in (0,1)$ small enough such that the $W_\infty$-ball of radius $\eps$ centered at $\mu_\tau:=\sfx^\tau_\sharp \mmu$ is contained in $\dom_{\#N}(\frF)$ and $\sfx^\tau$ is $\mmu$-essentially injective. We can then apply the same perturbation argument as in the proof of Theorem \ref{le:crucial}, and consider the measures $\Phi_0, \Phi_1, \ssigma_0, \nu_{s,\tau}, \Psi_{s, \tau}, \mmu^{s,\tau,0}, \tilde{\mmu}^{s,\tau, 0}, \ttheta^{s,\tau, 0}$ as defined therein. We can proceed with exactly the same computations and arrive at \eqref{eqs123}. The right hand sides of the equations in \eqref{eqs123} are again non-positive because the hypotheses of Lemma \ref{le:injective-strong}(1) are satisfied: for any $\ggamma \in \{ \mmu^{s,\tau,0}, \tilde{\mmu}^{s,\tau, 0}, \ttheta^{s,\tau, 0}\}$, its marginals belong to $\dom_{\#N}(\frF)$ by construction, $\sfx^t$ is $\ggamma$-essentially injective also by construction, and $\ggamma \in \CondGamma{\frF}{\sfx^0_\sharp \ggamma}{\sfx^1_\sharp \ggamma}$ because of the convexity of $\dom_{\#N}(\frF)$ along collisionless couplings. Then we get \eqref{eq:combdiss1} which gives immediately \eqref{eq:73bis}.
 \smallskip

\noindent\textbf{Claim 2.}
\emph{Case (2)}.
\smallskip

We can assume that $\mu_1 \notin \dom_{\#N}(\frF)$ and $\lambda=0$. Let us denote $\mu_t := \sfx^t_\sharp \mmu$, $t \in [0,1]$; we claim that there exists $\tau \in (0,1)$ such that $\mu_\tau \in \dom_{\#N}(\frF)$, $\sfx^\tau$ is $\mmu$-essentially injective, and $(\sfx^\tau, \sfx^1)_\sharp \mmu$ is optimal. Indeed, since $\mmu \in \prob_{\#N}(\X^2)$, $\sfx^t_\sharp \mmu$ is supported on less than $N$ distinct points only for a finite number of times $0<t_1<\dots<t_{K-1} <t_K= 1$; on the other hand, by Theorem \ref{thm:easy-but-not-obvious}, we can find $\bar{t} \in (0,1)$ such that $(\sfx^{\bar{t}}, \sfx^1)_\sharp \mmu$ is optimal. Applying condition (2) with $r= \max\{\bar{t}, t_{K-1}\}$, also using the last part of Theorem \ref{theo:chargeo}, we get the existence of the sought $\tau$. We can apply Claim 1 to $\mu_0, \mu_\tau$, and $(\sfx^0, \sfx^\tau)_\sharp \mmu$ to get
\[\directionalm{\Phi}{(\sfx^0, \sfx^\tau)_\sharp \mmu}{0}-\directionalp{ \Psi_\tau }{(\sfx^0, \sfx^\tau)_\sharp \mmu}{1}\le 0
\]
for every $\Phi \in \frF[\mu_0]$, $\Psi_\tau \in \frF[\mu_\tau]$. Since $(\sfx^\tau, \sfx^1)_\sharp \mmu$ is optimal and $\frF$ satisfies \eqref{Hdiss}, by Theorem \ref{theo:propflfr}(2) (more precisely, its finer version in \cite[Theorem 4.9(2)]{CSS}) we also have
\[\directionalm{\Psi_\tau}{(\sfx^\tau, \sfx^1)_\sharp \mmu}{0}-\directionalp{ \Psi }{(\sfx^\tau, \sfx^1)_\sharp \mmu}{1}\le 0,
\]
for every $\Psi_\tau \in \frF[\mu_\tau]$, $\Psi \in \frF[\mu_1]$. Applying Theorem \ref{thm:all}(3), summing the two expressions above, and using the $\mmu$-essential injectivity of $\sfx^\tau$ together with Theorem \ref{thm:all}(4), we get \eqref{eq:73bis}.
\end{proof}

\section{Construction of a totally \texorpdfstring{$\lambda$}{l}-dissipative \MPVF\ from a
  discrete core}\label{sec:constructionFlagr}
We have seen at the end of Section \ref{sec:3.2} (Corollary
\ref{cor:discrete-are-enough})
that a maximal totally $\lambda$-dissipative \MPVF
is determined by its restriction to the set of uniform discrete
measures.

In this section, we want to investigate the closely related question \ref{q1}, which leads, in a sense to the converse procedure. In other words: if we assign a \MPVF $\frF$ on a sufficiently
rich subset of discrete measures, is it possible to uniquely construct
a maximal extension of $\frF$? The answer to this question is the content of the main Theorems \ref{thm:main-discrete}, \ref{thm:perfavorelultimo}, \ref{thm:total-case}, and \ref{thm:demi-case}.

In the Hilbert setting, such kind of problems are well understood if
the domain of the initial operator is open and convex (see in particular \cite{Qi}, Proposition \ref{prop:qi} and Theorem \ref{theo:chefatica}). However, dealing with open sets
at the level of
$\prob_2(\X)$ will prevent the use of discrete measures.
We will circumvent this difficulty by a suitable localization of the open condition
in each subset $\prob_{\#N}(\X)$, which relies on the notion of
\emph{discrete core.}

Before giving the precise definition of core, let us fix some notation related to discrete measures: in order to allow for the greatest flexibility,
we consider collections of discrete measures indexed by
an unbounded directed subset $\cN\subset \N$ with respect to 
the partial order given by
 \begin{equation}
   \text{$m\preccurlyeq
 n\quad\Leftrightarrow\quad m\mid n$}\label{eq:166qui},
\end{equation}
where $m\mid n$ means that $n/m\in \N$. We write $m \prec n$ if $m
\preccurlyeq n$ and $m \ne n$.
Typical examples are the set of all natural integers $\cN:=\N$ or
the dyadic one $\cN:=\{2^n:n\in \N\}$.
We set
\begin{equation}
   \label{eq:137}
   \prob_{f,\cN}(\X):=\bigcup_{N\in \cN}\prob_{f,N}(\X),\quad
   \prob_{\#\cN}(\X):=\bigcup_{N\in \cN}\prob_{\# N}(\X),
 \end{equation}
 observing that, for every $N\in \cN$, $\prob_{f,N}(\X)$ is closed in
 $\prob_2(\X)$ and $\prob_{\#N}(\X)$ is a relatively
 open and dense subset of $\prob_{f,N}(\X)$.

 We can now give the definition of core.

\begin{definition}[$\cN$-core] 
  \label{ass:core}
  Let $\cN$ be an unbounded directed subset of $\N$ w.r.t.~the order
  relation $\preccurlyeq$ as in \eqref{eq:166qui}.
  A discrete \emph{$\cN$-core} is a set $\core\subset
  \prob_{\#\cN}(\X)$ such that $\overline{\core} \subset \prob_2(\X)$ is totally convex  and
  the family $\core_{N}:=\core \cap \prob_{\# N}(\X)$, $N\in \cN$, satisfies the
  following properties:
  \begin{enumerate}
  \item $\core_N$ is nonempty and relatively open in $\prob_{\#N}(\X)$
    (or, equivalently, in $\prob_{f,N}(\X)$);
\item $\core_N$ coincides with the relative interior in $\prob_{f,N}(\X)$ of $\overline{\core} \cap \prob_{\#N}(\X)$.
  \end{enumerate}
\end{definition}

\newtheorem*{continuancex}{Example \continuanceref{} (continued)}
\newenvironment{continuance}[1]
  {\newcommand\continuanceref{\ref{#1}}\continuancex}
  {\endcontinuancex}

\begin{example}[A simple core]\label{ex:core} A simple example of $\cN$-core is $\core:=\prob_{\#\cN}(\mathsf{U})$, where $\mathsf{U} \subset \X$ is a convex, open, non-empty subset, so that $\overline{\core}= \prob_2(\overline{\mathsf{U}})$ and $\core_N= \prob_{\# N}(\mathsf U)$ for every $N \in \cN$.
\end{example}

  We list here the main results of the section, which contain the answer to \ref{q1} and whose proof will be provided in Section \ref{sec:proofs8}. The first one shows how to recover a totally
  $\lambda$-dissipative \MPVF starting from a general (metrically)
  $\lambda$-dissipative \MPVF $\frF$
  whose domain is a $\cN$-core $\core$. 
\begin{theorem}[From dissipativity to total dissipativity]
  \label{thm:main-discrete} Let $\X$ be a separable Hilbert space, let $\frF \subset \prob_2(\TX)$ be a $\MPVF$ and let $\core \subset \prob_{\#\cN}(\X)$ be a $\cN$-core. Let us assume either one of the following hypotheses:
\begin{enumerate}[(i)]
    \item $\frF$ is $\lambda$-dissipative, $\dom(\frF)=\core$ and $\dim(\X)\ge 2$;
    \item $\frF$ is totally $\lambda$-dissipative and $\core \subset \dom(\frF) \subset \overline{\core}$.
\end{enumerate}
  For every $N\in \cN$ consider the \MPVF $\hat \frF_N$
  defined by the following formula:
  $\Phi\in \hat\frF_N[\mu]$ if and only if $\Phi \in \prob_{f,N}(\TX)$,
  $\mu\in \overline{\core_N}$
  and
  for every
  $\nu\in \core_N$, $\Psi\in \frF[\nu]$, 
   $\ttheta \in \Gamma_{f,N}(\Phi,\nu)$
  we have
  \begin{equation}
    \label{eq:32}
    \int_{\TX\times\X}\langle v_0-\bry\Psi(x_1),x_0-x_1\rangle\,\d \ttheta(x_0,v_0,x_1)\le \lambda
    \int_{\TX\times\X} |x_0-x_1|^2\,\d\ttheta(x_0,v_0,x_1).
  \end{equation}
  Then, we have the following properties:
  \begin{enumerate}
  \item
    For every $N\in \cN$,  for any $\Phi_0,\Phi_1\in\hat \frF_N$ and any coupling $\ttheta\in\Gamma(\Phi_0,\Phi_1)\cap \prob_{f,N}(\TX\times \TX)$, we have
    \begin{equation*}
        \int_{\TX^2} \langle v_1-v_0,x_1-x_0
        \rangle \,\d\bm\ttheta(x_0,v_0,x_1,v_1)\le \lambda \int_{\TX^2} |x_1-x_0|^2\d\bm\ttheta,
    \end{equation*}
    and $\dom(\hat\frF_N) $ contains 
    $\core_N$. 
    \item For every $\mu \in \overline{\core_N}$, let
    \[\maps{\hat{\frF}_N}[\mu]:= \left \{ \ff \in L^2(\X, \mu; \X) : (\ii_\X, \ff)_\sharp \mu \in \hat{\frF}_N[\mu] \right \};\]
     then,  $\ff\in L^2(\X, \mu; \X)$  belongs to $\maps{\hat{\frF}_N}[\mu]$ if and only if for every $\nu\in \core_N$, $\Psi\in \frF[\nu]$, $\mmu \in \Gamma_{f,N}(\mu,\nu)$ we have
  \begin{equation}
    \label{eq:32bis}
    \int_{\X^2}\langle \ff(x_0)-\bry\Psi(x_1),x_0-x_1\rangle\,\d \mmu(x_0,x_1)\le \lambda
    \int_{\X^2} |x_0-x_1|^2\,\d\mmu(x_0,x_1).
  \end{equation}
   Moreover, in order for $\ff$ to belong to
  $\maps{\maxim \frF N}[\mu]$, it is sufficient to check 
  \eqref{eq:32bis} only for all the measures $\nu\in \core_N$ 
  and all the couplings $\mmu\in \Gamma(\mu,\nu)$ 
  such that $\mmu$ is the unique element 
  of $\Gamma_o(\mu,\nu)$.
    \item
    $M\mid N$ implies $\dom(\hat \frF_M)\subset
    \dom(\hat\frF_N)$.
  \item
    The \MPVF 
    \begin{equation}
      \hat\frF_\infty:=\bigcup_{M\in \cN}\bigcap_{N\in \cN\,:\,M\mid
      N}\hat\frF_N\quad\text{with domain}\quad
    \dom(\hat\frF_\infty)=\bigcup_{M\in \cN}
    \dom(\hat\frF_M)\supset \core
    \label{eq:38}
  \end{equation}
  is totally $\lambda$-dissipative.
\item
    There exists a unique maximal totally $\lambda$-dissipative 
  \MPVF $\hat\frF$ extending $\hat\frF_\infty$ whose domain is contained in
  $\overline \core$. For every $\mu\in \overline\core$,
  $\hat\frF[\mu]$ consists of
  all the measures $\Phi\in \prob_2(\TX|\mu)$
  satisfying
  \begin{equation}
    \label{eq:34}
    \int_{\TX\times\X} \langle v-\ff(y),x-y\rangle\,\d\ttheta(x,v,y)\le
    \lambda\int_{\TX\times\X} |x-y|^2\,\d\ttheta    
  \end{equation}
  for every $\ttheta\in \Gamma(\Phi,\nu)$ with
  $\nu\in \dom(\hat\frF_\infty)$ and $(\ii_\X, \ff)_\sharp\nu\in
  \hat\frF_\infty$.
The \MPVF  $\hat\frF$ also coincides with the strong closure of
  $\hat \frF_\infty$ in $\prob_2(\TX).$
  Finally, if $\mu\in \core$ then the minimal selection $\hat\frF{\vphantom\frF}^\circ$ of $\hat\frF$ satisfies
  \[\hat\frF{\vphantom\frF}^\circ[\mu]\in \hat\frF_\infty[\mu].\]
  \end{enumerate}
\end{theorem}

The construction of $\hat\frF_\infty$ follows a ``restrict, then refine, then unite'' strategy to build a single, consistent multi-particle field that works for all particle numbers.

\begin{itemize}
    \item \textbf{The core for a fixed scale (restriction).} For a given number of particles $N$, we first define $\hat\frF_N$. This is the unique maximal extension of the original field $\frF$ when restricted to the core $\core_N$ (cf. Proposition \ref{prop:ext1} and Theorem \ref{prop:final}), characterized by condition \eqref{eq:32}. It represents the ``largest" field at level $N$ that still satisfies the dissipativity condition against all barycenters of elements of $\frF$ inside $\core_N$.

    \item \textbf{The consistency problem (refinement).} A configuration with $M$ particles can be seen as a configuration with $N$ particles whenever $N$ is a multiple of $M$ (by treating the $M$ particles as being made of smaller subunits). To be consistent, the field at level $M$ must be compatible with the field at every finer resolution $N$.

    \item \textbf{Ensuring compatibility (inner intersection).} To enforce consistency as above, we do not take $\hat\frF_M$ directly. Instead, for a fixed $M$, we consider all finer scales $N$ that are multiples of $M$. We then take the \emph{inner intersection}
    \[
    \bigcap_{N\in\cN\,:\,M \mid N} \hat\frF_N.
    \]
    This yields the part of the field at level $M$ that is compatible with the fields at all higher resolutions. This step makes the field more restrictive but guarantees consistency under refinement.

    \item \textbf{Combining all scales (union).} After performing this compatibility intersection for every $M$, we have a family of fields, one for each particle number, that are all mutually compatible. We can now safely take the union over all $M$ to obtain
    \[
    \hat\frF_\infty := \bigcup_{M\in \cN}\bigcap_{M\mid N}\hat\frF_N.
    \]
\end{itemize}

In the next result, we specify, in the general case, how $\frF$ and $\hat \frF$ are compatible in terms of $\lambda$-\EVI solutions. We show that $\frF$ indeed generates $\lambda$-\EVI solutions starting from every point of its domain -- which was not known a priori, since $\frF$ generally does not satisfy the hypotheses of \cite{CSS} or those of Section \ref{sec:totdissMPVF-flow}. These $\lambda$-\EVI solutions coincide with those generated by the maximal totally $\lambda$-dissipative \MPVF $\hat \frF$ constructed from $\frF$; moreover, when starting from a point in the core $\core$, they can be characterized purely in metric terms involving only $\frF$. Since $\core$ is dense in $\dom(\hat\frF)$, characterizing the Lagrangian solutions of the flow generated by $\hat \frF$ starting from every measure in $\core$ allows us to recover all other evolutions by approximation.
\begin{theorem}\label{thm:perfavorelultimo} Assume the hypothesis of Theorem \ref{thm:main-discrete}, let $\mu_0\in \overline{\core_N}$
for some $N\in \cN$. Then there exists a $\lambda$-\EVI solution $\mu:[0, +\infty) \to \overline{\core_N} \subset \prob_{f,N}(\X)$ for the restriction of $\frF$ to $\core_N$, starting from $\mu_0$, which is locally absolutely continuous in $(0,+\infty)$. Moreover, $\mu$ can be equivalently characterized by the following two properties:
    \begin{enumerate}
        \item $\mu$ is a Lagrangian solution of the flow generated by $\hat\frF$ (cf. Definition \ref{def:semig});
    \item $\mu$ is locally absolutely continuous 
    in $[0,+\infty)$ and locally Lipschitz continuous in 
    $(0,+\infty)$, 
    there exists a constant $C>0$ such that
    the Wasserstein velocity field $\vv$ of $\mu$ (cf.~Theorem \ref{thm:tangentv})
    satisfies
    \begin{equation}
    \label{eq:v-estimate}
        I_\lambda(t) \Big(\int_\X |\vv_t|^2\,\d\mu_t\Big)^{1/2}\le C\quad\text{a.e.~in }(0,1),
    \end{equation}
    $\mu_t\in \dom(\hat \frF_N)\subset \dom(\hat\frF)$
    for every $t>0$, and it holds
    \begin{equation}
    \label{eq:v-in-F}
    \vv_t =\hat\ff{\vphantom\ff}^\circ[\mu_t]
        \quad\text{for $\mathscr L^1$-a.e.~$t>0$},
    \end{equation}
    where $\hat\ff{\vphantom\ff}^\circ$ is the minimal selection map 
    induced by $(\hat\frF\vphantom{\frF})^\circ$ as in Theorem \ref{thm:minimal} and $I_\lambda(t)$ is as in \eqref{eq:defIlambda}.
    \end{enumerate}
\end{theorem}

We discuss two particular cases in more detail: the first one occurs when $\frF$ is  totally $\lambda$-dissipative.
\begin{theorem}[Unique maximal extension of a totally dissipative \MPVF]
  \label{thm:total-case}
      If $\frF$ is a 
      totally $\lambda$-dissipative $\MPVF$
  whose domain contains a dense $\cN$-core $\core$.
  Then the \MPVF $\hat \frF$ constructed
  as in Theorem \ref{thm:main-discrete} provides the unique maximal totally
  $\lambda$-dissipative extension of $\frF$ with domain included in $\overline{\core}$.
\end{theorem}

A second case occurs when we know that  $\frF$ is
a deterministic $\lambda$-dissipative
\MPVF: as in Theorem \ref{thm:demi-total} we obtain
that $\lambda$-dissipativity implies 
total $\lambda$-dissipativity; here however, we deal with a \MPVF (not necessarily single-valued)
defined in a much smaller domain.
\renewcommand{\gg}{\bm g}
\begin{theorem}[Deterministic dissipative \MPVF{s} on a core are totally dissipative]
  \label{thm:demi-case}
  Let us suppose that $\dim\X\ge2$ and
  $\frF\subset\prob_2(\TX)$  is a
  deterministic
  $\lambda$-dissipative \MPVF whose domain is a
  $\cN$-core $\core$.
 Then $\frF$ is totally $\lambda$-dissipative, $\hat\frF_\infty$ (cf. \eqref{eq:38}) is a totally $\lambda$-dissipative extension of $\frF$ and, for every $\mu\in
  \bigcup_{N\in \cN}\overline{\core_N}$, $\ff \in \maps{\hat{\frF}_\infty}[\mu]$ if and only if 
  \begin{equation}
    \label{eq:32tris}
    \int_{\X^2}\langle \ff(x_0)-\gg(x_1),x_0-x_1\rangle\,\d \mmu(x_0,x_1)\le \lambda
    \int_{\X^2} |x_0-x_1|^2\,\d\mmu(x_0,x_1)
  \end{equation}
  for all the measures $\nu\in \core$, $\gg \in \maps{\frF}[\nu]$, and all the couplings $\mmu\in \Gamma(\mu,\nu)$ 
  such that $\mmu$ is the unique element 
  of $\Gamma_o(\mu,\nu)$.
 The \MPVF $\hat\frF$ of Theorem \ref{thm:main-discrete}(5) provides the unique maximal totally $\lambda$-dissipative extension of $\frF$ with domain included in $\overline{\core}$.
  If moreover $\frF$ is single-valued and the restriction of $\frF$ to each set $\core_N$, $N \in \cN$, is demicontinous, then
  the restrictions of $\hat \frF_\infty$ and $\hat{\frF}{\vphantom \frF}^\circ$ to $\core$ coincide with $\frF$. 
\end{theorem}

We devote the remaining part of this section to the proof of the above
main theorems. We adopt a Lagrangian viewpoint, lifting the \MPVF
$\frF$ to the Hilbert space $\cH:=L^2(\Omega, \cB, \P;\X)$ and
parametrizing probability measures by random variables in $\cH$ as we
did in Section \ref{sec:3.2}.

\medskip
\makeatletter
\newcommand{\mylabel}[2]{#2\def\@currentlabel{#2}\label{#1}}
\makeatother

We proceed as follows:
\begin{enumerate}
    \item[\mylabel{expl:it:1}{(1)}] In Section \ref{sec:contropL}, we introduce the framework used for our construction, in particular the study of $\cN$-cores. We start from a Lagrangian description of discrete measures, viewed as elements of $\cH$ that take only finitely many distinct values. From this perspective, we derive several equivalent characterizations of $\cN$-cores in Propositions \ref{le:equivalent} and \ref{le:dimcore}. These characterizations will be used repeatedly in the proofs of the results that follow.
    \item[\mylabel{expl:it:2}{(2)}] Section \ref{sec:fn} is devoted to the construction of $\hat{\frF}_N$ as in Theorem \ref{thm:main-discrete}. We start by using the $\cN$-core–compatible Lagrangian representation $\fF$ of $\frF$ given in \eqref{eq:105} to define a suitable Lagrangian restriction of $\frF$ to discrete measures with exactly $N \in \cN$ distinct atoms. This restriction is denoted by $\fF_N$, defined in \eqref{eq:105}, and its properties are studied in Proposition \ref{prop:perunmotivo}. 
    In the subsequent Proposition \ref{prop:ext1}, we define its maximal extension $\hat{\fF}_N$, which will turn out to be the Lagrangian representation of $\hat{\frF}_N$ (cf. Theorem \ref{prop:final}), and analyze its properties. 
    The final three results of the section provide additional characterizations and properties of $\hat{\fF}_N$: the first, Proposition \ref{prop:allultimomomento}, under the general assumptions of Theorem \ref{thm:main-discrete}, and Corollaries \ref{cor:dim-total-case1} and \ref{cor:dim-demi-case1} under the stronger hypotheses of the main Theorems \ref{thm:total-case} and \ref{thm:demi-case}, respectively. These three results are used directly in the proofs of the corresponding main theorems.
   \item[\mylabel{expl:it:3}{(3)}] Section \ref{sec:finf} is devoted to the construction of $\hat{\frF}_\infty$ and $\hat{\frF}$ as in Theorem~\ref{thm:main-discrete}. We begin by showing in Proposition~\ref{prop:eulerstep} and Corollary~\ref{cor:propfcirc} that the resolvent and the minimal selection operators of $\hat{\fF}_N$ are compatible, in a suitable sense, across different values of $N$. We then introduce in~\eqref{eq:181} the Lagrangian representation $\hat{\fF}_\infty$ of $\hat{\frF}_\infty$, and recast in Corollary~\ref{cor:summarize} the properties of the resolvent and minimal selection in terms of $\hat{\fF}_\infty$. Thanks to these results, in Corollary~\ref{cor:help-me} we are able to define $\hat{\fF}$, which will turn out to be the Lagrangian representation of $\hat{\frF}$ (cf. Theorem \ref{prop:final}), and to study some of its properties. 
   \item[\mylabel{expl:it:4}{(4)}] Section \ref{sec:proofs8} contains Theorem~\ref{prop:final}, which includes the main Theorem~\ref{thm:main-discrete} and its proof, and the proofs of the remaining main Theorems~\ref{thm:perfavorelultimo}, \ref{thm:total-case}, and \ref{thm:demi-case}.
   \item[\mylabel{expl:it:5}{(5)}]  Section \ref{subsec:ex} contains a few examples of the theory just developed.
\end{enumerate}

\subsection{Lagrangian representations of \texorpdfstring{$\cN$}{N}-cores}\label{sec:contropL}
In this section, we initiate a Lagrangian approach to the description of discrete measures. To this end, we fix a standard Borel 
space $(\Omega, \cB)$ endowed with a nonatomic  probability measure $\P$ (see Definition \ref{def:sbs}).

Given $\cN$, an unbounded directed subset of $\N$ w.r.t.~the order relation $\preccurlyeq$ as in \eqref{eq:166qui},
we consider a $\cN$-segmentation of $(\Omega, \cB, \P)$ (see Definition \ref{def:segm}) that we denote by $(\mathfrak P_N)_{N \in \cN}$.
We define $\cB_N:=\sigma \left (\mathfrak P_N \right )$, $N\in\cN$, and we denote by $(\Omega, \cB, \P, (\mathfrak P_N)_{N \in \cN})$, with $\mathfrak P_N=\{\Omega_{N,n}\}_{n\in I_N}$ and $I_N:=\{0,\dots,N-1\}$, the $\cN$-refined 
probability space as in Definition \ref{def:segm} induced by $(\mathfrak P_N)_{N \in \cN}$ on $(\Omega, \cB, \P)$. We set
\[ \cH:= L^2(\Omega, \cB, \P; \X),\quad \cH_N:= L^2(\Omega, \cB_N, \P; \X), \quad N \in \cN, \quad \cH_\infty:= \bigcup_{N \in \cN} \cH_N, \]
and we recall that $\cH_\infty$ is dense in $\cH$ by Proposition
\ref{prop:strook}.

Even if the choice of a general standard Borel space allows for a great
generality, it would not be restrictive to
focus on the canonical example below, at least at a first reading.
\begin{example}\label{ex:canon}
  The canonical example of $\cN$-refined standard Borel probability  space is 
\[ ([0,1), \mathcal{B}([0,1)), \lambda , (\mathfrak {I}_N)_{N \in \cN}),\]
where $\lambda$ is the one dimensional Lebesgue measure restricted to $[0,1)$ 
and $\mathfrak {I}_N=\{I_{N,k}\}_{k \in I_N}$ with
$I_{N,k}:=[k/N,(k+1)/N)$, $k \in I_N$ and $N \in \cN$.
The space $\cH_N$ can then be identified with the class of functions
which are (essentially) constant in each subintervals
$I_{N,k}$, $k\in I_N$, of the partition $\mathfrak I_{N,k}$.
\end{example}

As in Section \ref{sec:invmpvf}, we parametrize measures in $\prob(\X)$ by random variables in $(\Omega,\cB, \P)$ and we use the notation $\iota:\cH\to \prob_2(\X)$ for the map sending $X \in \cH$ to $\iota(X)=X_\sharp \P = \iota_X \in \prob_2(\X)$. Recall that 
\begin{equation}
  \label{eq:74}
  W_2(\iota_X,\iota_Y)\le |X-Y|_\cH\quad\text{for every }X,Y\in \cH.
\end{equation}
If $(X,V)\in \cH\times \cH$, recall the notation $\iiota{X,V}=(X,V)_\sharp\P\in \prob_2(\TX)$.

\newcommand{\vecX}[1]{\X^{#1}}
We can identify $\cH_N$ with the space
$\X^N$: indeed, each $X\in\cH_N$ is associated with a vector
$\xx:I_N\to \X$ such that
$\xx(n)=X(\omega)$ whenever $\omega\in \Omega_{N,n}$.
In this case, we set
\[\IIN(\xx):=X.\] 
Clearly $\iota(\cH_N)=\prob_{f,N}(\X)$
and $\iota(\cH_\infty)=\prob_{f,\cN}(\X)$.

The isomorphism $\IIN$
preserves the scalar product on $\vecX N$
\begin{equation*}
  \la \xx,\yy\ra_{\vecX N}
  :=\udN\sum_{n=0}^{\dN-1} \la \xx(n),\yy(n)\ra
  =\E\big[\la \IIN(\xx),\IIN(\yy)\ra\big]=\la\IIN(\xx),\IIN(\yy)\ra_{\cH }\quad
  \xx,\yy \in \vecX N.
\end{equation*}
The conditional expectation $\Pi_N=\E[\cdot |\cB_N]$ provides the orthogonal
projection of an arbitrary map $X\in \cH $ onto $\cH_N$:
\begin{equation}\label{eq:projrep}
  \Pi_N(X)(\omega)=N \int_{\Omega_{N,n}}X\,\d\P\quad
  \text{if }\omega\in \Omega_{N,n}.
\end{equation}
Notice that 
\begin{equation*}
  \text{if $M\mid N$ then $\cB_M\subset \cB_N$ and $\Pi_M=\Pi_M\circ
    \Pi_N$}.
\end{equation*}
For every $X=\IIN(\xx)\in \cH_N$, the probability measure $\iota_X=X_\sharp \P$ takes the form
\[\iota_X=\frac{1}{N}\sum_{n=0}^{\dN-1}\delta_{\xx(n)}\in
\prob_{f,N}(\X).\]

\medskip
We denote by $\sfO_N\subset \vecX N$ the subset
of the injective maps and by
\begin{equation}\label{eq:ONlagr}
\cO N:=\IIN(\sfO_N)\subset \cH_N.
\end{equation}
Clearly, $\iota(\cO N)=\prob_{\#N}(\X)$.
Since the complement of $\sfO_N$ is the union
of a finite number of proper closed subspaces with empty interior
$S_{ij}:=\{\xx\in \vecX N:\xx(i)=\xx(j)\}$, $i\neq j$, of $\vecX N$,  then $\sfO_N$ is open and dense in $\vecX N$. 

\smallskip

Every permutation $\sigma\in \symg {I_N}$ acts on $\vecX N$
via $\sigma \xx(n):=\xx(\sigma(n))$ and can be thus extended to
$\cH_N$ via $\sigma (\IIN( \xx)):=\IIN(\sigma(\xx))$.
It is not difficult to see that,
for every $X,Y\in \cH_N$, 
$\iota_X=\iota_Y$ is equivalent to $Y=\sigma X$ for some
$\sigma\in \symg{I_N}$.

As in Section \ref{sec:invmpvf}, we denote by $\rmS(\Omega)$ the class of
$\cB$-$\cB$-measurable maps $g:\Omega\to\Omega$ which are
essentially injective and measure-preserving, meaning that there
exists a full $\P$-measure set $\Omega_0 \in \cB$ such that $g$ is
injective on $\Omega_0$ and $g_\sharp \P=\P$. Moreover, for every $N
\in \cN$, we denote by $\rmS_N(\Omega):=\rmS(\Omega, \cB, \P; \cB_N)$ the subset of $\rmS(\Omega)$ of $\cB_N$-$\cB_N$ measurable maps.

\begin{remark}\label{rmk:gsigma}
Clearly, if $X=\IIN(\xx)\in \cH_N$ and $g\in \rmS_N(\Omega)$ then
$X\circ g\in \cH_N$
and there exists a unique permutation $\sigma=\sigma_g\in\symg{I_N}$
such that $X\circ g=\sigma_gX=\IIN(\xx\circ\sigma_g)$.
Conversely, if $\sigma\in \symg{I_N}$ there exists
$g\in
\rmS_N(\Omega)$
such that $\sigma=\sigma_g$, as shown in Lemma \ref{cor:isomor}.
We set $G[\sigma]:=\big\{g\in \rmS_N(\Omega):\sigma_g=\sigma\big\}.$
\end{remark}

 As anticipated, the aim of this subsection is to prove equivalent characterizations of $\cN$-cores. The main result is the following.

\newcommand{\tcore}{\mathrm D}
\newcommand{\ecore}{\mathrm E}
\newcommand{\intecore}{\mathring{\ecore}}
\begin{proposition}[Equivalent characterizations of $\cN$-cores]
  \label{le:equivalent}
 Let $\core\subset \prob_{\#\cN}(\X)$;
  then the following properties are equivalent:
  \begin{enumerate}[$(a)$]
  \item the family of sets $\core_N = \core \cap \prob_{\#N}(\X)$ satisfies
\begin{enumerate}[\rm (1*)]
  \item $\core_N$ is relatively open in $\prob_{\#N}(\X)$
    (or, equivalently, in $\prob_{f,N}(\X)$),
  \item $\core_N$ is convex along collisionless couplings (cf. Definition \ref{def:collisionless}),
  \item if $M,N\in \cN$, $M\mid N$ then
    $\overline{\core_M}=\overline{\core_N}\cap
    \prob_{f,M}(\X)$,
  \item $\overline{\core_N}$ is convex along couplings in $\prob_{f,N}(\X \times \X)$;
  \end{enumerate}
  \item $\core$ is a $\cN$-core;
  \item there exists a subset $\tcore$ of $\prob_{f,\cN}(\X)$ such that $\core= \tcore \cap \prob_{\#\cN}(\X)$ and, for every $N\in \cN$, the set $\tcore_N:=\tcore\cap \prob_{f,N}(\X)$ satisfies the following two conditions:
    \begin{enumerate}[\rm (1')]
    \item $\tcore_N
      $ is relatively open in $\prob_{f,N}(\X)$,
    \item $\tcore_N
      $ is convex along couplings in $\prob_{f,N}(\X\times \X)$;
    \end{enumerate}
    
  \item
    there exists 
    a totally convex and closed
    subset $\ecore$ of $\prob_{2}(\X)$ such that $\core=\bigcup_{N\in \cN}
    {\intecore}_N\cap \prob_{\#N}(\X)$ and
    \begin{enumerate}[\rm (1'')]
    \item
      for every
      $N\in \cN$ the sets
      $$\intecore_N:=
      \text{relative interior of }\big(\ecore\cap \prob_{f,N}(\X)\big)
      \text{       in $\prob_{f,N}(\X)$}
      $$
      are not empty,
    \item $\ecore\cap \prob_{f,\cN}(\X)$ is dense in $\ecore$.
    \end{enumerate}
  \end{enumerate}
    In the above cases
    the sets $\core_N$, $\tcore_N$, $\intecore_N$, $\core$, $\tcore$ and
    $\ecore$ are linked by the following 
    relations
    \begin{align}
      \label{eq:39}
        &\core_N=\tcore_N\cap \prob_{\#N}(\X)=
        \intecore_N\cap \prob_{\#N}(\X),\quad
        \core=\bigcup_{N\in \cN}\core_N,
          ,\\
      \label{eq:50}
        &\tcore_N=\intecore_N=\text{ relative interior of
          $\overline{\core_N}$
          in $\prob_{f,N}(\X)$,}\quad
        \tcore=\bigcup_{N\in \cN}\tcore_N=\bigcup_{N\in \cN}\intecore_N
        ,\\\label{eq:51}
        &\overline{\core_N}=\overline{\tcore_N}=\ecore\cap \prob_{f,N},
        \\\label{eq:52}
                &\overline{\core}=\overline{\tcore}=\ecore.
    \end{align}
\end{proposition}

\begin{continuance}{ex:core} In the simple case of $\core= \prob_{\#\cN}(\mathsf{U})$, we have $\tcore_N = \intecore_N = \prob_{f,N}(\mathsf U)$, $\tcore = \prob_{f, \cN}(\mathsf U)$, and $\ecore=\prob_2(\overline{\mathsf U})$.
\end{continuance}

 The proof of Proposition \ref{le:equivalent} requires two preliminary lemmas.
The first one establishes an  interesting relation between projections and
permutations.  We denote by $\text{rel-int}(A;B)$ the relative interior of a set $A$ in $B$.  

\begin{lemma}
  \label{le:useful}
   Let $N,M\in\cN$ be such that $M\mid N$.
  If $\mathcal K$ is a convex subset of $\cH_N$ invariant
  by the action of $\Sym{I_N}$, then
  \begin{equation}
    \label{eq:48}
    \Pi_M\big(\mathcal K\big)=\mathcal K\cap \cH_M.
  \end{equation}
  Moreover,
  \begin{equation}
    \label{eq:56}
    \overline{\mathcal K}\cap \cH_M=
    \overline{\mathcal K\cap \cH_M}
  \end{equation}
  and,  if $\operatorname{rel-int}(\mathcal K; \cH_N)$ is not empty, we have
  \begin{equation}
    \label{eq:56bis}
    \operatorname{rel-int}(\mathcal K; \cH_N) \cap \cH_M =  \operatorname{rel-int}(\mathcal K\cap \cH_M; \cH_M).
  \end{equation} 
\end{lemma}
\begin{proof}
    Let us first compute the explicit representation of the orthogonal
  projection $\Pi_M(X)$ for every $X\in \cH_N$.
  If $K:=N/M$ we consider 
  the cyclic permutation $\sigma:I_N\to I_N$
  defined by
  \begin{displaymath}
    \sigma(n):=
    \begin{cases}
      mK+k+1&\text{if }n=mK+k,\ m\in I_M,\ 0\le k<K-1,\\
      mK&\text{if }n=mK+K-1,\ m\in I_M,
    \end{cases}
  \end{displaymath}
  and its powers $\sigma^p$, $p\in I_K$.
  It is not difficult to check that
  $\sigma^K=\sigma^0=\ii_{I_N}$
  and for every $Y\in \cH_M$ we have
  $\sigma^p Y=Y$ for every $p\in I_K$.
  Therefore, by \eqref{eq:projrep},  for every $X\in \cH_N$  and $\omega\in\Omega_{M,m}$, with $m\in\ I_M$,  we obtain the representation 
  \begin{align*}
    \Pi_M(X)(\omega)&= M\int_{\Omega_{M,m}} X\d\P\\
    &=\frac{N}{K}\int_{\cup_{p=0}^{K-1}\Omega_{N,mK+p}}X\d\P\\
    &=\frac{N}{K}\frac{1}{N}\sum_{p=0}^{K-1}X|_{\Omega_{N,mK+p}}\\
    &=\frac 1K\sum_{p=0}^{K-1}(\sigma^p X)(\omega).
  \end{align*} 
  If $\mathcal K$ is a convex subset of $\cH_N$ invariant
  by the action of $\Sym{I_N}$,
  we get $\Pi_M(X)\in \mathcal K$ for every $X\in \mathcal K$,
  so that $\Pi_M(\mathcal K)=\mathcal K\cap \cH_M$, hence we proved \eqref{eq:48}.

  In order to check \eqref{eq:56}, we observe
  that in general
  $ \overline{\mathcal K\cap \cH_M}\subset \overline{\mathcal K}\cap
  \cH_M$;
  on the other hand $\overline{\mathcal K}\cap
  \cH_M=\Pi_M(\overline{\mathcal K})\subset
  \overline{\Pi_M(\mathcal K)}=\overline{\mathcal K\cap \cH_M} $
  by \eqref{eq:48}.

  Similarly,  if we denote  $\mathring {\mathcal A}_M:=\operatorname{rel-int}(\mathcal K \cap \cH_M; \cH_M)$ and $\mathring {\mathcal B}_N:=\operatorname{rel-int}(\mathcal K; \cH_N)$,  
  as a general fact $\mathring{\mathcal B}_N\cap \cH_M\subset
  \mathring{\mathcal A}_M$ so that
  $\mathring {\mathcal A}_M$ is not empty, since by \eqref{eq:48} $\mathring{\mathcal B}_N\cap \cH_M= \Pi_M(\mathring{\mathcal B}_N)$ is not empty.
  On the other hand, by \eqref{eq:56},
  $
  \overline{\mathring{\mathcal B}_N\cap \cH_M}=
  \overline{\mathring{\mathcal B}_N} \cap \cH_M=
  \overline{\mathcal K} \cap \cH_M=
  \overline{\mathcal K \cap \cH_M}=
  \overline{\mathring{\mathcal A}_M}$
  so that the open convex sets $\mathring{\mathcal B}_N\cap \cH_M$ and
  $\mathring {\mathcal A}_M$ have the same closure and therefore coincide. 
\end{proof}

We introduce the following Lagrangian representation of a $\cN$-core: if $\core$ is a $\cN$-core and $N \in \cN$, we set
\begin{equation}
  \label{eq:138}
  \begin{aligned}
    \newODDom N:=
    \Big\{X\in \cH_N:\iota_X \in \core_N
    \Big\},\quad &\newODDom
    \infty:= \Big\{X\in \cH_\infty:\iota_X \in \core\Big\}=
    \bigcup_{N\in \cN}\newODDom N\\
    \newDDom N:=\conv{\newODDom N},\quad
    &\newDDom
    \infty:= 
    \bigcup_{N\in \cN}\newDDom N,\quad
    \mathcal E_\infty:=\overline{\newODDom \infty}.
  \end{aligned}
\end{equation}
Notice that $\newODDom N$ is in fact a subset of $\cO N$  (cf. \eqref{eq:ONlagr}), 
and $\newDDom N$ is a subset of $\cH_N$.

In the next results of this section, we investigate the properties of the sets defined in \eqref{eq:138}, inherited by those of $\cN$-cores. These sets will play a crucial role in the next Sections \ref{sec:fn} and \ref{sec:finf}, where we will study suitable Lagrangian representations of $\frF$ restricted to subsets of the $\cN$-core $\core$.

\begin{continuance}{ex:core} In the simple case of $\core= \prob_{\#\cN}(\mathsf U)$, denote by $\mathcal{U} \subset \cH$ the set of maps taking values in $\mathsf U$. Then, we have that $\newODDom N=\cO N \cap \mathcal{U}$, $\newODDom \infty$ is the set of injective maps in $\cH_\infty\cap \mathcal{U}$, $\newDDom N=\cH_N \cap \mathcal{U}$, $\newDDom \infty=\cH_\infty \cap \mathcal{U}$, and $\mathcal{E}_\infty$ is the set of maps in $\cH$ taking values in $\overline{\mathsf{U}}$.
\end{continuance}

 In this second preliminary lemma (together with its immediate corollary), we prove several properties of the Lagrangian representations of $\cN$-cores in \eqref{eq:138}. These will also contribute in proving the equivalence results stated in Proposition \ref{le:equivalent}. 

\begin{lemma}
  \label{le:trivial} Assume that $\core \subset \prob_{\#\cN}(\X)$ satisfies property $(a)$ in Lemma \ref{le:equivalent}.  Then for every $N\in \cN$ it holds:
  \begin{enumerate}
  \item
  $\newODDom N$ and $\newDDom N$ are relatively
    open subsets of $\cH_N$,
    invariant with respect to the action of permutations of $\symg{I_N}$.
  \item
  The relative interior of $\overline{\newODDom N}$ in
  $\cH_N$ coincides
  with
  $\newDDom N$, in particular $\newODDom N$ is dense in $\newDDom N$
 and $\overline{\newODDom N}=\overline{\newDDom N}$.
\item $
  \newDDom N\cap \cO N=\newODDom N$  and, if $X \in \newODDom N$ and $Y \in \overline{\newDDom N}$, there exists $\eps>0$ such that $X_t:=(1-t)X+tY \in \mathcal{C}_N$ for every $t \in (1-\eps, 1)$.
  \item If $M\in \cN$ and $M\mid N$ then
    $\DDom M=\DDom N\cap\cH_M=\Pi_M(\DDom N)$ and
    $\overline {\DDom M}=
    \overline{\DDom N}\cap\cH_M=\Pi_M\Big(\overline{\newDDom N}\Big)$.
  \item
    $\newODDom \infty \subset \newDDom \infty \subset
    \overline{\newDDom \infty} = \overline{\newODDom \infty}=\mathcal E_\infty$ and
    $\mathcal E_\infty$ is convex.
  \item
    $\newDDom N=\newDDom \infty\cap \cH_N=\Pi_N(\newDDom \infty)$ and
    $\overline{\newDDom N}=
    \mathcal E_\infty\cap \cH_N=\Pi_N(\mathcal E_\infty)$.
  \item $\mathcal E_\infty
    =\overline{\newDDom \infty} = \overline{\newODDom \infty}$ is law invariant.
  \end{enumerate}
\end{lemma}
\begin{proof}
  (1)  It is clear by construction that both $\newODDom N$ and $\DDom N$ are invariant w.r.t.~the action of permutations in $\symg{I_N}$.  
  The set $\newODDom N$ is relatively open, since
  the map $X\mapsto \iota_X$ is Lipschitz
  from $\cH_N$ to $\prob_{f,N}(\X)$, thanks to \eqref{eq:74}, and
  $\core_N$ is relatively open in $\prob_{f,N}(\X)$  by assumption (1*). 
   The set $\DDom N= \conv{\newODDom N}$  is relatively open in $\cH_N$ since
  it is the convex hull of  the  relatively open set  $\newODDom N$.
  
  \smallskip
  (2)  Since $\newDDom N$ is open by item (1) and convex by construction, it coincides with the interior of its closure. Therefore, we only need to show that $\overline{\newDDom N}= \overline{\newODDom N}$. Obviously, $\newODDom N \subset \newDDom N$ by construction, so that $\overline{\newODDom N} \subset \overline{\newDDom N}$. To show the reverse inclusion, it is enough to prove that $\overline{\newODDom N}$ is convex: indeed $\newDDom N= \conv{\newODDom N}$ is the smallest convex set containing $\newODDom N$ and then it must be contained in $\overline{\newODDom N}$, if the latter is convex. Let us show it: we take $X,Y\in\overline{\newODDom N}$, so that $\iota_X,\iota_Y\in\overline{\core_N}$, and we choose the coupling $\rrho:=\iotaT_{X,Y}\in\prob_{f,N}(\X\times\X)$. Let $t \in [0,1]$, since $\overline{\core_N}$ is convex along $\rrho$ by assumption (4*), we get 
  \[\sfx^t_\sharp\rrho=\iota_{(1-t)X+tY}\in\overline{\core_N}.\]
  Thus, there exists $(\mu_n)_{n\in\N}\subset \core_N$ such that $W_2(\mu_n,\iota_{(1-t)X+tY})\to0$ as $n\to+\infty$. Recalling Theorem \ref{thm:gpfinal}, there exists $(Z_n)_{n\in\N}\subset\cH_N$, $\iota_{Z_n}=\mu_n$, such that $Z_n\to (1-t)X+tY$. In particular, since $Z_n\in\newODDom N$, we conclude that $(1-t)X+tY\in\overline{\newODDom N}$. By arbitrarity of $X, Y$ and $t$, this gives the sought convexity.

  \smallskip
  (3)  As noted just after \eqref{eq:138}, we have  $\newODDom N\subset \newDDom N\cap \cO N$.
  Let now show that any element
  $X=\IIN(\xx)\in \newDDom N\cap \cO N$ belongs to $\newODDom N$. 
  If $\mathsf B_N$ is the open unit ball in $\vecX N$, since $\newDDom N\cap \cO N$ is open by item (1), 
  there
  exists 
  a sufficiently small $\eps>0$ such that the open set
  $\mathcal A_\eps:=\{
  (\IIN(\xx+\eps \zz),\IIN(\xx-\eps \zz)):\zz\in \mathsf B_N\}$
  is contained in $\big(\newDDom N\cap \cO
  N\big)^2$.
  Since $\newODDom N$ is relatively open and
  dense in $\newDDom N\cap \cO N$  by item (2),  the intersection of
  $\mathcal A_\eps$ with $\big(\newODDom N\big)^2$ is non-empty. 
  
  It follows that
  we can find $\zz\in \mathsf B_N$ such
  that
  $ X_0:= \IIN(\xx+\eps \zz)$ and $ X_1:=\IIN(\xx-\eps\zz)$ belong to $\newODDom N$. In particular, noting that $X= (X_0 + X_1)/2$ and denoting by $\rrho$ the coupling $\rrho:= \iota^2_{X_0,X_1}$, we see that $\rrho$ is collisionless (cf. Definition \ref{def:collisionless}) with $\sfx^0_\sharp \rrho = \iota_{X_0}$, $\sfx^1_\sharp \rrho =\iota_{X_1} \in \core_N$, and $\sfx^{1/2}_\sharp \rrho =\iota_X$. Since, by assumption (2*), $\core_N$ is convex along collisionless couplings, we deduce that 
  $\iota_X \in \core_N$, which gives $X \in \newODDom N$. 

  Now, we prove the second part of item (3). Let $X \in \newODDom N$ and $Y \in \overline{\newDDom N}$. Since $\newODDom N\subset \newDDom N$ by construction
    and $\newDDom N$ coincides with the interior of the convex set
    $\overline{\newDDom N}$ by (2),
    we deduce that all the points 
    $X_t$ belong to $\newDDom N$ for $t\in [0,1).$ 

    Since for $t$ in a neighborhood of $0$ we have that $X_t\in \newODDom N\subset \cO N$, 
    we deduce that $X_t\in \cO N$ with possible finite exceptions
    (observe that if two lines $t\mapsto (1-t) x_i+ty_i$, $i=1,2$, 
    in $\cH$ coincide at two distinct values of $t$ then they coincide everywhere).
    Therefore there exists $\eps>0$ such that $X_t\in \cO N$
    for every $t\in (1-\eps,1).$ Since $\newDDom N\cap \cO N=\newODDom N$ as just proved, we
    deduce that $X_t\in \newODDom N$ for every $t\in (1-\eps,1).$

  \smallskip
  (4)  The set $\newDDom N$ is convex by construction and invariant w.r.t.~the action of $\symg{I_N}$ by (1). These properties are clearly preserved by closure, so that we can apply Lemma \ref{le:useful} to both $\newDDom N$ and its closure $\overline{\newDDom N}$ to get
  \[ \Pi_M(\newDDom N) =\newDDom N\cap \cH_M, \quad \Pi_M(\overline{\newDDom N}) = \overline{\newDDom N}\cap \cH_M. \]
  Moreover, assumption (3*) gives that $\overline{\newODDom N}\cap \cH_M=
  \overline{\newODDom M}$; using this and the density of $\newODDom N$ in $\newDDom N$ (resp.~the density of $\newODDom M$ in $\newDDom M$) coming from (2), we get
\begin{equation}\label{eq:half4}
 \overline{\newDDom N}\cap \cH_M=
  \overline{\newODDom N}\cap \cH_M=
  \overline{\newODDom M}=\overline{\newDDom M}.   
\end{equation}
  Applying \eqref{eq:56bis} to $\overline{\newDDom N}$, we obtain
 that 
 \begin{equation}\label{eq:relint}
 \cH_M \cap \operatorname{rel-int}(\overline{\newDDom N}; \cH_N) = \operatorname{rel-int}(\overline{\newDDom N} \cap \cH_M; \cH_M).    
 \end{equation} 
By (2), we have $\operatorname{rel-int}(\overline{\newDDom N}; \cH_N)=\newDDom N$ and we have just shown above in \eqref{eq:half4} that $\overline{\newDDom N} \cap \cH_M=\overline{\newDDom M}$. Therefore \eqref{eq:relint} can be rewritten as
\[
\cH_M \cap \newDDom N = \operatorname{rel-int}(\overline{\newDDom M}; \cH_M).
\]
Again by (2), we have $\operatorname{rel-int}(\overline{\newDDom M}; \cH_M)=\newDDom M$, so that the above equality reads $\cH_M \cap \newDDom N= \newDDom M$. 

  \smallskip
  (5)  The only non-trivial facts to be proven are the inclusion $\overline{\newDDom \infty} \subset \overline{\newODDom \infty}$ and the convexity of $\overline{\newODDom \infty}$. To show the inclusion, we observe that 
  \[ \bigcup_{N \in \cN}\newDDom N \subset \bigcup_{N \in \cN}\overline{\newDDom N} = \bigcup_{N \in \cN}\overline{\newODDom N} \subset \overline{\bigcup_{N \in \cN}\newODDom N} =  \overline{\newODDom \infty}, \]
  where the first equality follows from (2). In particular, we deduce that 
  $\overline{\cup \overline{\newODDom N}} = \overline{\newODDom
    \infty}$. Hence,  to prove that $\overline{\newODDom \infty}$ is convex,
  it is enough to show that $\cup \overline{\newODDom N}$ is
  convex. If $X,Y \in \cup \overline{\newODDom N}$  and $t \in [0,1]$,  we can find $M,N
  \in \cN$ such that $X \in \overline{\newODDom N}$ and $Y \in
  \overline{\newODDom M}$, so that by (4), both $X$ and $Y$
  belong to $\overline{\newODDom {MN}}$. Since $\overline{\newODDom {MN}}=\overline{\newDDom {MN}}$ by (2) and $\newDDom {MN}$ is convex by construction, also $\overline{\newODDom {MN}}$ is convex, so that $(1-t)X+tY \in \overline{\newODDom {MN}} \subset \cup \overline{\newODDom N}$.

  \smallskip (6)
  The first property follows by
  the identity $\newDDom N= \newDDom L\cap \cH_N=
  \Pi_N(\newDDom L)$ for any $L \in \cN$ such that $N \mid L$, coming from (4),  and the fact that 
  $\newDDom \infty=\cup\Big\{\newDDom L: L\in \cN,\ N\mid L\Big\}$, since $\cN$ is
  a directed set.

  Setting $\mathcal D':=\cup_{N\in \cN}\overline{\newDDom N}$ and
  starting from the second identity of (4),
  the same argument shows that
  $\overline{\newDDom N}= \mathcal D'\cap \cH_N=
  \Pi_N(\mathcal D')$.  Taking into account the equality $\mathcal E_\infty = \overline{\mathcal D'}$ coming from (5), the conclusion follows if we show that
  \begin{equation}\label{eq:eqproj}
      \Pi_N(\mathcal D') = \Pi_N(\overline{\mathcal D'}), \quad \mathcal D'\cap \cH_N = \overline{\mathcal D'}\cap \cH_N.
  \end{equation}
The equality $\overline{\newDDom N} =\Pi_N(\mathcal D')$ gives that  $\Pi_N(\mathcal D')$ is closed, so that  $\Pi_N(\mathcal D') \subset \Pi_N\left(\overline{\mathcal D'}\right) \subset \overline{\Pi_N(\mathcal D')} = \Pi_N(\mathcal D')$, where the inclusion $\Pi_N\left(\overline {\mathcal D'}\right) \subset \overline{\Pi_N(\mathcal D')}$ is true by continuity of $\Pi_N$. This shows the first identity in \eqref{eq:eqproj}. Finally, since $\overline{\mathcal D '} \cap \cH_N$ is trivially a subset of $\cH_N$, we have
\[
\mathcal{D}' \cap \cH_N \subset \overline{\mathcal D '} \cap \cH_N = \Pi_N( \overline{\mathcal D '} \cap \cH_N) \subset \Pi_N(\overline{\mathcal D '}) = \Pi_N (\mathcal D') = \mathcal{D}' \cap \cH_N,
\]
which shows the second identity in \eqref{eq:eqproj}.

  \smallskip (7)
  The fact that $\mathcal E_\infty$ is law invariant follows from Lemma \ref{le:general-invariance} and (6), which shows that $\mathcal E_\infty\cap \cH_M=\overline{\newDDom M}$ which
  is invariant w.r.t.~$\Sym{I_M}$ by (1).
\end{proof}
As an immediate consequence of Lemma \ref{le:trivial} we have the following result.
\begin{corollary}[Cores are totally convex]
\label{cor:serviva-citarlo} If $\core$ is as in Lemma \ref{le:trivial}, then $\overline{\core}$ is totally convex. 
\end{corollary}

\begin{proof}
    Let $\mu,\nu\in\overline{\core}$ and $\gamma\in\Gamma(\mu,\nu)$. Consider $X,Y\in\cH$ such that $\gamma=\iiota{X,Y}$; in particular, $\iota_X=\mu$ and $\iota_Y=\nu$. Hence, there exists $(\mu_n)_{n\in\N},(\nu_n)_{n\in\N}\subset\core$ such that $W_2(\mu_n,\mu)$ and $W_2(\nu_n,\nu)$ both tend to zero as $n\to+\infty$. By Theorem \ref{thm:gpfinal}, there exist $(X_n)_{n\in\N},(Y_n)_{n\in\N}\subset\cH$ such that
    \[X_n\to X,\quad Y_n\to Y,\quad \iota_{X_n}=\mu_n\quad\text{and}\quad\iota_{Y_n}=\nu_n.\]
    Hence, by definition, $(X_n)_{n\in\N},(Y_n)_{n\in\N}\subset\newODDom \infty$ and thus we have $X,Y\in\mathcal E_\infty=\overline{\newODDom \infty}$.

    By the convexity of $\mathcal E_\infty$ (cf. Lemma \ref{le:trivial}(5)), we have that $X_t:=(1-t)X+tY\in\overline{\newODDom \infty}$, for $t\in[0,1]$. Thus, for any $t\in[0,1]$ there exists $(Z_n)_{n\in\N}\subset\newODDom \infty$ such that $Z_n\to X_t$. In particular,
    \[\iota_{Z_n}\in\core,\quad\text{and}\quad W_2(\iota_{Z_n},\iota_{X_t})\to0,\]
    thus $\iota_{X_t}\in\overline{\core}$. Hence the conclusion, noting that $\iota_{X_t}=\sfx^t_\sharp\gamma$.
\end{proof}

 We can now prove Proposition \ref{le:equivalent} and state and prove Proposition \ref{le:dimcore}, the two main results of this subsection describing equivalent characterization of $\cN$-cores. 

\begin{proof}[Proof of Proposition \ref{le:equivalent}]
We divide the proof in several claims. \\
\noindent\textbf{Claim 1.}
\emph{$(a)$ implies $(b)$, $(c)$ and $(d)$}.
\smallskip

   The fact that $(a)$ implies $(c)$ and $(d)$ follows by setting $\tcore:=\iota(\newDDom \infty)$ defined in \eqref{eq:138} and $\ecore:= \overline{\core}$, as a consequence of Lemma \ref{le:trivial} and Corollary \ref{cor:serviva-citarlo}. We prove that $(a)$ implies $(b)$: by Corollary \ref{cor:serviva-citarlo}, we have that $\overline{\core}$ is totally convex. Notice that the sets $\core_N$ are nonempty for every $N \in \cN$ thanks to (3*) and the fact that $\core$ is nonempty. Finally, by Lemma \ref{le:trivial}, we have that the relative interior in $\prob_{f,N}(\X)$ of $\overline{\core} \cap \prob_{\#N}(\X)$ is given by $\tcore_N \cap \prob_{\#N}(\X)=\core_N$ (cf. Lemma \ref{le:trivial}(3)).
 \smallskip

  \noindent\textbf{Claim 2.}
\emph{$(c)$ implies $(a)$}.
\smallskip

  If $\tcore$ is a subset of $\prob_{f,\cN}(\X)$ satisfying conditions
  $(1'),(2')$ and $\core= \tcore \cap \prob_{\# \cN}$, we see that $\core_N = \tcore_N \cap \prob_{\# N}(\X)$ for every $N \in \cN$. Clearly $\core_N$ is relatively open and convex along collisionless couplings in $\prob_{f,N}(\X)$. Also, since $\prob_{\#N}(\X)$ is obviously dense in $\prob_{f,N}(\X)$ and $\tcore_N$ is open, we see that $\core_N$ is dense $\tcore_N$ i.e.~$\overline{\core_N} = \overline{\tcore_N}$. It is also clear that $\overline{\core_N}$ is convex along couplings in $\prob_{f,N}(\X \times \X)$. Finally 
  $\overline {\tcore_N}\cap \prob_{f,M}(\X)=
  \overline{\tcore_M}$ thanks to the convexity of $\tcore_N$ and
  $\tcore_M$, as an application of
  \eqref{eq:56bis} to their Lagrangian representations.
  \smallskip

  \noindent\textbf{Claim 3.}
\emph{$(d)$ implies $(c)$}.
\smallskip

Let $\ecore$ be a totally convex and closed subset of $\prob_2(\X)$ satisfying conditions $(1''),(2'')$ and $\core= \cup_{N \in \cN} \mathring{\ecore}_N \cap \prob_{\#N}(\X)$. We define $\tcore_N$ and $\tcore$ as in \eqref{eq:50}. The only thing to check
  is that 
  \begin{equation}\label{eq:rott1}
  \tcore\cap \prob_{f,N}(\X)=\tcore_N.
  \end{equation}
  Denote by $\mathcal E_\infty$ the Lagrangian parametrization of $\ecore$ (hence, law invariant) and denote by $\mathcal E_N:=\mathcal E_\infty\cap \cH_N$, which is closed and convex. The relative interior $\mathring {\mathcal E}_N$ of $\mathcal E_N$ in $\cH_N$ provides a Lagrangian parametrization of $\mathring\ecore_N=\tcore_N$. Hence, proving \eqref{eq:rott1} is equivalent to prove that $\mathcal D' \cap \cH_N =\mathring {\mathcal E}_N$, where $\mathcal D':=\bigcup_{N\in\cN}\mathring{\mathcal E}_N$.
  Using \eqref{eq:56bis}, if $M\mid N$ we get $\mathring{\mathcal E}_N\cap\cH_M=\mathring{\mathcal E}_M$, also observing that $\mathcal E_N$ is invariant by the action of $\Sym{I_N}$, as a consequence of the law invariance of $\mathcal E_\infty$. Therefore we deduce that
$\mathcal D'\cap \cH_M=\mathring{\mathcal E}_M$.
\smallskip

\noindent\textbf{Claim 4.}
\emph{$(b)$ implies $(d)$}.
\smallskip

 It is clear that setting $\ecore:=\overline{\core}$ we have that $\ecore$ it totally convex and closed. Moreover, since $\intecore_N$ contains the relative interior in $\prob_{f,N}(\X)$ of $E \cap \prob_{\#N}(\X)$ (coinciding with $\core_N$), $\intecore_N$ is not empty. Since the intersection of $\intecore_N$ with $\prob_{\#N}(\X)$ is given by $\core_N$, we immediately see that $\cup_N (\intecore_N \cap \prob_{\#N}(\X))= \core$. Finally 
\[
    \overline{\ecore \cap \prob_{f,\cN}(\X)} = \overline{\cup_N \overline{ \ecore \cap \prob_{\#N}(\X)}}
    = \overline{\cup_N \overline{\core_N}}\\
    = \overline{\core},  
\]
where we have used again that the intersection of $\intecore_N$ with $\prob_{\#N}(\X)$ is given by $\core_N$ and that the closure of $\ecore \cap \prob_{\#N}(\X)$ coincides with the closure of its (relative) interior.
\end{proof}

\begin{proposition}\label{le:dimcore}
 Let $\core\subset \prob_{\#\cN}(\X)$; if $\dim(\X) \ge 2$, then condition \rm{(4*)} in Lemma \ref{le:equivalent} follows by \rm{(1*)}-\rm{(3*)}.
  \end{proposition}

\begin{proof}
Assume that \rm{(1*)}-\rm{(3*)} hold. We need to prove that $\overline{\core_N}$ is convex along couplings in $\prob_{f,N}(\X \times \X)$ for every $N \in \cN$. This is equivalent to prove the convexity of $\overline{\newODDom N}$ so that it is sufficient to show that, for  every $X_0,X_1\in \newODDom N$ and $t \in [0,1]$, their linear interpolation $X_t:=(1-t)X_0+t X_1$ belongs to $\overline{\newODDom N}$. By Proposition \ref{prop:perturbation}, we can find small perturbations $X_1(s)$ of $X_1$, $s\in [0,1]$, such that $X_1(s)\in \newODDom N$, $X_1(s)\to X_1$ as $s\downarrow0$,  and the perturbed interpolation $X_{s,t}:=(1-t)X_0+tX_1(s)$ belongs to $\newODDom N$ for every $t\in [0,1]$ and $s>0$. It follows that the coupling $\mmu_s=\iotaT_{X_0,X_1(s)}$ belongs to $\prob_{\#N}(\X\times \X)$ and it is collisionless for every $s>0$ and therefore $\mu_{s,t}=\sfx^t_\sharp \mmu_s$ belongs to $\core_N$ for every $t$.  Since $\mu_{s,t}=\iota_{X_{s,t}}$ we have $X_{s,t}\in \newODDom N$. Passing to the limit as $s\downarrow 0$ we conclude that $X_t\in \overline{\newODDom N}$. 
\end{proof}

\subsection{Lagrangian representations of discrete \MPVF{s}: construction of \texorpdfstring{$\hat\frF_N$}{F}}\label{sec:fn}
Let us now study in more detail the Lagrangian representations of a
\MPVF $\frF\subset\prob_2(\TX)$ defined on a $\cN$-core.
If $\Phi\in \frF$
we can consider the (non-empty) set of all the maps $(X,V)\in \cH^2$ such that
$\iotaT_{X,V}=\Phi$.
A particular case is obtained when the first marginal $\mu=\sfx_\sharp\Phi$ of $\Phi$ belongs to
$\prob_{f,N}(\X)$. In this case, $X$ has the form $X=\IIN(\xx)\in \cH_N$, so that $\mu=\iota_X=
\frac 1N\sum_{k\in I_N}\delta_{\xx(k)} $, and we can construct $V$ from the representation of $\Phi$ given by
\begin{equation*}
  \Phi=\frac 1N\sum_{k\in I_N} \Phi_k,\quad
  \sfx_\sharp \Phi_{k}=\delta_{\xx(k)},
\end{equation*}
for a family $\{\Phi_k\}_{k\in I_N}\subset \prob(\TX)$, by setting $V(\omega):=V_k(\omega)$ if $\omega\in \Omega_{N,k}$, where $V_k\in
L^2(\Omega_{N,k},\P|_{\Omega_{N,k}};\X)$ are maps such that
$(V_k)_\sharp \P|_{\Omega_{N,k}}=\frac 1N \,\sfv_\sharp\Phi_k$.
\medskip

Recall that (cf. Definition \ref{def:pairings}), given $\ttheta\in \prob_2(\X \times \X)$ and  $\Phi \in \relcP2{\sfx^0_\sharp \ttheta}{\TX}$,
\[\directionalm{\Phi}{\ttheta}0:=
\min \left \{ \int_{\TX\times\X}
      \scalprod{x_0-x_1}{v_0} \de \ssigma(x_0,v_0,x_1)
      \Big| \begin{array}{l}\ssigma \in \prob_2(\TX\times\X),\\(\sfx^0,\sfx^1)_\sharp\ssigma=\ttheta,\,(\sfx^0,\sfv^0)_\sharp\ssigma=\Phi\end{array}\right \}.\]

Thus, in the general case when $\Phi\in\prob_2(\TX)$, it is easy to check that if $\iotaT_{X,V}=\Phi$ and $Y\in \cH$ then
\begin{equation}
  \label{eq:167}
  \directionalm{\Phi}{\iotaT_{X,Y}}0\le\la V,X-Y\ra_{\cH}.
\end{equation}
A particular important case occurs when $X\in \cO N$ and $Y\in \cH_N$: in this case $\Phi_k$ is uniquely determined by the disintegration of $\Phi$ w.r.t.~$\mu$, and $V|_{\Omega_{N,k}}$ coincides with $V_k$, where $V_k$ is as above.  Thus, 
\begin{equation}\label{eq:projbar}
    \Pi_N(V)(\omega)=\bry{\Phi}(\xx(k))\quad \text{if }\omega\in \Omega_{N,k},
\end{equation}
    where $\bry{\Phi}$ is the barycenter of $\Phi$ as in Definition \ref{def:wassmom} and $\Pi_N(\cdot)$ is defined in \eqref{eq:projrep}. Moreover,
since $X-Y\in\cH_N$ and $\Pi_N(V)$ is the orthogonal projection of $V$ onto $\cH_N$, we have
\begin{equation*}
\la V,X-Y\ra_{\cH} =\la \Pi_N (V),X-Y\ra_{\cH}.
\end{equation*}

It is easy to check that, in this case, 
\begin{equation}
  \label{eq:111}
  \directionalm{\Phi}{\iotaT_{X,Y}}0=\la \Pi_N (V),X-Y\ra_{\cH}=
  \la V,X-Y\ra_{\cH}
  \quad\text{if }\iotaT_{X,V}=\Phi,\ X\in
  \cO N,\ Y\in \cH_N,
\end{equation}
 where the first equality follows by \eqref{eq:29} since the map $\sfx^0$ is
  $\iotaT_{X,Y}$-essentially injective.

\medskip
We define now one of the main objects of study of this subsection: the operator $\fF_N$ whose maximal extension $\hat \fF_N$ (defined in Proposition \ref{prop:ext1} below) is the Lagrangian counterpart of the operator $\hat \frF_N$ in the main Theorem \ref{thm:main-discrete}: for every $N \in \cN$, we set
\begin{equation}
  \label{eq:105}
  \fF:=
  \Big\{(X,V)\in \newODDom\infty\times \cH:
  \iotaT_{X,V} \in \frF\Big\},\quad
  \pfrF N:=\Big\{\left(X,\Pi_N (V)\right):X\in \newODDom N,\
  (X,V) \in \fF\Big\}.
\end{equation} 
We stress that they are essential tools in the proofs of the main Theorems \ref{thm:main-discrete}, \ref{thm:perfavorelultimo}, \ref{thm:total-case}, and \ref{thm:demi-case} as most of the properties  of $\frF_N, \hat \frF_\infty$, and $\hat \frF$ will be derived by the corresponding properties of their Lagrangian representations $\fF_N, \hat \fF_\infty$, and $\hat \fF$, which we obtain using the Hilbertian structure of $\cH$.

In the following result, we study some immediate properties of $\fF_N$.

\begin{proposition}\label{prop:perunmotivo} Assume the same hypotheses of Theorem \ref{thm:main-discrete}. Then $\fF_N \subset \cH_N \times \cH_N$ as in \eqref{eq:105} is $\lambda$-dissipative, has open domain $\dom(\fF_N)=\newODDom N$, and it is invariant by permutations: if $(X,V) \in \fF_N$ and $\sigma \in \symg{I_N}$, then $(\sigma X, \sigma V) \in \fF_N$.
\end{proposition}
\begin{proof}
 We take $(X,V), (Y,W) \in \fF_N$; by definition, we can find $V_0, W_0 \in \cH$ such that, defined $\Phi:=\iota^2_{X,V_0}$ and $\Psi:=\iota^2_{Y,W_0}$, we have that $\Phi,\Psi\in \frF$ and $V=\Pi_N(V_0), W=\Pi_N(W_0)$. Since by definition $X,Y \in \mathcal{C}_N \subset \mathcal{O}_N$, we can use \eqref{eq:111} and Theorem \ref{thm:all}(1), to obtain
\begin{align*}
    \la V-W, X-Y \ra_{\cH} = \directionalm{\Phi}{\iotaT_{X,Y}}0 -\directionalp{\Psi}{\iotaT_{X,Y}}1.
\end{align*}
In case (ii) of Theorem \ref{thm:main-discrete}, the total $\lambda$-dissipativity of $\frF$ immediately gives that the above quantity is bounded above by $\lambda |X-Y|^2_\cH$. In case (i) of Theorem \ref{thm:main-discrete}, we can apply Theorem \ref{le:crucial2}(1) to get the same bound: indeed, $\dom_{\#N}(\frF)=\core_N$ is convex along collisionless couplings by Proposition \ref{le:equivalent}$(2*)$, $\core_N$ is open in $\prob_{\#N}(\X)$ by Proposition \ref{le:equivalent}$(1*)$ so that $\iota_X, \iota_Y$ are indeed in the interior of $\dom_{\#N}(\frF)$, and $\iotaT_{X,Y} \in \Gamma_{\#N}(\iota_X, \iota_Y)$ by construction.
Overall, we obtained 
\begin{equation}
  \label{eq:107}
  (X,V),\ (Y,W)\in \pfrF N\quad\Rightarrow\quad
  \la V-W,X-Y\ra_{\cH}\le \lambda |X-Y|_{\cH}^2,
\end{equation}
so that $\pfrF N$ is $\lambda$-dissipative. In any of the cases (i) and (ii) of Theorem \ref{thm:main-discrete}, if $(X,V)\in \pfrF N$ and $\sigma \in \symg{I_N}$, then
there exists $W\in\cH$ such that $\iotaT_{X,W}\in\frF$ and $V=\Pi_N
(W)$. By Lemma \ref{cor:isomor}, we can write $\sigma X=X\circ g\in\newODDom N$
for some $g\in G[\sigma]$ 
and $\iotaT_{X\circ g,W\circ g}\in\frF$. To conclude, it suffices to notice that $\Pi_N(W\circ g)=\sigma V$. 
\end{proof}

We can now define the maximal extension of $\fF_N$, the operator $\hat \fF_N$. As we will prove in Theorem \ref{prop:final}, the Eulerian image of
  $\hat\fF_N$
  is the \MPVF $\hat\frF_N$
  defined in Theorem \ref{thm:main-discrete}.

\begin{proposition}
  \label{prop:ext1}
  Under the same assumptions of Theorem \ref{thm:main-discrete},
  for every $N\in \cN$ the $\lambda$-dissipative operator $\pfrF {N}$ 
  admits a unique maximal $\lambda$-dissipative extension $\maxim \fF N$
  in $\cH_N\times \cH_  N$ 
  with $\DDom N\subset\dom(\maxim\fF N)\subset\overline {\DDom N}$.
  The operator $\maxim\fF N$ can be equivalently characterized by
  \begin{equation}
    \label{eq:112}
    (X,V)\in \maxim\fF N
    \quad
    \Leftrightarrow
    \quad
    X\in \overline {\DDom N},\ V\in \cH_N,\ 
    \la V-W,X-Y\ra_{\cH}\le \lambda |X-Y|_{\cH}^2
    \quad \forall\,
    (Y,W)\in \pfrF N,
  \end{equation}
  and, whenever $X\in \DDom N$,
  $\maxim\fF N (X)=
  \clconv{\bar \fF_N (X)}$,
  where   
  \begin{equation}
    \label{eq:114}
    \bar {\fF}_N (X):=\Big\{V\in \cH_N:
      \exists\, (X_n,V_n)_{n\in\N}\subset \pfrF N:
      X_n\to X,\ V_n\weakto V\Big\}.
    \end{equation}
  $\maxim\fF N$ is invariant with respect to permutations, i.e.
  \begin{equation}
    \label{eq:77}
    (X,V)\in \maxim\fF N,\ \sigma\in\symg{I_N}
    \quad\Rightarrow\quad
    (\sigma X,\sigma V)\in \maxim\fF N
  \end{equation}
  and for every $X,Y\in \DDom N$, 
  we have
  \begin{equation}
    \label{eq:124}
    V\in \maxim\fF N (X),\
    \Psi\in \frF[\iota_Y]\quad
    \Rightarrow\quad
    \la V,X-Y\ra_{\cH}+\directionalm\Psi{\iotaT_{Y,X}}0\le \lambda |X-Y|_{\cH}^2.
  \end{equation}
  Finally, if $M\mid N$, $X\in \overline{\DDom M}$, and
  $(X,V)\in \maxim\fF N$ then
  $\Pi_M (V)\in \maxim \fF M (X)$.
  Conversely, if $X\in \DDom M$ and
  $W\in \maxim \fF M (X)$ then 
  there exists $V\in \cH_N$ such that 
  \begin{equation}
    \label{eq:113}
    (X,V)\in \maxim\fF N,\quad W=\Pi_M (V).
  \end{equation}
\end{proposition}

\begin{proof}
  \eqref{eq:112} and \eqref{eq:114} follow from the fact that
  $\DDom N$ is convex and open and
  the domain of $\pfrF N$ is dense in $\DDom N$,
  see Lemma \ref{le:trivial}
  and Theorem \ref{theo:chefatica} in the Appendix.

  Using \eqref{eq:112} it is immediate to check that $\maxim\fF N$
  satisfies \eqref{eq:77}, since for every $(X,V)\in \maxim\fF N$
  and $(Y,W)\in \pfrF N$ 
  \begin{align*}
    \la \sigma V-W,\sigma X-Y\ra_{\cH}
    &=
      \la V-\sigma^{-1}W,X-\sigma^{-1}Y\ra_{\cH}
    \le \lambda |X-\sigma^{-1}Y|_{\cH}^2 = \lambda |\sigma X - Y|_{\cH}^2,
  \end{align*}  
  since $\pfrF {N}$ and the scalar product in $\cH_N$ are invariant by
  the action of permutations in $\symg{I_N}$.
  \smallskip
  
   We now take $\Psi\in\frF[\iota_Y]$, $Y\in \DDom N$, and prove \eqref{eq:124} first in case $(X,V)\in \pfrF N$.
Then  \eqref{eq:124} follows immediately 
  since there exists $W\in\cH$ such that $\Phi:=\iotaT_{X,W}\in \frF$, $V=\Pi_N (W)$, and
  \eqref{eq:111} yields
  $\la V,X-Y\ra_{\cH}=\directionalm{\Phi}{\iotaT_{X,Y}}0$
  so that
  \begin{equation}\label{eq:mix1}
    \la V,X-Y\ra_{\cH}+
    \directionalm{\Psi}{\iotaT_{Y,X}}0=
    \directionalm{\Phi}{\iotaT_{X,Y}}0+\directionalm{\Psi}{\iotaT_{Y,X}}0 \le  \lambda |X-Y|_{\cH}^2.
  \end{equation}
   Notice that in case (ii) of Theorem \ref{thm:main-discrete}, the last inequality is obvious; while, in case (i) of Theorem \ref{thm:main-discrete}, the last inequality in \eqref{eq:mix1} follows by Theorem \ref{le:crucial2}(2) and recalling Theorem \ref{thm:all}(1): indeed $\dom_{\#N}(\frF)= \core_N$ which is convex along collisionless couplings by Proposition \ref{le:equivalent}$(3^*)$, open in $\prob_{\#N}(\X)$ by Proposition \ref{le:equivalent}$(1^*)$, $\iota_X \in \core_N$, $\iota_Y \in \dom_{f, N}(\frF)$, $\iotaT_{X,Y} \in \Gamma_{\#N}(\iota_X, \iota_Y)$ and condition (2) in Theorem \ref{le:crucial2} is satisfied thanks to Lemma \ref{le:trivial}(3).  

  If  $X\in \DDom N$
  and $V\in \bar \fF_N (X)$ according to \eqref{eq:114},
  then there exist
  $(X_n,V_n)_{n\in\N}\subset \pfrF N$, $X_n\in \newODDom N$,
  such that $X_n\to X$
  and $V_n\weakto V$. We can pass to the limit in \eqref{eq:mix1}
  written for $(X_n,V_n)$ and using Theorem \ref{thm:all}(5)
 we obtain that $(X,V)$ satisfies
  \eqref{eq:mix1} as well.
  Finally, since \eqref{eq:mix1} holds for every 
  $V\in \bar \fF_N (X)$, it also holds for every
  $V\in \clconv{\bar \fF_N (X)}$.  This completes the proof of \eqref{eq:124}. 
  \smallskip 

    Let us now suppose that $M\mid N$,
  $(X,V)\in \maxim\fF N$ and
  $X\in \DDom M$. We want to show that $W:=\Pi_M (V)$ belongs to
  $\maxim\fF M (X)$ by using \eqref{eq:112}.
  If $(Y,U)\in \pfrF{M}$ with $Y\in \newODDom M$,
  we have $U=\Pi_M (U')$ with $\iotaT_{Y,U'}=:\Phi\in \frF$, so that
  \eqref{eq:124} yields
  \begin{equation}
    \label{eq:119}
    \la V,X-Y\ra_{\cH}+
    \directionalm{\Phi}{\iotaT_{Y,X}}0\le \lambda |X-Y|_{\cH}^2.
  \end{equation}
  Since $Y\in \cO M$ and $X\in \cH_M$, we have
  $\directionalm{\Phi}{\iotaT_{Y,X}}0=\la U,Y-X\ra_{\cH}$ by \eqref{eq:111};
  since $X-Y\in \cH_M$, we also have
  $\la V,X-Y\ra_{\cH}=\la \Pi_M (V),X-Y\ra_{\cH}$ and we get
  \begin{equation}
    \label{eq:119bis}
    \la W,X-Y\ra_{\cH}+
    \la U,Y-X\ra_{\cH}=
    \la V,X-Y\ra_{\cH}+
    \directionalm{\Phi}{\iotaT_{Y,X}}0\le \lambda |X-Y|_{\cH}^2.
  \end{equation}
  Hence, by \eqref{eq:112} $(X,W)\in\maxim\fF M$.
  In particular, the above property shows that
  if $\Gg:\DDom N\to \cH_N$  is an arbitrary single-valued selection of $\maxim \fF N$,
 the restriction $\bm{\mathcal{G}}:=\left(\Pi_M\circ\Gg\right)|_{\DDom M}$
 is a selection of $\maxim\fF M$.
 We fix such a selection. 
To conclude we need to prove that the property holds also if $X\in\overline{\DDom M}$. Recall that by Lemma \ref{le:trivial}(3), $\overline{\dom(\pfrF {M})}=\overline{\newODDom M}=\overline{\DDom M}$. Then if $X\in\overline{\DDom M}$, by Corollary \ref{cor:maximal} we have that
 $W$ belongs to $\maxim \fF M (X)$ if and only if
 \begin{equation*}
   \la W-\bm{\mathcal{G}}(Y),X-Y\ra_{\cH}\le  \lambda |X-Y|_{\cH}^2
   \quad\text{for every }Y\in \DDom M,
 \end{equation*}
i.e., if and only if
 \begin{equation}
   \label{eq:125}
   \la W-\Gg (Y),X-Y\ra_{\cH}\le  \lambda |X-Y|_{\cH}^2
   \quad\text{for every }Y\in \DDom M.
 \end{equation}
If $V\in \maxim \fF N (X)$, then using Corollary \ref{cor:maximal} we have
\[\la V-\Gg (Y),X-Y\ra_{\cH}\le \lambda |X-Y|_{\cH}^2 
   \quad\text{for every }Y\in \DDom N\supset\DDom M,\]
hence \eqref{eq:125} holds and we get $\Pi_M (V)\in \maxim\fF M (X)$. 
  
  Let us now show the converse implication.
 If $X\in \DDom M$ and
 $W\in \maxim \fF M (X)$, we need to prove that $W\in\Pi_M\left(\maxim\fF N (X)\right)$. 
Since $\overline{\dom(\Gg)}=\overline{\DDom N}$, by Corollary \ref{cor:maximal} and Theorem \ref{theo:chefatica} applied to $\Gg$, we get $\Pi_M\left(\maxim\fF N (X)\right)=\Pi_M\left(\hat\Gg (X)\right)=\Pi_M\left(\clconv{\bar\Gg (X)}\right)$, where
\[\bar \Gg (X):=\Big\{Z\in \cH_N:
 \exists \,(X_n)_{n\in\N}\subset \DDom N:
 X_n\to X,\ \Gg (X_n)\weakto Z\Big\}.\]
Similarly, by Corollary \ref{cor:maximal} and Theorem \ref{theo:chefatica} we get
\begin{align*}
\maxim\fF M (X)&=\hat{\bm{\mathcal{G}}} (X)=\clconv{\overline{\bm{\mathcal{G}}}(X)}=\clconv{\left\{Z\in \cH_M:
 \exists \,(X_n)_{n\in\N}\subset \DDom M:
 X_n\to X,\ \bm{\mathcal{G}}(X_n)\weakto Z\right\}}\\
&\subset \Pi_M\left(\clconv{\bar\Gg (X)}\right),
\end{align*}
where the proof of the last equality can be pursued as follows. We first observe that 
\begin{align*}
&\left\{Z\in \cH_M:
 \exists \,(X_n)_{n\in\N}\subset \DDom M:
 X_n\to X,\ \bm{\mathcal{G}}(X_n)\weakto Z\right\}\\
&\subset\Pi_M\left(\left\{W\in \cH_N:
 \exists \,(X_n)_{n\in\N}\subset \DDom N:
 X_n\to X,\ \Gg (X_n)\weakto W\right\}\right)=\Pi_M\left(\bar\Gg (X)\right),
\end{align*}
by using the local boundedness of $\Gg$ as a selection of $\hat\Gg$ (see Theorem \ref{thm:brezis2}(3)) and the fact that $\Pi_M$ is a linear and continuous operator. Then we notice that 
\[\clconv{\Pi_M\left(\bar\Gg (X)\right)}=\overline{\Pi_M\left(\conv{\bar\Gg (X)}\right)}=\Pi_M\left(\clconv{\bar\Gg (X)}\right),\]
where the first equality follows by linearity of $\Pi_M$ and, for the second, we exploit again the local boundedness of $\bar\Gg$ as a selection of $\hat\Gg$ and the linearity and continuity of $\Pi_M$. Hence the conclusion.
\end{proof}

It is remarkable that, under the general assumptions of Theorem \ref{thm:main-discrete},
$\maxim \fF N$ can also be characterized by those $(X,V) \in \overline{\DDom N} \times \cH_N$ satisfying inequality \eqref{eq:112}
\emph{restricted to those $Y\in \newODDom N$ for which $\iotaT_{X,Y}$ 
is the unique optimal coupling between $\iota_X$ and $\iota_Y$.} This is stated in the next Proposition \ref{prop:allultimomomento} and it is directly used in the proof of Theorem \ref{thm:main-discrete}.
\begin{proposition}
    \label{prop:allultimomomento}
    We assume the same hypothesis of Theorem \ref{thm:main-discrete}.
    Let $X\in \overline{\newDDom N}$ and $V\in \cH_N$ 
    be satisfying
    \begin{equation}
        \label{eq:112bis}
    \begin{gathered}
        \la V-W,X-Y\ra_{\cH}\le \lambda |X-Y|_{\cH}^2\\
    \text{for every }
    (Y,W)\in \pfrF N
    \text{ s.t. }
    \iotaT_{X,Y}
    \text{ is the unique element of } \Gamma_o(\iota_X,\iota_Y).
    \end{gathered}
    \end{equation}
    Then $(X,V)\in \maxim\fF N.$
\end{proposition}
\begin{proof}
    Let us consider an arbitrary element $(Y,W)\in \fF_N$;  by Lemma \ref{le:trivial}(3), there exists $\eps>0$ such that $Y_t:=(1-t)X+tY\in \newODDom N$  for every $t\in (0,\eps).$

    By Theorem \ref{thm:easy-but-not-obvious}, we can thus find
    $\tau\in (0,\eps)$ such that 
    $Y_\tau\in \newODDom N$ and $\iotaT_{X,Y_\tau}$ is the unique
    optimal coupling between $\iota_X$ and $\iota_{Y_\tau}$. Let $W_\tau\in \fF_N(Y_\tau)$, then  by \eqref{eq:112bis} we have
    \begin{equation}\label{eq:III3.1}
        \la V-W_\tau, X-Y_\tau \ra_{\cH}\le\lambda |X-Y_\tau|_{\cH}^2.
    \end{equation}
Moreover, since $(Y,W),(Y_\tau,W_\tau)\in\fF_N$, we can apply the $\lambda$-dissipativity of $\fF_N$ (cf. Proposition \ref{prop:perunmotivo}) and get
\begin{equation}\label{eq:III3.2}
    \la W_\tau-W,Y_\tau-Y \ra_{\cH}\le\lambda|Y_\tau-Y|_{\cH}^2.
\end{equation}
Combining \eqref{eq:III3.1} and \eqref{eq:III3.2}, we finally get
    
\begin{align*}
    \la V-W, X-Y \ra_{\cH} &= \la W_\tau-W, X-Y \ra_{\cH} + \la V-W_\tau,X-Y \ra_{\cH} \\
    &= \frac{1}{1-\tau} \la W_\tau-W,Y_\tau-Y \ra_{\cH} + \frac{1}{\tau} \la V-W_\tau, X-Y_\tau \ra_{\cH} \\
    &\le \lambda |X-Y|_{\cH}^2.
\end{align*}
Since $(Y,W)$ is an arbitrary element of $\fF_N$, we deduce that $(X,V)\in \hat\fF_N$ by \eqref{eq:112}.
\end{proof}

In the next two corollaries, we work separately under the additional assumptions of Theorems \ref{thm:total-case}  and \ref{thm:demi-case} to provide additional properties of $\hat \fF_N$ which will be used in the proofs of the aforementioned main theorems.
We work first under the assumptions of Theorem
\ref{thm:total-case}, i.e. assuming that $\frF$ is a 
      totally $\lambda$-dissipative $\MPVF$
  whose domain contains a dense $\cN$-core $\core$.
Let us recall that,
by Corollary \ref{cor:bary-works-well},
if $\frF$ is totally $\lambda$-dissipative also
$\tilde\frF:=\frF\cup\bri \frF$
is totally $\lambda$-dissipative.
\begin{corollary}
  \label{cor:dim-total-case1}
  Under the assumptions of Theorem \ref{thm:total-case},
  let $\tilde\mmo$ be
  the Lagrangian representation of $\tilde\frF=\frF\cup\bri \frF$,
  and let $\mmo'$ be any $\lambda$-dissipative extension of
  $\tilde{\mmo}$.
  For every $N \in \cN$, $Y\in \overline{\newDDom N}$,
  $(Y,W)\in {\mmo}'$,
  we have
  $(Y,\Pi_N (W))\in \maxim \fF N$ and, in particular,
  \begin{equation}
    \label{eq:43}
     \langle V-\Pi_N (W), X-Y\rangle_{\cH}\le \lambda|X-Y|_{\cH}^2
    \quad
    \text{for every 
      $(X,V)\in \maxim \fF N$, $Y\in \overline{\newDDom N}$,
      $(Y,W)\in {\mmo'}$},
  \end{equation}
  where $\maxim \fF N$ is constructed as in Proposition \ref{prop:ext1} starting from the restriction of $\frF$ to $\core$.
\end{corollary}
\begin{proof} Observe that, by construction, $\fF$ (constructed
  starting from the restriction of the \MPVF $\frF$ to $\core$)
  and $\fF_N$ are subsets of $\tilde\mmo$ hence of $\mmo'$;
  this implies that $\fF_N$ is dissipative with 
  ${\mmo'}$ in the sense that
  \begin{equation}
    \label{eq:41}
    \langle X-Y,V-W\rangle_{\cH}\le \lambda|X-Y|_{\cH}^2
    \quad
    \text{for every $(X,V)\in \fF_N$, $(Y,W)\in {\mmo'}$}.
  \end{equation}
  Restricting \eqref{eq:41} to $Y\in \overline{\newDDom N}$,
the very definition of $\maxim\fF N$ in \eqref{eq:112} yields 
    $(Y,\Pi_N (W))\in \maxim \fF N$;
     in particular, we get \eqref{eq:43}.
\end{proof}

Let us now show that,  if we work under the assumptions of Theorem \ref{thm:main-discrete}, also requiring that $\frF$ is deterministic, then $\maxim \fF N$ coincides with $\fF$ on 
$\newODDom N$. This occurs in particular under the assumptions of Theorem
\ref{thm:demi-case}, i.e. when $\dim\X\ge2$ and
  $\frF\subset\prob_2(\TX)$  is a
  deterministic
  $\lambda$-dissipative \MPVF whose domain is a
  $\cN$-core $\core$. 
\begin{corollary}
  \label{cor:dim-demi-case1} Under the assumptions of Theorem \ref{thm:main-discrete}, assume also that the \MPVF $\frF$ is deterministic. Then 
  $\maxim \fF N$ is an extension of $\fF_N=\fF$ on $\newODDom N$, for every $N \in \cN$.
  Under the further assumptions 
  that $\frF$ is a single-valued \PVF and demicontinuous on each $\core_N$, then $\fF_N$ 
  coincides with $\hat \fF_N$ on $\newODDom N$.
\end{corollary}
\begin{proof}
  The first statement is an immediate consequence of Proposition \ref{prop:ext1}; the equality $\fF_N=\fF$ on $\newODDom N$ follows from the fact that $\frF$ is a deterministic \MPVF by assumption. Let us now assume that $\frF$ is single-valued
  and its restriction to $\core_N$ is demicontinuous. 
  Let $X$ be an element of $\newODDom N$, $\mu=\iota_X$; $\frF[\mu]$ contains a unique element $\Phi$ which may be represented as $\bri{\Phi}=(\ii_\X, \bry\Phi)_\sharp\mu$ so that there is a unique element $V=\bry\Phi\circ X\in \cH_N$ such that $\iotaT_{X,V}=\Phi$. This shows that $\fF (X)$ is single-valued.
   Recalling the definition of $\bar\fF_N$ in \eqref{eq:114}, if $W \in \bar\fF_N (X)$, we can find a sequence $\left(X_n,\fF (X_n)\right)_{n\in\N}=(X_n, \ff_n \circ X_n)_{n\in\N}$ converging in the strong-weak topology of $\cH\times \cH$ to $(X,W)$, for maps $\ff_n \in L^2(\X, \mu_n; \X)$ with $\mu_n=\iota_{X_n}$. On the other hand, since $\frF$ is demicontinuous and deterministic, we have that $\frF[\iota_{X_n}]=(\ii_{\X}, \ff_n)_\sharp \mu_n \to (\ii_{\X}, \ff)_\sharp \mu=\frF[\iota_X]$ in $\prob_2^{sw}(\TX)$ for a map $\ff \in L^2(\X, \mu; \X)$. If $\psi \in \rmC_b(\X;\X)$, we can test the convergence in $\prob_2^{sw}(\TX)$ against $\zeta(x,y):=\langle \psi(x), y \rangle$ so that
 \[ \langle \psi(X_n), \ff_n \circ X_n \rangle_{\cH} = \int_{\X} \zeta \de (\ii_{\X}, \ff_n)_\sharp \mu_n \to \int_{\X} \zeta \de (\ii_{\X}, \ff)_\sharp \mu = \langle \psi(X), \ff \circ X \rangle_{\cH}. \]
 On the other hand $\psi(X_n) \to \psi(X)$ and $\ff_n \circ X_n \weakto W$ so that we deduce that
 \[ \langle \psi(X), \ff \circ X \rangle_{\cH} = \langle \psi(X), W \rangle_{\cH} \quad \text{ for every } \psi \in \rmC_b(\X; \X).\]
By arbitrariness of $\psi$, we deduce that $W=\ff \circ X = \fF (X)$. We thus deduce that $\bar\fF_N (X)$ coincides with $\fF (X)$ and then it contains a unique element $V$, and therefore by \eqref{eq:114} $\hat\fF_N
  (X)=\clconv{\bar\fF_N (X)}=V$ as well.
\end{proof}

\subsection{Lagrangian representation of the maximal extension}\label{sec:finf}
This section is devoted to the construction of $\hat \fF_\infty$ and $\hat \fF$, the Lagrangian representations of $\hat \frF_\infty$ and $\hat \frF$, as in Theorem \ref{thm:main-discrete}. We start with an important invariance property of the resolvents
of $\maxim\fF N$ with respect to $N$.

\begin{proposition} \label{prop:eulerstep}
  We keep the same assumptions of Theorem
  \ref{thm:main-discrete}.
  For every $X\in \cH_\infty$ and every $0<\tau<1/\lambda^+$ 
  there exists a unique $X_\tau\in\cH_\infty$ such that, for any $N\in \cN$,
  \begin{equation}
    \label{eq:120}
    X\in\cH_N \Rightarrow X_\tau\in \dom(\maxim\fF N)\subset\cH_N\,\text{ and }\,X_\tau-X\in \tau\, \maxim\fF N (X_{\tau}).
  \end{equation}
  Moreover 
  \begin{equation}
    \label{eq:118}
    |X_\tau(\omega')-X_\tau(\omega'')|\le \frac{1}{1-\lambda \tau} |X(\omega')-X(\omega'')|\quad
    \text{for every }\omega',\omega''\in \Omega.
  \end{equation}
\end{proposition}
\begin{proof}
Since $X\in \cH_\infty$, there exists $N\in\cN$ such that $X\in\cH_N$.
  Since $\maxim\fF N$ is maximal $\lambda$-dissipative, recalling Theorem \ref{thm:brezis1}(1),
  there exists a unique solution $X_{\tau,N}\in \dom(\maxim\fF N)$ of
  \begin{displaymath}
     X_{\tau,N}-X\in \tau\, \maxim\fF N(X_{\tau,N}).
  \end{displaymath}
  The invariance of $\maxim\fF N$ by permutations, stated in \eqref{eq:77}, shows
  that
  $(\sigma X)_{\tau,N}=\sigma (X_{\tau,N})$ for every $\sigma\in
  \symg{I_N}$.
  In particular, by $\lambda$-dissipativity of $\maxim \fF N$ we have
  \begin{align*}
    \la \sigma X_{\tau,N}-\sigma X-
    (X_{\tau,N}-X),\sigma X_{\tau,N}-X_{\tau,N}\ra_{\cH}\le \lambda \tau |\sigma X_{\tau,N}-X_{\tau,N}|_{\cH}^2
  \end{align*}
  so that
  \begin{displaymath}
    (1-\lambda \tau)\, |\sigma X_{\tau,N}-X_{\tau,N}|_{\cH}\le |\sigma X-X|_{\cH}\quad
    \text{for every }\sigma\in \symg{I_N}.
  \end{displaymath}
  If $\omega'\in \Omega_{N,i}$, $\omega''\in \Omega_{N,j}$, $i,j\in I_N$, and
  we choose as $\sigma$ the transposition which
  shifts $i$ with $j$, we get
  \begin{displaymath}
    \frac 2N (1-\lambda \tau)^2  |X_{\tau,N}(\omega')-X_{\tau,N}(\omega'')|^2\le
    \frac 2N |X(\omega')-X(\omega'')|^2
  \end{displaymath}
  which yields \eqref{eq:118}.

  Let us now suppose that $X\in \cH_M$ with $M\mid N$.
  Then $X_{\tau,N}$ belongs to $\cH_M$ by \eqref{eq:118}, so that $X_{\tau,N}\in\overline{\DDom N}\cap\cH_M=\overline{\DDom M}$ by Lemma \ref{le:trivial}(4). By Proposition \ref{prop:ext1}, for every $Y\in \DDom M$ and $W\in \maxim \fF M (Y)$
  we can find $V\in \maxim\fF N (Y)$
  such that $W=\Pi_M (V)$, so that by $\lambda$-dissipativity of $\maxim \fF N$ we have
  \begin{equation}
    \label{eq:126}
    \la X_{\tau,N}-X-\tau V,X_{\tau,N}-Y\ra_{\cH}\le \lambda\tau |X_{\tau,N}-Y|_{\cH}^2.
  \end{equation}
  Since $X_{\tau,N}-Y\in \cH_M$, we can replace
  $V$ with $W=\Pi_M (V)$ in \eqref{eq:126}, thus obtaining that
  $X_{\tau,N}-X\in \tau \maxim \fF M (X_{\tau,N})$ by Corollary \ref{cor:maximal},
  i.e.~$X_{\tau,N}=X_{\tau,M}$, by the uniqueness of the resolvent (see also Theorem \ref{thm:brezis1}(1)). 
  If $M,N$ are arbitrary and $X\in \cH_M\cap
\cH_N$, then
  setting $R:=MN$ the previous argument shows that
  $X_{\tau,M}=X_{\tau,R}=X_{\tau,N}$.
\end{proof}
As a corollary, we obtain the corresponding invariance property for the minimal selection.

\newcommand{\hfFcirc}{\hat\fF{\vphantom\fF}^\circ}
\begin{corollary} \label{cor:propfcirc}
  We keep the same assumptions of Theorem \ref{thm:main-discrete},
  let $M \in \cN$ and let $X \in \dom(\maxim \fF M)$. Then 
\begin{enumerate}
    \item $X \in \dom(\maxim \fF N)$ for every $N \in \cN$ s.t.~$M\mid N$.
    \item $\hfFcirc (X):= \lim_{\tau \downarrow 0} \frac{X_{\tau}-X}{\tau} \in \maxim \fF M (X)$. In particular $\hfFcirc (X) \in \maxim \fF N (X)$ for every $N \in \cN$ s.t.~$M \mid N$.
    \item $|\hfFcirc (X)|_{\cH} \le |V|_{\cH}$ for every $V \in \maxim \fF N (X)$ and for every $N \in \cN$ s.t.~$M \mid N$.
    \item $(1-\lambda\tau) |X_\tau-X|_{\cH} \le \tau |\hfFcirc (X)|_{\cH}$ for every $0<\tau<1/\lambda^+$.
    \end{enumerate}
    Moreover, for every $X, Y \in \bigcup_{N \in \cN} \dom(\maxim \fF N)$, we have
    \begin{equation}\label{eq:dissel}
    \scalprod{\hfFcirc (X)-\hfFcirc (Y)}{X-Y}_{\cH}\le \lambda |X-Y|_{\cH}^2.
    \end{equation}
\end{corollary}
\begin{proof}
By Theorem \ref{thm:brezis2}(5) there exists the limit
\[ \lim_{\tau \downarrow 0} \frac{X_\tau-X}{\tau}= \hfFcirc (X) \in \maxim \fF M (X)\]
and (4) holds. If $N \in \cN$ is s.t.~$M \mid N$, then $X \in \dom(\maxim \fF M) \subset \overline{\DDom M} \subset \overline{\DDom N}$, by Lemma \ref{le:trivial}. Moreover by Proposition \ref{prop:eulerstep}, we have that 
\[ \frac{X_{\tau}-X}{\tau} \in \maxim \fF N (X_\tau) \quad \forall \, 0<\tau <1/\lambda^+.\]
In particular
\[\scalprod{ \frac{X_{\tau}-X}{\tau}-W}{X_\tau-Y}_{\cH} \le \lambda |X_\tau-Y|_{\cH}^2  \quad \forall (Y,W) \in \fF_N \quad \forall \,  0< \tau < 1/\lambda^+,\]
so that, passing to the limit as $\tau \downarrow 0$, we get
\[ \scalprod{\hfFcirc (X)-W}{X-Y}_{\cH} \le \lambda|X-Y|_{\cH}^2 \quad \forall (Y,W) \in \fF_N,\]
since $X_\tau \to X$ as $\tau \downarrow 0$ by Theorem \ref{thm:brezis2}(4). This proves that $(X, \hfFcirc (X)) \in \maxim \fF N$ and, in particular, that $X \in \dom(\maxim \fF N)$. This proves (1) and (2).
\smallskip

 Concerning item (3): let $N \in \cN$ be s.t.~$M \mid N$; since $X\in\dom(\maxim \fF M)\subset\cH_N$, by \eqref{eq:120} we have 
\[\jJ^{\maxim \fF N}_\tau(X)=X_\tau,\]
where $\jJ^{\maxim \fF N}_\tau$ is the resolvent operator of $\maxim \fF N$. In particular, by (2) we have that $\hfFcirc=\lim_{\tau\downarrow0}\frac{\jJ^{\maxim \fF N}_\tau-\ii_\cH}{\tau}$ in $\dom(\maxim \fF M)$. Since by (2) we have $\hfFcirc (X) \in \maxim \fF N (X)$, we can conclude that $\hfFcirc (X)$ is the element of minimal norm in $\maxim \fF N (X)$ by Theorem \ref{thm:brezis2}(2)(5). 
\smallskip

Finally, if $X, Y \in \bigcup_{N \in \cN} \dom (\maxim \fF N)$, then there exist $N, M \in \cN$ s.t.~$X \in \dom (\maxim \fF N)$ and $Y \in \dom (\maxim \fF M)$ so that, taking $R:=MN$, we have
\[ \left(X, \hfFcirc (X)\right), \left(Y, \hfFcirc (Y)\right) \in \maxim \fF R\]
by (2). The $\lambda$-dissipativity of $\maxim \fF R$ gives \eqref{eq:dissel}.
\end{proof}
Thanks to the above results, we are now able to define the operator $\maxim\fF\infty\subset \cH\times \cH$
\begin{equation}
  \label{eq:181}
  \maxim\fF\infty:=\Big\{(X,V)\in \cH_\infty\times \cH_\infty:
  \exists\,M\in \cN: (X,V)\in \maxim \fF N
  \ \forall\, N\in \cN,\ M\mid N\Big\}.
\end{equation}

Equivalently, $\maxim\fF\infty$ has domain
$\dom(\maxim\fF\infty)=\bigcup_{N\in \cN}\dom(\maxim\fF N)$ and
\begin{equation}
  \label{eq:40}
  \maxim\fF\infty (X)=\bigcup_{M\in \cN}\bigcap_{M\mid N}\maxim \fF N (X)
  \quad \text{for every }X\in \dom(\maxim\fF\infty).
\end{equation}
Notice that $\maxim\fF\infty$ is the Lagrangian representation of the \MPVF
$\hat\frF_\infty$ defined by Theorem \ref{thm:main-discrete}.

We can recast the previous results in terms of $\hat \fF_\infty$ in the following statement.
\begin{corollary}
  \label{cor:summarize}
  We keep the same assumptions of Theorem \ref{thm:main-discrete}.
  The operator $\maxim\fF\infty$ defined by \eqref{eq:181} or
  \eqref{eq:40} satisfies the following properties:
  \begin{enumerate}
  \item $\maxim\fF\infty$ is $\lambda$-dissipative with domain
    $\dom(\maxim\fF\infty)=\bigcup_{N\in \cN}\dom(\maxim\fF N)$ and  $\newODDom\infty \subset \newDDom \infty \subset \dom(\maxim\fF\infty) \subset \overline{\dom(\maxim\fF\infty)} = \overline{\newODDom\infty}= \overline{\newDDom\infty}$.
  \item The map 
      $\hfFcirc$
      defined by Corollary \ref{cor:propfcirc}
      provides the minimal selection $(\maxim\fF\infty)^\circ$.
    \item 
      For every $X\in \cH_\infty$ and every
    $0<\tau<1/\lambda^+$ there exists a unique
    $X_\tau\in \dom(\maxim\fF\infty)$ such that
    $X_\tau-X\in \tau\,\maxim\fF\infty (X_\tau)$.
  \end{enumerate}
\end{corollary}
\begin{proof}
Item (1) follows by Proposition \ref{prop:ext1} and Lemma \ref{le:trivial}.
Item (2) comes by \eqref{eq:181} and Corollary \ref{cor:propfcirc}.
Item (3) is a consequence of Proposition \ref{prop:eulerstep}.
\end{proof}

In the following corollary, we are finally able to define the Lagrangian representation $\hat \fF$ of $\hat \frF$ as in Theorem \ref{thm:main-discrete} as the maximal extension of $\hat \fF_\infty$.

\begin{corollary}\label{cor:help-me}
Under the assumptions of Theorem \ref{thm:main-discrete}, there exists a unique maximal extension 
$\hat\fF$ of $\maxim\fF\infty$  with $\dom(\hat\fF)\subset \overline{\dom(\maxim\fF\infty)}$  and it satisfies the following:
\begin{enumerate}
    \item $\dom(\hat\fF)\subset \overline{\dom(\maxim\fF\infty)}=
\overline{\newODDom \infty}$,
\begin{equation}\label{eq:ninfty-hat}
\cH_N\cap\dom(\hat\fF)=\dom(\maxim\fF N),\quad \cH_\infty\cap\dom(\hat\fF)=\dom(\maxim\fF \infty),
\end{equation}
and, if $X\in\cH_\infty$ and $0<\tau<1/\lambda^+$, then 
\begin{equation}\label{eq:mmmh}
\jJ_\tau (X)=X_\tau,
\end{equation}
where $\jJ_\tau$ is the resolvent operator of $\hat\fF$ and $X_\tau$ is as in Proposition \ref{prop:eulerstep}.
\item When restricted to $\dom(\maxim\fF N)$ (resp.~$\dom(\maxim \fF \infty)$), the minimal  selection of $\hat\fF$ coincides with the minimal selection  $\maxim\fF N^\circ$ of $\maxim\fF N$ (resp.~$(\maxim \fF \infty)^\circ= \hfFcirc$ as in Corollary \ref{cor:summarize}(2)).\\
\item The following characterization holds
\begin{equation}
  \label{eq:182}
  (X,V)\in \hat\fF\,\Leftrightarrow\,\begin{array}{l}
  X\in \overline{\newODDom \infty},\\ 
  \la V-W,X-Y\ra_{\cH}\le \lambda |X-Y|_{\cH}^2
  \text{ for every }(Y,W)\in \maxim\fF\infty;\end{array}
\end{equation}
or, equivalently,
\begin{equation}
  \label{eq:182bis}
  (X,V)\in \hat\fF\,\Leftrightarrow\,\begin{array}{l}
  X\in \overline{\newODDom \infty},\\ 
  \la V-\hfFcirc (Y),X-Y\ra_{\cH}\le \lambda |X-Y|_{\cH}^2
  \text{ for every } Y\in \dom(\hat{\fF}_\infty).\end{array}
\end{equation}
\item $\hat\fF=\overline{\maxim\fF\infty}^{\cH\times \cH}$.
\end{enumerate}
\end{corollary} 
\begin{proof}
Thanks to Corollary \ref{cor:summarize}, the existence and uniqueness of the maximal extension $\hat\fF$ of $\maxim\fF\infty$ with domain
$\dom(\hat\fF)\subset \overline{\dom(\maxim\fF\infty)}$ and characterized by \eqref{eq:182} follows by Lemma \ref{le:easier}, with $D=\cH_\infty$.

Notice that \eqref{eq:mmmh} holds since, by Corollary \ref{cor:summarize}(3), when $X\in\cH_\infty$ then $X_\tau$ plays the role of the resolvent for $\maxim\fF\infty$ and we just proved that $\hat\fF$ is a maximal extension of $\maxim\fF\infty$.
We prove the equivalences in \eqref{eq:ninfty-hat}: let $X\in\cH_N\cap\dom(\hat\fF)$ and $0<\tau<1/\lambda^+$, then
\[\frac{\jJ_\tau X-X}{\tau}\]
belongs to $\maxim\fF N (X_\tau)$ thanks to Proposition \ref{prop:eulerstep} and \eqref{eq:mmmh}, moreover it
is bounded since $X\in\dom(\hat\fF)$ (cf. Theorem \ref{thm:brezis2}(5)). By maximality of $\maxim\fF N$ and applying again Theorem \ref{thm:brezis2}(5), we deduce that $X\in\dom(\maxim\fF N)$, hence $\cH_N\cap\dom(\hat\fF)\subset \dom(\maxim\fF N)$. The reverse inclusion is trivial.

Item (2) comes from item (1) and Theorem \ref{thm:brezis2}(5). The assertion involving $\maxim\fF \infty$ comes from Corollary \ref{cor:summarize}(2) and the proof of Lemma \ref{le:easier}.

The characterization in \eqref{eq:182bis} is a consequence of Corollary \ref{cor:speriamo-che-sia-ultimo},  with $D=\cH_\infty$, and of \eqref{eq:ninfty-hat}.

Finally, item (4) comes by Lemma \ref{le:easier} and the density of $\cH_\infty$ in $\cH$.
\end{proof}

\begin{remark}\label{rem:coincidence}
Notice that Corollary \ref{cor:help-me}(2) makes the notation $\hfFcirc$, used in Corollary
\ref{cor:propfcirc}, coherent with the
one used in Appendix \ref{sec:brezis} to denote the minimal selection of $\hat \fF$.
\end{remark}

\subsection{Proofs of the main Theorems \ref{thm:main-discrete}, \ref{thm:perfavorelultimo}, \ref{thm:total-case}, \ref{thm:demi-case}}\label{sec:proofs8}

We collect here the proofs of the main Theorems \ref{thm:main-discrete}, \ref{thm:perfavorelultimo}, \ref{thm:total-case}, \ref{thm:demi-case}, whose statements appear at the beginning of Section \ref{sec:constructionFlagr}.
We start with Theorem \ref{thm:main-discrete}, whose statement is contained in the following.

\begin{theorem} \label{prop:final}
Under the assumptions of Theorem \ref{thm:main-discrete},
$\hat\fF$ is a law invariant maximal  $\lambda$-dissipative operator according to Definition
\ref{def:inv} and
the Eulerian images 
$\hat\frF_N,\hat\frF_\infty,\hat\frF$ of
$\maxim\fF N,\maxim\fF\infty,\hat\fF$  respectively (cf. Definition \ref{def:representations}) satisfy
the properties stated in Theorem \ref{thm:main-discrete}.

  Finally, if $X \in \newODDom N$ for some $N \in \cN$ and $\Phi \in \frF[\iota_X]$, then
    \begin{equation}\label{eq:ineqbary}
    |\hfFcirc (X)|_{\cH}^2 \le \int_\X
    |\bry{\Phi}|^2 \de \iota_X,
    \end{equation}
    where $\bry{\Phi}$ is the barycenter of $\Phi$ as in Definition \ref{def:wassmom}.
  \end{theorem}
  \newcommand{\hfF}{\hat \fF}
\begin{proof}  We divide the proof in several claims. \\
\noindent\textbf{Claim 1.}
\emph{$\hat \fF$ is a law invariant maximal $\lambda$-dissipative operator.}
\smallskip

The operator $\hat\fF$ is maximal by definition (cf.~Corollary \ref{cor:help-me}), we need to prove it is law invariant. To this aim, it is sufficient to prove that $\mathcal{A}:=\hfF_\infty\subset \cH_\infty\times \cH_\infty$ satisfies the assumptions of Lemma \ref{le:general-invariance}
  (see also Remark \ref{rem:trivial}). Indeed, since $\hfF$ is the closure of $\hfF_\infty$ by Corollary \ref{cor:help-me}(4), 
  this yields that $\hfF$ is law
  invariant. We prove that $\hfF_\infty\cap (\cH_M\times \cH_M)$ are invariant with respect to $\Sym{I_M}$, for every $M\in\cN$. 
   By definition of $\hfF_\infty$ in \eqref{eq:181},  if $(X,V)\in \hfF_\infty\cap (\cH_M\times \cH_M)$,
  there exists some $M'\in \cN$ such that $(X,V)\in \hfF_N$ for all
  $N$ multiple of $M'$. In particular, choosing  $M'':=M\,M'\in \cN$, we have 
  $(X,V)\in \hfF_N$ for all
  $N$ multiple of $M''$. On the other hand,  any permutation
  $\sigma
  \in \Sym{I_M}$ induce an admissible permutation  of $\Sym{I_N}$, for all
  $N$ multiple of $M''$; therefore, by \eqref{eq:77}, we have that
 $(\sigma X,\sigma Y)$ belongs to $\hfF_N$ 
  for every $N$ multiple of $M''$. We deduce that $(\sigma X,\sigma
  Y)\in \hfF_\infty$ so that $\hfF_\infty\cap(\cH_M\times \cH_M)$ 
  is invariant by $\Sym{I_M}$. 
\smallskip

\noindent\textbf{Claim 2.}
\emph{$\hat \frF_N=\iota^2(\hat \fF_N)$.}
\smallskip

We prove the two inclusions. Let $\Phi \in \iota^2(\hat \fF_N)$ and let $(X,V) \in \hat \fF_N$ be s.t.~$\iotaT_{X,V} = \Phi$. Recalling the properties of $\hat \fF_N$ in Proposition \ref{prop:ext1}, we see that, since $\hat \fF_N \subset \cH_N \times \cH_N$, we have $\Phi \in \prob_{f,N}(\TX)$ and, since $X \in \dom(\hat \fF_N) \subset \overline{\mathcal D_N}=\overline{\mathcal C_N}$ (see Lemma \ref{le:trivial}(2)),  we have $\mu:=\sfx_\sharp \Phi = \iota_X \in \overline{\core_N}$.
Let now $\Psi \in \frF$ be such that $\nu:=\sfx_\sharp \Psi \in \core_N$ and $\ttheta \in \Gamma_{f,N}(\Phi, \nu)$. Let $(X',V',Y') \in \cH_N^3$ be s.t.~$(X',V', Y')_\sharp \P= \ttheta$; since $\iotaT_{X',V'}= \Phi \in \prob_{f,N}(\TX)$, up to a permutation in $\symg{I_N}$ and by the invariance by permutation of $\hat \fF_N$ in \eqref{eq:77}, we can assume that $(X',V') \in \hat \fF_N$ and $Y' \in \mathcal C_N$.
By the discussion at the beginning of Section \ref{sec:fn}, we can construct $W' \in \cH$ such that $\iotaT_{Y',W'}= \Psi$; by \eqref{eq:projbar} and $\mathcal C_N \subset \mathcal O_N$, we deduce that $\Pi_N (W') = \bry{\Psi} \circ Y'$, so that, by definition of $\fF_N$ in \eqref{eq:105}, we get that $(Y', \bry{\Psi} \circ Y') \in \fF_N$. By \eqref{eq:112} we deduce
\begin{align*}
    \int_{\TX\times\X} \la v_0-\bry{\Psi}(x_1), x_0-x_1 \ra \de \ttheta(x_0, v_0, x_1) &= \la V'-\bry{\Psi} \circ Y', X'-Y'\ra_{\cH}
    \\
    & \le \lambda |X'-Y'|^2_\cH \\
    &= \lambda \int_{\TX\times\X} |x_0-x_1|^2 \de \ttheta(x_0, v_0, v_1),
\end{align*}
which is \eqref{eq:32}.
This proves that $\iota^2(\hat \fF_N) \subset \hat \frF_N$.
Let us show the reverse inclusion: let $\Phi \in \hat \frF_N$; since $\Phi \in \prob_{f,N}(\TX)$ and $\mu:=\sfx_\sharp \Phi \in \overline{\core_N}$, we can find $(X,V) \in \overline{\mathcal D_N} \times \cH_N = \overline{\mathcal C_N} \times \cH_N$ (see Lemma \ref{le:trivial}(2)) such that $\iotaT_{X,V} = \Phi$. Let $(Y,W) \in \fF_N$; by definition of $\fF_N$ in \eqref{eq:105}, we can find $W' \in \cH$ such that $(Y,W') \in \mathcal{C}_N \times \cH$, $\Psi:=\iotaT_{Y, W'} \in \frF$, and $W=\Pi_N(W')$. In particular, $\nu:= \sfx_\sharp \Psi \in \core_N$. Again by \eqref{eq:projbar} and the fact that $\mathcal C_N \subset \mathcal O_N$, we deduce that $W=\Pi_N(W')= \bry{\Psi} \circ Y$. Setting $\ttheta:=(X,V,Y)_\sharp \P \in \Gamma_{f,N}(\Phi, \nu)$, \eqref{eq:32} gives
\begin{align*}
    \la V-W, X-Y\ra_{\cH} &=\la V-\bry{\Psi} \circ Y, X-Y\ra_{\cH} \\
    &= \int_{\TX\times\X} \la v_0-\bry{\Psi}(x_1), x_0-x_1 \ra \de \ttheta(x_0, v_0, x_1) 
    \\
    &\le \lambda \int_{\TX\times\X} |x_0-x_1|^2 \de \ttheta(x_0, v_0, v_1) \\
    &= \lambda |X-Y|^2_\cH,
\end{align*}
which, by \eqref{eq:112}, gives that $(X,V) \in \hat \fF_N$ i.e.~$\Phi \in \iota^2(\hat \fF_N)$. This proves that $\hat \frF_N \subset \iota^2(\hat \fF_N)$.
\smallskip

\noindent\textbf{Claim 3.}
\emph{$\hat \frF_N$ satisfies property (1) in Theorem \ref{thm:main-discrete}.}
\smallskip

First of all we observe that, if $\mu \in \core_N$, then there exists $X \in \mathcal C_N$ such that $\iota_X = \mu$. In particular $X \in \dom(\hat \fF_N)$ by Proposition \ref{prop:ext1}; hence there exists $V \in \cH_N$ such that $(X,V) \in \hat \fF_N$, so that $\Phi:=\iotaT_{X,V} \in \iota^2(\hat \fF_N) = \hat \frF_N$ by Claim 2. Therefore $\mu=\iota_X= \sfx_\sharp \Phi \in \dom(\hat \frF_N)$. This proves that $\core_N \subset \dom(\hat \frF_N)$.

If $\Phi_0, \Phi_1 \in \hat \frF_N$ and $\ttheta \in \Gamma(\Phi_0, \Phi_1) \cap \prob_{f,N}(\TX \times \TX)$, we can find $X_0, X_1, V_0, V_1 \in \cH$ such that $(X_0, V_0, X_1, V_1)_\sharp \P = \ttheta$; since $\iotaT_{X_0, V_0} = \Phi_0$, $\iotaT_{X_1, V_1}=\Phi_1$, and $\hat \frF_N$ is invariant by permutations in $\symg{I_N}$ by \eqref{eq:77}, we can assume that $(X_0, V_0), (X_1, V_1) \in \hat \fF_N$. The $\lambda$-dissipativity of $\hat \fF_N$ stated in Proposition \ref{prop:ext1} gives
\begin{align*}
\int_{\TX^2} \la v_1-v_0, x_1-x_0 \ra \de \ttheta(x_0, v_0, x_1, v_1) &=\la V_1-V_0, X_1-X_0 \ra_{\cH} 
\\
&\le \lambda |X_0-X_1|^2 \\
&= \int_{\TX^2} |x_0-x_1|^2 \de \ttheta(x_0, v_0, x_1, v_1).
\end{align*}
\smallskip

\noindent\textbf{Claim 4.}
\emph{$\hat \frF_N$ satisfies property (2) in Theorem \ref{thm:main-discrete}.}
\smallskip

Suppose that $\mu \in \overline{\core_N}$,  $\ff \in \maps{\hat{\frF}_N}[\mu]$, $\Psi \in \frF$ with $\nu:=\sfx_\sharp \Psi \in \core_N$, and $\mmu \in \Gamma_{f,N}(\mu, \nu)$.
Set $\Phi :=(\ii_\X, \ff)_\sharp \mu \in \hat \frF_N$ and $\ttheta:=(\sfx^0, \ff \circ \sfx^0, \sfx^1)_\sharp \mmu \in \Gamma_{f,N}(\Phi, \nu)$. Then, by \eqref{eq:32}, we get \eqref{eq:32bis}.
To get the opposite implication, take $\mu \in \overline{\core_N}$, $\ff \in L^2(\X, \mu; \X)$ and assume that \eqref{eq:32bis} holds for every $\Psi \in \frF$ such $\nu := \sfx_\sharp \Psi \in \core_N$ and all $\mmu \in \Gamma(\mu, \nu)$ such that $\mmu$ is the unique element of $\Gamma(\mu, \nu)$.
Set $\Phi:=(\ii_\X, \ff)_\sharp \mu$, and take $X \in \overline{\mathcal C}_N$ such that $\iota_X= \mu$, so that, setting $V:=\ff \circ X$, we have $\Phi=\iotaT_{X, V}$. Let $(Y,W) \in \fF_N$ be such that $\mmu:=\iotaT_{X,Y}$ is the unique element of $\Gamma_o(\iota_X, \iota_Y)$; by definition of $\fF_N$ in \eqref{eq:105}, we can find $W' \in \cH$ such that $(Y,W') \in \mathcal{C}_N \times \cH$, $\Psi:=\iotaT_{Y, W'}\in \frF$, and $W=\Pi_N(W')$. In particular, $\nu:= \sfx_\sharp \Psi \in \core_N$ and, again by \eqref{eq:111} and $\mathcal C_N \subset \mathcal O_N$, we deduce that $W=\Pi_N(W')= \bry{\Psi} \circ Y$. By \eqref{eq:32bis}, we get
\begin{align*}
    \la V-W, X-Y \ra &= \int_{\X^2} \la \ff(x_0)-\bry{\Psi}(x_1), x_0-x_1 \ra \de \mmu(x_0, x_1)\\
    &\le \lambda \int_{\X^2} |x_0-x_1|^2 \de \mmu(x_0, x_1)\\
    &= \lambda |X-Y|^2_\cH,
\end{align*}
which, by arbitrariry of $(Y,W)$ and Proposition \ref{prop:allultimomomento}, gives that $(X,V) \in \hat \fF_N$ i.e.~that $\Phi \in \hat \frF_N$, hence that $\ff \in \maps{\hat{\frF}_N}[\mu]$.
\smallskip

\noindent\textbf{Claim 5.}
\emph{$\hat \frF_N$ satisfies property (3) in Theorem \ref{thm:main-discrete}.}
\smallskip

This follows from the inclusion $\dom(\hat \fF_M) \subset \dom(\hat \fF_N)$ in Corollary \ref{cor:propfcirc}(1) and the equality $\hat \frF_N = \iota^2(\hat \fF_N)$ in Claim 2.

\smallskip

\noindent\textbf{Claim 6.}
\emph{$\hat \frF_\infty = \iota^2(\hat \fF_\infty)$ and $\hat \frF_\infty$ satisfies property (4) in Theorem \ref{thm:main-discrete}.}
\smallskip

We prove the two inclusions to show the equality $\hat \frF_\infty = \iota^2(\hat \fF_\infty)$. Let $\Phi \in \hat \frF_\infty$; then there exists $M \in \cN$ such that $\Phi \in \hat \frF_N$ for every $N \in \cN$ such that $M \mid N$. By Claim 2, for every $N \in \cN$ such that $M \mid N$, there exists $(X_N, V_N) \in \hat \fF_N$ such that $\iotaT_{X_N, V_N} = \Phi$. Set $(X,V):= (X_M, V_M)$ and let $N \in \cN$ be such that $M \mid N$. Then $(X,V) \in \hat \fF_M \subset \cH_M \times \cH_M \subset \cH_N \times \cH_N$, $(X_N, V_N) \in \hat \fF_N \subset \cH_N \times \cH_N$, and $\iotaT_{X,V}= \iotaT_{X_N, V_N}= \Phi$.
In particular, there exists a permutation $\sigma \in \symg{I_N}$ such that $(X,V)=(\sigma X_N, \sigma V_N)$. The invariance of $\hat \fF_N$ w.r.t.~permutations of $\symg{I_N}$ in \eqref{eq:77} gives that $(X,V) \in \hat \fF_N$. By arbitrariness of $N$, we have proven that $(X,V) \in \hat \fF_N$ for every $N \in \cN$ such that $M \mid N$ which, by definition of $\hat\fF_\infty$ in \eqref{eq:181}, gives $(X,V) \in \hat \fF_\infty$. Therefore $\Phi \in \iota^2(\hat \fF_\infty)$. This proves that $\hat \frF_\infty \subset \iota^2(\hat \fF_\infty)$.
Let us show the reverse inclusion: let $\Phi \in \iota^2(\hat \fF_\infty)$ and let $(X,V) \in \hat \fF_\infty$ be such that $\iotaT_{X,V}= \Phi$. By definition of $\hat \fF_\infty$ in \eqref{eq:181}, we have that there exists $M \in \cN$ such that $(X,V) \in \hat \fF_N$ for every $N \in \cN$ such that $M \mid N$. By Claim 2, we have that $\Phi \in \hat \frF_N$ for every $N \in \cN$ such that $M \mid N$ so that $\Phi \in \hat \frF_\infty$. This proves that $\iota^2(\hat \fF_\infty) \subset \hat \frF_\infty$.

Since $\hat \fF_\infty$ is $\lambda$-dissipative by Corollary \ref{cor:summarize} and we have proven that $\hat \frF_\infty = \iota^2(\hat \fF_\infty)$, by Proposition \ref{prop:basic-relation}, we get that $\hat \frF_\infty$ is totally $\lambda$-dissipative.

The equality $\dom(\hat \frF_\infty) = \cup_{M \in \cN} \dom(\hat \frF_M)$ follows from the identity $\hat \frF_\infty = \iota^2(\hat \fF_\infty)$ just proven and the corresponding characterization of the domain of $\hat \fF_\infty$ in Corollary \ref{cor:summarize}(1).

The inclusion $\core \subset \cup_{M \in \cN} \dom(\hat \frF_M)$ can be proven as follows: if $\mu \in \core$, then there exists $M \in \cN$ such that $\mu \in \core_M$, see also \eqref{eq:39}; thus there exists $X \in \mathcal C_M$ such that $\iota_X=\mu$. Therefore, $X \in \dom(\hat \fF_M)$ since, by definition of $\mathcal D_M$ in \eqref{eq:138} and by Proposition \ref{prop:ext1}, we have $\mathcal C_M \subset \mathcal D_M \subset \dom(\hat \fF_M)$. By Claim 2, we deduce $\mu \in \dom(\hat \frF_M)$.

\smallskip

\noindent\textbf{Claim 7.}
\emph{$\hat \frF = \iota^2(\hat \fF)$.}
\smallskip

By Corollary \ref{cor:help-me}, $\hat \fF$ is the unique maximal $\lambda$-dissipative operator extending $\hat \fF_\infty$ with domain included in $\overline{\mathcal C_\infty}$.
By Theorem \ref{thm:maximal-dissipativity}(2), the \MPVF $\iota^2(\hat \fF)$ is maximal totally $\lambda$-dissipative and, since $\hat \fF$ extends $\hat \fF_\infty$, it extends $\hat \frF_\infty$.
If $\mu \in \dom(\iota^2(\hat \fF))$, then we can find $X \in \dom(\hat \fF) \subset \overline{\mathcal C_\infty}$ such that $\iota_X=\mu$; therefore, there exists a sequence $(X_n)_{n\in\N} \subset \mathcal C_\infty$ such that $X_n \to X$. In particular, $\iota_{X_n} \in \core$ (see \eqref{eq:138}) and $W_2(\iota_{X_n}, \mu) \to 0$ as $n \to +\infty$; hence $\mu \in \overline{\core}$.
This proves that $\iota^2(\hat \fF)$ is a maximal totally $\lambda$-dissipative extension of $\hat \frF_\infty$ with domain included in $\overline{\core}$. Uniqueness can be proven as follows: suppose $\frG \subset \prob_2(\TX)$ is another maximal totally $\lambda$-dissipative extension of $\hat \frF_\infty$ with domain included in $\overline{\core}$, and let $\tilde \fF \subset \cH \times \cH$ be its Lagrangian representation.
By Theorem \ref{thm:maximal-dissipativity}(2), we get that $\tilde \fF$ is maximal $\lambda$-dissipative. Now assume that $(X,V) \in \hat \fF_\infty$; then $\iotaT_{X,V}\in \hat \frF_\infty \subset \frG$, so that $(X,V) \in \tilde \fF$. This shows that $\tilde \fF$ extends $\hat \fF_\infty$. On the other hand, if $X \in \dom(\tilde \fF)$, then $\iota_X \in \dom(\frG) \subset \overline{\core}$; hence there exists $(\mu_n)_{n\in\N} \subset \core$ such that $W_2(\mu_n, \mu) \to 0$ as $n \to +\infty$. By definition of $\mathcal{C}_\infty$ in \eqref{eq:138} and Theorem \ref{thm:gpfinal}, we can find $(X_n)_{n\in\N} \subset \mathcal C_\infty$ such that $X_n \to X$ and $\iota_{X_n}= \mu_n$. In particular, $X \in \overline{ \mathcal C_\infty}$; this shows that $\dom(\tilde \fF) \subset \overline{ \mathcal C_\infty}$. We have proven that $\tilde \fF$ is a maximal $\lambda$-dissipative operator extending $\hat \fF_\infty$ with domain included in $\overline{\mathcal C_\infty}$. By the uniqueness part of Corollary \ref{cor:help-me}, we deduce that $\tilde \fF = \hat \fF$, hence that $\frG=\iota^2(\hat \fF)=\hat \frF$.

\smallskip

\noindent\textbf{Claim 8.}
\emph{$\hat \frF$ satisfies property (5) in Theorem \ref{thm:main-discrete}.}
\smallskip

Let $\mu \in \overline{\core}$ and let $\Phi\in \prob_2(\TX|\mu)$ be such that \eqref{eq:34} holds for every $\nu \in \dom(\hat \frF_\infty)$, $\ff \in \maps{\hat \frF_\infty}[\nu]$, and $\ttheta \in \Gamma(\Phi, \nu)$.
We take $(X,V) \in \overline{\mathcal C_\infty} \times \cH$ such that $\iotaT_{X,V} = \Phi$ and any $Y \in \dom(\hat \fF_\infty)$ so that $\nu:=\iota_Y \in \dom(\hat \frF_\infty)$. By Corollary \ref{cor:help-me}(2), we have that $\left(Y,\hat \fF^\circ(Y)\right) \in \hat \fF_\infty$ so that $\iotaT_{Y,\hat \fF^\circ (Y)}\in \hat \frF_\infty$ by Claim 6. Moreover, by equation \eqref{eq:9} in Theorem \ref{thm:invTOlawinv}, $\hat \fF^\circ(Y)= \bb^\circ[\nu] \circ Y$; in particular $\bb^\circ[\nu] \in \maps{\hat \frF_\infty}[\nu]$. Setting $\ttheta:=(X,V,Y)_\sharp \P \in \Gamma(\Phi, \nu)$, by \eqref{eq:34}, we have
\begin{align*}
\la V-\fF^\circ(Y), X-Y \ra_\cH &= \int_{\TX\times\X} \la v-\bb^\circ[\nu](y), x-y \ra \de \ttheta(x,v,y)
\\
& \le \int_{\TX\times\X} |x-y|^2 \de\ttheta(x,v,y) \\
&= |X-Y|_\cH^2,
\end{align*}
which, by arbitrariness of $Y\in \dom(\hat \fF_\infty)$ and \eqref{eq:182bis}, gives that $(X,V) \in \hat \fF$. Hence, by Claim 7, we have that $\Phi \in \hat \frF$. The converse implication simply follows by the total $\lambda$-dissipativity of $\hat \frF$ and the inclusion $\hat \frF_\infty \subset \hat \frF$. 
\smallskip

The fact that $\hat \frF$ coincides with the strong closure of $\hat \frF_\infty$ in $\prob_2(\TX)$ follows from the analogous property for $\hat \fF$ and $\hat \fF_\infty$ stated in Corollary \ref{cor:help-me}(4). Indeed, if $\Phi$ belongs to the strong closure of $\hat \frF_\infty$ in $\prob_2(\TX)$, we can find a sequence $(\Phi_n)_{n\in\N} \subset \hat \frF_\infty$ such that $\Phi_n \to \Phi$ in $\prob_2(\TX)$. By Theorem \ref{thm:gpfinal} and Claim 6, we can find a sequence $(X_n, V_n)_{n\in\N} \subset \hat \fF_\infty$ and $(X,V) \in \cH \times \cH$ such that $\iotaT_{X_n, V_n}= \Phi_n$, $\iotaT_{X,V} = \Phi$, and $(X_n, V_n) \to (X,V)$. In particular $(X,V) \in \overline{\maxim\fF\infty}^{\cH\times \cH}$ which coincides with $\hat \fF$ by Corollary \ref{cor:help-me}(4). Thus $\Phi \in \hat \frF$ by Claim 7. On the other hand, if $\Phi \in \hat \frF$, by Claim 7, we can find $(X,V) \in \hat \fF$ such that $\iotaT_{X,V} = \Phi$. By Corollary \ref{cor:help-me}(4), there exists a sequence $(X_n, V_n)_{n\in\N} \subset \hat \fF_\infty$ such that $(X_n, V_n) \to (X,V)$. In particular, $\Phi_n:=\iotaT_{X_n, V_n} \in \hat \frF_\infty$ by Claim 6, and $\Phi_n \to \Phi$ in $\prob_2(\TX)$. This proves that $\Phi$ belongs to the strong closure of $\hat \frF_\infty$ in $\prob_2(\TX)$.
\smallskip

Now, let $\mu \in \core$ and let $X \in \mathcal C_\infty$ be such that $\iota_X=\mu$. By Theorem \ref{thm:minimal}(2) and \eqref{eq:9}, we have
\[\hat\frF{\vphantom\frF}^\circ[\mu] = (\ii_\X, \bb^\circ[\mu])_\sharp \mu = \iotaT(X, \bb^\circ[\mu] \circ X)= \iotaT(X,\hat \fF^\circ(X)). \]
Moreover, by \eqref{eq:38}, we have $\mu \in \dom(\hat \frF_\infty)$ so that, using Corollary \ref{cor:help-me}(2), we get that $(X, \hat \fF^\circ(X)) \in \hat \fF_\infty$. In particular, $\hat\frF{\vphantom\frF}^\circ[\mu] \in \hat \frF_\infty$ by Claim 6.

\smallskip

\noindent\textbf{Claim 9.}
\emph{\eqref{eq:ineqbary} holds.}
\smallskip

Let  $X \in \newODDom N \subset \DDom N$ for some $N \in \cN$, and observe that, since $\DDom N$ is open by Lemma \ref{le:trivial}, then $\jJ_\tau (X) \in \DDom N$ for $0<\tau<1/\lambda^+$ sufficiently small, since $\jJ_\tau (X) \to X$ as $\tau \downarrow 0$, where $\jJ_\tau$ is the resolvent of $\hat\fF$. We can thus apply \eqref{eq:124} and get
\[ \frac{1}{\tau}\scalprod{\jJ_\tau (X)-X}{\jJ_\tau (X)-X}_{\cH} + \directionalm{\Phi}{\iotaT_{X, \jJ_\tau (X)}}{0} \le \lambda |X-J_\tau (X)|_{\cH}^2.\]
Since we have shown that $\hat\fF$ is an invariant maximal $\lambda$-dissipative operator, by Theorem \ref{thm:invTOlawinv}, there exists a Lipschitz function $f$ such that $\jJ_\tau (X)=f\circ X$; thus $\iotaT_{X, \jJ_\tau (X)}$ is concentrated on a map so that, by Theorem \ref{thm:all}(4), we have 
\[ \directionalm{\Phi}{\iotaT_{X, \jJ_\tau (X)}}{0} = \scalprod{\bry{\Phi}}{X-\jJ_\tau (X)}_{\cH}.\]
We hence get
\[ \frac{1}{\tau} |\jJ_\tau (X)-X|_{\cH}^2 \le |X-\jJ_\tau (X)|_{\cH}\,\bigg(|\bry{\Phi}|+\lambda |X-\jJ_\tau (X)|_{\cH}\bigg);\]
dividing by $|X-\jJ_\tau (X)|_{\cH}$ and passing to the limit as $\tau \downarrow 0$, we obtain \eqref{eq:ineqbary} (cf. Theorem \ref{thm:brezis2}(5)).\\
\end{proof}

We conclude this section with the proofs of Theorems \ref{thm:perfavorelultimo}, \ref{thm:total-case} and \ref{thm:demi-case}.

\begin{proof}[Proof of Theorem \ref{thm:perfavorelultimo}] 
 The existence of a curve as in \emph{(1)} comes from the fact that $\overline{\core_N} \subset \overline{\dom(\hat \frF)}$ and the maximal total $\lambda$-dissipativity of $\hat \frF$. Let us collect the properties of $\mu$, as they are in the first part of the statement, in the following item: 
\begin{enumerate}
    \item[(0)] $\mu:[0, +\infty) \to \overline{\core_N}$ is a $\lambda$-\EVI solution for the restriction of $\frF$ to $\core_N$, which is locally absolutely continuous in $(0,+\infty)$.
\end{enumerate}
We devote the rest of the proof to prove the equivalence between \emph{(0)}, \emph{(1)}, and \emph{(2)}.

 \smallskip
 
 \noindent\textbf{Claim 1.}
\emph{(1) $\Leftrightarrow$ (2)}.
\smallskip

To see that (2) implies (1),  it is sufficient to notice that by \eqref{eq:v-in-F} $\mu$ satisfies the inclusion $(\ii_\X, \vv_t)_\sharp \mu_t \in \hat{\frF}[\mu_t]$ for a.e.~$t>0$, so that it is clearly a $\lambda$-\EVI solution for $\hat{\frF}$ (see also \cite[Theorem 5.4(1)]{CSS}); by Theorem \ref{thm:final-total}, we get that $\mu$ is a Lagrangian solution of the flow generated by $\hat{\frF}$. We are left to check that (1) implies (2).

Since $\mu_0\in \overline{\core_N}$, we can represent $\mu_0$ as $\iota_{X_0}$
for some $X_0\in \overline{\newODDom N} =\overline{\newDDom N}=\overline{\dom(\hat\fF_N)}$ (cf. Lemma \ref{le:trivial} and Proposition \ref{prop:ext1});
if $(\bm S_t)_{t\ge0}$ is the semigroup generated by $\hat\fF$ we have
$\mu_t=\iota_{X_t}$ where $X_t=\bm S_t (X_0)$.

By Corollary \ref{cor:help-me}(1), the restriction of the resolvent $\jJ_\tau$
of $\hat\fF$ to $\cH_N$ coincides with the resolvent of $\hat\fF_N$: using the exponential formula (cf.~\eqref{eq:convJS}), we obtain that the restriction of 
the semigroup $(\bm S_t)_{t\ge0}$ 
to $\overline{\newDDom N}$ coincides with the semigroup generated by $\hat\fF_N.$
Since the interior of the domain of $\hat\fF_N$ in $\cH_N$ is not empty (cf. Proposition \ref{prop:ext1} and Lemma \ref{le:trivial}), we can 
apply Theorem \ref{thm:brezreg}. We thus obtain that
$\bm S_t (X_0)$ is locally absolutely continuous in $[0,+\infty)$ and it is locally Lipschitz in $(0,+\infty)$. Moreover, it satisfies $I_\lambda(t)|\dot X_t|_\cH\le C$ in $(0,1)$ 
for a suitable constant $C$ (so that we get \eqref{eq:v-estimate}), it belongs to $\dom(\hat{\fF}_N)$ for every $t>0$,
and it solves the equation
$$\dot X_t=\hat{\fF}_N^\circ (X_t) \quad\text{for $\leb ^1$-a.e.~$t>0$},$$
where $\hat{\fF}_N^\circ$ denotes the minimal selection of $\hat{\fF}_N$.
Corollary \ref{cor:help-me}(2) then shows that $\dot X_t=(\hat{\fF})^\circ (X_t)$ 
as well, so that we get 
\begin{equation*}
    \partial_t \mu_t+\nabla\cdot(
    \mu_t \,\hat\ff{\vphantom\ff}^\circ(\cdot, \mu_t))=0
    \quad\text{in $(0,+\infty)\times \X$},
\end{equation*}
and therefore \eqref{eq:v-in-F}: indeed the tangent space $\Tan_{\mu_t}\prob_2(\X)$ (cf.~Theorem \ref{thm:tangentv} and \cite[Theorem 8.3.1, Propositions 8.4.5, 8.4.6]{ags}) coincides with $L^2(\X, \mu_t; \X)$ since $\supp(\mu_t)$ has finite cardinality.
 
\smallskip

\noindent\textbf{Claim 2.}
\emph{(2) $\Rightarrow$ (0)}.
\smallskip
 
We know that $\mu$ solves the continuity equation with velocity field $\vv_t=\hat\ff{\vphantom\ff}^\circ[\mu_t]$ 
so that, by Corollary \ref{cor:help-me}(2), we have 
$(\ii_\X, \vv_t)_\sharp \mu_t \in \hat{\frF}_N$.  
Let $\Phi \in \frF$ with $\nu:=\sfx_\sharp \Phi \in \core_N$ and let $t \in A(\mu) \subset [0,+\infty)$, where $A(\mu)$ is the full $\leb^1$-measure set given by Theorem \ref{thm:all}(6a). By Theorem \ref{thm:main-discrete}(2) we have that
\begin{equation}\label{eq:eeeeciao}
 \int_{\X^2} \la \vv_t(x),x-y \ra \de \mmu_t(x,y) \le \int_{\X^2} \left ( -\la \bry{\Phi}(y),y-x\ra + \lambda |x-y|^2 \right )\de \mmu_t(x,y)
\end{equation} 
for every $\mmu_t \in \Gamma_{f,N}(\mu_t, \nu)$. Choosing $\mmu_t$ optimal, by Theorem \ref{thm:all}(6a) we have that 
\[ \frac{\de}{\de t} \frac{1}{2} W_2^2(\mu_t, \nu) = \bram{(\ii_\X,\vv_t)_\sharp \mu_t}{\nu} \le \int_{\X^2} \la \vv_t(x),x-y \ra \de \mmu_t(x,y).\]
On the other hand, since $\mmu_t$ is concentrated on a map w.r.t.~$\nu$, \eqref{eq:29} gives that 
\[ \int_{\X^2} \la \bry{\Phi}(y),y-x\ra \de \mmu_t(x,y) = \directionalm{\Phi}{{\mathsf s}_\sharp\mmu_t}0,\]
where $\mathsf s: \X^2 \to \X^2$ is defined by $\mathsf s (x_0,x_1):=(x_1,x_0)$.
So that, using \eqref{eq:eeeeciao}, we obtain that 
\[\frac{\de}{\de t} \frac{1}{2} W_2^2(\mu_t, \nu) \le -\directionalm{\Phi}{{\mathsf s}_\sharp\mmu_t}0 + \lambda W_2^2(\mu_t,\nu).\]
 Noting that ${\mathsf s}_\sharp\mmu_t\in\Gamma_o(\nu,\mu_t)$, we have
\[\directionalm{\Phi}{{\mathsf s}_\sharp\mmu_t}0\ge\min_{\ggamma_t\in\Gamma_o(\nu,\mu_t)}\directionalm{\Phi}{\ggamma_t}0=\bram{\Phi}{\mu_t},\]
where the last equality is given by Theorem \ref{thm:all}(2). We  finally obtain
\[ \frac{\de}{\de t} \frac{1}{2} W_2^2(\mu_t, \nu) \le - \bram{\Phi}{\mu_t} +\lambda W_2^2(\mu_t,\nu);\]
this implies that $\mu$ is a $\lambda$-\EVI solution for the restriction of $\frF$ to $\core_N$.

\smallskip

\noindent\textbf{Claim 3.}
\emph{(0) $\Rightarrow$ (1)}.
\smallskip

 We apply \cite[Lemma 5.3, (5.5a)]{CSS} 
    obtaining that for every $t$ in a set $A(\mu)\subset [0,+\infty)$ of full $\leb^1$-measure, every  $\nu\in \core_N$ and $\Phi\in \frF[\nu]$, we have
    \begin{equation}
        \label{eq:cisiamo}
        \big[(\ii_{\X}, \vv_t)_\sharp \mu_t,\nu\big]_r
        +\big[\Phi,\mu_t\big]_r\le \lambda W_2^2(\mu_t,\nu),
    \end{equation}
    where $\vv_t$ is the Wasserstein velocity field of $\mu$.
    Let $t \in A(\mu)$ be fixed; restricting \eqref{eq:cisiamo}
    to all the measures $\nu$ for which $\Gamma_o(\mu_t,\nu)$ contains
    a unique element (denoted by $\mmu$), 
    Theorem \ref{thm:all}(4) yields 
    \begin{align*}
        \big[(\ii_\X, \vv_t)_\sharp \mu_t,\nu\big]_r
        &=
        \int_{\X^2} \langle \vv_t(x_0),x_0-x_1\rangle \,\d\mmu(x_0,x_1),\\
        \big[\Phi,\mu_t\big]_r&=
        \int_{\X^2} \langle \bry\Phi(x_1),x_1-x_0\rangle \,\d\mmu(x_0,x_1).
    \end{align*}
  Proposition \ref{prop:allultimomomento} and \eqref{eq:cisiamo} then yield that $(\ii_\X, \vv_t)_\sharp \mu_t
    \in \hat\frF_N[\mu_t].$

    Let us now consider the Lagrangian solution $\tilde\mu_t:=S_t(\mu_0)$ 
    of the flow driven by $\hat\frF$.
   By the Claim 1, 
    we know that $\tilde \mu$ is absolutely continuous, 
    $\tilde \mu_t\in \dom(\hat\frF_N)\subset 
    \dom(\hat\frF)\cap \overline{\core_N}$ for $t>0$,  
    and satisfies \eqref{eq:v-in-F}.

    We can then compute the derivative of $W_2^2(\mu_t,\tilde\mu_t)$:
    for $\mathscr L^1$-a.e.~$t>0$, we can choose an arbitrary $\mmu_t
    \in \Gamma_o(\mu_t,\tilde \mu_t)$, in particular
    a coupling in $\prob_{f,N}(\X\times \X)$, obtaining, by Theorem \ref{thm:all}(6b), 
    \begin{displaymath}
        \frac\d{\d t}\frac 12 W_2^2(\mu_t,\tilde\mu_t)
        =
        \int_{\X^2} \langle \vv_t(x_0)-\hat\ff{\vphantom\ff}^\circ[\tilde\mu_t](x_1),x_0-x_1
        \rangle\,\d\mmu_t(x_0,x_1)\le \lambda W_2^2(\mu_t,\nu)
    \end{displaymath}
    by $\lambda$-dissipativity of $\hat{\frF}_N$, since $(\ii_\X, \hat\ff{\vphantom\ff}^\circ[\tilde{\mu}_t])_\sharp\tilde{\mu}_t \in \hat{\frF}_N$ by Corollary \ref{cor:help-me}(2).  We thus have that $\mu_t=\tilde\mu_t$ for every $t\ge0$
    and $\vv_t=\hat\ff{\vphantom\ff}^\circ[\mu_t]$.
\end{proof}

\begin{remark}\label{rem:correct} Consider the example of $\frac{1}{2}$-dissipative \PVF $\frF$, with $\X=\R$ discussed in Remark \ref{rem:sticky}.  We already know that $\frF$ cannot be maximal totally $1/2$-dissipative, since the evolution driven by $\frF$ splits mass, a contradiction with Theorem \ref{thm:existence-Lagrangian}. Thanks to Theorem \ref{thm:perfavorelultimo} we can also deduce that it is not even totally $1/2$-dissipative: the evolution driven by $\frF$ and the one driven by the maximal totally $1/2$-dissipative \MPVF $\hat \frF$ should coincide, but this is again impossible by Theorem \ref{thm:existence-Lagrangian}. In particular, Theorem \ref{thm:perfavorelultimo} can fail when  $\dim(\X)=1$ and $\frF$ is not totally dissipative.
\end{remark}

\begin{proof}[Proof of Theorem \ref{thm:total-case}]
Let $\mmo'$ be a law invariant maximal $\lambda$-dissipative extension of the Lagrangian representation of $\frF$ with domain included in the convex set $\overline{\newODDom\infty}$, whose existence is given by Theorem \ref{thm:maximal-dissipativity}. Notice that $\iota^2(\mmo')$ is maximal totally $\lambda$-dissipative and contains $\frF$ so that it also contains $\bri{\frF}$ by Theorem \ref{thm:bary-proj}. We deduce that $\mmo'$ is the Lagrangian representation of a $\lambda$-dissipative extension of $\frF\cup \bri{\frF}$. 

We want to show that $\mmo'\subset \hat\fF$ and we split the argument in a few steps. 

  \smallskip
  
   \noindent\textbf{Claim 1.}
\emph{for every $Y\in \dom(\mmo')\cap 
  \Big(\bigcup_{N\in \cN}\overline{\newDDom N}\Big)$
  and $W\in \mmo' (Y)$, we have $W\in \hat\fF (Y).$}
  \smallskip

  Let $Y$ and $W$ be as above and let $X \in \dom(\hat{\fF}_\infty)$. We can find some $M,L \in \cN$ such that $Y \in \dom(\mmo') \cap \overline{\newDDom M}$ and $X \in \dom(\maxim \fF L)$. In particular $Y \in \dom(\mmo') \cap \overline{\newDDom N}$ and $X \in \dom(\maxim \fF N)$ for every $N \in \cN$ such that $ML \mid N$ (cf.~Corollary \ref{cor:propfcirc} and Lemma \ref{le:trivial}). By \eqref{eq:43} we have
  \begin{equation}
    \label{eq:44}
         \langle X-Y,\hfFcirc (X)-\Pi_N (W)\rangle_{\cH}\le \lambda|X-Y|_{\cH}^2
    \quad
    \text{for every $N \in \cN$ such that $ML \mid N$}.
  \end{equation}
  Passing to the limit as $N\to+\infty$ in $\cN$ and using \eqref{eq:182bis} we deduce that 
  $(Y,W)\in \hat \fF.$ 

  \smallskip

   \noindent\textbf{Claim 2.}
  \emph{$\dom(\mmo')\cap 
  \Big(\bigcup_{N\in \cN}\overline{\newDDom N}\Big)=
  \dom(\mmo')\cap \cH_\infty$.}
  \smallskip
  
  It is sufficient to prove that 
  $\dom(\mmo')\cap \overline{\newDDom N}=
  \dom(\mmo')\cap \cH_N$ for every $N\in \cN$ 
  and since 
  $\overline{\newDDom N}\subset \cH_N$ 
  it is sufficient to prove the inclusion 
  \begin{equation}
  \label{eq:closure-domains}
      \dom (\mmo')\cap \cH_N\subset \overline{\newDDom N}.
  \end{equation}
  We first show that 
  \begin{equation}
  \label{eq:unacosaallavolta}
      \overline{\dom(\hat\fF)}\cap \cH_N \subset  \overline{\newDDom N}.
  \end{equation}
  Indeed, by Proposition \ref{prop:eulerstep} and Corollary \ref{cor:help-me}, for every 
  $X\in \overline{\dom(\hat\fF)}\cap \cH_N$ and $\tau>0$, 
  $\jJ_\tau (X)$ belongs to $\dom(\maxim\fF N)\subset \overline{\newDDom N}$:  passing to the limit as $\tau\downarrow0$, since $X \in \overline{\dom(\hat \fF)}$, we conclude that $X$
  belongs to $\overline{\newDDom N}$ as well, thus proving \eqref{eq:unacosaallavolta}. Since $\dom(\mmo')\subset \overline{\newDDom \infty}
  =\overline{\dom(\hat\fF)}$, by \eqref{eq:unacosaallavolta}, we get  
  $\dom (\mmo')\cap \cH_N\subset 
    \overline{\dom(\hat\fF)}\cap \cH_N\subset 
    \overline{\newDDom N}$, which shows \eqref{eq:closure-domains}.
  
  \smallskip
  
  \noindent\textbf{Claim 3.}
\emph{$\mmo'\subset \hat \fF.$}
\smallskip

  Setting $\mmo'_0:=\mmo'\cap (\cH_\infty\times \cH)$,
  Claims 1 and 2 yield
  $\mmo'_0\subset \hat\fF $.
  On the other hand, the maximal $\lambda$-dissipativity and the law invariance of $\mmo'$ show (cf.~Theorem \ref{thm:invTOlawinv}) that
  $\cH_\infty$ is invariant under the action of the resolvent of $\mmo'$;
  since $\cH_\infty$ is also dense in $\cH$, we can apply 
  \eqref{eq:152} of Lemma \ref{le:easier} obtaining that $\mmo'$
  coincides with the strong closure of $\mmo'_0$ in $\cH \times \cH$ which is also contained
  in $\hat\fF$, since $\hat \fF$ is maximal $\lambda$-dissipative.
 
\end{proof}

\begin{proof}[Proof of Theorem \ref{thm:demi-case}]
Let us first check that $\frF \subset \hat \frF_\infty$.
It is sufficient to prove that if $\mu\in \core_M$ and $M\mid N$, $M,N\in \cN$, 
then every element $\Phi=(\ii_\X, \ff)_\sharp\mu\in \frF[\mu]$ 
belongs to $\hat{\frF}_N[\mu]$. 
Adopting a Lagrangian viewpoint (thanks to Theorem \ref{prop:final}), if $X\in \newODDom M$
we want to show that $V=\ff\circ X$ belongs to $\maxim \fF N (X)$.
This follows easily from 
the fact that $\newODDom  M\subset \overline{\newDDom N}$, 
the $\lambda$-dissipativity of $\frF$
and Proposition \ref{prop:allultimomomento}.
Since $\hat{\frF}_\infty$ is totally $\lambda$-dissipative, the inclusion $\frF \subset \hat \frF_\infty$ shows that $\frF$ is totally $\lambda$-dissipative and $\hat{\frF}_\infty$ is a totally $\lambda$-dissipative extension of $\frF$.
By construction, $\hat{\frF}$ is a maximal totally $\lambda$-dissipative extension of $\frF$ and its uniqueness follows as a particular case of Theorem \ref{thm:total-case}.
The characterization in \eqref{eq:32tris} follows by definition of $\hat{\frF}_\infty$ and Proposition \ref{prop:allultimomomento}.
Let us now check the second statement, under the assumptions that $\frF$ is also single-valued and demicontinuous in $\core_N$. By Corollary \ref{cor:propfcirc}, we know that, on each $\newODDom N$, the minimal selection $\hfFcirc$ is a subset of  $\maxim\fF N$ and therefore, by Corollary \ref{cor:dim-demi-case1}, $\hfFcirc (X)=\fF (X)$ for every $X\in \newODDom \infty$.
 
\end{proof}

\subsection{Examples and applications}
\label{subsec:ex}

This subsection is devoted to several examples to which the developed theory applies. In particular, in the following examples, we provide some \MPVF to which Theorem \ref{thm:main-discrete} and \ref{thm:perfavorelultimo} apply. More specifically, for these examples we have existence of $\lambda$-\EVI solutions without the boundedness assumptions required in our previous work \cite{CSS}; we also have and a fine description of the solutions coming from the Lagrangian perspective.

We can now fully justify the example given in the Introduction.
\begin{example}\label{ex:introo}
Assume that $\dim \X\ge2$ and that $\frF$ is a $\lambda$-dissipative 
single-valued deterministic \PVF
induced by a map $\ff:\Sp{\X,\core}\to \X$, where $\core$ is a core as in Definition \ref{ass:core}.
This means that 
$\ff$ induces a vector field $\ff^N:\mathsf C_N\to \X^N$
defined on $\mathsf C_N:=\mathscr I^{-1}_N(\newODDom N)$ (where $\newODDom N$ is as in \eqref{eq:138}),
which is an open subset of $\X^N$, whose vectors
have distinct coordinates:
for every $\xx=(x_1, \dots, x_N)\in \mathsf C_N$ we have
$$\ff^N(\xx):=(\ff(x_n,\iota \circ \mathscr I_N(\xx)))_{n=1, \dots, N}.$$
Clearly $\ff^N$ is invariant with respect to permutations, in the sense that $\ff^N(\sigma \xx)= \sigma \ff^N(\xx)$, for every $\xx \in \mathsf C_N$ and every $\sigma \in \Sym{I_N}$.
If $\frF$ is demicontinuous in $\core_N$,
$\ff^N$ is demicontinuous (i.e.~strongly-weakly continuous) 
in $\mathsf C_N$.

 Theorem \ref{thm:perfavorelultimo}  shows that 
starting from 
$\mu^N=\frac 1N\sum_{n=1}^N\delta_{\bar{x}^N_n}\in \core_N$
the evolution $\mu^N_t=S_t(\mu^N)$, at least for a short time  when no collisions occur,
has the form
\begin{displaymath}
    \mu^N_t=\frac 1N\sum_{n=1}^N\delta_{x^N_n(t)}
    \quad\text{where}\quad
    \dot x^N_n(t)=\ff^N_n(\xx^N(t)), \quad \xx^N(0)=\bar{\xx}^N:=(\bar{x}_1^N, \dots, \bar{x}_N^N) .
\end{displaymath}
Such an evolution admits a unique extension (see Theorem \ref{thm:demi-case}) which in fact corresponds to the unique 
maximal (and invariant by permutation) extension
of the $\lambda$-dissipative vector field
$\ff^N$ to $\overline{\mathsf C_N}$.
It is then possible to follow the path of
each single particle by using the Lagrangian flow
starting from $\mu_0\in\core$ and defining $N$ locally Lipschitz curves  $x^N_n\in\Lip_{loc}([0,T];\X)$, 
$x^N_n(t)=\ss_t(x^N_n,\mu_0)$.
If now $\mu^N\to \mu_0$ as $N\to+\infty$
with a uniform control of the initial velocities, 
i.e.~
$$\sup_N \frac 1N\sum_{n=1}^N |\ff^N_n(\bar{\xx}^N)|^2 <+\infty,$$
then 
the measures 
$\mu^N_t$ will converge to $\mu_t=S_t(\mu_0)$
for every $t\ge0$ in $\prob_2(\X)$ and,
by Theorem \ref{thm:stability},
the measures carried on the discrete trajectories
$\frac 1N \sum_{n=1}^N \delta_{x^N_n}
\in \prob_2(\rmC([0,T];\X))$
will converge to $\mathrm s_\sharp \mu_0$
where $\mathrm s$ is the Lagrangian map
starting from $\mu_0$ as in \eqref{eq:pedante}.
\end{example}

\begin{example}[A kinetic model of collective motion]\label{ex:collmo}
    Consider 
    in the phase space $\X:=\R^d\times \R^d$
    the evolution 
    of $N$-particles characterized
    by position-velocity coordinates 
    $(x_n,v_n)\in \X$, $n=1,\dots, N$, satisfying
    the system
    \cite{DCBC06,CCR11}
    \begin{equation}
        \label{eq:kinetic}
        \left\{
        \begin{aligned}
            \dot x_n(t)&=v_n(t),\\
            \dot v_n(t)&=(\alpha-\beta|v_n(t)|^2)v_n(t)
            +\frac 1N \sum_{m=1}^N 
            \bm h(x_n(t)-x_m(t)),
        \end{aligned}
        \right.
    \end{equation}
    with $\alpha\ge 0,\beta> 0$
    and $\bm h:\R^d\to\R^d$ a given Lipschitz vector field.
    For a given $\mu\in \prob_2(\X)$ 
    we can consider the 
    lower semicontinuous and 
     $(-\alpha)$-totally convex functional 
    $\phi:\prob_2(\X)\to (-\infty,+\infty]$
    \begin{equation}
        \label{eq:v-potential}
        \phi(\mu):= 
        \int_{\X}
        \Big(\frac \beta4 |v|^4
        -\frac \alpha2 |v|^2\Big)\,\d\mu(x,v),
    \end{equation}
    whose proper domain is 
    $\dom(\phi):=
    \displaystyle 
    \Big\{\mu\in \prob_2(\X):\int_{\X} |v|^4\,\d\mu(x,v)<+\infty\Big\}$.
    The minimal selection of 
    $-\btpartial\phi(\mu)$ 
    is given by $(\ii_\X, \gg)_\sharp \mu$ with  
    \begin{equation}
        \label{eq:minimal-kinetic}
        \gg(x,v;\mu):=
        \Big(0,(\alpha -\beta|v|^2) v\Big)
    \end{equation}
    with proper domain $\dom(\btpartial\phi)=
    \displaystyle\Big\{\mu\in \prob_2(\X):\int_\X |v|^6\,\d\mu(x,v)<+\infty\Big\}.$
       
    We can also define the 
    deterministic \PVF induced as in \eqref{eq:f-induces-F} by
    $\bm h:\mathcal S(\X)\to \X$ 
    \begin{equation}
        \label{eq:kinetic-field}
        \bm h(x,v;\mu):=
        \Big(v,\int_{\X}
        \bm h(x-y)\,\d\mu(y,w)\Big).
    \end{equation}
    It is easy to check that 
    a collection of $N$ particles
    $(x_n(t),v_n(t))$ satisfies
    \eqref{eq:kinetic} if and only if 
    the measure 
    $\mu_t=\frac1N\sum_{n=1}^N \delta_{(x_n(t),v_n(t))}$
    is a Lagrangian solution
    of the system 
    \begin{equation*}
    (\dot{x}_n(t),\dot{x}_n(t))=\ff(x_n(t),v_n(t),\mu_t)\quad
    \text{a.e.~in }(0,+\infty)
\end{equation*}
    associated with the 
    deterministic \PVF
    \begin{equation}
        \label{eq:sum}
        \ff(x,v;\mu):=\gg(x,v;\mu)+
        \bm h(x,v;\mu),\quad\mu\in \dom(\btpartial\phi).
    \end{equation}
    Since the Lagrangian representation 
    of $\ff$ corresponds to the sum
    of a maximal $\alpha$-dissipative 
    operator 
    (the subdifferential of 
    $\psi=\phi\circ\iota$)
    and a Lipschitz operator,
    it is maximal $\alpha$-dissipative
    thanks to \cite[Lemma 2.4, Chapter~II]{BrezisFR},
    so that the deterministic \PVF associated
    with \eqref{eq:sum}
    is totally $\alpha$-dissipative
    and we can apply all the results of 
    Section \ref{sec:totdissMPVF-flow}.
\end{example}
In the following we give an example of totally dissipative \MPVF $\frF$ having a core contained in its domain.
 \begin{example}
Let $W:\X \to (-\infty, + \infty]$ be a proper, lower semicontinuous, even and convex function and denote by $\dom(W)$ its proper domain. Let $\Bb\subset\X\times\X$ be a maximal dissipative set (see Appendix \ref{sec:brezis}) and  suppose that $0 \in \intt{\dom(W)}$ and $\intt{\dom(\Bb)}\ne \emptyset$. Possible examples of $W$ and $\Bb$ are given by the indicator of a convex set in $\X$ (or a function diverging at the boundary of a convex set) and the gradient of a convex function in $\X$ (or its sum with a linear and antisymmetric function) respectively. 
Let $\uu_W$ be an odd single-valued  measurable selection of $\partial W$ and let $\vv_{\Bb}$ be an arbitrary single-valued selection of $\Bb$. We define the set
\[E:=\left\{\mu\in\prob_c(\X)\,:\,\supp\mu \subset \intt{\dom(\Bb)}, \, \supp \mu - \supp \mu \subset \intt{\dom(W)}\right\},\]
 where $\prob_c(\X)$ denotes the subset of measures in $\prob(\X)$ with compact support. 
We define the single-valued probability vector field $\frF$ as follows:
\[\frF[\mu]:=\begin{cases}
\left(\ii_\X,-(\uu_W \ast \mu)+\vv_{\Bb}\right)_\sharp\mu,&\textrm{if }\mu\in E\\
\emptyset&\textrm{otherwise}
\end{cases},\qquad \mu\in\prob_2(\X).\]
Notice that the convolution between $\uu_W$ and $\mu$ is well posed since  the support of $\mu$ is compact and by definition of $E$; moreover  $(\uu_W \ast \mu)+\vv_{\Bb}\in L^2(\X,\mu;\X)$ if $\mu\in E$; indeed $\vv_{\Bb}$ and $\uu_W$ are both locally bounded in the interior of the respective domains (see Corollary \ref{cor:phi} and Theorem \ref{thm:brezis2}(3) and recall that $\intt{\dom(\partial W)}=\intt{\dom(W)}$), so that $\dom(\frF)=E$ and $\frF\subset\prob_2(\TX)$. It is not difficult to check that $\frF$ is totally  dissipative: for every $\ggamma \in \Gamma(\mu, \nu)$ and every $\mu, \nu \in E$,
\begin{align*}
&\frac{1}{2} \int_{\X^2} W(y_1-y_2) \de (\nu \otimes \nu)(y_1,y_2) - \frac{1}{2} \int_{\X^2} W(x_1-x_2) \de (\mu \otimes \mu)(x_1,x_2)\\
    &\ge \frac{1}{2} \int_{\X^4} \la \uu_W(x_1-x_2), (y_1-y_2) -(x_1-x_2) \ra \de (\ggamma \otimes \ggamma)(x_1,y_1,x_2,y_2) \\
    &= \frac{1}{2} \int_{\X^3} \la \uu_W(x_1-x_2),y_1-x_1 \ra \de \mu(x_2)\de \ggamma(x_1,y_1)  \\ 
    &\quad +\frac{1}{2} \int_{\X^3} \la \uu_W(x_2-x_1), y_2-x_2 \ra \de \mu(x_1) \de \ggamma(x_2,y_2)\\
    &= \int_{\X^2}\la (\uu_W \ast \mu)(x),y-x \ra \de \ggamma(x,y),
\end{align*}
where we have used Fubini's theorem and the fact that $\uu_W$ is odd. This immediately gives that
\begin{equation}\label{eq:tobeused}
\int_{\X^2}\la (-\uu_W \ast \mu)(x)+ (\uu_W \ast \nu)(y),x-y \ra \de \ggamma(x,y) \le 0.
\end{equation}
Thus
\begin{align*}
& \int_{\X^2} \la - (\uu_W \ast \mu)(x)+ \vv_{\Bb}(x)+(\uu_W \ast \nu)(y)- \vv_{\Bb}(y), x-y \ra \de \ggamma(x,y)  \\
&= \int_{\X^2} \la (-\uu_W \ast \mu)(x)+ (\uu_W \ast \nu)(y),x-y \ra \de \ggamma(x,y) \\
&\quad + \int_{\X^2} \la \vv_{\Bb}(x)-\vv_{\Bb}(y), x-y \ra \de \ggamma(x,y) \\
& \le 0,
\end{align*}
where we have used \eqref{eq:tobeused} and the dissipativity of $\Bb$.\\
Given any unbounded directed subset $\cN \subset \N$, we can define $\tcore$ as 
\[\tcore:=\left\{\mu\in\prob_{f,\cN}(\X)\,:\,\supp\mu \subset \intt{\dom(\Bb)}, \, \supp \mu - \supp \mu \subset \intt{\dom(W)}\right\}.\]

Trivially, since $\tcore\subset\prob_c(\X)$, then
$\tcore\subset\dom(\frF)\cap\prob_{f,\cN}(\X)$. Moreover, for any
$N\in\cN$, the set $\tcore\cap\prob_{f,N}(\X)$ is open in
$\prob_{f,N}(\X)$  and convex along couplings in $\prob_{f,N}(\X
\times \X)$, since both $\intt{\dom(\partial W)}$ and
$\intt{\dom(\Bb)}$ are convex sets (see Corollary \ref{cor:phi} and
Theorem \ref{thm:brezis2}(3)). Thus, setting $\core:=\tcore \cap
\prob_{\#\cN}(\X)$ and recalling
Lemma \ref{le:equivalent}, then Definition \ref{ass:core} is satisfied for $\core$. 
\end{example}

\begin{example}
    \label{ex:ultimo}
     Assume $\dim \X\ge2$.
    Let $\mathsf U\subset \X$ be an open convex
    subset of $\X$ containing $0$ 
    (e.g.~an open ball of radius $r>0$
    centered at $0$)
    and let $\mathrm A$ be the set 
    of all measures $\mu\in \prob_2(\X)$ such that 
    $$\supp\mu-\int_\X x\,\d\mu(x)\subset \mathsf U.$$
    In the case $\mathsf U$ is an open ball, $\mathrm A$
    imposes the constraint that 
    the support of $\mu$ is contained in the ball with same
    radius as $\mathsf U$ centered at the barycenter of $\mu$.
    We can then consider the 
    set $\tcore:= \bigcup_{N\in \N} (\mathrm A\cap \prob_{f,N})$
    and inducing
    a corresponding core $\core$ as in Lemma
    \ref{le:equivalent}.

    Let 
     $\ff:\Sp \X\to \X$
     be a map as in Theorem \ref{thm:demi-total} 
     inducing a $\lambda$-dissipative demicontinuous
    \PVF $\frF$ by \eqref{eq:f-induces-F}. 

    The restriction of $\ff$ to $\Sp{\X,\core}$
    induces a unique maximal totally $\lambda$-dissipative
    \MPVF $\frF'$,
    whose evolution corresponds to the evolution driven by $\ff$ and constrained by $\mathrm A$.
\end{example}

We conclude with an example 
of two probability vector fields $\frF,\frG$ generating the 
same evolution semigroup.
The assumptions could be considerably refined: we just discuss a simple case, 
for ease of exposition.
\begin{example}[Superposition of \PVF{s}]
    \label{ex:SGD}
    Let $(\Theta,\mathcal T,\mm)$
    be a probability space and
    let 
    $\ff:\X\times \Theta\to \X$
    be a $\mathcal B(\X)\otimes \mathcal T$-measurable map satisfying 
    the properties
    \begin{gather*}
     \ff(\cdot,\theta):\X\to\X \quad\text{is $\lambda$-dissipative and demicontinuous for 
        $\mm$-a.e.~$\theta\in \Theta,$}        \\
  \text{there exists $A>0$ such that }
        |\ff(x,\theta)|\le A(1+|x|^2)
        \text{ for every }x\in \X
        \text{ and $\mm$-a.e.~$\theta\in \Theta$.}
    \end{gather*}
    We denote by 
    $\pi^\X:\X\times \Theta\to \X$
    the projection on the first component,
    $\pi^\X(x,\theta):=x$,
    and we set
    \begin{equation}
        \label{eq:superposition}
        \begin{aligned}
        \frF[\mu]:={}&
        (\pi^\X,\ff)_\sharp (\mu\otimes 
        \mm),\quad\mu\in\prob_2(\X).
        \end{aligned}
    \end{equation}
    Clearly
    \begin{equation*}
        |\frF[\mu]|_2^2=
        \int_\X\Big(
        \int_\Theta |\ff(x,\theta)|^2\,\d\mm(\theta)\Big)\,\d\mu(x)
        \le A(1+\mathsf m_2^2(\mu))<+\infty
    \end{equation*}
    so that $\dom(\frF)=\prob_2(\X)$.
    Using the plan $\Sigma:= (\sfx^0, \ff(\sfx^0, \ii_{\Theta}), \sfx^1, \ff(\sfx^1, \ii_{\Theta}))_\sharp (\mmu \otimes \mm)$ where $\mmu \in \Gamma_o(\mu_0, \mu_1)$, we see that 
    $\frF$ is $\lambda$-dissipative.
    Its barycentric selection (cf.~\eqref{eq:brimpvf})
    $\frG:=\bri \frF$ is a deterministic
    \PVF induced by the 
    demicontinuous map
    \begin{equation}
        \gg(x):=\int_\Theta \ff(x,\theta)\,\d\mm(\theta).
    \end{equation}
    $\frG$ is a maximal totally $\lambda$-dissipative \PVF (cf. Theorem \ref{thm:trivial-but-useful-to-fix}).
    Whenever $\ff(\cdot,\theta)$ is not constant 
    in a set $\Theta_0\subset \Theta$ of positive $\mm$-measure 
    (and therefore $\frF\neq \frG$), then
    $\frF$ cannot be totally $\lambda$-dissipative since this would lead to a contradiction with the maximality of its barycentric projection $\frG$. Applying
    \cite[Corollary 5.23, Theorem 5.27]{CSS},
    we know that $\frF$ generates
    a unique $\lambda$-EVI flow
    whose trajectories 
    have the barycentric property, and therefore coincide with the 
    Lagrangian solutions 
    of the flow generated by $\frG$, i.e.~$\frF$ and
    $\frG$ generate the same evolution semigroup.  
    It would not be difficult to check that 
    $\frG$ coincides with the 
    operator $\hat \frF$
    of Theorem \ref{thm:main-discrete}
    constructed from the restriction of $\frF$
    to the core of 
    discrete measures.
\end{example}
\section{Geodesically convex functionals with a core dense in energy are totally convex}\label{sec:jko}
In this section, we provide sufficient conditions for the total $(-\lambda)$-convexity property (cf.~Section \ref{sec:jkobis}), $\lambda\in\R$, of a functional $\phi:\prob_2(\X) \to (-\infty, + \infty]$ which is proper, lower semicontinuous and  geodesically $(-\lambda)$-convex (see \cite[Definition 9.1.1]{ags}) with proper domain $\dom(\phi):=\{\mu\in\prob_2(\X)\,:\,\phi(\mu)<+\infty\}$, where we assume $\dim(\X)\ge2$. This ensures the applicability of the results of Section \ref{sec:jkobis}, in particular Theorem \ref{prop:GFiotabis}.

Recall that
$\phi:\prob_2(\X)\to(-\infty,+\infty]$ is geodesically
$(-\lambda)$-convex if for any $\mu_0,\mu_1$ in $\dom(\phi)$ there exists $\mmu\in\Gamma_o(\mu_0,\mu_1)$ such that
\[\phi(\mu_t)\le (1-t)\phi(\mu_0)+t\phi(\mu_1)+\frac{\lambda}{2}t(1-t) W_2^2(\mu_0,\mu_1)\qquad\forall \, t\in[0,1],\]
where $\mu_t:=\sfx^t_{\sharp}\mmu$. 
\begin{theorem}
[Geodesic convexity vs total convexity]\label{prop:ab} 
  Assume that
  $\dim \X\ge 2$, $\phi:\prob_2(\X)\to (-\infty,+\infty]$
  is a proper l.s.c.~geodesically $(-\lambda)$-convex functional such that
  $\dom(\phi)$ contains
  a $\cN$-core $\core$ (see Definition \ref{ass:core}) which is \emph{dense in energy,} meaning that for every $\mu \in \dom(\phi)$ there exists $(\mu_n)_n \subset \core$ such that
    \[ \mu_n \to \mu \quad \text{ and } \quad \phi(\mu_n) \to \phi(\mu).\]
Then $\phi$ is totally $(-\lambda)$-convex (cf.~Section \ref{sec:jkobis}).
\end{theorem}
\begin{proof}
  Notice that $\phi$ is geodesically (resp.~totally) $(-\lambda)$-convex
  if and only if $\phi_{\lambda}:=\phi +\frac{\lambda}{2}\sqm{\cdot}$
  is geodesically (resp.~totally) convex. Moreover the
  assumptions
  of the present Theorem
  hold for $\phi$ if and only if they hold for $\phi_\lambda$. We can
  thus prove the Theorem only in case $\lambda=0$.
We proceed in a
  few steps, keeping the notation of Section
  \ref{sec:contropL}.
  First of all, we introduce
  a standard Borel space  $(\Omega, \cB)$ endowed with a nonatomic probability measure $\P$
  as in Definition \ref{def:sbs}
  and let $\cH:= L^2(\Omega, \cB, \P;
  \X)$. We lift $\phi$ to the l.s.c.~functional
  $\psi:\cH\to (-\infty,+\infty]$
  defined
  as
\begin{equation}
  \label{eq:45}
  \psi(X):=\phi(\iota_X)\quad\text{for every }X\in \cH.
\end{equation}
\noindent\textbf{Claim 1.}
\emph{The restriction of
  $\psi$ to $\newODDom
  N$
  is continuous and locally convex.}
  \smallskip

By construction the function $\psi$ is finite and lower
semicontinuous in $\newODDom N$. It is also clear, recalling
Lemma \ref{le:quantitative}, that for every $X\in \newODDom N$
there is an open ball $\mathcal U$ of $\cH_N$ and centered at $X$ such that
$\mathcal U\subset \newODDom N$
and the restriction of $\psi$ to $\mathcal U$ is convex.
Since $\mathcal U$ is open,
it follows that $\psi$ is locally convex and continuous in
$\newODDom N$.

\smallskip

\noindent\textbf{Claim 2.}
\emph{For every $X_0,X_1 \in \newODDom N$ we have}
  \begin{equation}
  \psi((1-t)X_0+tX_1)
  \le (1-t)\psi(X_0)+t\psi(X_1).\label{eq:46}
\end{equation}

Let $X_0,X_1\in \newODDom N$; setting $A:=\supp(\iota_{X_0})$
and $B:=\supp(\iota_{X_1})$ we can apply Proposition
\ref{prop:perturbation} and use the fact that $\newODDom N$ is relatively open
to find $X_1'\in \newODDom N$ such that
$X_1(s):=(1-s)X_1+s X_1'\in \newODDom N$ for every $s\in [0,1]$
and $X_{s,t}:=(1-t)X_0+tX_1(s)$ belongs to $\mathcal{O}_N$
for every $t\in [0,1]$ and $s\in (0,1]$.
Since $\core_N$ is convex along collisionless couplings, we deduce
that $X_{s,t}\in \newODDom N$ for every $s,t\in (0,1)$
and $\psi(X_{s,t})\le (1-t)\psi(X_0)+t\psi(X_1(s))$.
Passing to the limit as $s\downarrow0$, using the the lower
semicontinuity
of $\psi$ and its continuity
in $\newODDom N$ we deduce
\eqref{eq:46}.

\smallskip

\noindent\textbf{Claim 3.}
\emph{Let $K \in \N$, $X_1,X_2,\cdots X_K\in \newODDom N$
  and $\beta_1,\cdots, \beta_K\ge0$ with $\sum_{k=1}^K\beta_k=1$ .
  For every $\eps>0$ there exist $X_k'\in \newODDom N$ with
  $|X_k-X_k'|<\eps$, $k=1,\cdots, K$, such that
  $\sum_{k=1}^K\beta_kX_k'\in \newODDom N$.}
  \smallskip

It is sufficient to observe that the map $S_K:\cH^K\to \cH$,
$S_K(X_1,\cdots,X_K):=\sum_{k=1}^K\beta_k X_k$ is linear, continuous,
and surjective,  in particular it is an open map.
If $X_1,X_2,\cdots X_K\in \newODDom N$ and
$\mathcal U_\eps$ is an open ball of radius $\eps$
around the corresponding vector in $\cH^K$ and contained in $(\newODDom N)^K$,
$S_K\big(\mathcal U_\eps)$ is open in
$\newDDom N=\conv{\newODDom N}$ so that its intersection with
the open and dense subset $\newODDom N$ (see Lemma \ref{le:trivial}(2)) is not empty.

\smallskip

\noindent\textbf{Claim 4.}
\emph{For every $K \in \N$, $X_1,X_2,\cdots X_K\in \newODDom N$
  and $\alpha_1,\cdots, \alpha_K\ge0$ with $\sum_{k=1}^K\alpha_k=1$
  we have}
\begin{equation}
  \psi\left(\sum_{k=1}^K \alpha_k X_k\right)\le \sum_{k=1}^K
  \alpha_k\psi(X_k).\label{eq:47}
\end{equation}
\smallskip

We argue by induction on the number $K$. By Claim 2 the statement is true if
$K=2$. Let us assume that it is true for $K\in \N$
and let us consider $X_k\in \newODDom N$, $1\le k\le K+1$ and
corresponding coefficients $\alpha_k$. It is not restrictive to assume
$0<\alpha_{K+1}<1$ and we set $\beta_k:=\alpha_k/(1-\alpha_{K+1})$,
$1\le k\le K$, 
so that $\beta_k\ge0$ and $\sum_{k=1}^K\beta_k=1$.

We can use Claim 3 and for every $\eps>0$
we can find $X_k'(\eps)\in \newODDom N$ with $|X_k'(\eps)-X_k|<\eps$
such that $X'(\eps):=\sum_{k=1}^K\beta_kX_k'(\eps)\in \newODDom N$.

Using Claim 2, we get
\begin{displaymath}
\psi\big((1-\alpha_{K+1})X'(\eps)+\alpha_{K+1}X_{K+1}\big)
\le (1-\alpha_{K+1})\psi\big(X'(\eps)\big)+\alpha_{K+1}\psi\big(X_{K+1}\big).
\end{displaymath}
Using the induction step we also get
\begin{displaymath}
  (1-\alpha_{K+1})\psi\big(X'(\eps)\big)
  \le \sum_{k=1}^K \alpha_k \psi\big(X_k'(\eps)\big).
\end{displaymath}
Combining the two inequalities and passing to the limit as
$\eps\downarrow0$
using the lower semicontinuity of $\psi$ and its continuity
in $\newODDom N$ we conclude.

\smallskip

\noindent\textbf{Claim 5.}
\emph{$\psi$ is convex in $\overline{\newDDom N}$.}
\smallskip

Let us consider 
the convex envelope of the restriction of $\psi$ to $\newDDom N=\conv{\newODDom N}$ 
defined by 
\begin{equation*}
    \psi_N(X):=\inf\Big\{\sum_{k=1}^K\alpha_k\psi(X_k):
    X_k\in \newODDom N,\ \alpha_k\ge 0,\ \sum_{k=1}^K\alpha_k=1,\ 
    \sum_{k=1}^K\alpha_k X_k=X,\ K\in \N\Big\}, \quad X \in \newDDom N.
\end{equation*}
By the Claim 4, $\psi(X)\le \psi_N(X)$ for every $X\in \newDDom N.$
We then consider the lower semicontinuous envelope
$\bar\psi_N:\overline{\newDDom N}\to (-\infty,+\infty]$ 
of $\psi_N$
defined by 
\begin{equation*}
    \bar\psi_N(X):=\inf\Big\{\liminf_{n\to+\infty}\psi_N(X_n):
    (X_n)_{n\in\N}\subset \newDDom N,\ X_n\to X\quad\text{as }n\to+\infty\Big\}, \quad X \in \overline{\newDDom N}.
\end{equation*}
Since $\psi$ is lower semicontinuous and $\psi_N$ is continuous in $\newODDom N$, 
we have 
\begin{equation}
    \label{eq:coincidence}
    \psi(X)\le \bar\psi_N(X)\quad\text{for every }X\in \overline{\newDDom N},\quad
    \bar\psi_N(X)=\psi_N(X)=\psi(X)\quad\text{if }X\in \newODDom N.
\end{equation}

We want to show that $\psi\equiv \bar\psi_N$ in $\overline{\newDDom N}.$ 
Let us suppose that  $X\in \overline{\newDDom N}$, with $\psi(X)<+\infty.$ We take 
$Y\in \newODDom N$, so that $X_t:=(1-t)X+tY\in \newDDom N$ for every $t\in (0,1]$ 
(since $\overline{\newDDom N}$ is convex and its relative interior coincides with $\newDDom N$ by Lemma \ref{le:trivial})
 and $X_t\in \newODDom N$ with possibly finite exceptions.
 Therefore, 
possibly replacing $Y$ with $X_{t_0}$ for a sufficiently small $t_0>0$, 
it is not restrictive to assume that $X_t\in \newODDom N$ for every $t \in (0,1]$ and $\iotaT_{X,Y}$ is the unique optimal coupling between its marginals (see Lemma \ref{thm:easy-but-not-obvious}) , so that 
$\psi$ is convex along $(X_t)_{t \in [0,1]}$ since $\phi$ is geodesically convex. We deduce that 
\begin{equation*}
    \bar{\psi}_N(X_t)=\psi(X_t)\le (1-t)\psi(X)+t\psi(Y)\quad\text{for every }t\in (0,1],
\end{equation*}
so that $\bar\psi_N(X)\le \liminf_{t\downarrow0}\bar{\psi}_N(X_t)\le \psi(X)$.

\smallskip

\noindent\textbf{Claim 6.}
\emph{$\psi$ is convex}.
\smallskip

Let $X,Y\in \dom(\psi)$,
and let $\mu=\iota_X, \nu=\iota_Y \in \prob_2(\X)$. We thus have that $\mu, \nu \in \dom(\phi) \subset \overline{\core}$.

By density, we can find sequences $(\mu_n)_{n\in\N}, (\nu_n)_{n\in\N} \subset \core$
such that $W_2(\mu_n, \mu) \to 0$, $W_2(\nu_n, \nu) \to 0$,
$\phi(\mu_n) \to \phi(\mu)$ and  $\phi(\nu_n) \to \phi(\nu)$ as $n \to
+ \infty$.
By the last part of Theorem \ref{thm:gpfinal}, we can find sequences
$(X_n)_{n\in\N}, (Y_n)_{n\in\N} \subset \newODDom \infty$
such that $\iota_{X_n}= \mu_n$, $\iota_{Y_n}=\nu_n$, $X_n \to X$ and
$Y_n \to Y$. 
Since  $X_n\in \newODDom {M(n)}$,
$Y_n\in \newODDom {N(n)}$ for some $M(n),N(n)\in \cN$
and $\cN$ is a directed set, we can find $P(n)\in \cN$ such that 
$M(n)\mid P(n)$, $N(n)\mid P(n)$; so that 
$X_n,Y_n\in \overline{\newDDom {P(n)}}$. By Claim 5, we  
we have that
\begin{align*}
  \psi((1-t)X_n+tY_n) \le (1-t) \psi(X_n) + t \psi(Y_n),\quad \text{for any }n\in\N.
\end{align*}
Passing to the limit as $n\to+\infty$
and using the lower semicontinuity of $\psi$ yield the sought convexity.
\end{proof}
\begin{remark}[Geodesic convexity implies total convexity for continuous functionals] 
\label{rem:surprise}
Let $\phi:\prob_2(\X)\to \R$ be a lower semicontinuous
  and geodesically $(-\lambda)$-convex functional
  which is approximable by discrete measures, i.e.~for every
  $\mu\in \prob_2(\X)$ there exists a sequence
  $\mu_n\in \prob_{\#\N}(\X)$ converging to $\mu$ such that
  $\phi(\mu_n)\to\phi(\mu)$ (e.g.~$\phi$ is continuous).
  Then
  $\phi$ satisfies the assumptions of Theorem \ref{prop:ab}
  with $\core=\prob_{\#\N}(\X)$. 
  This in particular gives that such kind of functionals
  are totally $(-\lambda)$-convex and locally Lipschitz.
  
As a consequence, we notice that non totally $(-\lambda)$-convex
functionals cannot be approximated in the Mosco sense by everywhere
finite, continuous and geodesically $(-\lambda)$-convex functionals
defined on $\prob_2(\X)$
(this is because total $(-\lambda)$-convexity
is preserved by the Mosco limit).
\end{remark}

\begin{remark} An analogous result as in Remark \ref{rem:surprise} has been obtained independently in \cite{Parker}. There, the author proves the equivalence of geodesic convexity and total convexity, assuming that the functional is additionally differentiable, with no restrictions on $\dim \X$. Notice that, if the functional is just continuous, the result doesn't hold in general if $\dim \X=1$, as shown in \cite[Example 3.9]{Parker}.
\end{remark}
\medskip

As previously mentioned, thanks to Theorem \ref{prop:ab} we are
allowed to apply all the results obtained in Section \ref{sec:jkobis}
to the totally $(-\lambda)$-convex functional $\phi$. In particular,
we get existence and uniqueness of the $\lambda$-\EVI solution for
the \MPVF $\frF:=-\bm{\partial}\phi$ starting from
$\mu_0\in\overline{\dom(\phi)}$ and its Lagrangian characterization as
the law of the semigroup generated by $-\partial\psi$, where $\psi$ is defined as in \eqref{eq:45}.

We conclude the section by
showing that the total subdifferential
$-\btpartial\phi:=\iotaT(-\partial \psi)$ coincides with the
operator $\hat\frF$ obtained by the $\cN$-core construction of
Theorem \ref{thm:main-discrete}.
\begin{proposition}\label{prop:jkoprimaprop}
  Let us suppose that
  $\dim \X\ge 2$, $\phi:\prob_2(\X)\to (-\infty,+\infty]$
  is a proper, l.s.c.~geodesically $(-\lambda)$-convex functional such that
  $\dom(\bm\partial\phi)$ contains
  a $\cN$-core $\core$ which is dense in energy in the sense that for every $\mu \in \dom(\phi)$ there exists $(\mu_n)_{n\in\N} \subset \core $ s.t.
  \[ \mu_n \to \mu, \quad \phi(\mu_n) \to \phi(\mu).\]
  The maximal totally $\lambda$-dissipative
  \MPVF $\hat\frF$, obtained by Theorem \ref{thm:main-discrete}
  starting from the minimal selection $-\bm\partial^\circ\phi$ restricted to $\core$,
  coincides with $-\btpartial\phi$ defined as in Section \ref{sec:jkobis}.
  Equivalently,  
  if $\psi:=\phi\circ\iota$
  and $\hat\fF$ is the Lagrangian representation of
  $\hat\frF$, 
  then 
\[ \hat\fF = -\partial \psi.\]
\end{proposition}
\begin{proof}
  By Theorem \ref{prop:ab}, we have that $\phi$ is totally $(-\lambda)$-convex so that we can apply the results of Section \ref{sec:jkobis}. By Propositions \ref{prop:invphi} and \ref{prop:JtauvsJKObis} we know that $\bm\partial^\circ\phi$ coincides with
  $\btpartial^\circ\phi$ and $\btpartial^\circ\phi$ is totally  $\lambda$-dissipative.

  Theorem \ref{thm:total-case} shows that
  $\hat\frF$ provides the unique maximal totally $\lambda$-dissipative
  extension of the restriction of $\btpartial^\circ\phi$
  to $\core$ with domain included in $\overline{\core}$. Therefore, $\hat\frF$ must coincide with $\btpartial\phi$, since $\btpartial\phi$ is maximal totally $\lambda$-dissipative as well (cf. Proposition \ref{prop:invphi}) and observing that by Proposition \ref{prop:JtauvsJKObis}(3) we have $\dom(\btpartial \phi)= \dom(\bm\partial \phi) \subset \overline{\core}$.
\end{proof}

\appendix 

\section{Dissipative operators in Hilbert spaces and extensions}\label{sec:brezis}

This appendix recalls and establishes useful results on $\lambda$-dissipative operators in Hilbert spaces, which are used throughout the paper. We divide the appendix into three parts. Section \ref{sec:lambdabrezis} lists classical results on $\lambda$-dissipative operators; these are stated for the case $\lambda=0$ in the monograph \cite{BrezisFR}. We stress that the proofs for a general $\lambda\in\R$ are adaptations of the $\lambda=0$ case, and the emphasis should be placed on the statements rather than on the proofs, which we include only for completeness. In the short Section \ref{sec:mmofin}, we state and prove two results concerning the behavior of $\lambda$-dissipative operators when restricted to closed subspaces of the ambient space, and when the space is finite-dimensional. Finally, in Section \ref{sec:mmoext} we discuss the problem of uniqueness and characterization of the maximal extension of dissipative operators in several situations; the only non‑original result here is Proposition \ref{prop:robert}.

\subsection{Classical results on $\lambda$-dissipative operators}\label{sec:lambdabrezis}

In this section, we recall useful definitions, properties and results on $\lambda$-dissipative operators in Hilbert spaces used in Sections \ref{sec:invmpvf} and \ref{sec:constructionFlagr}, with $\lambda\in\R$. Our main reference is \cite{BrezisFR}.

Let $\H$ be a Hilbert space with norm $|\cdot|$ and scalar product $\scalprod{\cdot}{\cdot}$. Given $E\subset \H$, we denote by $\conv E$ the convex hull of $E$  and
by $\clconv E$ 
its closure. Given an operator $\Bb\subset
\H\times \H$ (which we identify with its graph) we define its sections $\Bb (x):=\{v\in \H:(x,v)\in
\Bb\}$,  its domain
$\dom(\Bb):=\{x\in \H:\Bb (x)\neq \emptyset\}$, and its inverse $\Bb^{-1}:=\{ (v,x) \in \H \times \H : (x,v) \in \Bb\}$.  
 An operator $\Bb\subset \H\times
\H$
is $\lambda$-dissipative ($\lambda \in \R$) if
\begin{equation}
  \label{eq:140}
  \la v-w,x-y\ra\le \lambda |x-y|^2\quad \text{for every }(x,v),\ (y,w)\in \Bb.
\end{equation}
A $\lambda$-dissipative operator $\Bb$ is maximal if it is maximal w.r.t.~inclusion in the class of $\lambda$-dissipative operators  or, equivalently, (see e.g.~\cite[Chap.~II, Def.~2.2]{BrezisFR}) if
\begin{equation}
  \label{eq:141}
  (x,v)\in \H\times \H,\quad
  \la v-w,x-y\ra\le \lambda |x-y|^2 \quad \text{for every }(y,w)\in
  \Bb\quad\Rightarrow\quad
  (x,v)\in \Bb.
\end{equation}
\begin{remark}[Dissipativity, monotonicity]\label{rem:transff} Let $\Bb \subset \H \times \H$; we define $-\Bb:= \{ (x,-v) : (x,v) \in \Bb\}$ and we say that $\Bb$ is $\lambda$-monotone if $-\Bb$ is $(-\lambda)$-dissipative. It is easy to check that $\Bb$ is $\lambda$-dissipative if and only if $\Bb^{\lambda}:= \Bb-\lambda \ii_\H$ is $0$-dissipative (or simply, dissipative) if and only if $-\Bb^\lambda$ is $0$-monotone (or simply, monotone). The same holds for maximal $\lambda$-dissipativity, maximal dissipativity and maximal monotonicity (with analogous definition). Observe also that $\dom(\Bb)=\dom(\Bb^\lambda)=\dom(-\Bb^\lambda)$.
\end{remark}

We list in the following theorems a few well known properties of $\lambda$-dissipative operators that have been extensively used in the previous sections. Since these results are more commonly known for $\lambda=0$ (cf. \cite{BrezisFR}), we prefer to state them here in the general case. For this reason, in the proofs, we point out only the changes that have to be made compared to the case $\lambda = 0$. Recall that $\lambda^+ := \lambda \vee 0$ and we set $1/\lambda^+ = + \infty$ if $\lambda^+=0$.
\begin{theorem}\label{thm:brezis1}
    Let $\Bb\subset \H \times \H$ be a $\lambda$-dissipative operator. Then:
    \begin{enumerate}
        \item $\Bb$ is maximal if and only if the resolvent operator $\jJ_\tau:=(\ii_\H-\tau \Bb)^{-1}$ is a $(1-\lambda \tau)^{-1}$-Lipschitz continuous map defined on the whole $\H$ for every $0<\tau<1/\lambda^+$;
        \item there exists a maximal extension $\hat{\Bb}$ of $\Bb$ (meaning that $\Bb \subset \hat{\Bb}$ and $\hat{\Bb}$ is maximal $\lambda$-dissipative) whose domain is included in $\clconv{\dom(\Bb)}$.
    \end{enumerate}
\end{theorem}
\begin{proof} 
(1) We can use Remark \ref{rem:transff} and apply \cite[Proposition 2.2]{BrezisFR} to $-\Bb^\lambda$ and then obtain that $\Bb$ is maximal $\lambda$-dissipative if and only if $((1+\lambda \vartheta)\ii_\H-\vartheta \Bb)^{-1}$ is a contraction on $\H$ for every $\vartheta>0$. Since $x \mapsto x/(1-\lambda x)$ is a bijection between $(0,1/\lambda^+)$ and $(0,+\infty)$, this is equivalent to saying that $((1-\lambda \tau)^{-1}(\ii_\H-\tau \Bb))^{-1}$ is a contraction on $\H$ for every $0<\tau< 1/\lambda^+$ which is to say that $\jJ_\tau$ is a $(1-\lambda \tau)^{-1}$-Lipschitz map defined on the whole $\H$.\\
\medskip
(2) This follows immediately from Remark \ref{rem:transff} and \cite[Corollary 2.1]{BrezisFR}.
\end{proof}

\begin{remark}[Characterization of the resolvent]\label{rem:charres} Property (1) in Theorem \ref{thm:brezis1} can be equivalently stated saying that, for every $x \in \H$ and $\tau \in (0, 1/\lambda^+)$, $\jJ_\tau(x)$ is the unique solution $y$ of the inclusion $(y-x)/\tau \in \Bb(y)$ or, equivalently, that $(\jJ_\tau(x), (\jJ_\tau(x)-x)/\tau)$ is the unique pair $(y,v)$ satisfying $y= x + \tau v$, $v \in \Bb(y)$.
\end{remark}

\begin{theorem}\label{thm:brezis2}
    Let $\Bb$ be a maximal $\lambda$-dissipative operator. Then:
    \begin{enumerate}
        \item $\Bb$ is closed in the strong-weak (or the weak-strong) topology in $\H \times \H$;
        \item for every $x \in \dom(\Bb)$, the section $\Bb (x)$ is closed and convex so that it contains a unique element of minimal norm denoted by $\Bb^\circ (x)$;
         \item if $\intt{\conv{\dom(\Bb)}} \ne \emptyset$, then $\intt{\dom(\Bb)}$ is convex, $\intt{\dom(\Bb)} = \intt{\overline{\dom(\Bb)}}\ne \emptyset$ and $\Bb$ is locally bounded in the interior of its domain;
        \item $\overline{\dom(\Bb)}$ is convex and for every $x \in \overline{\dom(\Bb)}$, $\jJ_\tau (x) \to x $ as $\tau \downarrow 0$;
        \item for every $0<\tau<1/\lambda^+$, the Moreau-Yosida approximation of $\Bb$, $\Bb_\tau:=\frac{\jJ_\tau-\ii_\H}{\tau}$, is maximal $\frac{\lambda}{1-\lambda \tau}$-dissipative and $\frac{2-\lambda \tau}{\tau(1-\lambda \tau)}$-Lipschitz continuous. Moreover, for every $x \in~\dom(\Bb)$,
        \begin{equation*}
        \begin{split}
        &(1-\lambda \tau)\left|\Bb_\tau (x)\right| \uparrow \left|\Bb^\circ (x)\right|,\quad\text{as }\tau \downarrow 0,\\
        &\Bb_\tau (x) \to \Bb^\circ (x),\quad\text{as }\tau \downarrow 0,\\
        &\left|\Bb_\tau (x) - \Bb^\circ (x)\right|^2 \le \left|\Bb^\circ (x)\right|^2 -(1-2\lambda\tau)\left|\Bb_\tau (x)\right|^2,\quad\text{for } 0<\tau<1/\lambda^+.
        \end{split}
        \end{equation*}
        If $x \notin \dom(\Bb)$, then $\left|\Bb_\tau (x)\right| \to + \infty$. Finally, $\Bb_\tau \to \Bb$ in the graph sense:
\begin{equation*}
\text{for every }(x,v)\in\Bb\text{ there exists }(x_\tau)_{\tau>0}\subset\H\text{ such that }x_\tau\to x,\,\Bb_\tau (x_\tau)\to v,\text{ as }\tau\downarrow0.
\end{equation*}
\item $\Bb^\circ$ is a \emph{principal selection} of $\Bb$ i.e.
        \begin{equation}
\label{eq:143}
  (x,v)\in \overline {\dom(\Bb)}\times\H,\quad
  \la v-\Bb^\circ (y),x-y\ra\le \lambda |x-y|^2 \quad \text{for every }y\in
  \dom(\Bb)\quad\Rightarrow\quad
  (x,v)\in \Bb.
\end{equation}
    \end{enumerate}
\end{theorem}
\begin{proof}
(1) and (2) follow immediately from \eqref{eq:141}. \\
(3) follows immediately by Remark \ref{rem:transff} and \cite[Proposition 2.9]{BrezisFR}.\\
(4) follows by Remark \ref{rem:transff} and \cite[Theorem 2.2]{BrezisFR} observing that
\[ \lim_{\tau \downarrow 0} \jJ_\tau (x) = \lim_{\vartheta \downarrow 0} (1+\lambda \vartheta) (\ii_\H+\vartheta (-\Bb^\lambda))^{-1}(x) = x.\]
(5) The Lipschitz constant of $\Bb_\tau$ can be estimated by $\frac{1}{\tau}(L+1)$, where $L$ is the Lipschitz constant of $\jJ_\tau$, so that the value of the constant follows by Theorem \ref{thm:brezis1}(1). The fact that $\Bb_\tau$ is $\lambda/(1-\lambda \tau)$ dissipative is a consequence of the inequality
\[ \langle \Bb_\tau (x) - \Bb_\tau (y) , x-y \rangle = \frac{1}{\tau} \la \jJ_\tau (x) - \jJ_\tau (y), x-y \ra  -\frac{1}{\tau} |x-y|^2 \le \frac{\lambda}{1-\lambda \tau} |x-y|^2,  \]
where we used the Lipschitz continuity of $\jJ_\tau$. Maximality of $\Bb_\tau$ follows by Remark \ref{rem:transff} and \cite[Proposition 2.6]{BrezisFR}. The fact that $(1-\lambda\tau)|\Bb_\tau x|$ is increasing and bounded from above by $\left|\Bb^\circ (x)\right|$ follows precisely as in the proof of \cite[Proposition 2.6]{BrezisFR}: exploiting the dissipativity inequality
\[ \langle \Bb^\circ (x) - \Bb_\tau (x), x-\jJ_\tau (x) \rangle \le \lambda |x-\jJ_\tau (x)|^2\]
one gets that $\left|\Bb_\tau (x)\right|^2 (1-\lambda\tau) \le \la \Bb^\circ (x), \Bb_\tau (x) \ra $ for every $x \in \dom(\Bb)$. Substituting to $\Bb$, in the same inequality, the $\lambda/(1-\lambda \eta)$-dissipative operator $\Bb_\eta$, we get that 
\[ \left|\Bb_{\eta+\tau} (x) \right|^2 (1-\lambda(\tau + \eta)) \le (1-\lambda \eta) \langle \Bb_{\eta}(x) , \Bb_{\eta+\tau}(x) \rangle \quad \text{ for every } x \in \H \text{ and every } 0<\eta, \tau < 1/\lambda^+.\]
This shows that the quantity $(1-\lambda\tau)\left|\Bb_\tau (x)\right|$ is nondecreasing as $\tau \downarrow 0$ for every $x \in \H$. This means in particular that there exists the limit $\ell:= \lim_{\tau \downarrow 0} \left|\Bb_\tau (x)\right| \in [0,+\infty]$. The above estimate also gives that 
\begin{equation}\label{eq:giulialavuole}
\left|\Bb_{\eta+\tau}(x) - \Bb_{\eta}(x)\right|^2 \le \left|\Bb_\eta (x)\right|^2 - \frac{1-\lambda(\eta+2\tau)}{1-\lambda \eta}\left|\Bb_{\eta+\tau}(x)\right|^2 \quad \text{ for every } x \in \H,
\end{equation}
so that $\left(\Bb_\tau (x)\right)_\tau$ is Cauchy whenever it is bounded. Thus, if $x \in \dom(\Bb)$, then $(1-\lambda \tau)\left|\Bb_\tau (x)\right| \le \left|\Bb^\circ (x)\right|$ so that $\Bb_\tau (x) \to v$ for some $v \in \H$. By (1), $(x,v) \in \Bb$ and $|v| \le \left|\Bb^\circ (x)\right|$ which implies that $v= \Bb^\circ (x)$. On the other hand,  if $x \notin \dom(\Bb)$, we have that $\left|\Bb_\tau (x) \right| \to + \infty$: indeed, if by contradiction $\left|\Bb_\tau (x) \right|$ is bounded, then we have shown that $\Bb_\tau (x)$ must converge to some $v \in \H$ so that we also have $\jJ_\tau (x) = \tau \Bb_\tau (x) + x \to x$. Since $\left(\jJ_\tau (x), \Bb_\tau (x)\right) \in \Bb$ and $\left(\jJ_\tau (x), \Bb_\tau (x)\right) \to (x,v)$, by (1) we deduce that $(x,v) \in \Bb$, a contradiction. Observe that passing to the limit as $\eta \downarrow 0$ in \eqref{eq:giulialavuole}, we get that $\left|\Bb_\tau (x) - \Bb^\circ (x)\right|^2 \le \left|\Bb^\circ (x)\right|^2 -(1-2\lambda\tau)\left|\Bb_\tau (x)\right|^2$. To conclude the proof of (5) we only need to show the graph convergence of $\Bb_\tau$ to $\Bb$. Let $(x,v) \in \Bb$ and let us define $x_\tau:= x-\tau v$. Then $x_\tau \to x$ and $\jJ_\tau (x_\tau) = x$. Then $\Bb_\tau (x_\tau) = (x-x_\tau)/\tau=v$.\\
(6) Follows exactly as in \cite[Proposition 2.7]{BrezisFR}: performing similar computations, we get
\[ \frac{1}{2} \langle y_1-y_2, x_1-x_2 \rangle \le -\langle y_1+y_2, x-\jJ_\tau (x) \rangle + \lambda (|\jJ_\tau (x) - x_1|^2 + |\jJ_\tau (x) -x_2|^2)\]
for every $(x_1,y_1), (x_2,y_2) \in \Mm$, where
\[ \Mm=\{ (y,w) \in \overline{\dom(\Bb)}\times \H : \langle \Bb^\circ (z) - w,z-y \rangle \le \lambda |z-y|^2 \quad \text{ for every } z \in \dom(\Bb) \},\]
and $x:= (x_1+x_2)/2$. Passing to the limit as $\tau \downarrow 0$ we obtain that $\Mm$ is $\lambda$-dissipative so that, since $\Bb \subset \Mm$, we get that $\Mm=\Bb$.
\end{proof}

For the next result, we recall that a proper functional $\psi:\H \to (-\infty, + \infty]$ is said to be $\lambda$-convex if the map $x \mapsto \psi(x)-\frac{\lambda}{2}|x|^2$ is convex. Its Fr\'echet subdifferential $\partial \psi$ is characterized by
\[ (x,v) \in \partial \psi \quad \Leftrightarrow \quad x \in \dom(\psi)\text{ and } \psi(y)-\psi(x) \ge \la v,y-x \ra + \frac{\lambda}{2}|x-y|^2 \quad \text{ for every } y \in \H.
\]
 In the next corollary, for $0 < \tau < 1/\lambda^+$, we connect the resolvent $\jJ_\tau$ of the (opposite of the) subdifferential $-\partial \psi$ with the Moreau--Yosida regularization of $\psi$, i.e. 
\[
\psi_\tau(x) := \inf_{y \in \H}  \Psi(\tau,x;y), \quad x \in \H,
\]
where
\begin{equation}\label{eq:PsitauMY}
\Psi(\tau,x;y):= \frac{1}{2\tau} |x-y|^2 + \psi(y).
\end{equation}

\begin{corollary}\label{cor:phi}
Let $\psi:\H \to (-\infty, + \infty]$ be a proper, lower semicontinuous and $(-\lambda)$-convex function,  $0 < \tau < 1/\lambda^+$.  Then $-\partial \psi$ is a maximal $\lambda$-dissipative operator. Moreover, denoting by $\Bb:= -\partial \psi$, we have that
\begin{equation*}
    \lim_{\tau \downarrow 0} \frac{\psi(x)-\psi(\jJ_\tau (x))}{\tau} = |\Bb^\circ (x)|^2 \quad \text{ for every } x \in \dom(\Bb),
\end{equation*}
\[ \frac{1}{2\tau}|x-\jJ_\tau (x)|^2 +\psi(\jJ_\tau (x)) <   \frac{1}{2\tau}|x-y|^2 +\psi(y) \quad \text{ for every } x,y \in \H, \, y \ne \jJ_\tau (x).\]
 In particular, $\psi_\tau(x) = \Psi(\tau,x;\jJ_\tau(x))$, for every $x\in\H$. 
\end{corollary}
\begin{proof} Notice that $\psi^\lambda:= \psi + \frac{\lambda}{2} |\cdot|^2$ is convex and that $\partial\psi^\lambda = \partial \psi + \lambda \ii_\cH$ so that by \cite[Example 2.3.4]{BrezisFR} and Remark \ref{rem:transff}, the operator $-\partial \psi^\lambda$ is maximal dissipative and thus $-\partial \psi$ is maximal $ \lambda$-dissipative. By definition of subdifferential of a $(-\lambda)$-convex function, we have that for every $0 < \tau < 1/\lambda^+$ it holds
\begin{align*}
    \psi(x)-\psi(\jJ_\tau (x)) &\ge \la \Bb_\tau (x), \jJ_\tau (x)-x \ra - \frac{\lambda}{2}|\jJ_\tau (x)-x|^2 = \tau |\Bb_\tau (x)|^2-  \frac{\lambda}{2}|\jJ_\tau (x)-x|^2, \\
    \psi(\jJ_\tau (x)) - \psi (x) &\ge \la \Bb^\circ (x) , x-\jJ_\tau (x) \ra - \frac{\lambda}{2}|\jJ_\tau (x)-x|^2 = -\tau \la \Bb^\circ (x) , \Bb_\tau (x) \ra  - \frac{\lambda}{2}|\jJ_\tau (x)-x|^2.
\end{align*}
Dividing the first (resp.~the second) inequality by $\tau>0$ (resp.~$-\tau<0$) and passing to the $\liminf$ (resp.~to the $\limsup$) as $\tau \downarrow 0$, gives the desired equality thanks to Theorem \ref{thm:brezis2}(5). The fact that the limit diverges outside the domain of $\Bb$ follows again by Theorem \ref{thm:brezis2}(5) and the first inequality above. The last assertion follows simply observing that $y\mapsto \Psi(\tau,x;y)$,  defined in \eqref{eq:PsitauMY},  is proper and strictly convex, so that $z$ is a strict minimum point for $\Psi(\tau,x;\cdot)$ if and only if $0 \in \partial \Psi(\tau,x;z)$, which is satisfied if and only if $z=\jJ_\tau (x)$.
\end{proof}

\begin{theorem}\label{thm:brezis3}
    Let $\Bb$ be a maximal $\lambda$-dissipative operator and let $x_0 \in \dom(\Bb)$. There exists a unique locally Lipschitz function $x:[0,+\infty)\to\H$, with $x(0)=x_0$, such that:
    \begin{enumerate}
        \item $x(t) \in \dom(\Bb)$ for every $t>0$;
        \item $\dot{x}(t) \in \Bb \left(x(t)\right)$ for a.e.~$t>0$;
        \item the map $t \mapsto \Bb^\circ \left(x(t)\right)$ is right continuous, $t\mapsto x(t)$ is right differentiable at every $t \ge 0$ and its right derivative at $t$ coincides with $\Bb^\circ (x(t))$ for every $t \ge 0$;
        \item the function $t \mapsto e^{-\lambda t}|\Bb^\circ \left(x(t)\right)|$ is decreasing in $[0,+\infty)$.
    \end{enumerate}
Moreover, if $x,y:[0,+\infty)\to\H$ are solutions of the differential inclusion in (2), then 
\[ |x(t)-y(t)| \le e^{\lambda t} |x(0)-y(0)| \quad \text{ for every } t \ge 0.\]    
\end{theorem}
\begin{proof} The proof of the last assertion is trivial. The proof of the points (1),(2),(3) and (4) is completely analogous to the one of \cite[Theorem 3.1]{BrezisFR} with only few differences that we point out in case $\lambda \ne 0$. In what follows, we take $0<\tau,\eta< 1/\lambda^+$. To prove existence one starts from the approximate problems 
\[ \dot{x}_\tau(t) -\Bb_\tau (x_\tau(t))=0, \quad x_\tau(0)=x, \]
which have unique smooth solutions thanks to e.g.~\cite[Theorem 1.6]{BrezisFR} together with the estimate
\begin{equation}\label{eq:estimate}
|\Bb_\tau \left(x_\tau(t)\right)| = |\dot{x}_\tau(t)| \le e^{\frac{\lambda t}{1-\lambda \tau}} |\Bb_\tau (x_0)| \le \frac{e^{\frac{\lambda t}{1-\lambda \tau}}}{1-\lambda \tau}|\Bb^\circ (x_0)| \quad \text{ for every } t \ge 0,
\end{equation}
still provided by \cite[Theorem 1.6]{BrezisFR} and Theorem \ref{thm:brezis2}(5). Performing the same computations of the proof of \cite[Theorem 3.1]{BrezisFR}, using $\lambda$-dissipativity instead of monotonicity, one obtains 
\[ |x_\tau(t) - x_\eta(t)| \le C(\lambda,t)\,\left|\Bb^\circ (x_0)\right|\,\sqrt{\tau+\eta}\quad \text{ for every }t \ge 0,\]
where $C(\lambda, t)$ is a positive constant that depends in a continuous way only on $\lambda$ and $t$. This proves that $x_\tau$ converges locally uniformly to $x$ on $[0,+\infty)$ with the estimate
\begin{equation}\label{eq:estsemi}
|x_\tau(t) - x(t) | \le C(\lambda,t)\,\left|\Bb^\circ (x_0)\right|\, \sqrt{\tau} \quad \text{ for every } t \ge 0.
\end{equation}
Since 
\[ |\jJ_\tau (x_\tau) - x_\tau | = \tau \left|\Bb_\tau (x_\tau) \right| \le \tau\,\frac{e^{\frac{\lambda t}{1-\lambda \tau}}}{1-\lambda \tau}\left|\Bb^\circ (x_0)\right|,\]
we also get that $\jJ_\tau (x_\tau)$ converges to $x$ locally uniformly in $[0,+\infty)$ and this, together with the estimate \eqref{eq:estimate} and Theorem \ref{thm:brezis2}(1), shows that $x(t) \in \dom(\Bb)$ and $\left|\Bb^\circ (x(t))\right| \le e^{\lambda t}\left|\Bb^\circ (x_0)\right|$ for every $t \ge 0$; in particular this proves (1). Since $|\dot{x}_\tau|$ is uniformly bounded on every interval $[0,T]$ by \eqref{eq:estimate}, it converges weakly$^*$ in $L^{\infty}([0,T]; \H)$ (and thus also weakly in $L^{2}([0,T]; \H)$) to a function $v \in L^{\infty}([0,T]; \H)$ which turns out to be the almost everywhere derivative of $x$ in $[0,T]$ (cf.~\cite[Appendix]{BrezisFR}) so that, applying Theorem \ref{thm:brezis2}(1) to the extension of $\Bb$ to $L^2([0,T];\H)$ (see \cite[Examples 2.1.3, 2.3.3]{BrezisFR} and Remark \ref{rem:transff}), we obtain (2) and also the inequality
\begin{equation}\label{eq:seconde}
|\dot{x}(t)| \le e^{\lambda t} \left|\Bb^\circ (x_0)\right| \quad \text{ for a.e.~} t>0.
\end{equation}
Observing now that, for every $t_0 \ge 0$, $t \mapsto x(t+t_0)$ is a solution of (2) with initial datum $x(t_0)$, we get that $\left|\Bb^\circ \left(x(t+t_0)\right)\right| \le e^{\lambda t} \left|\Bb^\circ \left(x(t_0)\right)\right|$ which proves (4). It remains only to prove (3). The right continuity of $t \mapsto \left|\Bb^\circ \left(x(t)\right)\right|$ follows precisely as in \cite[Theorem 3.1]{BrezisFR}: it is enough to prove it at $t=0$; if $0 < t_n < 1$ is such that $t_n \downarrow 0$, then $\left|\Bb^\circ (x(t_n))\right| \le e^{\lambda_+} \left|\Bb^\circ (x_0)\right|$ by (4), so that, up to a unrelabeled subsequence, $\Bb^\circ (x(t_n))$ converges weakly to some $v \in \H$. Since $x(t_n) \to x_0$ and thanks to Theorem \ref{thm:brezis2}(1), $v$  belongs to $\Bb (x_0)$. However $|v|\le \left|\Bb^\circ (x_0)\right|$ so that it must be $v= \Bb^\circ (x_0)$. The strong convergence follows observing that $\limsup |\Bb^\circ (x(t_n))| \le |v| = \left|\Bb^\circ (x_0)\right|$. Since the limit is independent of the subsequence, we obtain convergence of the whole sequence. We still follow the proof of \cite[Theorem 3.1]{BrezisFR} to prove the right differentiability of $x$ and the inclusion for its right derivative: for every $t_0,h>0$ we have that 
\[ |x(t_0+h)-x(t_0)|= \left | \int_{t_0}^{t_0+h} \dot{x}(s) \de s \right | \le \frac{e^{\lambda h}-1}{\lambda}|\Bb^\circ (x(t_0))|, \]
where we have applied \eqref{eq:seconde} to $t \mapsto x(t+t_0)$. If $t_0$ is a point of differentiability for $x(t)$ such that $\dot{x}(t_0) \in \Bb \left(x(t_0)\right)$, dividing by $h$ and passing to the limit as $h \downarrow 0$ in the above inequality, we get that $|\dot{x}(t_0)| \le \left|\Bb^\circ \left(x(t_0)\right)\right|$ so that $\dot{x}(t_0) =\Bb^\circ \left(x(t_0)\right)$. We can thus integrate this equality in $[t_0, t_0+h]$ for every $t_0 \ge 0$ and every $0<h<1$ to obtain that
\[ \lim_{h \downarrow 0} \frac{x(t_0+h)-x(t_0)}{h} = \lim_{h \downarrow 0} \int_{0}^{1} \Bb^\circ \left(x(t_0+sh)\right) \de s = \Bb^\circ \left(x(t_0)\right),\]
where we used the right continuity of $t \mapsto \Bb^\circ (x(t))$ and the dominated convergence theorem that we can apply since $\left|\Bb^\circ \left(x(t_0+rh)\right)\right| \le e^{\lambda_+} \left|\Bb^\circ \left(x(t_0)\right)\right|$ by (4). This concludes the proof of (3).
\end{proof}

\begin{theorem}\label{thm:brezis4}
    If $\Bb$ is maximal $\lambda$-dissipative, there exists a semigroup of Lipschitz transformations $\Sgp_t: \overline{\dom(\Bb)} \to \overline{\dom(\Bb)}$ such that, for every $x \in \dom(\Bb)$, the curve $t\mapsto x(t):=\Sgp_t (x)$ is the unique solution of the differential inclusion $\dot{x}(t) \in \Bb \left(x(t)\right)$, for a.e.~$t>0$, starting from $x$. Moreover, we have
\begin{equation}\label{smgprop}
    |\Sgp_t (x) - \Sgp_t (y)| \le e^{\lambda t} |x-y| \quad \text{ for every } x,y \in \overline{\dom(\Bb)} \text{ and every } t \ge 0.
\end{equation}
Finally, for every $x \in \overline{\dom(\Bb)}$ we have that
\begin{equation}\label{eq:convJS}
\jJ_{t/n}^n (x) \to \Sgp_t (x)\quad\text{as } n \to + \infty
\end{equation}
and for every $T \ge 0$ there exist $N(\lambda, T) \in \N$, $C(\lambda, T)>0$ (with $C(0,T)=2T$) such that
\begin{equation}\label{eq:estimateforjn}
| \jJ_{t/n}^n (x) -\Sgp_t (x)| \le C(\lambda, T)\frac{|\Bb^\circ (x)|}{\sqrt{n}} \quad \text{ for every } 0 \le t \le T,\, n \ge N(\lambda, T),\, x \in \dom(\Bb).   
\end{equation}
\end{theorem} 
\begin{proof} The first assertion follows by extending by continuity the semigroup (whose existence follows by Theorem \ref{thm:brezis3}) from $\dom(\Bb)$ to the whole $\overline{\dom(\Bb)}$ (see also \cite[Remark 3.2]{BrezisFR}). The second assertion for $\lambda<0$ follows immediately from \cite[Corollaries 4.3, 4.4]{BrezisFR} applied to $-\Bb$. We only prove the second assertion in case $\lambda>0$ following the same strategy of \cite[Corollaries 4.3, 4.4]{BrezisFR}. We fix $x_0 \in \dom(\Bb)$ and we consider as in the proof of Theorem \ref{thm:brezis3} the approximated problems
\[ \dot{x}_\tau(t) - \Bb_\tau \left(x_\tau(t)\right)=0, \quad x_\tau(0)=x_0,  \]
where we are assuming from now on that $0<\tau<1/\lambda$.
By \cite[Theorem 1.7]{BrezisFR} we have that
\begin{align*}
|x_\tau(t) - \jJ_\tau^n(x_0)| &\le (1-\lambda \tau)^{-n} e^{\lambda t} |x_0-\jJ_\tau (x_0)| \left ( \left (n-\frac{t}{\tau (1-\lambda \tau)} \right )^2 + \frac{t}{\tau (1-\lambda\tau)} \right )^{1/2} \\
&\le |\Bb^\circ (x_0)|(1-\lambda \tau)^{-n-1}e^{\lambda t}   \left ( \left (\tau n -\frac{t}{1-\lambda \tau} \right )^2 + \frac{t \tau}{1-\lambda \tau} \right )^{1/2}, 
\end{align*}
where we have also used that $\jJ_\tau$ is $(1-\lambda \tau)^{-1}$-Lipschitz continuous (see Theorem \ref{thm:brezis1}(1)) and Theorem \ref{thm:brezis2}(5). Using this inequality together with \eqref{eq:estsemi} with $\tau=t/n$ we get that for every $T \ge0$ we can find an integer $N(\lambda, T)$ and a positive constant $C(\lambda, T)$ such that
\[ |\jJ_\tau (x_0) - \Sgp_t (x_0)| \le C(\lambda, T) \frac{|\Bb^\circ (x_0)|}{\sqrt{n}}  \quad \text{ for every } n \ge N(\lambda, T) \text{ and every } t\in [0,T].\]
This proves \eqref{eq:estimateforjn} and also the convergence of $\jJ_{t/n}^n (x_0)$ to $\Sgp_t (x_0)$, whenever $x_0 \in \dom(\Bb)$. In case $y_0 \in \overline{\dom(\Bb)}$ and $x_0 \in \dom(\Bb)$ we can estimate
\begin{align*}
|\jJ_{t/n}^n (y_0) - \Sgp_t (y_0)| &\le |\jJ_{t/n}^n (y_0) - \jJ_{t/n}^n (x_0)| + |\Sgp_t (y_0) - \Sgp_t (x_0)| + |\Sgp_t (x_0) - \jJ^n_{t/n}(x_0)| \\
& \le |x_0-y_0|  \left ( (1-\lambda t/n)^{-n} + e^{\lambda t} \right ) + |\Sgp_t (x_0) - \jJ^n_{t/n}(x_0)|,
\end{align*}
where we have used again Theorem \ref{thm:brezis1}(1). Passing to the limit as $n \to + \infty$ gives that
\[ \limsup_{n \to + \infty} |\jJ_{t/n}^n (y_0) - \Sgp_t (y_0)| \le 2e^{\lambda t} |x_0-y_0|\]
ans passing to the $\inf$ w.r.t.~$x_0 \in \dom(\Bb)$ gives the sought convergence.
\end{proof}

The following result corresponds to \cite[Theorem 3.3]{BrezisFR} and concerns the regularizing effect for the semigroup generated by maximal $\lambda$-dissipative operators whose domain has nonempty interior.

\begin{theorem}\label{thm:brezreg} Let $\mmo$ be a maximal $\lambda$-dissipative operator such that $\intt{\dom(\Bb)} \ne \emptyset$ and let $x_0 \in \overline{\dom(\Bb)}$. Then the curve $x(t):=\Sgp_t (x_0)$, $t \ge 0$ (cf.~Theorem \ref{thm:brezis4}) has the following properties:
\begin{enumerate}
\item $x$ is locally absolutely continuous in $[0,+\infty)$ and locally Lipschitz in $(0,+\infty)$;
\item $x(t) \in \dom(\Bb)$ for every $t>0$;
\item there exists a constant $C>0$ (depending solely on $x_0, \lambda$ and $\Bb$) such that
\begin{equation}\label{eq:ilambda}
    I_\lambda (t) |\dot{x}(t)| \le C \quad \text{ for a.e.~$t \in (0,1)$},
\end{equation}
where 
\begin{equation}\label{eq:defIlambda}
I_\lambda (t) := \int_0^t \mathrm e^{\lambda(s-t)} \de s = \begin{cases} \frac{1-\mathrm{e}^{-\lambda t}}{\lambda} \quad &\text{ if } \lambda \ne 0, \\
t \quad &\text{ if } \lambda=0, \end{cases} \quad t \ge 0.
\end{equation}
\end{enumerate}
\end{theorem}
\begin{proof} The proof closely follows the one of \cite[Theorem 3.3]{BrezisFR} and it is divided in several claims.

\noindent\textbf{Claim 1.}
\emph{For every $y \in \intt{\dom(\Bb)}$ there exist $\varrho,M>0$ such that
 \[\varrho|v| \le \la v,y-x \ra + M(|x-y|+ \varrho) + \lambda^+(|x-y|+ \varrho )^2 \quad \text{ for every } (x,v) \in \Bb. \]}
 
Let $y \in \intt{\dom(\Bb)}$ and let $(x,v) \in \Bb$ be fixed. By Theorem \ref{thm:brezis2}(3), there exist $\varrho,M>0$ such that, for every $z \in \H$ with $|z|=1$ and every $w \in \mmo (y-\varrho z)$, it holds $|w| \le M$. Testing the $\lambda$-dissipativity of $\mmo$ with $(x,v),(y-\varrho z,w) \in \mmo$, we get
\[ \la v-w,x-y+\varrho z \ra \le \lambda |x-y+\varrho z|^2\]
so that
\begin{align*}
    \varrho \la v,z \ra &\le \la v, y-x\ra +\lambda^+ ( |x-y|^2+2\varrho \la x-y,z \ra + \varrho^2 |z|^2) + M (|x-y|+ \varrho |z|) \\
    & \le \la v, y-x\ra +M (|x-y|+ \varrho) + \lambda^+(|x-y|+\varrho)^2.
\end{align*}
Passing to the supremum in $z \in \H$ with $|z|=1$ proves the claim.\\

We consider, as in the proof of Theorem \ref{thm:brezis3}, the approximated problems
\[ \dot{x}_\tau(t) - \Bb_\tau (x_\tau(t))=0, \quad x_\tau(0)=x_0,  \]
where we are assuming from now on that $0<\tau<1/\lambda^+$.

\smallskip

\noindent\textbf{Claim 2.}
\emph{For every $T>0$, the curves $ x_\tau$ and $ \jJ_\tau (x_\tau)$ converge to $t \mapsto \Sgp_t (x_0)$ uniformly in $[0,T]$ as $\tau \downarrow 0$.}
\smallskip

Let us first show that $x_\tau$ converges to $t \mapsto \Sgp_t (x_0)$ uniformly in $[0,T]$: let us denote by $(\Sgp^\tau_t)_{t \ge 0}$ the semigroup associated by Theorem \ref{thm:brezis4} to the maximal $\frac{\lambda}{1-\lambda \tau}$-dissipative operator $\Bb_\tau$ (cf.~Theorem \ref{thm:brezis2}(5)), so that in particular $x_\tau (t)= \Sgp^\tau_t (x_0)$ for every $t \ge 0$. 
For every $y_0 \in \dom(\Bb)$ and $t \in [0,T]$, we estimate
\begin{align*}
|x_\tau(t) - \Sgp_t (x_0)| &\le |\Sgp^\tau_t (x_0) - \Sgp^\tau_t (y_0) | + | \Sgp^\tau_t (y_0) - \Sgp_t (y_0)| + |\Sgp_t (y_0) - \Sgp_t (x_0) | \\
& \le e^{\frac{\lambda}{1-\lambda \tau} t } |x_0-y_0| + C(\lambda, t) |\mmo^\circ (y_0)| \sqrt{\tau} + e^{\lambda t} |x_0-y_0| \\ 
& \le \left ( e^{\frac{\lambda^+}{1-\lambda \tau} T} + e^{\lambda^+ T} \right ) |x_0-y_0| + \sup_{t \in [0,T]} C(\lambda, t) |\mmo^\circ (y_0)|\sqrt{\tau},
\end{align*}
where we have used \eqref{smgprop} for $\mmo$ and $\mmo_\tau$ and \eqref{eq:estsemi}. 
Passing first to $\sup_{t \in [0,T]}$, then to the limit as $\tau \downarrow 0$ and finally to the infimum w.r.t.~$y_0 \in \dom(\Bb)$, gives the sought uniform convergence of $x_\tau$ to $t \mapsto \Sgp_t (x_0)$ in $[0,T]$. The argument for $\jJ_\tau (x_\tau)$ is similar: for every $t \in [0,T]$ and every $y_0 \in \dom(\Bb)$ we estimate
\begin{align*}
    &|\jJ_\tau \left(x_\tau (t)\right) - \Sgp_t (x_0)|\\
    &\le |\jJ_\tau \left(x_\tau (t)\right) - \jJ_\tau \left(\Sgp_t (x_0)\right)| + |\jJ_\tau \left(\Sgp_t (x_0)\right) - \jJ_\tau \left(\Sgp_t (y_0)\right)| + |\jJ_\tau \left(\Sgp_t (y_0)\right) - \Sgp_t (y_0) | + |\Sgp_t (y_0) - \Sgp_t (x_0)| \\ 
    & \le \frac{1}{1-\lambda \tau} |x_\tau(t)-\Sgp_t (x_0)| + \left (\frac{e^{\lambda t}}{1-\lambda \tau} + e^{\lambda t} \right )|x_0-y_0| + \tau |\Bb_\tau \left(\Sgp_t (y_0)\right)| \\
    & \le \frac{1}{1-\lambda \tau} |x_\tau(t)-\Sgp_t (x_0)| + \left (\frac{e^{\lambda t}}{1-\lambda \tau} + e^{\lambda t} \right )|x_0-y_0| + \frac{\tau e^{\lambda t}}{1-\lambda \tau} |\mmo^\circ (y_0)|\\
    & \le \frac{1}{1-\lambda \tau} \sup_{t \in [0,T]}|x_\tau(t)-\Sgp_t (x_0)| + \left (\frac{e^{\lambda^+ T}}{1-\lambda \tau} + e^{\lambda^+ T} \right )|x_0-y_0| + \frac{\tau e^{\lambda^+ T}}{1-\lambda \tau} |\mmo^\circ (y_0)|
\end{align*}
where we have used the $(1-\lambda \tau)^{-1}$-Lipschitzianity of $\jJ_\tau$ coming from Theorem \ref{thm:brezis1}(1), \eqref{smgprop} for $\Bb$, the definition of $\Bb_\tau$, Theorem \ref{thm:brezis2}(5) and Theorem \ref{thm:brezis3}(4) applied to $\Bb$ (notice that this is possible since $y_0 \in \dom(\Bb)$). Passing first to $\sup_{t \in [0,T]}$, then to the limit as $\tau \downarrow 0$ and finally to the infimum w.r.t.~$y_0 \in \dom(\Bb)$, concludes the proof of the claim.

\smallskip

\noindent\textbf{Claim 3.}
\emph{For every $T>0$ there exists a constant $M>0$ (not depending on $\tau$) such that $|\Bb_\tau \left(x_\tau (T)\right)| \le M$ for every $0<\tau<1/\lambda^+$.}
\smallskip

We fix some $y \in \intt{\dom(\Bb)}$ and we apply Claim 1 to $(x,v):= \left(\jJ_\tau \left(x_\tau(t)\right), \Bb_\tau \left(x_\tau (t)\right)\right) \in \Bb$, with $t \in [0,T]$ and $0<\tau<1/\lambda^+$ so that
\[ \varrho\, \left|\mmo_\tau \left(x_\tau (t)\right)\right| \le -\frac{1}{2} \frac{\de}{\de t} |x_\tau(t)-y|^2 + M\varrho + M\left|\jJ_\tau \left(x_\tau (t)\right)-y\right|+ \lambda^+ \left ( \left|\jJ_\tau \left(x_\tau (t)\right)-y\right|+\varrho \right )^2.\]
Integrating in $[0,T]$ and using Theorem \ref{thm:brezis3}(4) applied to $\mmo_\tau$, we get
\begin{align*}
     &\varrho\, \left|\mmo_\tau \left(x_\tau (T)\right)\right| \,I_{\frac{\lambda}{1-\lambda \tau}}(T) \\
     &\le \frac{1}{2} |x_0-y|^2 +M\varrho T + \int_0^T \left [  M\left|\jJ_\tau \left(x_\tau (t)\right)-y\right|+ \lambda^+ \left ( \left|\jJ_\tau \left(x_\tau (t)\right)-y\right|+\varrho \right )^2 \right ] \de t.
\end{align*}
By Claim 2, the right hand side of the previous inequality is uniformly bounded (w.r.t.~$\tau \in (0,1/\lambda^+)$) so that we conclude the proof of the claim.

\smallskip

\noindent\textbf{Claim 4.}
\emph{Proof of items (1), (2) and (3).}
\smallskip

By Claim 3, we have that for every $t>0$, up to an unrelabeled subsequence, $\mmo_\tau \left(x_\tau(t)\right) \weakto v$ for some $v \in \H$. By Claim 2, we have that $\jJ_\tau \left(x_\tau (t)\right) \to \Sgp_t (x_0)$ so that we deduce by Theorem \ref{thm:brezis2}(1) that $\Sgp_t (x_0) \in \dom(\Bb)$; this proves (2). We can then fix some $y \in \intt{\dom(\Bb)}$ and apply Claim 1 to $(x,v):=(x(t), \dot{x}_+(t))$, $t>0$, where $\dot{x}_+(t)$ is the right derivative of $t\mapsto x(t)$ at $t$. Indeed, since $\Sgp_t (x_0) = \Sgp_{t-\delta} \left(\Sgp_\delta (x_0)\right)$ and $\Sgp_\delta (x_0) \in \dom(\Bb)$ for every $0<\delta < t$ by (2), we can apply  Theorem \ref{thm:brezis3}(3) to get that $(x(t), \dot{x}_+(t))\in\mmo$. We then obtain
\[ \varrho |\dot{x}_+(t)| \le -\frac{1}{2} \frac{\de}{\de t} |x(t)-y|^2 + M\varrho + M |x(t)-y|+ \lambda^+( |x(t)-y|+\varrho)^2.\]
Integrating the above inequality in $[s,1]$ for any $0<s<1$, we get
\begin{align*}
    \varrho \int_s^1 |\dot{x}_+(t)| \de t \le \frac{1}{2} |x_0-y|^2 + M\varrho + \int_s^1 \left [ M |x(t)-y|+ \lambda^+( |x(t)-y|+\varrho)^2 \right] \de t.
\end{align*}
Thanks to \eqref{smgprop} and Theorem \ref{thm:brezis3}(4) we have that for every $t \in [s,1]$ it holds
\[ |x(t)-y| \le e^{\lambda t} |x_0-y| + |\Sgp_t (y) - y| \le e^{\lambda^+ } (|x_0-y| + |\mmo^\circ (y)|).\]
This proves that there exists some constant $C>0$ (depending solely on $x_0,\lambda,y,\varrho$ and $M$) such that 
\[ \int_s^1 |\dot{x}_+(t)| \de t \le C \quad \text{ for every } s \in (0,1).\]
Since the constant is independent on $s$, we conclude that $x$ is absolutely continuous in $(0,1)$; using also Theorem \ref{thm:brezis3}, this proves (1).  To prove (3), it is enough to use the above estimate with Theorem \ref{thm:brezis3}(3),(4).
\end{proof}

\begin{corollary}\label{cor:samesgp} Let $\Bb_1$ and $\Bb_2$ be maximal $\lambda$-dissipative operators with $\overline{\dom(\Bb_1)}=\overline{\dom(\Bb_2)}$ and let $\Sgp^1_t$ and $\Sgp^2_t$ be the semigroups of Lipschitz transformations associated to $\Bb_1$ and $\Bb_2$ respectively given by Theorem \ref{thm:brezis4}. If for every $x \in \overline{\dom(\Bb_1)}=\overline{\dom(\Bb_2)}$ there exists $\delta>0$ such that $\Sgp^1_t(x) = \Sgp^2_t(x)$ for every $0\le t<\delta$, then $\Bb_1=\Bb_2$.
\end{corollary}
\begin{proof} This can be proven as in \cite[Theorem 4.1]{BrezisFR}: let $x \in \dom(\Bb_1)$ and let $y \in \dom(\Bb_2)$; by hypotesis, we can find some $\delta>0$ such that $\Sgp_t^1 (x) =\Sgp_t^2 (x)$ and $\Sgp_t^1 (y) =\Sgp_t^2 (y)$ for every $0\le t < \delta$. Thus, for every $0\le t < \delta$, we have
 \begin{align*}
     \la \frac{\Sgp_t (x) -x}{t} - \frac{\Sgp_t (y) -y}{t}, x-y \ra &\le\frac{1}{t} \left|\Sgp_t (x) - \Sgp_t (y)\right| |x-y| - \frac{1}{t}|x-y|^2 \\
     & \le \frac{e^{\lambda t}-1}{t}|x-y|^2,
 \end{align*}
 where we have used that $\Sgp_t:=\Sgp^1_t=\Sgp^2_t$ is $e^{\lambda t}$-Lipschitz by \eqref{smgprop}.
 Passing to the limit as $t \downarrow 0$ and using Theorem \ref{thm:brezis3}(3), we get that
 \[ \la \Bb_1^\circ (x)-\Bb_2^\circ (y), x-y \ra \le \lambda |x-y|^2.\]
 By \eqref{eq:141} we get that $\dom(\Bb_1)=\dom(\Bb_2)$ and thus that $\Bb_1^\circ = \Bb_2^\circ$. By \eqref{eq:143} we thus get that $\Bb_1= \Bb_2$.
\end{proof} 

\subsection{Dissipative operators and closed/finite-dimensional spaces}\label{sec:mmofin}

\begin{proposition}\label{prop:brezz} Let $\mmo$ be a maximal $\lambda$-dissipative operator, let $\cY \subset \H$ be a closed subspace and suppose that $\cY$ is invariant for the resolvent of $\mmo$, i.e.~$\jJ_\tau (x) \in \cY$ for every $x \in \cY$. Then the operator $\mmo_\cY:=\mmo \cap (\cY \times \cY)$ has the following properties:
\begin{enumerate}[$(i)$]
    \item $\mmo_\cY$ is maximal $\lambda$-dissipative in $\cY$;
    \item the resolvent (resp.~the semigroup) of $\mmo$ coincides with the resolvent (resp.~the semigroup) of $\mmo_\cY$ when restricted to $\cY$.
    \item $\dom(\mmo_\cY)=\dom(\mmo) \cap \cY$;
    \item $\overline{\dom(\mmo_\cY)}= \overline{\dom(\mmo)} \cap \cY$;
    \item $(\mmo_\cY)^\circ (x) =\mmo^\circ (x) $ for every $x\in \dom(\mmo_\cY)$.
\end{enumerate}
\begin{proof}    
It is clear that the restriction of $\jJ_\tau$, the resolvent of $\mmo$, to $\cY$ provides the resolvent operator for $\mmo_\cY$ and it is a $(1-\lambda\tau)^{-1}$-Lipschitz map defined on the whole $\cY$: by Theorem \ref{thm:brezis1}(1), $\mmo_\cY$ is maximal $\lambda$-dissipative in $\cY$. This proves $(i)$ and $(ii)$, also using the exponential formula (cf.~Theorem \ref{thm:brezis4}). To prove $(iii)$, it is enough to show the inclusion ``$\supset$'': if $x \in \dom(\mmo) \cap \cY$, then $(\jJ_\tau (x)-x)/\tau \in \cY$ is bounded by Theorem \ref{thm:brezis2}(5) and, by the same result together with $(ii)$, it must be that $x \in \dom(\mmo_\cY)$. The inclusion ``$\subset$'' in $(iv)$ follows by $(iii)$, while the inclusion ``$\supset$'' follows simply noticing that, if $x \in \overline{\dom(\mmo)}\cap \cY$, then $\jJ_\tau (x) \to x$ by Theorem \ref{thm:brezis2}(4) and $\jJ_\tau (x) \in \dom(\mmo) \cap \cY=\dom(\mmo_\cY)$. Assertion $(v)$ follows again by Theorem \ref{thm:brezis2}(5).
\end{proof}   
\end{proposition}

\begin{corollary}\label{cor:vipregobasta} Let $\mmo$ be a maximal $\lambda$-dissipative operator and suppose that $\H$ has finite dimension. Then the conclusions of Theorem \ref{thm:brezreg} hold.
\end{corollary}
\begin{proof} Up to a translation, we can assume that $0 \in \dom(\Bb)$. Let $\cY$ be the subspace generated by $\dom(\Bb)$. Since $\H$ is finite dimensional, then $\cY$ is closed. We can thus apply Proposition \ref{prop:brezz} and obtain that $\mmo_\cY:= \mmo \cap (\cY \times \cY)$ is maximal $\lambda$-dissipative in $\cY$, has the same domain of $\mmo$ and its semigroup coincides with the semigroup generated by $\mmo$. Since $\H$ is finite dimensional, the relative interior of $\conv{\dom(\mmo_\cY)}$ in $\cY$ is nonempty and thus we conclude by Theorem \ref{thm:brezis2}(3) that the relative interior of $\dom(\mmo_\cY)$ in $\cY$ is nonempty, so that we can apply Theorem \ref{thm:brezreg} to $\mmo_\cY$ and obtain the conclusion of such theorem for the semigroup generated by $\mmo$.
\end{proof}

\subsection{Extensions of dissipative operators}\label{sec:mmoext}

The following proposition is a slight generalization of \cite[Lemma 2.3]{attouch} but we report its proof for the reader's convenience.

\begin{proposition} \label{prop:robert} Let $\Bb \subset \H \times \H$ be maximal $\lambda$-dissipative and let $\Gg \subset \Bb$ be s.t.~$\dom(\Gg)$ is dense in $\dom(\Bb)$. Then for every $x \in \intt{\dom(\Bb)}$ it holds
\begin{equation}
    \Bb (x) = \clconv{ \left \{ v \in \H \mid \exists (x_n,v_n)_{n\in\N} \subset\Gg \text{ s.t. } x_n \to x, \, v_n \weakto v \right \}}.
\end{equation}
\end{proposition}
\begin{proof} Let $x \in \intt{\dom(\Bb)}$ and let us define
\[ \Mm (x):= \clconv{ \left \{ v \in \H \mid \exists (x_n,v_n)_{n\in\N} \subset \Gg \text{ s.t. } x_n \to x, \, v_n \weakto v \right \}}.\]
If $(x_n, v_n)_{n\in\N} \subset \Gg \subset \Bb$ with $x_n \to x$ and $v_n \weakto v$, by $\lambda$-dissipativity of $\Bb$, we have that
\[ \scalprod{v_n-w}{x_n-y} \le  \lambda|x_n-y|^2 \quad \forall \, (y,w) \in \Bb.\]
Passing to the limit we get
\[ \scalprod{v-w}{x-y} \le \lambda|x-y|^2 \quad \forall \, (y,w) \in \Bb,\]
so that $v \in \Bb (x)$ by \eqref{eq:141}. This, together with the closure and convexity of $\Bb (x)$ given by Theorem \ref{thm:brezis2}(2), proves that $\Mm (x) \subset \Bb (x)$. Let us prove the other inclusion by contradiction: suppose that there is some $v \in \Bb (x)$ s.t.~$v \notin \Mm (x)$. The sets $\{v\}$ and $\Mm (x)$ are disjoint, closed, convex and $\{v\}$ is also compact. By Hahn-Banach's theorem we can find some $z \in \H$ with $|z|=1$ s.t.
\begin{equation} \label{eq:contr}
    \scalprod{v}{z} > \scalprod{u}{z} \quad \forall \, u \in \Mm (x).
\end{equation}
Since $x \in \intt{\dom(\Bb)}$, if we define $z_n:= x+ z/n$, we have that $z_n \in \intt{\dom(\Bb)}$ for $n$ sufficiently large. We can thus find $x_n \in \dom(\Gg)$ s.t.~$|x_n-z_n| < n^{-2}$. Clearly $x_n \to x$ and it is easy to check that $(x_n-x)/|x_n-x| \to z$. Since $x_n \in \dom(\Gg)$, we can find $v_n \in \Gg (x_n)$. Since $\Bb$ is maximal, it is locally bounded (cf.~Theorem \ref{thm:brezis2}(3)) at $x$. Given that $\Gg \subset \Bb$ and since $x_n \to x$, the sequence $(v_n)_{n\in\N}$ is bounded so that, up to an unrelabeled subsequence, it converges weakly to some point $u \in \H$. By $\lambda$-dissipativity of $\Bb$ we have
\[ \scalprod{v-v_n}{x-x_n} \le  \lambda|x-x_n|^2  \quad \forall \, n \in \N,\]
so that, dividing by $|x_n-x|$ and passing to the limit, we obtain
\[ \scalprod{v-u}{z} \le 0,\]
a contradiction with \eqref{eq:contr} since, obviously, $u \in \Mm (x)$.
\end{proof}

The following proposition is an immediate consequence of \cite[Theorem 1]{Qi} and Remark \ref{rem:transff}.
\begin{proposition}\label{prop:qi} Let $\Bb \subset \H \times \H$ be $\lambda$-dissipative with open non empty convex domain. Then there exists a unique maximal $\lambda$-disipative $\hat{\Bb} \supset \Bb$ with $\dom(\hat{\Bb}) \subset \overline{\dom(\Bb)}$ and it is characterized by
\[ \hat{\Bb} = \left \{ (x,v) \in \overline{\dom(\Bb)} \times \H \mid \scalprod{v-w}{x-y} \le \lambda |x-y|^2 \quad \forall \, (y,w) \in \Bb \right\}.\]
\end{proposition}

As a consequence of Propositions \ref{prop:robert} and \ref{prop:qi} we can prove the following.

\begin{theorem} \label{theo:chefatica} Let $\Bb \subset \H \times \H$ be $\lambda$-dissipative with 
\[ C:= \overline{\dom(\Bb)} \text{ convex}, \quad \intt{\dom(\Bb)} \ne \emptyset.\]
Then there exists a unique maximal $\lambda$-dissipative $\hat{\Bb} \supset \Bb$ with $\dom(\hat{\Bb}) \subset C$ and it is characterized by
\begin{equation}\label{eq:thechar}
 \hat{\Bb} = \left \{ (x,v) \in C \times \H \mid \scalprod{v-w}{x-y} \le \lambda |x-y|^2 \quad \forall \, (y,w) \in \Bb \right\}.
\end{equation}
Moreover, for every $x \in \intt{\dom(\hat{\Bb})}$ it holds
\begin{equation}\label{eq:sectionasconvex}
    \hat{\Bb} (x) = \clconv{ \left \{ v \in \H \mid \exists (x_n,v_n)_{n\in\N} \subset \Bb \text{ s.t. } x_n \to x, \, v_n \weakto v \right \}}.
\end{equation}
Finally
\begin{equation} \label{eq:inclusions}
 \intt{C} = \intt{\dom(\hat{\Bb})} \subset \dom(\hat{\Bb}) \subset \overline{\dom(\hat{\Bb})} = C.
\end{equation}
\end{theorem}
\begin{proof}
Let $\Bb'$ be a $\lambda$-dissipative  maximal extension of $\Bb$ with $\dom(\Bb') \subset C$, whose existence is granted by Theorem \ref{thm:brezis1}(2); by $\lambda$-dissipativity of $\Bb'$ and since $\Bb \subset \Bb'$, then $\Bb' \subset \hat{\Bb}$, where $\hat{\Bb}$ is defined as in \eqref{eq:thechar}. We need to prove the other inclusion.\\
Since $\dom(\Bb) \subset \dom(\Bb') \subset C$, we have that $\overline{\dom(\Bb')}=C$. Moreover, given that $\Bb'$ is maximal $\lambda$-dissipative and since the interior of its domain is nonempty, we have by Theorem \ref{thm:brezis2}(3) that
\[ \intt{\dom(\Bb')} \text{ is convex }, \quad \intt{\dom(\Bb')} = \intt{\overline{\dom(\Bb')}} = \intt{C}.\]
It is then clear that $\Bb_0:= \Bb' \cap (\intt{\dom(\Bb')} \times \H)$ is $\lambda$-dissipative with open and nonempty convex domain so that, by Proposition \ref{prop:qi}, there exists a unique maximal $\lambda$-dissipative  $\Bb'' \supset \Bb_0$ with $\dom(\Bb'') \subset \overline{\dom(\Bb_0)} = \overline{\intt{\dom(\Bb')}}=\overline{\intt{C}} =C$ ($C$ is convex) and it is characterized by
\begin{equation}\label{eq:almostdone} \Bb'' = \left \{ (x,v) \in C \times \H \mid \scalprod{v-w}{x-y} \le \lambda |x-y|^2 \quad \forall \, (y,w) \in \Bb_0 \right\}.
\end{equation}
Since $\Bb' \supset \Bb_0$, $\Bb'$ is maximal $\lambda$-dissipative  and $\dom(\Bb') \subset C$, it must be that $\Bb' = \Bb''$.\\
By \eqref{eq:almostdone}, we need to prove that
\begin{equation}\label{eq:theend} \hat{\Bb} \subset \left \{ (x,v) \in C \times \H \mid \scalprod{v-w}{x-y} \le  \lambda|x-y|^2 \quad \forall \, (y,w) \in \Bb_0 \right\}.
\end{equation}
To this aim we apply Proposition \ref{prop:robert} to the maximal $\lambda$-dissipative $\Bb'$ and its subset $\Bb$ noticing that $\dom(\Bb)$ is dense in $\dom(\Bb')$. In this way, we obtain that
\begin{equation} \label{eq:reachb0}
    \Bb_0 (y) = \clconv{ \overline{\Bb}(y)}, \quad  y \in \dom(\Bb_0),
\end{equation}
where
\[ \overline{\Bb} (y)= \left \{ u \in \H \mid \exists (y_n,u_n)_{n\in\N} \subset \Bb \text{ s.t. } y_n \to y, \, u_n \weakto u \right \}.\]
If $(x,v) \in \hat{\Bb}$ and $(y,w) \in \dom(\Bb_0) \times \H$ is such that $w \in \overline{\Bb} (y)$, we can find a sequence $(y_n, u_n)_{n\in\N} \subset \Bb$ s.t.~$y_n \to y$ and $u_n \weakto w$; then, by the very definition of $\hat{\Bb}$, we have
\[ \scalprod{v-u_n}{x-y_n} \le \lambda |x-y_n|^2 \quad \forall \, n \in \N,\]
so that, passing to the limit, we get
\[ \scalprod{v-w}{x-y} \le\lambda |x-y|^2.\]
This proves that, if $(x,v) \in \hat{\Bb}$, then
\begin{equation} \label{eq:please}
 \scalprod{v-w}{x-y} \le \lambda |x-y|^2  \quad \forall \, w \in \overline{\Bb}(y), \quad \forall \, y \in \dom(\Bb_0).
\end{equation}
Finally, if $(x,v) \in \hat{\Bb}$ and $(y,w) \in \Bb_0$, we can find a sequence $(N_n)_{n\in\N} \subset \N$, numbers $(\alpha_i^n)_{i=1}^{N_n} \subset [0,1]$ and points $(w_i^n)_{i=1}^{N_n} \subset \overline{\Bb}(y)$ s.t.
\[ \sum_{i=1}^{N_n}\alpha_i^n =1 \, \, \forall n \in \N, \quad \lim_{n \to + \infty} \sum_{i=1}^{N_n} \alpha_i^n w_i^n = w.\]
By \eqref{eq:please}
\[ \scalprod{v-w_i^n}{x-y} \le\lambda |x-y|^2  \quad \forall i=1, \dots, N_n, \quad \forall n \in \N, \]
so that, multiplying by $\alpha_i^n$ and summing up w.r.t.~$i$ we obtain
\[ \scalprod{ v- \sum_{i=1}^{N_n} \alpha_i^n w_i^n }{x-y} \le \lambda |x-y|^2 \quad \forall n \in \N.\]
Passing to the limit as $n \to + \infty$, we obtain
\[ \scalprod{v-w}{x-y} \le \lambda |x-y|^2,\]
so that \eqref{eq:theend} holds. Finally notice that \eqref{eq:sectionasconvex} is already stated in \eqref{eq:reachb0} since we just proved that $\Bb'=\Bb''=\hat{\Bb}$.
\end{proof}

As a consequence, we  have the following corollary.

\begin{corollary}\label{cor:maximal} Let $\Bb \subset \H \times \H$ be as in Theorem \ref{theo:chefatica} and let $\Gg: \intt{C} \to \H$ be a single-valued selection of the maximal $\lambda$-dissipative extension $\hat{\Bb}$ of $\Bb$. Then the unique maximal $\lambda$-dissipative extension of $\Gg$ with domain included in $C$, $\hat \Gg$, coincides with $\hat \Bb$ and in particular
\begin{equation}\label{eq:condmax}
(x,v) \in \hat \Bb \Leftrightarrow x \in C, \scalprod{v-\Gg (y)}{x-y}\le \lambda |x-y|^2 \quad  \forall y \in \intt{C}.
\end{equation}
\end{corollary}
\medskip

Let us consider a different situation when
we do not assume that $\dom(\Bb)$ 
contains interior points
but there exists a subset $D$ dense in $\dom(\Bb)$
which is invariant with respect to the resolvent map $\resolvent\tau$, i.e.
\begin{equation}
  \label{eq:147}
  \overline{D}\supset \dom(\Bb)\ 
  \text{and}\quad
  \forall \, x\in D,\  0<\tau<1/\lambda^+  \quad
  \exists \,x_\tau\in D:\,\,
  x_\tau-\tau \Bb (x_\tau)\ni x.
\end{equation}
Since $\Bb$ is $\lambda$-dissipative, the point $x_\tau$
solving the inclusion in
\eqref{eq:147} is unique and defines a map $\resolvent\tau:D\to
D\cap \dom(\Bb)$.
\begin{lemma}
  \label{le:easier}
  Let $\Bb \subset \H \times \H$ be $\lambda$-dissipative
  with $C:=\overline{\dom(\Bb)}$ convex, let us
  assume that $D\subset \H$ satistifies \eqref{eq:147}, and let us
  set $\Bb_0:=\Bb\cap (D\times \H)$. The following hold:
  \begin{enumerate}
  \item $\Bb$ admits a unique maximal $\lambda$-dissipative
    extension $\hat \Bb$ with $\dom(\hat{\Bb}) \subset C$
    characterized by
    \begin{equation}
      \label{eq:145}
      \begin{aligned}
        \hat{\Bb} = \Big\{ &(x,v) \in C \times \H \mid \la v-v_0, x-x_0
        \ra \le \lambda |x-x_0
        |^2 \text{ for every } (x_0,v_0)\in \Bb_0
        \Big \}.
      \end{aligned}
    \end{equation}
    \item If moreover the interior of $\overline D$ contains $C$,
      we have
    \begin{equation}
      \label{eq:152}
      \hat\Bb=
      \Big\{(x,v)\in \H\times \H:
      \exists\,(x_n,v_n)_{n\in\N}\subset \Bb_0:x_n\to x,v_n\to v\quad\text{as }n\to+\infty\Big\}.
    \end{equation}
  \end{enumerate}
\end{lemma}
\begin{proof}

  We first prove item (1).  
  Let $\Bb'$ be any maximal $\lambda$-dissipative extension of $\Bb$ with domain
  included in $C$ (whose existence is granted by Theorem \ref{thm:brezis1}(2)) and let $\jJ_\tau'$ be the resolvent
  associated with $\Bb'$. By
  dissipativity of $\Bb'$ and since $\Bb_0  \subset \Bb \subset \Bb'$, we have that $\Bb' \subset \hat{\Bb}$ defined as in \eqref{eq:145}. We need to prove the other inclusion.\\
  Clearly, the restriction of $\jJ_\tau'$ to
  $D$
  coincides with $\resolvent\tau$; since $ \jJ_\tau'$
  is Lipschitz and $D$ is dense in $C$, it is the unique Lipschitz extension of
  $\resolvent\tau$ to $\overline D \supset C$.
  
  If $(x,v) \in \hat{\Bb}$, \eqref{eq:145}
  and the fact that for every $y\in D$,
  $\frac 1\tau(\jJ_\tau (y)-y)\in \Bb \left(\jJ_\tau (y)\right)$ yield by density that
  \begin{equation}
    \label{eq:149}
    \la v-\tau^{-1} (\jJ_\tau' (y)-y),x-\jJ_\tau' (y)\ra \le 
     \lambda |x- \jJ_\tau' (y)|^2
    \quad
    \forall\, y\in \dom(\Bb'), \quad \forall \,  0<\tau<1/\lambda^+,
  \end{equation}
  and passing to
  the limit as $\tau\downarrow0$ we obtain that
  \begin{equation}
    \label{eq:150}
    \la v-\Bb'^\circ (y),x-y\ra \le  \lambda |x-y|^2 \quad
    \forall\, y\in \dom(\Bb'),
  \end{equation}
  where we also used Theorem \ref{thm:brezis2}(4), (5).
  We can then apply \eqref{eq:143} and conclude that $(x,v)\in
  \Bb'$.
\smallskip

 We prove item (2).
  Since $\overline{\Bb_0}\subset
  \hat\Bb$, it is sufficient to prove the opposite inclusion
  $\hat\Bb\subset \overline{\Bb_0}$.
  Let $(x,v)\in \hat \Bb$, let $0<\tau<1/\lambda^+$ and set $y:=x- \tau v$.
  Clearly $\jJ_\tau'  (y)=x$;
  since $\overline D$ contains a
  neighborhood of every element of $\dom(\hat \Bb)\subset C$,
  for sufficiently small $\tau>0$
  there exists a sequence $(y_n)_{n\in\N} \subset D$ converging to
  $y$ as $n\to+\infty$.
  Setting $x_{n}:=\jJ'_\tau (y_{n})$ and
  $v_{n}:=(x_{n}-y_n)/\tau \in \Bb (x_{n})$, we clearly have
  $\lim_{n\to+\infty} x_{n}=x,$ $\lim_{n\to+\infty}v_{n}=v$.
\end{proof}

\begin{corollary}
  \label{cor:speriamo-che-sia-ultimo}
  Let $\Bb \subset \H \times \H$ be
  maximal $\lambda$-dissipative, let us
  assume that $D\subset \H$ satistifies \eqref{eq:147} and the
  interior of $\overline D$ contains $C:=\overline{\dom(\Bb)}$. The following hold:
  \begin{enumerate}
  \item For every $x\in\dom(\Bb)$ there exists a sequence
    $x_n\in D\cap \dom(\Bb)$ converging to $x$ such that
    $\Bb^\circ (x_n)\to \Bb^\circ (x)$ as $n\to+\infty$.
  \item $\Bb$ can be determined by the restriction of the minimal
    section $\Bb^\circ$ to $D$ i.e.
    \begin{equation}
      \label{eq:145bis}
      \begin{aligned}
        \Bb = \Big\{ &(x,v) \in \overline{\dom(\Bb)} \times \H \mid
        \la v-\Bb^\circ (x_0),x-x_0
        \ra \le \lambda |x-x_0
        |^2 \text{ for every } x_0\in D\cap \dom(\Bb)
        \Big \}.
      \end{aligned}
    \end{equation}
  \end{enumerate}
\end{corollary}
\begin{proof}
  We first prove item (1). Since $\Bb$ is maximal $\lambda$-dissipative, the closure
  of its domain $C$ is convex (see Theorem \ref{thm:brezis2}(4)).
  We can thus apply the second item of the previous Lemma
  \ref{le:easier}
  (in this case $\hat\Bb=\Bb$)
  to find a sequence $(x_n,v_n)_{n\in\N}\subset \Bb\cap (D\times \H)$
  such that
  $x_n\to x$ and $v_n\to \Bb^\circ (x)$.
  Let us first prove that $\Bb^\circ (x_n)\weakto \Bb^\circ (x)$ weakly in
  $\H$ as $n\to +\infty$: extracting an unrelabeled subsequence, since $|\Bb^\circ (x_n)|\le |v_n|$ is bounded,
  we can suppose that there exists an increasing subsequence $\left(n(k)\right)_{k\in\N}$
  and an element $v\in \H$ such that
  $\Bb^\circ \left(x_{n(k)}\right)\weakto v$ as $k\to+\infty$.
  Since the graph of $\Bb$ is strongly-weakly closed (cf.~Theorem \ref{thm:brezis2}(1)), we
  deduce that $(x,v)\in \Bb$ so that $|v|\ge|\Bb^\circ (x)|$.
  On the other hand, the lower semicontinuity of the norm yields
  \[ |\Bb^\circ (x)| \le |v|\le \liminf_{k\to+\infty}\left|\Bb^\circ \left(x_{n(k)}\right)\right| \le \limsup _{k\to+\infty}\left|\Bb^\circ \left(x_{n(k)}\right)\right|
  \le \limsup_{k\to+\infty}|v_{n(k)}|=|\Bb^\circ (x)|.\]

  We deduce that $\Bb^\circ \left(x_{n(k)}\right)\weakto \Bb^\circ (x)$
  and $\lim_{k\to+\infty}\left|\Bb^\circ \left(x_{n(k)}\right)\right|=|\Bb^\circ (x)|$ so that the
  convergence is also strong. Since the starting (unrelabeled) subsequence was arbitrary, we deduce the strong convergence of the whole sequence.

  Item (2) now follows easily by approximation using
  the item (1) and Theorem \ref{thm:brezis2}(6).
\end{proof}

\section{Borel partitions and almost optimal couplings}\label{sec:appborel}
In this appendix we summarize some of the results of \cite{CSS2piccolo} related to standard Borel spaces, Borel partitions and optimal couplings between probability measures that have been used throughout the whole paper. We refer to \cite[Section 3]{CSS2piccolo} for the proofs.

\begin{definition}\label{def:sbs}
A \emph{standard Borel space} $(\Omega, \cB)$ is a measurable space that is isomorphic (as a measurable space) to a Polish space. Equivalently, there exists a Polish topology $\tau$ on $\Omega$ such that the Borel sigma algebra generated by $\tau$ coincides with $\cB$. We say that a  probability   measure $\P$ on $(\Omega, \cB)$ is nonatomic  if $\P(\{\omega\}) = 0$ for every $\omega \in \Omega$ (notice that $\{\omega\} \in \cB$ since it is compact in any Polish topology on $\Omega$).
\end{definition}
If $(\Omega, \cB)$ is a standard Borel space endowed with a nonatomic probability measure $\P$, we denote by $\rmS(\Omega, \cB, \P)$ the class of $\cB$-$\cB$-measurable maps $g:\Omega\to\Omega$ which are essentially injective and measure-preserving, meaning that there exists a full $\P$-measure set $\Omega_0 \in \cB$ such that $g$ is injective on $\Omega_0$ and $g_\sharp \P =\P$. If $\mathcal{A} \subset \cB$ is a sigma algebra on $\Omega$ we denote by $\rmS(\Omega, \cB, \P ; \mathcal{A})$ the subset of $\rmS(\Omega, \cB, \P)$ of $\mathcal{A}-\mathcal{A}$ measurable maps.

We will often use the notation
\[ I_N:=\{0, \dots, N-1\}, \quad N \in \N, \, N \ge 1 \]
while $\symg{I_N}$ denotes the set of permutations of $I_N$
i.e.~bijective maps $\sigma: I_N \to I_N$.
We will consider the partial order on $\N$ given by
 \begin{equation*}
   \text{$m\preccurlyeq
 n\quad\Leftrightarrow\quad m\mid n$}
\end{equation*}
where $m\mid n$ means that $n/m\in \N$. We write $m \prec n$ if $m
\preccurlyeq n$ and $m \ne n$.

This first result shows a correspondence between permutations and measure-preserving isomorphisms.
\begin{lemma}\label{cor:isomor}
Let $(\Omega, \cB)$ be a standard Borel space endowed with a nonatomic probability measure $\P$,
and let $\mathfrak P_N=\{\Omega_{N,k}\}_{k \in I_N} \subset \cB$ be a \emph{$N$-partition} of $(\Omega, \cB)$ for some $N \in \N$, i.e.
\[\bigcup_{k\in I_N}\Omega_{N,k}=\Omega, \quad \Omega_{N,k}\cap \Omega_{N,h}=\emptyset \text{ if }h,k\in I_N,\, h\neq k;\]
assume moreover that $\P(\Omega_{N,k})= \P(\Omega)/N$ for every $k \in I_N$. If $\sigma \in \symg{I_N}$, there exists a measure-preserving isomorphism $g \in \rmS(\Omega, \cB, \P; \sigma(\mathfrak P_N))$ such that 
\[ (g_k)_{\sharp}\P|_{\Omega_{N,k}} = \P|_{\Omega_{N,\sigma(k)}} \quad \forall k \in I_N, \]
where $g_k$ is the restriction of $g$ to $\Omega_{N,k}$.
\end{lemma}

We introduce now the notion of \emph{refined} standard Borel measure space which turns out to be useful when dealing with approximation of general measures with discrete ones.

\begin{definition}\label{def:segm}Let $(\Omega, \cB)$ be a standard Borel space endowed with a nonatomic probability measure $\P$, and let $\cN \subset \N$ be an unbounded directed set w.r.t.~$\preccurlyeq$. We say that a collection of partitions $(\mathfrak P_N)_{N \in \cN}$ of $\Omega$, with corresponding sigma algebras $\cB_N:= \sigma(\mathfrak P_N)$, is a \emph{$\cN$-segmentation of $(\Omega, \cB, \P)$} if 
\begin{enumerate}
\item $\mathfrak P_N= \{\Omega_{N,k}\}_{k \in I_N}$ is a $N$-partition of $(\Omega, \cB)$ for every $N \in \cN$,
\item $\P(\Omega_{N,k})=\P(\Omega)/N$ for every $k\in I_N$ and every $N \in \cN$,
\item if $M\mid N$ and $K:=N/M$ then
  $\bigcup_{k=0}^{K-1}\Omega_{N,mK+k}=\Omega_{M,m}$, $m\in
  I_{M}$,
\item $\sigma \left ( \left \{ \cB_N \mid N \in \cN \right \} \right) = \cB$.
\end{enumerate} 
In this case we call $(\Omega, \cB, \P , (\mathfrak P_N)_{N \in \cN})$ a \emph{$\cN$-refined standard Borel probability  space}.
\end{definition}

\begin{proposition}\label{prop:strook} For any standard Borel space $(\Omega, \cB)$ endowed with a nonatomic probability measure $\P$ and any unbounded directed set $\cN \subset \N$ w.r.t.~$\preccurlyeq$, there exists a $\cN$-segmentation of $(\Omega, \cB,\P)$. If $\cN \subset \N$ is an unbounded directed subset w.r.t.~$\preccurlyeq$, then there exists a totally ordered diverging sequence $(b_n)_{n\in\N} \subset \cN$ satisfying
\begin{itemize}
    \item  $b_{n} \prec b_{n+1}$ for every $n \in \N$,
    \item for every $N\in \N$ there exists $n\in \N$ such that $N\mid b_n.$
\end{itemize}
In particular,
for every $\cN$-refined standard Borel measure space $(\Omega, \cB, \mm, (\mathfrak P_N)_{N \in \cN})$ it holds that
$(\cB_{b_n})_{n \in \N}$ is a filtration on $(\Omega, \cB)$,
\begin{equation}\label{eq:goodseq}
\text{for every $N \in \cN$ there exists $n \in \N$ such that } \cB_N \subset \cB_{b_n},
\end{equation}
and 
$\sigma \left ( \left \{ \cB_{b_n} \mid n \in \N \right \} \right ) = \cB$.

For every every separable Hilbert space $\X$, we 
thus have that 
\begin{equation}
    \bigcup_{N \in \cN} L^{2}(\Omega, \cB_N, \mm; \X) \text{ is dense in } L^{2}(\Omega, \cB, \mm; \X).
\end{equation}
\end{proposition}

The next theorem contains approximation results for couplings by means of maps in different situations.

\begin{theorem}\label{thm:gpfinal} Let $(\Omega, \cB, \P, (\mathfrak P_N)_{N \in \cN})$ be a $\cN$-refined standard Borel probability space. Then:
\begin{enumerate}
    \item For every $\ggamma \in \Gamma(\P, \P)$ there exist a totally ordered strictly increasing sequence $(N_n)_{n\in\N} \subset \cN$ and maps $g_n \in \rmS(\Omega, \cB, \P; \cB_{N_n})$ such that, for every separable Hilbert space $\X$ and every $X,Y \in L^2(\Omega, \cB, \P; \X)$ it holds
\begin{equation}\label{eq:gangboconv}
  (X,Y)_\sharp (\ii_{\Omega}, g_n)_\sharp \P \to
 (X\otimes Y)_\sharp \ggamma
   \text{ in } \prob_2(\X^2).
\end{equation}
\item If $\X$ is a separable Hilbert space and $X,X'\in L^2(\Omega, \cB, \P; \X)$, then for every $\bm\mu\in \Gamma(X_\sharp \P,X'_\sharp \P)$ there exist a totally ordered strictly increasing sequence $(N_n)_n{n\in\N}\subset \cN$ and maps $g_n \in \rmS(\Omega, \cB, \P; \cB_{N_n})$ such that
\begin{equation}\label{eq:gangboconv2}
(X,X'\circ g_n)_\sharp
\P \to \bm \mu \text{ in } \prob_2(\X^2).
\end{equation} 
In particular, if $X_\sharp \P = X'_\sharp \P$, there exist a totally ordered strictly increasing sequence $(N_n)_{n\in\N} \subset \cN$ and maps $g_n \in \rmS(\Omega, \cB, \P; \cB_{N_n})$ such that
$X'\circ g_n\to X$ in $L^2(\Omega, \cB, \P; \X)$ as $n\to+\infty$.
\end{enumerate}
Finally, if $(\Omega, \cB)$ is a standard Borel space endowed with a nonatomic probability measure $\P$, $\X$ is a separable Hilbert space, $\mu, \nu \in \prob_2(\X)$ and $X \in L^2(\Omega, \cB,\P; \X)$ is s.t.~$X_\sharp \P = \mu$, then, for every $\eps>0$, there exists $Y \in L^2(\Omega, \cB,\P; \X)$ s.t. $Y_\sharp \P = \nu$ and
\[ |X-Y|_{L^2(\Omega, \cB,\P; \X)} \le W_2(\mu, \nu) + \eps.\]
\end{theorem}

Before stating the next result, we fix a $\cN$-refined standard Borel probability space $(\Omega, \cB, \P, (\mathfrak P_N)_{N \in \cN})$ and we set
\[
\cH_N:= L^2(\Omega, \cB_N, \P; \X), \quad N \in \cN, \quad \cH_\infty:= \bigcup_{N \in \cN} \cH_N. \]
We show that a sufficient condition for a a set $\mathcal A\subset
\cH_\infty$  to be law invariant according to Definition \ref{def:inv} is that its sections $\mathcal A\cap \cH_N$
are invariant by the action of $\Sym{I_N}$, meaning that, for every $N \in \cN$ and $g \in \rmS(\Omega, \cB, \P ; \mathcal{B}_N)$, it holds
\[
X \in \mathcal A\cap \cH_N \Rightarrow X \circ g \in \mathcal A\cap \cH_N.
\]
 
\begin{lemma}
  \label{le:general-invariance}
  Let $\mathcal A\subset \cH_\infty$ be a set such that 
  $\mathcal A_N:=\mathcal A\cap \cH_N$ are invariant
  w.r.t.~$\Sym{I_N}$ for every $N\in \cN$. Then 
  $\overline{\mathcal A}$
  is law invariant.
\end{lemma}
\begin{remark}
  \label{rem:trivial}
  The same statement applies to subsets of $\cH_\infty\times \cH_\infty$.
\end{remark}
\begin{proof}
  Since $\overline{\mathcal A}$ is a closed set, by Lemma \ref{le:noncera},
  it is sufficient to prove that
  it is invariant by measure-preserving isomorphisms:
  for every $X\in \overline{\mathcal
  A}$ and $g\in \mathrm  \rmS(\Omega, \cB, \P) $ we want to show that $X\circ g\in \overline{\mathcal A}$.  
  It is enough to prove that there exist
  $(Z_n)_{n\in\N}\subset \mathcal A$ s.t. $Z_n \to X \circ g$. Let $(X_n)_{n\in\N}$ be a sequence in $\mathcal A$ such that $X_n \to X$; since $\mathcal A \subset \cH_\infty$, for every $n \in \N$, there exists some $N_n \in \cN$ such that $X_n \in \mathcal A_{N_n}$. Let $(b_k)_{k\in\N} \subset \cN$ be the sequence given by Proposition \ref{prop:strook}; by  Theorem \ref{thm:gpfinal}(1) applied to $(\Omega, \cB, \P, (\mathfrak P_{b_k})_{k \in \N})$ and $\ggamma:=(\ii_{\Omega}, g)_\sharp \P$, we can find a strictly increasing sequence $(M_j )_{j\in\N}\subset \N$ and maps $g_j \in  \rmS(\Omega, \cB, \P ; \mathcal{B}_{b_{M_j}})$ such that 
  \[ (U,W)_\sharp (\ii_{\Omega}, g_j)_\sharp \P \to (U,W)_\sharp (\ii_{\Omega}, g)_\sharp \P \, \text{ in } \prob_2(\X^2)\]
 for every $U,W \in \cH$. Since $(M_j )_{j\in\N}$ is strictly increasing and \eqref{eq:goodseq} holds, then we can find a strictly increasing sequence $(j(n))_{n\in\N}$ such that $g_{j(n)} \in  \rmS(\Omega, \cB, \P ; \mathcal{B}_{N_n})  $. Thus setting $g'_n := g_{j(n)}$, $n \in \N$, by the invariance of $\mathcal A_{N_n}$, we get that $Z_n:= X_n \circ g'_n \in \mathcal{A}_{N_n} \subset \mathcal A$ and of course we have
\begin{equation}\label{eq:theoneiwant} (U,W)_\sharp (\ii_{\Omega}, g'_n)_\sharp \P \to (U,W)_\sharp (\ii_{\Omega}, g)_\sharp \P  \,\text{ in } \prob_2(\X^2)
\end{equation}
 for every $U,W \in \cH$.
We are left with showing that 
\begin{equation}\label{eq:ord1}
Z_n \to X \circ g \text{ in }\cH.
\end{equation}
Since $|Z_n - X \circ g'_n|_{\cH}= |X_n-X|_{\cH}$, in order to get \eqref{eq:ord1}
it is enough to show that $X\circ  g'_n \to X \circ g$ which, on the
other hand, is implied by $X\circ  g'_n \weakto X \circ g$, since $|X\circ  g'_n |_{\cH} = |X|_{\cH} = |X \circ g|_{\cH}$.
Let $Y \in \cH$ and let us take $U=Y, W=X$ in \eqref{eq:theoneiwant} so that 
\[ \scalprod{X\circ  g'_n}{Y}_{\cH}= \int_{\X^2}\scalprod{x}{y} \de ((Y,X) \circ (\ii_\Omega, g'_n))_{\sharp} \P \to \int_{\X^2}\scalprod{x}{y} \de ((Y,X) \circ (\ii_\Omega, g))_{\sharp} \P = \scalprod{X \circ g}{Y}_{\cH}, 
\]
since $\varphi(x,y) := \scalprod{x}{y}$ is a real valued function on
$\X^2$ with less than quadratic growth (see e.g.~\cite[Proposition
7.1.5, Lemma 5.1.7]{ags}). This shows that $X\circ  g'_n \weakto X
\circ g$ as desired, thus \eqref{eq:ord1} and so $X \circ g\in \overline{\mathcal A}$.
\end{proof}


\begin{thebibliography}{10}

\bibitem{AFMS21}
L.~Ambrosio, M.~Fornasier, M.~Morandotti, and G.~Savaré.
\newblock Spatially inhomogeneous evolutionary games.
\newblock {\em Communications on Pure and Applied Mathematics},
  74(7):1353--1402, 2021.

\bibitem{ags}
L.~Ambrosio, N.~Gigli, and G.~Savar{\'e}.
\newblock {\em Gradient Flows In Metric Spaces and in the Space of Probability
  Measures}.
\newblock Lectures in Mathematics. ETH Z{\"u}rich. Birkh{\"a}user Basel, 2008.

\bibitem{attouch}
H.~{Attouch}.
\newblock {Familles d'op\'erateurs maximaux monotones et mesurabilite}.
\newblock {\em {Ann. Mat. Pura Appl. (4)}}, 120:35--111, 1979.

\bibitem{AOZ}
{Aussedat, Averil}, {Jerhaoui, Othmane}, and {Zidani, Hasnaa}.
\newblock Viscosity solutions of centralized control problems in measure
  spaces.
\newblock {\em ESAIM: COCV}, 30:91, 2024.

\bibitem{AK-pmp}
Y.~Averboukh and D.~Khlopin.
\newblock Pontryagin maximum principle for the deterministic mean field type
  optimal control problem via the lagrangian approach.
\newblock {\em Journal of Differential Equations}, 430:113205, 2025.

\bibitem{AMQ2021}
Y.~Averboukh, A.~Marigonda, and M.~Quincampoix.
\newblock Extremal shift rule and viability property for mean field-type
  control systems.
\newblock {\em J. Optim. Theory Appl.}, 189(1):244--270, 2021.

\bibitem{Averb2022}
Y.~V. Averboukh.
\newblock A mean field type differential inclusion with upper semicontinuous
  right-hand side.
\newblock {\em Vestn. Udmurt. Univ. Mat. Mekh. Komp'yut. Nauki},
  32(4):489--501, 2022.

\bibitem{BadF-hjb}
Z.~Badreddine and H.~Frankowska.
\newblock Solutions to {H}amilton-{J}acobi equation on a {W}asserstein space.
\newblock {\em Calc. Var. Partial Differential Equations}, 61(1):Paper No. 9,
  41, 2022.

\bibitem{BadF-viability}
Z.~Badreddine and H.~Frankowska.
\newblock Viability and invariance of systems on metric spaces.
\newblock {\em Nonlinear Analysis}, 225:113--133, 2022.

\bibitem{BauWang2009}
H.~H. Bauschke and X.~Wang.
\newblock The kernel average for two convex functions and its application to
  the extension and representation of monotone operators.
\newblock {\em Trans. Amer. Math. Soc.}, 361(11):5947--5965, 2009.

\bibitem{BauWang2010}
H.~H. Bauschke and X.~Wang.
\newblock Firmly nonexpansive and {K}irszbraun-{V}alentine extensions: a
  constructive approach via monotone operator theory.
\newblock In {\em Nonlinear analysis and optimization {I}. {N}onlinear
  analysis}, volume 513 of {\em Contemp. Math.}, pages 55--64. Amer. Math.
  Soc., Providence, RI, 2010.

\bibitem{Benilan}
P.~B\'{e}nilan.
\newblock Solutions int\'{e}grales d'\'{e}quations d'\'{e}volution dans un
  espace de {B}anach.
\newblock {\em C. R. Acad. Sci. Paris S\'{e}r. A-B}, 274:A47--A50, 1972.

\bibitem{bonnet2020mean}
B.~Bonnet and H.~Frankowska.
\newblock Differential inclusions in {W}asserstein spaces: the
  {C}auchy-{L}ipschitz framework.
\newblock {\em J. Differential Equations}, 271:594--637, 2021.

\bibitem{BF-pmp}
B.~Bonnet and H.~Frankowska.
\newblock Necessary optimality conditions for optimal control problems in
  {W}asserstein spaces.
\newblock {\em Appl. Math. Optim.}, 84(suppl. 2):S1281--S1330, 2021.

\bibitem{BF-viability}
B.~Bonnet and H.~Frankowska.
\newblock Viability and exponentially stable trajectories for differential
  inclusions in wasserstein spaces.
\newblock In {\em 2022 IEEE 61st Conference on Decision and Control (CDC)},
  pages 5086--5091, 2022.

\bibitem{BF2023}
B.~Bonnet-Weill and H.~Frankowska.
\newblock Carath\'eodory theory and a priori estimates for continuity
  inclusions in the space of probability measures.
\newblock {\em Nonlinear Anal.}, 247:Paper No. 113595, 32, 2024.

\bibitem{BrezisFR}
H.~Br\'ezis.
\newblock {\em Op\'erateurs maximaux monotones et semi-groupes de contractions
  dans les espaces de {H}ilbert}, volume No. 5 of {\em North-Holland
  Mathematics Studies}.
\newblock North-Holland Publishing Co., Amsterdam-London; American Elsevier
  Publishing Co., Inc., New York, 1973.
\newblock Notas de Matem\'atica, No. 50. [Mathematical Notes].

\bibitem{Breziss}
H.~Brezis.
\newblock {\em Functional Analysis, Sobolev Spaces and Partial Differential
  Equations}.
\newblock Universitext. Springer New York, 2010.

\bibitem{CCR11}
J.~A. Ca\~{n}izo, J.~A. Carrillo, and J.~Rosado.
\newblock A well-posedness theory in measures for some kinetic models of
  collective motion.
\newblock {\em Math. Models Methods Appl. Sci.}, 21(3):515--539, 2011.

\bibitem{Camilli_MDE}
F.~Camilli, G.~Cavagnari, R.~De~Maio, and B.~Piccoli.
\newblock Superposition principle and schemes for measure differential
  equations.
\newblock {\em Kinet. Relat. Models}, 14(1):89--113, 2021.

\bibitem{CapuaniM-lift}
R.~Capuani and A.~Marigonda.
\newblock Constrained mean field games equilibria as fixed point of random
  lifting of set-valued maps.
\newblock {\em IFAC-PapersOnLine}, 55(30):180--185, 2022.
\newblock 25th International Symposium on Mathematical Theory of Networks and
  Systems MTNS 2022.

\bibitem{carda}
P.~Cardaliaguet.
\newblock Notes on mean field games.
\newblock From P.-L. Lions’ lectures at College de France,
  \href{https://www.ceremade.dauphine.fr/~cardaliaguet/MFG20130420.pdf}{link to
  the notes.}

\bibitem{CD18}
R.~Carmona and F.~Delarue.
\newblock {\em Probabilistic theory of mean field games with applications.
  {I}}, volume~83 of {\em Probability Theory and Stochastic Modelling}.
\newblock Springer, Cham, 2018.
\newblock Mean field FBSDEs, control, and games.

\bibitem{CLOS}
G.~Cavagnari, S.~Lisini, C.~Orrieri, and G.~Savar\'{e}.
\newblock Lagrangian, {E}ulerian and {K}antorovich formulations of multi-agent
  optimal control problems: equivalence and gamma-convergence.
\newblock {\em J. Differential Equations}, 322:268--364, 2022.

\bibitem{CMPrainbow}
G.~Cavagnari, A.~Marigonda, and B.~Piccoli.
\newblock Generalized dynamic programming principle and sparse mean-field
  control problems.
\newblock {\em J. Math. Anal. Appl.}, 481(1):123--137, 45, 2020.

\bibitem{CMQ}
G.~Cavagnari, A.~Marigonda, and M.~Quincampoix.
\newblock Compatibility of state constraints and dynamics for multiagent
  control systems.
\newblock {\em J. Evol. Equ.}, 21(4):4491--4537, 2021.

\bibitem{CSS}
G.~Cavagnari, G.~Savar\'{e}, and G.~E. Sodini.
\newblock Dissipative probability vector fields and generation of evolution
  semigroups in {W}asserstein spaces.
\newblock {\em Probab. Theory Related Fields}, 185(3-4):1087----1182, 2023.

\bibitem{CSS2piccolo}
G.~Cavagnari, G.~Savar\'e, and G.~E. Sodini.
\newblock Extension of monotone operators and {L}ipschitz maps invariant for a
  group of isometries.
\newblock {\em Canad. J. Math.}, 77(1):149--186, 2025.

\bibitem{CSSnew}
G.~Cavagnari, G.~Savaré, and G.~E. Sodini.
\newblock Stochastic euler schemes and dissipative evolutions in the space of
  probability measures, 2025.
\newblock arxiv.org/abs/2505.20801.

\bibitem{Chizat}
L.~Chizat and F.~Bach.
\newblock On the global convergence of gradient descent for over-parameterized
  models using optimal transport.
\newblock {\em Advances in neural information processing systems}, pages
  3036--3056, 2018.

\bibitem{DCBC06}
M.~R. D'Orsogna, Y.-L. Chuang, A.~L. Bertozzi, and L.~Chayes.
\newblock Self-propelled particles with soft-core interactions: Patterns,
  stability, and collapse.
\newblock {\em Phys. Rev. Lett.}, 96:104302–1/4, 2006.

\bibitem{FLOS}
M.~Fornasier, S.~Lisini, C.~Orrieri, and G.~Savar\'{e}.
\newblock Mean-field optimal control as gamma-limit of finite agent controls.
\newblock {\em European J. Appl. Math.}, 30(6):1153--1186, 2019.

\bibitem{FSS22}
M.~Fornasier, G.~Savar\'e, and G.~E. Sodini.
\newblock Density of subalgebras of {L}ipschitz functions in metric {S}obolev
  spaces and applications to {W}asserstein {S}obolev spaces.
\newblock {\em J. Funct. Anal.}, 285(11):Paper No. 110153, 76, 2023.

\bibitem{gangbotudo}
W.~Gangbo and A.~Tudorascu.
\newblock On differentiability in the {W}asserstein space and well-posedness
  for {H}amilton-{J}acobi equations.
\newblock {\em J. Math. Pures Appl. (9)}, 125:119--174, 2019.

\bibitem{Jimenez}
C.~Jimenez.
\newblock Equivalence between strict viscosity solution and viscosity solution
  in the {W}asserstein space and regular extension of the {H}amiltonian in
  {$L^2_{\Bbb P}$}.
\newblock {\em J. Convex Anal.}, 31(2):619--670, 2024.

\bibitem{JMQ}
C.~Jimenez, A.~Marigonda, and M.~Quincampoix.
\newblock Optimal control of multiagent systems in the {W}asserstein space.
\newblock {\em Calc. Var. Partial Differential Equations}, 59(2):Paper No. 58,
  45, 2020.

\bibitem{JKO98}
R.~Jordan, D.~Kinderlehrer, and F.~Otto.
\newblock The variational formulation of the {F}okker-{P}lanck equation.
\newblock {\em SIAM J. Math. Anal.}, 29(1):1--17, 1998.

\bibitem{JKO}
R.~Jordan, D.~Kinderlehrer, and F.~Otto.
\newblock The variational formulation of the {F}okker-{P}lanck equation.
\newblock {\em SIAM J. Math. Anal.}, 29(1):1--17, 1998.

\bibitem{Lions}
P.-L. Lions.
\newblock Th\'eorie des jeux \`a champs moyen et applications.
\newblock 2007.

\bibitem{McCann97}
R.~J. McCann.
\newblock A convexity principle for interacting gases.
\newblock {\em Adv. Math.}, 128(1):153--179, 1997.

\bibitem{NaldiSavare}
E.~Naldi and G.~Savar\'{e}.
\newblock Weak topology and {O}pial property in {W}asserstein spaces, with
  applications to gradient flows and proximal point algorithms of geodesically
  convex functionals.
\newblock {\em Atti Accad. Naz. Lincei Rend. Lincei Mat. Appl.},
  32(4):725--750, 2021.

\bibitem{Natile-Savare09}
L.~Natile and G.~Savar\'{e}.
\newblock A {W}asserstein approach to the one-dimensional sticky particle
  system.
\newblock {\em SIAM J. Math. Anal.}, 41(4):1340--1365, 2009.

\bibitem{Nochetto-Savare06}
R.~H. Nochetto and G.~Savar\'{e}.
\newblock Nonlinear evolution governed by accretive operators in {B}anach
  spaces: error control and applications.
\newblock {\em Math. Models Methods Appl. Sci.}, 16(3):439--477, 2006.

\bibitem{NSV00}
R.~H. Nochetto, G.~Savar\'{e}, and C.~Verdi.
\newblock A posteriori error estimates for variable time-step discretizations
  of nonlinear evolution equations.
\newblock {\em Comm. Pure Appl. Math.}, 53(5):525--589, 2000.

\bibitem{Otto01}
F.~Otto.
\newblock The geometry of dissipative evolution equations: the porous medium
  equation.
\newblock {\em Comm. Partial Differential Equations}, 26(1-2):101--174, 2001.

\bibitem{Parker}
G.~Parker.
\newblock Some convexity criteria for differentiable functions on the
  2-{W}asserstein space.
\newblock {\em Bull. Lond. Math. Soc.}, 56(5):1839--1858, 2024.

\bibitem{Piccoli_MDI}
B.~Piccoli.
\newblock Measure differential inclusions.
\newblock {\em 2018 IEEE Conference on Decision and Control (CDC)}, pages
  1323--1328, 2018.

\bibitem{Piccoli_2019}
B.~Piccoli.
\newblock Measure differential equations.
\newblock {\em Arch. Ration. Mech. Anal.}, 233(3):1289--1317, 2019.

\bibitem{Piccoli-soa23}
B.~Piccoli.
\newblock Control of multi-agent systems: results, open problems, and
  applications.
\newblock {\em Open Math.}, 21(1):Paper No. 20220585, 26, 2023.

\bibitem{pogodaev2016}
N.~Pogodaev.
\newblock Optimal control of continuity equations.
\newblock {\em NoDEA Nonlinear Differential Equations Appl.}, 23(2):Art. 21,
  24, 2016.

\bibitem{Qi}
L.~{Qi}.
\newblock {Uniqueness of the maximal extension of a monotone operator}.
\newblock {\em {Nonlinear Anal., Theory Methods Appl.}}, 7:325--332, 1983.

\bibitem{santambrogio}
F.~Santambrogio.
\newblock {\em Optimal transport for applied mathematicians}, volume~87 of {\em
  Progress in Nonlinear Differential Equations and their Applications}.
\newblock Birkh\"{a}user/Springer, Cham.
\newblock Calculus of variations, PDEs, and modeling.

\bibitem{Sznitman91}
A.-S. Sznitman.
\newblock Topics in propagation of chaos.
\newblock In {\em \'{E}cole d'\'{E}t\'{e} de {P}robabilit\'{e}s de
  {S}aint-{F}lour {XIX}---1989}, volume 1464 of {\em Lecture Notes in Math.},
  pages 165--251. Springer, Berlin, 1991.

\end{thebibliography}
\end{document}